\documentclass[english,reqno]{amsart}

\usepackage{amsmath, appendix, ulem}
\usepackage[nobysame]{amsrefs}
\usepackage{amssymb, color}
\usepackage[margin=1in]{geometry}

\usepackage{mathrsfs}
\usepackage{graphicx}
\usepackage{subfig}
\usepackage{float}
\usepackage{epsf}
\usepackage{hyperref}
\usepackage{titletoc}

\graphicspath{{../Figures/}}

\numberwithin{equation}{section}
\newtheorem{theorem}{Theorem}[section]
\newtheorem{proposition}{Proposition}[section]
\newtheorem{lemma}{Lemma}[section]
\newtheorem{corollary}{Corollary}[section]
\newtheorem{remark}{Remark}[section]
\newtheorem{lem}{Lemma}[section]
\newtheorem{thm}{Theorem}[section]
\newtheorem{prop}[thm]{Proposition}
\newtheorem{cor}[thm]{Corollary}
\theoremstyle{remark}
\newtheorem{rmk}{Remark}[section]

\renewcommand{\tilde}{\widetilde}
\renewcommand{\hat}{\widehat}
\renewcommand{\bar}{\overline}

\newcommand{\nn}{\nonumber}
\newcommand{\R}{{\mathbb R}}

\newcommand{\N}{{\mathbb N}}

\newcommand{\del}{\partial}
\newcommand{\Denote}{\stackrel{\Delta}{=}}
\newcommand{\dt}{ \, {\rm d} t}
\newcommand{\dx}{ \, {\rm d} x}

\newcommand{\dv}{ \, {\rm d} v}
\newcommand{\ds}{\, {\rm d} s}

\newcommand{\dbmu}{\, {\rm d} \bar\mu}

\newcommand{\dtheta}{\, {\rm d} \theta}
\newcommand{\dtau}{\, {\rm d} \tau}

\newcommand{\dxi}{\, {\rm d} \xi}

\newcommand{\One}{\boldsymbol{1}}
\newcommand{\La}{\left\langle}
\newcommand{\Ra}{\right\rangle}
\newcommand{\la}{\langle}
\newcommand{\ra}{\rangle}

\newcommand{\Id}{{\mathcal{I}}}
\newcommand{\Ss}{{\mathbb{S}}}

\newcommand{\Eps}{\epsilon}
\newcommand{\Ni}{\noindent}

\newcommand{\CalA}{{\mathcal{A}}}

\newcommand{\CalC}{{\mathcal{C}}}
\newcommand{\CalD}{{\mathcal{D}}}

\newcommand{\CalI}{{\mathcal{I}}}

\newcommand{\CalL}{{\mathcal{L}}}
\newcommand{\CalM}{{\mathcal{M}}}

\newcommand{\CalQ}{{\mathcal{Q}}}

\newcommand{\CalS}{{\mathcal{S}}}

\newcommand{\Null} {{\rm Null} \, }
\newcommand{\Span} {{\rm Span} \, }

\newcommand{\dsigma}{\, {\rm d}\sigma}

\newcommand{\IntTRRS}{\int_{\T^3} \!\! \int_{\R^3} \!\! \int_{\R^3} \!\! \int_{\Ss^2_+}}
\newcommand{\IntRRS}{\int_{\R^3} \!\! \int_{\R^3} \!\! \int_{\Ss^2_+}}
\newcommand{\IntRS}{\int_{\R^3} \!\! \int_{\Ss^2_+}}

\newcommand{\RR}{\mathbb{R}}
\newcommand{\NN}{\mathbb{N}}
\newcommand{\ZZ}{\mathbb{Z}}
\newcommand{\T}{\mathbb{T}}

\newcommand{\abs}[1]{\left\lvert#1\right\rvert}
\newcommand{\norm}[1]{\left\lVert#1 \, \right\rVert}
\newcommand{\vint}[1]{\left\langle#1\right\rangle}

\newcommand{\vpran}[1]{\left(#1\right)}

\newcommand{\Tnorm}[1]{{\left\vert\kern-0.25ex\left\vert\kern-0.25ex\left\vert #1 
    \right\vert\kern-0.25ex\right\vert\kern-0.25ex\right\vert}}

\newcommand{\Semi}{\CalS_{\CalL}}

\begin{document}

\title[non-cutoff Boltzmann]{Non-cutoff Boltzmann Equation with Polynomial Decay Perturbation}

\author{Ricardo Alonso}
\address{Departamento de Matematica, PUC-Rio, Rua Marquex de Sao Vicente 225, Rio de Janeiro, CEP 22451-900, Brazil}
\email{ralonso@mat.puc-rio.br}

\author{Yoshinori Morimoto}
\address{2-9-16 Asahi, Otsu, 520-0533, Japan}
\email{morimoto.yoshinori.74r@st.kyoto-u.ac.jp}

\author{Weiran Sun}
\address{Department of Mathematics, Simon Fraser University, 8888 University Dr., Burnaby, BC V5A 1S6, Canada}
\email{weirans@sfu.ca}

\author{Tong Yang}
\address{Department of Mathematcs, City University of Hong Kong, 81 Tat Chee Avenue, Kowloon, Hong Kong}
\email{matyang@cityu.edu.hk}

\date{\today}
%
%
%
\begin{abstract}
The Boltzmann equation without an angular cutoff is considered when the initial data is a small perturbation of a global Maxwellian with an algebraic decay in the velocity variable. A well-posedness theory in the perturbative framework is obtained for both mild and strong angular singularities by combining three ingredients: the moment propagation, the spectral gap of the linearized operator, and the regularizing effect of the linearized operator when the initial data is in a Sobolev space with a negative index. A carefully designed pseudo-differential operator plays an central role in capturing the regularizing effect. Moreover, some intrinsic symmetry with respect to the collision operator and an intrinsic functional in the coercivity estimate are essentially used in the commutator estimates for the collision operator with velocity weights.

{\bf key words: } moment propagation, coercivity, spectral gap, commutator estimates, regularizing effect.
\end{abstract}

\maketitle
%

\tableofcontents


\section{Introduction} \label{sec:setting}

This paper aims to present a complete well-posedness theory to the Boltzmann equation without an angular cutoff when the initial perturbation of a global equilibrium state is small and decays only algebraically in the velocity variable. 
Precisely, we consider the Cauchy problem for  the non-cutoff Boltzmann equation 
\begin{equation}
\begin{gathered} \label{eq:Boltzmann}
   \del_t F + v \cdot \nabla_x F  = Q(F, F) \,,
\\
   F|_{t = 0} = F_0 \ge0\,,
\end{gathered}
\end{equation}
where $(x, v) \in \T^3 \times \R^3$ and the collision operator is given by
\begin{align} \label{def:Q}
   Q(F, F) 
= \int_{\R^3} \int_{\Ss^2} 
    b(\cos\theta) |v - v_\ast|^\gamma 
    \vpran{F_\ast' F' - F_\ast F} \dsigma \dv_\ast \,.
\end{align}
Our analysis applies to the non-angular cutoff cross-section for hard potential, that is, when $b$ and $\gamma$ satisfy
\begin{align} \label{assump:b-gamma}
   0 < \gamma \leq 1 \,,
\qquad
   b(\cos\theta) \sin\theta \sim \frac{1}{\theta^{1 + 2s}} \,,
\qquad
   0 < s < 1 \,.
\end{align}
In the perturbative framework, let $\mu= (2\pi)^{-3/2}e^{-|v|^2/2}$ be the normalized equilibrium and $f$ be the perturbation by
writing
\begin{align*}
    F = \mu + f \,.
\end{align*}
Equation \eqref{eq:Boltzmann} becomes
\begin{equation}
\begin{gathered} \label{eq:perturbation}
   \del_t f + v \cdot \nabla_x f 
= Q(\mu, f) + Q(f, \mu) +Q(f,f)=Lf+ Q(f,f)\,,
\\
   f |_{t=0} = f_0(x, v) \,.
\end{gathered}
\end{equation}
To study equation \eqref{eq:perturbation} when the initial data only has an algebraic decay in the velocity variable, we first point out its main difference from the classical decomposition $F=\mu +\sqrt{\mu}f $ that implies a Gaussian tail in the perturbation. First of all, with the Gaussian tail decomposition, the corresponding linearized operator given by
\begin{align*}
   L^{(\mu)}f 
&= \frac{1}{\sqrt{\mu}} \vpran{Q(\mu, \sqrt{\mu} f) 
      + Q(\sqrt{\mu} f, \mu)}
\\
& = \int_{\R^3} \int_{\Ss^2}
	b(\cos\theta) |v - v_\ast|^\gamma
	\vpran{\sqrt{\mu'_\ast} f' + \sqrt{\mu'} f'_\ast
          - \sqrt{\mu} f_\ast - \sqrt{\mu_\ast} f} \dsigma\dv_\ast
\end{align*}
is self-adjoint and has the null space 
\begin{align*}
  \Null (L^{(\mu)}) 
 = \Span \left\{\sqrt{\mu}, \,\, \sqrt{\mu} \, v, \,\, \sqrt{\mu} \, |v|^2 \right\} \,.
\end{align*}
For this self-adjoint linear operator, one has the following strong coercivity estimate that implies the gain of both regularity and moment of order $s$ in the velocity variable (cf. \cite{AMUXY2011,AMUXY2012JFA,GS2011}): 
\begin{align*}
   \vint{f, \,\, L^{(\mu)} f}_{L^2_v}
\leq
  - c_0 \vpran{\norm{f}^2_{H^{s}_{\gamma/2}} 
  +  \norm{f}^2_{L^2_{s+ \frac{\gamma}{2}}}} \,,
\qquad
   f \in \vpran{\text{Null} (L^\mu)}^\perp \,.
\end{align*}
This coercivity property is essentially used in the well-posedness theory for the non-cutoff Boltzmann equation with Gaussian tail, cf. \cite{AMUXY2011, AMUXY2012JFA,GS2011}. However, if we only assume an algebraic decay in the perturbation by writing
$F = \mu + f$, then the corresponding linearized operator given by
\begin{align*}
   L f &= Q(\mu, f) + Q(f, \mu)
 \\
   & = \int_{\R^3} \int_{\Ss^2}
	\vpran{\mu'_\ast  f' + \mu' \, f'_\ast - \mu \, f_\ast - \mu_\ast f}
	b(\cos\theta) |v -v_\ast|^\gamma \dsigma \dv_\ast,
\end{align*}
is no longer self-adjoint. In addition, the coercivity only gains regularity rather than moments. Precisely, cf. \cite{AMUXY-Kyoto-2012}, one has
\begin{align*}
   \vint{Q(\mu, f), \,\, f}
 = -c_1 J^\gamma_1(f)+ \mbox{mod} \{\norm{f}^2_{L^2_{\gamma/2}}\}
\leq
  -c_0 \norm{f}^2_{H^{s}_{\gamma/2}}
  + C \norm{f}^2_{L^2_{\gamma/2}},
\end{align*}
where 
\begin{align*}			
  J^\gamma_1(f)
 =\int_{\R^6}\int_{\Ss^2} b(\cos\theta)|v-v_*|^\gamma \mu_* (f(v')-f(v))^2 \dsigma\dv_*\dv \,.
\end{align*}
Apparently, the linearized operator $L$ can no longer be used to control any moment gain. 
Therefore, in the commutator estimates for the collision operator with either weights or some pseudo-differential operators, the estimation becomes more subtle, especially in the strong singularity setting. For this, we will show that the functional $J_1^\gamma$ plays an important role. In fact, this function corresponds to the first component in the isotropic norm defined in \cite{AMUXY2011-CMP} in the setting with a Gaussian tail. Note that even though $J_1^\gamma(f)$ has a lower bound as $\|f\|^2_{H^2_{\gamma/2}}$, its upper bound in the Sobolev norm can only be shown as $\|f\|^2_{H^2_{\gamma/2+s}}$ because of the Laplacian operator on a sphere. Therefore, in the commutator estimates for the collision operator and some pseudo-differential operators given in Section 5, we will keep the precise form of $J_1^\gamma$ rather than using the usual weighted Sobolev norms. On the other hand, we would like to mention that for mild singularity, that is,
when $0<s<1/2$, using the lower bound in weighted
Sobolev norm as $\|f\|^2_{H^2_{\gamma/2}}$ is sufficient.

Now let us review some works related to our paper. First of all, many of the well-posedness theories on the Boltzmann equation
established so far are based on Grad's angular cutoff 
assumption.  For this, there is the classical work on the renormalized solutions developed by DiPerna-Lions \cite{DL} for large initial data with finite mass, energy and entropy. In the perturbative framework,
the pioneering work was obtained by Ukai \cite{Ukai} for  $L^\infty$-solutions by using the spectrum of the linearized operator and a bootstrapping argument following the local existence result by Grad, \cite{Grad}. And an $L^2$-framework by using the energy method and micro-macro decompositions was established in \cite{Guo,LY,LYY}.

Without Grad's angular cutoff assumption,  the spectrum of the linearized operator around a global
Maxwellian was studied by Pao \cite{Pao}
in 1970s.  Later, the existences of weak and analytic
(Gevrey)
solutions  were obtained by Arkeryd and Ukai
in 1980s repectively, cf. \cite{Arkeryd,Ukai-2}.

In 1990s,  P.-L. Lions used the entropy dissipation 
to show the gain of regularity:
$$
\|\sqrt{F}\|^2_{\dot{H}^\delta(|v|<M)} \le C_M\|F\|^\theta_{L^1}(
\|F\|_{L^1} +D(F))^{1-\theta},
$$
where
$$
D(F)=-\int_{\R^3} Q(F,F)\log F dv,
$$
for some  constants $0<\delta<\frac{s}{1+s}$ and $M>0$ and $0<\theta<1$.  
Around the same time, Desvillettes firstly proved the regularization of solutions to some simplified kinetic models.

In early 2000s, the regularization induced by the grazing collisions was analyzed by using the entropy production and it was developed by many people, including Alexandre, Bouchut, Desvillettes,
Golse, P.-L. Lions, Mouhot, Villani, Wennberg, cf. \cite{Villani2002} and the references therein. In particular, some elegant formula were obtained in the work by Alexandre-Desvillettes-Villani-Wennberg \cite{ADVW2000} such as the cancellation lemma. In addition,  it was proved that
$$
\|\sqrt{F}\|^2_{{H}^{s}(|v|<M)} \le C_{F,M}(
\|F\|^2_{L^1_2} +D(F)),
$$
which was later finalized in \cite[Corollary 2.4]{AMUXY-Kyoto-2012} in the precise form as
$$
\|\sqrt{F}\|^2_{{H}^{s}(\R^3)} \le C_F(
\|F\|^2_{L^1_{\gamma}} +D(F)).
$$

For the  well-posedness theories of the Boltzmann equation without an angular cutoff, the existence of renormalized solutions was
obtained by Alexandre-Villani in \cite{AV2002}. In 2011-12,  two different approaches were introduced by Gressman-Strain \cite{GS2011} and Alexandre-Morimoto-Ukai-Xu-Yang \cite{AMUXY2012JFA, AMUXY2011-AA} independently to obtain  the well-posedness theory for small perturbations of a global equilibrium state with Gaussian tails. The regularizing effect was also obtained in our previous works, cf. \cite{HMUY-2008,AMUXY2011}. Note that in the setting with a Gaussian tail decay, the well-posedness theories hold for both cases when the space variable is in torus and the whole space, because the self-adjoint linearized operator yields both gain of regularity and moments. However,  it remains open to establish $L^\infty$-solutions to the Boltzmann equation without an angular cutoff in an analog to Ukai's result on the angular cutoff Boltzmann equation. 

When the perturbation has only an algebraic decay in the velocity variable, there is a recent important progress made by Gualdani-Mischler-Mouhot in \cite{GMM} on the spectral gap of the linearized operator around a global Maxwellian. Their result leads to the well-posedness theory on various kinetic equations with algebraic-decay perturbations when the space variable is in a torus, an example of which is the cutoff Boltzmann equation. In fact, the spectral gap in both the velocity variable in $\R^3$ and the space variable in a torus was
obtained in \cite{MN} under the cutoff assumption by analyzing  the mixing between the convection and the coercivity in the velocity variable of the linearized operator. 

Without an angular cutoff, a well-posedness theory was recently obtained in \cite{HTT} for the case of the mild angular singularity where $0 < s < 1/2$. The main result of our paper gives a different approach to establish well-posedness that applies for both mild and strong angular singularity. There is also a recent work \cite{He} that gives a well-posedness theory using yet a third method. We would like to mention that the spectral gap of the linearized operator in both the velocity and space variables is essential in the analysis, so at this moment it is not known how to show the well-posedness in the whole space if only algebraic decay in the velocity variable is assumed. Furthermore, for the angular cutoff case with a Gaussian tail, the gain of moment implies that the case $\gamma+2s\ge 0$  corresponds to the hard potential. With an algebraic decay, we can only show the existence of a spectral gap with the condition $\gamma\ge 0$ rather than $\gamma+2s\ge 0$.

There are  three main components in our proof of the main well-posedness theorem. The first one is the gain of the moment due to the hard potential, with an error term of the same order of the moment as in the energy function. In this step, the term with a good sign leading to the gain of regularity due to the non-angular cutoff is simply neglected. Second, the gain of regularity is obtained by the standard coercivity estimate. It also produces an error which can be bounded by the moment estimate in the first step.  Finally, to control all the error generated in the first two steps, we apply the spectral gap of the linearized collision operator for solutions with an algebraic decay and study the semi-group operator as used in \cite{GMM}. In this last step, in order to deal with the strong singularity in the collision operator, we establish an estimate of the linearized equation with initial data in a Sobolev space with a negative index in the velocity variable. 

\subsection*{Function spaces} To define the function spaces considered in this paper, we introduce the linear operator
\begin{align} \label{def:L}
   \CalL = -v \cdot \nabla_x  + L \,.
\end{align}
Then the linearized equation for \eqref{eq:perturbation} is
\begin{align*}
    \del_t h = \CalL h \,,
\quad
   h|_{t=0} = h_0(x, v) \,.
\end{align*}
Let $\Semi$ be the associated semigroup on $L^2(\!\dv; H^2(\!\dx))$. Denote $W$ as the weight function such that
\begin{align} \label{def:W}
    W(v) 
   =\vint{v}^{m_0}, \quad m_0 > \max\{4s,1\} \,. 
\end{align}
Define a  function space
\begin{align*}
   Y_l
= \{f \in L^2(\!\dx \dv) \big| \, W^{l - |\alpha|} \del^\alpha_x f \in L^2(\!\dx\dv), \,\, \abs{\alpha} = 0, 1, 2\} \,,
\quad
   l >  2 \,.
\end{align*}
For some  $l_0 \in \NN$ to be specified later, as in
\cite{GMM}, define  a
 norm to cope with the spectral gap by 
\begin{align*}
    \Tnorm{h} = \vpran{\norm{h}_{Y_l}^2 + A \int_0^\infty \norm{\Semi(\tau)h}^2_{L^2(W^{l_0}\dv;H^2(\dx))} \dtau}^{1/2} \,,
\end{align*}
with $A$ being a large constant to be determined later. Here $h$ satisfies
\begin{align} \label{cond:orthog-1}
  \int_{\T^3} \int_{\R^3} h \mu \dv \dx
= \int_{\T^3} \int_{\R^3} v_i h \ mu\dv \dx
= \int_{\T^3} \int_{\R^3} |v|^2 h \mu \dv \dx
= 0 \,,
\qquad i = 1, 2, 3 \,.
\end{align}
Note that the integral in the definition of the  norm $\Tnorm{\cdot}$ is well-defined and
equivalent to the $\|\cdot\|_{Y_l}$-norm 
if  $\CalL$ has a spectral gap.

With these notations, we state the main theorems of this paper. The first one is the local existence result.
\begin{thm}
  Suppose $0 < s < 1$ and $0<\gamma\le 1$. Then there exists a sufficiently small constant $\epsilon_0>0$  such that  if $ f_0 \in Y_l$ with
\begin{equation*}
   \norm{f_0}_{Y_l} \le \epsilon_0,
\end{equation*}
then there exist constants $T,\,\epsilon_1>0$ such that the Cauchy problem \eqref{eq:perturbation} admits a  unique solution 
\begin{align} \label{bound:f-Y-l}
   &f \in  L^\infty([0,T];Y_l) \enskip \mbox{satisfying} \enskip \norm{f}_{L^\infty([0,T];Y_l)} \le \epsilon_1,
\end{align}
and
\begin{align} \label{def:A-l}
    A_l(f)  
\Denote 
   \int_0^T\sum_{\alpha} \norm{W^{l - |\alpha|}\del^{\alpha}_x f(\tau)}^2_{L^2(\dx; H^s_{\gamma/2}(\dv))}d\tau < \epsilon_1^2
	\,.
\end{align}
\end{thm}
Note that in the setting of this paper, 
the smallness assumption on the perturbation is needed even for local existence. The next  result is about the global existence and large time behaviour of the solution.
\begin{thm}
	Suppose $0 < s < 1$ and $0<\gamma\le 1$.  For some $l$ being suitably large and $\Eps_0 > 0$ small enough,  if 
	$F^{in} = \mu + f^{in} \geq 0$ satisfies
	\begin{align*}
	\norm{f^{in}}_{Y_l} < \Eps_0 \,,
	\qquad
	\int_{\T^3} \int_{\R^3} f^{in} \phi(v) \dv \dx = 0 
	\end{align*}
	for any $\phi \in \Null (L) = \Span \left\{\mu, \,\, \mu \, v, \,\, \mu \, |v|^2 \right\}$, then the non-cutoff Boltzmann equation has a unique non-negative solution $F\in L^\infty([0, \infty), Y_l)$ such that
\begin{align*}	
   \norm{F-\mu}_{Y_l}\le c e^{-\lambda t} \norm{f^{in}}_{Y_l}, 
\end{align*}
holds for some constants $\lambda>0$ and $c>0$.
\end{thm}
 
\begin{rmk}
	The decay rate $\lambda$ of the perturbation can be made more precise as in \cite{GMM,HTT}. In particular, it can be chosen as any constant less thanthe spectral gap of the linearized operator $\mathcal{L}$. Since the idea is similar as in~\cite{GMM,HTT}, we will not elaborate on it in this paper.
\end{rmk}

The rest of paper is arranged as follows. In the next section, we will give
some preliminary estimates for later use. Bounds related to the collision operator will be given in the Section 3. The spectral gap without an angular cutoff in the algebraic decay function
space is given in Section 4. In Section 5, we prove a precise
regularization estimate of the linearized collision operator with
initial data in a Sobolev space with a negative index. The closed-form energy estimate will then be given in Section 6 and the proof of local and global existence with uniqueness and non-negativity will be given in the last section. Finally, in the Appendix we give some basic  estimates
about some differential operators and estimates related to the functional $J^{\gamma}_1(f)$.


\section{Some useful estimates} \label{sec:propositions}
In this section, we list some useful estimates that are needed for later estimation. 
For this, we introduce the notation
\begin{align*}
   \norm{g}_{H^\beta_k(\!\dv)}
= \norm{\vint{v}^k g}_{H^\beta(\!\dv)} \,,
\qquad
   \beta \in \R \,.
\end{align*}

The first proposition is about the equivalence of weight and differential operators up to commutation.

\begin{prop} [\cite{HMUY-2008}] \label{prop:equivalence}
Suppose $\alpha, \theta > 0$. Then there exists a generic constant $C$ independent of $f$ such that
\begin{align*}
   \frac{1}{C} \norm{\vint{D_v}^\theta \vint{v}^\alpha f}_{L^2(\!\dv)}
\leq
    \norm{\vint{v}^\alpha \vint{D_v}^\theta f}_{L^2(\!\dv)}
\leq
   C \norm{\vint{D_v}^\theta \vint{v}^\alpha f}_{L^2(\!\dv)} \,,
\end{align*}
that is, the above two norms are equivalent.
\end{prop}

The second proposition is the trilinear estimate for hard potential with non-cutoff cross section.

\begin{prop} [\cite{AMUXY2011, MS2016}] \label{prop:trilinear}
Denote $a^+ = \max\{a, 0\}$. Then the bilinear operator $Q$ satisfies
\begin{align*}
   \abs{\int_{\R^3} Q(f, g) h \dv}
\leq
  C \vpran{\norm{f}_{L^1_{(m - \gamma/2)^++\gamma + 2s}}
  + \norm{f}_{L^2}}
     \norm{g}_{H^{s+\sigma}_{\gamma/2 + 2s + m}}
     \norm{h}_{H^{s-\sigma}_{\gamma/2 - m}} 
\end{align*}
for any $\sigma \in [\min\{s-1, -s\}, s]$, $m, \gamma, s \geq 0$.
Here,  $f, g, h$  are any functions so that the corresponding
norms are well-defined.  The constant $C$ is independent of $f, g, h$. 
\end{prop}

In later analysis, we often use two types of change of variables given in 

\begin{prop}[\cite{ADVW2000}] \label{prop:change-variable} 
Suppose $f$ is smooth enough such that the integrals below are well-defined. Then

\Ni (a) (Regular change of variables)
\begin{align*}
   \int_{\R^3} \int_{\Ss^2} b(\cos\theta) |v-v_\ast|^\gamma f(v') \dsigma \dv 
= \int_{\R^3} \int_{\Ss^2} b(\cos\theta) \frac{1}{\cos^{3+\gamma}(\theta/2)}|v-v_\ast|^\gamma f(v) \dsigma \dv \,.
\end{align*}

\Ni (b) (Singular change of variables)
\begin{align*}
   \int_{\R^3} \int_{\Ss^2} b(\cos\theta) |v-v_\ast|^\gamma f(v') \dsigma \dv_\ast 
= \int_{\R^3} \int_{\Ss^2} b(\cos\theta) \frac{1}{\sin^{3+\gamma}(\theta/2)}|v-v_\ast|^\gamma f(v_\ast) \dsigma \dv_\ast \,.
\end{align*}
\end{prop}

The proof of the main theorems relies on estimates of the
solution in some weighted Sobolev spaces. For this, we need to consider
the difference of the weight before and after collision, in particular, to seek for the cancellation of the angular singularity. Additional symmetry is also needed for strong singularity. To this end, we establish a technical lemma about the difference of the weights that is essential for the analysis. First, note that
\begin{align} 
   |v'|^{2} 
 = |v|^{2}\cos^{2}\frac{\theta}{2} + |v_{*}|^{2}\sin^{2}\frac{\theta}{2}+2\cos\frac{\theta}{2}\sin\frac{\theta}{2}|v-v_*| v\cdot\omega,
\label{formula:classical-1}
\end{align}
and
\begin{align}
  \vint{v'}^{2} 
= \vint{v}^{2}\cos^{2}\frac{\theta}{2} 
     + \vint{v_\ast}^{2} \sin^{2}\frac{\theta}{2}
     + 2\cos\frac{\theta}{2}\sin\frac{\theta}{2}|v-v_*| v\cdot\omega,
\label{formula:classical-2}
\end{align}
where $\omega = \frac{\sigma - (\sigma \cdot k) k}{\abs{\sigma - (\sigma \cdot k) k}}$ with $k = \frac{v - v_\ast}{|v - v_\ast|}$. Here $\omega$ satisfies that $\omega \perp (v - v_\ast)$.

\begin{rmk} \label{rmk:v-v-ast-omega}
Since $\omega \perp (v - v_\ast)$, we have $v \cdot \omega = v_\ast \cdot \omega$. Hence, we have the freedom to choose
when to use $v \cdot \omega $ or $v_\ast \cdot \omega$ in later  estimates. 
\end{rmk}

\begin{lem} \label{lem:diff-v-k}
Suppose $k > 3$ and $(v, v_\ast), (v', v'_\ast)$ are the velocity pairs before and after the collision. Let $\omega$ be the same vector as in~\eqref{formula:classical-1} and~\eqref{formula:classical-2}
Then,
\begin{align}\label{precise}
\langle v' \rangle^{2k} - \langle v \rangle^{2k}  
&=   2k \langle v \rangle^{2k-2}  |v-v_*| \big (v \cdot \omega \big)\cos^{2k-1} \tfrac{\theta}{2}  \sin\tfrac{\theta}{2} 
+ \vint{v_\ast}^{2k}\sin^{2k} \tfrac{\theta}{2}  + \mathfrak{R}_1 + \mathfrak{R}_2 + \mathfrak{R}_3 \,,
\end{align} 
where there exists a constant $C_k$ only depending on $k$ such that
\begin{align} \label{bound:mathfrak-R-2-3}
  \abs{\mathfrak{R}_1}
\leq
  & C_k \vint{v} \vint{v_\ast}^{2k-1} \sin^{2k-3} \tfrac{\theta}{2},
\qquad
  \abs{\mathfrak{R}_2}
\leq
  C_k \vint{v}^{2k-2} \vint{v_\ast}^2 \sin^2\tfrac{\theta}{2},
\quad \,
  \abs{\mathfrak{R}_3}
\leq
  C_k \vint{v}^{2k-4} \vint{v_\ast}^4 \sin^2\tfrac{\theta}{2}.
\end{align}
\end{lem}
\begin{proof}
By the Taylor expansion and~\eqref{formula:classical-2}, we have
\begin{align} \label{eq:diff-v-k-precise-1}
& \quad \, \vint{v'}^{2k} - \vint{v}^{2k} \cos^{2k}\tfrac{\theta}{2} \nn
\\
&= k \Big(\langle v \rangle^2 \cos^2 \tfrac{\theta}{2}\Big)^{k-1} \langle v_*\rangle^2 \sin^2 \tfrac{\theta}{2} 
+ 2 k \Big(\langle v \rangle^2 \cos^2 \tfrac{\theta}{2}\Big)^{k-1}
   \cos \tfrac{\theta}{2} \sin\tfrac{\theta}{2}|v-v_*| \vpran{v \cdot\omega} \nn
\\
& \quad + k(k-1) \int_0^1 (1-t) \Big(
\vint{v}^2 \cos^2\tfrac{\theta}{2} + t \big(\langle v_*\rangle^2 \sin^2 \tfrac{\theta}{2} + 
2\cos\tfrac{\theta}{2}\sin \tfrac{\theta}{2}|v-v_*| \vpran{v \cdot\omega} \big) \Big)^{k-2} \dt  \nn
\\
&\qquad \qquad \qquad \times 
\Big( \langle v_*\rangle^2 \sin^2 \tfrac{\theta}{2} + 
2\cos\tfrac{\theta}{2}\sin\tfrac{\theta}{2}|v-v_*| \vpran{v \cdot\omega} \Big)^2 \nn
\\
& \Denote 
\mathfrak{D}_1 + \mathfrak{D}_2 + \mathfrak{D}_3 \,. 
\end{align}
Note that $\mathfrak{D}_2$ gives the first term on the right hand side of~\eqref{precise} and $\mathfrak{D}_1$ is part of $\mathfrak{R}_1$. To estimate $\mathfrak{D}_3$, we use the mean value theorem for the integrand in $\mathfrak{D}_3$ such that
\begin{align} \label{term:integral-1}
 & \quad \,
 \Big(
\vint{v}^2 \cos^2\tfrac{\theta}{2} + t \big(\langle v_*\rangle^2 \sin^2 \tfrac{\theta}{2} + 
2\cos\tfrac{\theta}{2}\sin\tfrac{\theta}{2}|v-v_*| v \cdot\omega\big) \Big)^{k-2} \nn
\\
&= t^{k-2}\vpran{\vint{v_\ast} \sin \tfrac{\theta}{2}}^{2k-4}
+ (k-2)\int_0^1 \Big(
 t \langle v_*\rangle^2 \sin^2 \frac{\theta}{2} + \tau\big(
\langle v \rangle^2 \cos^2\frac{\theta}{2} +
2 t\cos\frac{\theta}{2}\sin\frac{\theta}{2}|v-v_*| v \cdot\omega\big) \Big)^{k-3} \dtau  \nn
\\
& \hspace{5.5cm} \times \big(
\langle v \rangle^2 \cos^2\frac{\theta}{2}+
2 t\cos\frac{\theta}{2}\sin\frac{\theta}{2}|v-v_*| v \cdot\omega\big)  \nn
\\
& \Denote 
  t^{k-2}\vpran{\vint{v_\ast} \sin \tfrac{\theta}{2}}^{2k-4} + \mathfrak{D}_{3,1}.
\end{align}
By a direct estimate on $\mathfrak{D}_{3,1}$, we have
\begin{align} \label{bound:D-3-1}
  \mathfrak{D}_{3,1}
\leq
  C_k \vpran{\vint{v}^{2k-4} + \vint{v} \vpran{\vint{v_\ast} \sin\tfrac{\theta} {2}}^{2k-5}} \,.
\end{align}
Denoting $\mathfrak{h}=\langle v_*\rangle^2 \sin^2 \tfrac{\theta}{2} + 
2\cos\tfrac{\theta}{2}\sin\tfrac{\theta}{2}|v-v_*| \vpran{v \cdot\omega} $ and applying the bound on $\mathfrak{D}_{3,1}$ in $\mathfrak{D}_3$, we have
\begin{align*}
   \mathfrak{D}_3
&= k (k-1) \vpran{\int_0^1 (1-t) t^{k-2}\vpran{\vint{v_\ast} \sin \tfrac{\theta}{2}}^{2k-4} \dt} \mathfrak{h}^2
   + k(k-1) \vpran{\int_0^1 (1-t) \mathfrak{D}_{3,1} \dt} \mathfrak{h}^2
\\
&= \vpran{\vint{v_\ast} \sin \tfrac{\theta}{2}}^{2k-4} \mathfrak{h}^2
     + k(k-1) \vpran{\int_0^1 (1-t) \mathfrak{D}_{3,1} \dt} \mathfrak{h}^2 
\\
& \Denote 
  \mathfrak{D}_{3,2} + \mathfrak{D}_{3,3} \,.
\end{align*}
When estimating $\mathfrak{h}^2$ in $\mathfrak{D}_{3,2}$, we use $v \cdot \omega$ in its second term and obtain
\begin{align*}
   \mathfrak{h}^2 = \vpran{\vint{v_\ast} \sin\tfrac{\theta}{2}}^4 + \mathfrak{h}_1 
\end{align*}
with
\begin{align*}
   \mathfrak{h}_1
&\leq
   C \vint{v} \sin^2\tfrac{\theta}{2}
   \vpran{\vint{v}^3 + \vint{v}\vint{v_\ast}^2 + \vint{v_\ast}^3 \sin\tfrac{\theta}{2}}.
\end{align*}
Therefore,
\begin{align*}
   \mathfrak{D}_{3,2}
= \vpran{\vint{v_\ast} \sin\tfrac{\theta}{2}}^{2k} 
  + \vpran{\vint{v_\ast} \sin\tfrac{\theta}{2}}^{2k-4} \mathfrak{h}_1 \,, 
\end{align*}
where
\begin{align*}
  \vpran{\vint{v_\ast} \sin\tfrac{\theta}{2}}^{2k-4} \mathfrak{h}_1
&\leq
   C \vint{v}^4 \vint{v_\ast \sin\tfrac{\theta}{2}}^{2k-4} 
   + C\vint{v}^2 \vint{v_\ast \sin\tfrac{\theta}{2}}^{2k-2} 
   + C \vint{v} \vpran{\vint{v_\ast} \sin\tfrac{\theta}{2}}^{2k-1} 
\\
& \leq
   C \vint{v} \vint{v_\ast}^{2k-1} \sin^{2k-3}\tfrac{\theta}{2}
   + C \vint{v}^{2k-2} \vint{v_\ast}^2 \sin^2\tfrac{\theta}{2}
   + C \vint{v}^{2k-4} \vint{v_\ast}^4 \sin^2\tfrac{\theta}{2} \,.
\end{align*}
Hence, the second term of $\mathfrak{D}_{3,2}$  contributes only
to the remainder term $\mathfrak{R}_1 + \mathfrak{R}_2 + \mathfrak{R}_3$ in~\eqref{precise}. Finally, when estimating the term $\mathfrak{D}_{3,3}$, we replace $v \cdot \omega$ in $\mathfrak{h}$ by $v_\ast \cdot \omega$ (see Remark~\ref{rmk:v-v-ast-omega}). Then $\mathfrak{h}$ is directly bounded by
\begin{align*}
  \mathfrak{h}^2
\leq
  C \vint{v}^2 \vpran{\vint{v_\ast} \sin\tfrac{\theta}{2}}^2 
  + C \vint{v_\ast}^2 \vpran{\vint{v_\ast} \sin\tfrac{\theta}{2}}^2 \,.
\end{align*}
Together with~\eqref{bound:D-3-1}, we obtain the bound of $\mathfrak{D}_{3,3}$ as
\begin{align*}
   \abs{\mathfrak{D}_{3,3}}
\leq 
   C_k \vint{v} \vint{v_\ast}^{2k-1} \sin^{2k-3}\tfrac{\theta}{2}
   + C_k \vint{v}^{2k-2} \vint{v_\ast}^2 \sin^2\tfrac{\theta}{2}
   + C_k \vint{v}^{2k-4} \vint{v_\ast}^4 \sin^2\tfrac{\theta}{2} \,.  
\end{align*}
In summary, we have
\begin{align*}
   \mathfrak{D}_2 = 2k \langle v \rangle^{2k-2}  |v-v_*| \big (v \cdot \omega \big)\cos^{2k-1} \tfrac{\theta}{2}  \sin\tfrac{\theta}{2} \,,
\qquad
  \mathfrak{D}_1 + \mathfrak{D}_3
= \vint{v_\ast}^{2k}\sin^{2k} \tfrac{\theta}{2}  + \mathfrak{R}_1 + \mathfrak{R}_2 + \mathfrak{R}_3 \,,
\end{align*}
which completes the proof of the lemma. 
\end{proof}

Next we recall a coercivity estimate obtained in~\cite{AMUXY2012JFA}.
\begin{prop} [\cite{AMUXY2012JFA}] \label{prop:coercivity-1}
Suppose $F$ satisfies 
\begin{align*}
   F \geq 0 \,,
\qquad
   \norm{F}_{L^1} \geq D_0 \,,
\qquad
  \norm{F}_{L^1_2} + \norm{F}_{L\log L} \leq E_0 \,.
\end{align*}
Then there exist two constants $c_0$ and $C$ such that
\begin{align*}
  \int_{\R^3} Q(F, f) \, f \dv
\leq
  -c_0 \norm{f}^2_{H^s_{\gamma/2}}
  + C \norm{f}^2_{L^2_{\gamma/2}} \,.
\end{align*}
\end{prop}

Finally, we have two technical lemmas that will be used in the spectral analysis of the linearized operator ~$\CalL$. 
\begin{lem}\label{A2}
For any $h\in H_v^{s}(\mathbb{R}^{3})$ with $s\in(0,1)$, let $h^\pm$ denote the positive and negative parts of $h$. Then it holds
\begin{equation}\label{lA2-e1}
\tfrac{1}{2}\|h\|_{H_v^{s}}^{2}\leq \sum_{g\in\{h^{\pm}\}}\|g\|^{2}_{H_v^{s}}\leq 2\|h\|^{2}_{H_v^{s}}\,.
\end{equation}
\end{lem}
\begin{proof}
Firstly of all, for $f\in H^{s}(\mathbb{R}^{3})$, we have
\begin{equation}\label{lA2-e2}
\|(-\Delta)^{s/2} f\|^{2}_{L^{2}} = c_{d,s}\int_{R^{2d}}\frac{|f(y)-f(x)|^{2}}{|y-x|^{d+2s}}\text{d}y\text{d}x\,,
\end{equation}
with $c_{d,s} = \tfrac{4^{s}\Gamma(d/2+s)}{\pi^{d/2}|\Gamma(-s)|}$.  Observe that
\begin{equation*}
|f^{\pm}(y)-f^{\pm}(x)|\leq|f(y)-f(x)|\,,\quad \text{for any}\quad (y,x)\in\mathbb{R}^{3}\times \mathbb{R}^{3} \,.
\end{equation*}
As a consequence, it readily follows from \eqref{lA2-e2} that
\begin{equation*}
\|(-\Delta)^{s/2} f^{\pm}\|^{2}_{L^{2}} \leq \|(-\Delta)^{s/2} f\|^{2}_{L^{2}}\,.
\end{equation*}
This leads to the second inequality in \eqref{lA2-e1}.  Furthemore,
\begin{align*}
|f(y)-f(x)| = |(f^{+}(y) - f^{+}(x)) - (f^{-}(y) -f^{-}(x))|
\leq |f^{+}(y) - f^{+}(x)| + |f^{-}(y) -f^{-}(x)|\,.
\end{align*} 
Thus,
\begin{equation*}
\|(-\Delta)^{s/2} f\|^{2}_{L^{2}}\leq 2\|(-\Delta)^{s/2} f^{+}\|^{2}_{L^{2}} + 2\|(-\Delta)^{s/2} f^{-}\|^{2}_{L^{2}}\,,
\end{equation*}
which, after using \eqref{lA2-e2}, gives the first inequality in \eqref{lA2-e1}.
\end{proof}

\begin{lem}\label{A1}
Let $h\in H^{s}_{v}(\mathbb{R}^{3})$ and set $H = \mu^{-1/2}h$. Let $b_2$ be the truncated collision kernel defined in~\eqref{DAK1}. Then there exist constants $C>0$ and $c>0$ depending only on the mass and energy of $\mu$ such that for every $\varepsilon\in(0,1)$,
\begin{align*}
\int_{\mathbb{R}^{6}}\int_{\mathbb{S}^{2}}b_{2}(\cos\theta)\mu(v_{*})\mu(v)\big(H(v') - H(v)\big)^{2}
\geq c\,\sum_{g\in\{h^{\pm}\}}\Big\|\widehat{g}(\xi)\,|\xi|^{s} \textbf{1}\big\{|\xi|\geq\tfrac{1}{\varepsilon}\big\}\Big\|^{2}_{L^{2}_{\xi}} - C\|\theta^{2} b_{2}\|_{L^{1}_{\theta}}\|h\|^{2}_{L^{2}_{v}}\,.
\end{align*}
\end{lem}

\begin{proof}
As in \cite[Proposition 1]{ADVW2000}, we expand the square 
in the above integrand and then apply
 Bobylev's identity together with the Cauchy-Schwarz inequality to obtain
\begin{align*}
\int_{\mathbb{R}^{6}}\int_{\mathbb{S}^{2}}b(\cos\theta)\mu(v_{*})\mu(v)\big(H(v') - H(v)\big)^{2} 
& = 2\int_{\mathbb{R}^{6}}\int_{\mathbb{S}^{2}}b(\cos\theta)\big(\mu(v_{*})|h|^{2}(v) - \mu^{1/2}(v_{*})\mu^{1/2}(v'_{*})h'h\big)
\\
&\geq 2\int_{\mathbb{R}^{6}}\int_{\mathbb{S}^{2}}b(\cos\theta)\sum_{g\in\{h^{\pm}\}}\big(\mu(v_{*})g^{2}(v) - \mu(v'_{*})g(v')g(v)\big)\\
&= 2\int_{\mathbb{R}^{3}}\int_{\mathbb{S}^{2}}b(\cos\theta)\sum_{g\in\{h^{\pm}\}}\big(\widehat{\mu}(0)\big|\widehat{g}(\xi)\big|^{2} - \widehat{\mu}(\xi^{-})\widehat{g}(\xi^{+})\widehat{g}(\xi)\big)\,.
\end{align*}
By applying the  Cauchy-Schwarz inequality to the second term in
 the above summation,  we  obtain 
\begin{align*}
\int_{\mathbb{R}^{3}}\int_{\mathbb{S}^{2}}b(\cos\theta)\mu(v_{*})\mu(v)\big(H(v') - H(v)\big)^{2}
&\geq \int_{\mathbb{R}^{3}}\int_{\mathbb{S}^{2}}b(\cos\theta)\sum_{g\in\{h^{\pm}\}}\big(\widehat{\mu}(0) - \widehat{\mu}(\xi^{-})\big)\big|\widehat{g}(\xi)\big|^{2}\\
& \quad \,+ \int_{\mathbb{R}^{3}}\int_{\mathbb{S}^{2}}b(\cos\theta)\sum_{g\in\{h^{\pm}\}}\big(\widehat{\mu}(0) - \widehat{\mu}(\xi^{-})\big)\big|\widehat{g}(\xi^{+})\big|^{2}\\
& \quad \, + \int_{\mathbb{R}^{3}}\int_{\mathbb{S}^{2}}b(\cos\theta)\sum_{g\in\{h^{\pm}\}}\widehat{\mu}(0)\big(\big|\widehat{g}(\xi)\big|^{2} - \big|\widehat{g}(\xi^{+})\big|^{2}\big)\,.
\end{align*}
For the last term on the right side, we apply the cancellation lemma  from \cite[Lemma 1]{ADVW2000} to obtain
\begin{equation}\label{A1e1}
\int_{\mathbb{R}^{3}}\int_{\mathbb{S}^{2}}b(\cos\theta)\widehat{\mu}(0)\big(\big|\widehat{g}(\xi)\big|^{2} - \big|\widehat{g}(\xi^{+})\big|^{2}\big) \geq -C\|\theta^{2} b\|_{L^{1}_{\theta}}\|\mu\|_{L^{1}_{v}}\|g\|^{2}_{L^{2}_{v}}\,,
\end{equation}
for some generic constant $C>0$.  The first and second terms are both positive that can be  treated similarly.  Consider the second term by applying the change of variables $\xi\rightarrow\xi^{+}$ and
 the fact that $\widehat{\mu}(\xi)$ is decreasing in $|\xi|$  to obtain
\begin{align*}
\int_{\mathbb{R}^{3}}\int_{\mathbb{S}^{2}}b(\cos\theta)\big(\widehat{\mu}(0)- \widehat{\mu}(\xi^{-})\big)\big|\widehat{g}(\xi^{+})\big|^{2}
\geq\int_{\mathbb{R}^{3}}\int_{\mathbb{S}^{2}}\frac{2^{d-1}b(\cos(2\theta))}{\cos^d(\theta)}\big(\widehat{\mu}(0)- \widehat{\mu}(\xi^{-})\big)\big|\widehat{g}(\xi)\big|^{2}\,,
\end{align*}
where $\cos\theta = \widehat{\xi}\cdot\sigma$.  Now, set $b=b_{2}$ which is supported in $\{|\sin\theta|\leq\varepsilon\}$.  Then, 
\begin{align}\label{A1e2}
\begin{split}
\int_{\mathbb{S}^{2}}\frac{2^{d-1}b(\cos(2\theta))}{\cos^d(\theta)}\big(\widehat{\mu}(0)- \widehat{\mu}(\xi^{-})\big)
&\geq c\int^{|\xi|\varepsilon}_{0}\frac{1}{\theta^{1+2s}}\big(1- \mu_{o}(\theta)\big)\big|\xi\big|^{2s}
\\
&\geq c\int^{1}_{0}\frac{1}{\theta^{1+2s}}\big(1- \mu_{o}(\theta)\big)\big|\xi\big|^{2s}\textbf{1}_{\{|\xi|\geq\frac{1}{\varepsilon}\}}\,.
\end{split}
\end{align}
Here $\mu_o(\theta)=e^{-\theta^{2}/2}$ is the radial profile of the Fourier transform of $\mu$.  Similar estimate holds for the second term.  Thus, \eqref{A1e1} and \eqref{A1e2} give the result of the lemma.
\end{proof}

For the
cross section $B(\cos\theta, |v - v'|)$ with an angular
cutoff, we will use the notation $Q^\pm$ defined as follows throughout the paper:
\begin{align*}
   Q^+(f, g) 
   = \int_{\R^3} \int_{\Ss^2} B(\cos\theta, |v - v'|) f'_\ast g' \dsigma \dv_\ast \,,
\qquad
   Q^-(f, g) 
   = \int_{\R^3} \int_{\Ss^2} B(\cos\theta, |v - v'|) f_\ast g \dsigma \dv_\ast \,.
\end{align*}
Note that $Q(f, g) = Q^+(f, g) - Q^-(f, g)$.


\section{Upper bounds on $Q$}

In this section, we will derive some bounds on the collision operator in some weighted $L^2$-norms. For simplicity of notations, we denote 
$\dbmu = \dsigma \dv_\ast \dv$.

The first estimate is about a commutator on the 
collision operator with a weight function.
\begin{prop} \label{prop:trilinear-weight}
Suppose $0 < s < 1$ and $k >\frac{9}{2} +\frac{\gamma}{2} + 2s$. Then 
\begin{align} \label{bound:commutator-weight}
& \quad \,
   \IntRRS b(\cos\theta) |v - v_\ast|^\gamma
         \vpran{\vint{v'}^{k} - \vint{v}^{k}} f_\ast g \, h' \dbmu \nn
\\
&\leq
     \vpran{\int_{\Ss^2} b(\cos\theta) \sin^{k- \frac{3}{2}- \frac{\gamma}{2}} \tfrac{\theta}{2} \dsigma}
    \norm{g}_{L^1_\gamma}
    \norm{\vint{v}^k f}_{L^2_{\gamma/2}} 
   \norm{h}_{L^2_{\gamma/2}}
   + C_k \norm{g}_{L^1_{1+\gamma}}
    \norm{\vint{v}^k f}_{L^2} 
   \norm{h}_{L^2} \nn
\\
& \quad \,
   + C_k \norm{f}_{L^1_{4+\gamma}} \norm{\vint{v}^k g}_{L^2} \norm{h}_{L^2}
  +   C_k \norm{f}_{L^1_\gamma}
  \norm{\vint{v}^k g}_{L^2_{\gamma/2}}
  \norm{h}_{L^2_{\gamma/2}}
\\
& \quad \,
  + C_k \norm{f}_{L^1_{3+\gamma+2s} \cap L^2}
  \norm{\vint{v}^k g}_{H^{s'}_{\gamma'/2}}
  \norm{h}_{L^2_{\gamma/2}} \,. \nn
\end{align}
The parameters $s', \gamma'$ satisfy the following conditions: if $0 < s < 1/2$, then
\begin{align*}
   (s', \gamma') = (0, \gamma) \,,
\qquad
   0 < s < 1/2 \,;
\end{align*}
if $1/2 \leq s < 1$, then
\begin{align*}
   s' = 2s - 1 + \frac{\Eps}{2} \in (0, s) \,,
\qquad
   \Eps \in (0, 2(1-s)) \,,
\qquad
   \frac{\gamma'}{2} = \frac{\gamma}{2} + (s' - 1) \in (0, \gamma/2) \,.
\end{align*}
\end{prop}
\begin{proof} 
Denote 
\begin{align}\label{for-section-4-2}
   \Gamma =\Gamma(f,g,h)
=  \IntRRS b(\cos\theta) |v - v_\ast|^\gamma
         \vpran{\vint{v'}^{k} - \vint{v}^{k}} f_\ast g \, h' \dbmu \,.
\end{align}
Then by Lemma~\ref{lem:diff-v-k},
\begin{align*}
  \Gamma 
& = \IntRRS b(\cos\theta) |v - v_\ast|^\gamma
         \vpran{k \langle v \rangle^{k-2}  |v-v_*| \big (v_\ast \cdot \omega \big)\cos^{k-1} \frac{\theta}{2}  \sin  \frac{\theta}{2} 
} f_\ast g \, h' \dbmu
\\
& \quad \,
  + \IntRRS b(\cos\theta) |v - v_\ast|^\gamma
         \vpran{\vint{v_\ast}^{k}\sin^{k} \vpran{\tfrac{\theta}{2}}} f_\ast g \, h' \dbmu
\\
& \quad \,
   + \sum_{m=1}^3 \IntRRS b(\cos\theta) |v - v_\ast|^\gamma
         \mathfrak{R}_m f_\ast g \, h' \dbmu
\\
& \quad \,
  + \IntRRS b(\cos\theta) |v - v_\ast|^\gamma
         \vpran{1-\cos^k \tfrac{\theta}{2}} f_\ast \vpran{g \vint{v}^{k}} \, h' \dbmu
\Denote 
   \sum_{m=1}^6 \Gamma_m \,.
\end{align*}
Here, we have replaced $v \cdot \omega$ by $v_\ast \cdot \omega$ in $\Gamma_1$.
Now we estimate all the $\Gamma_m$ separately. First, by the Cauchy-Schwarz inequality and the singular change of variables, we have
\begin{align} \label{bound:Gamma-2}
  \abs{\Gamma_2}
&\leq \left(
\IntRRS b(\cos\theta) 
\sin^{k-\frac{3}{2} -\frac{\gamma}{2}} \vpran{\tfrac{\theta}{2}}
|v - v_\ast|^\gamma |g|\vint{v_*}^k f_*|^2 dvdv_* d\sigma \right)^{1/2}\nn \\
&\qquad \times \left(
\IntRRS b(\cos\theta) 
\sin^{k+\frac{3}{2} -\frac{\gamma}{2}} \vpran{\tfrac{\theta}{2}}
|v' - v|^\gamma |g| |h(v') |^2 dvdv_* \sigma\right)^{1/2}  \nn \\
&\leq \int_{\Ss^2} b(\cos \theta)\sin^{k-\frac{3}{2} -\frac{\gamma}{2}}\vpran{\tfrac{\theta}{2}}
d\sigma  
\left(
\iint_{\RR^3 \times \RR^3}
|v - v_\ast|^\gamma |g| |\vint{v_*}^k f_*|^2 dvdv_*\right)^{1/2} \nn \\
&\qquad \qquad \times 
\left( \iint_{\RR^3 \times \RR^3}
|v - v'|^\gamma |g||h(v')||^2 dvdv' \right)^{1/2} \nn \\
&
\le \vpran{\int_{\Ss^2} b(\cos\theta) \sin^{k-\frac{3}{2} -\frac{\gamma}{2}}
\vpran{\tfrac{\theta}{2} }\dsigma} 
  \norm{g}_{L^1_\gamma}
  \norm{\vint{v}^k f}_{L^2_{\gamma/2}}
  \norm{h}_{L^2_{\gamma/2}} \,,
\end{align}
which holds when $k > \frac{3}{2} + \frac{\gamma}{2} + 2s$.  Similarly,  
\begin{align} \label{bound:Gamma-3}
  \abs{\Gamma_3}
&\leq
  \IntRRS b(\cos\theta) |v - v_\ast|^\gamma
         \abs{\mathfrak{R}_1} |f_\ast| |g| \, |h'| \dbmu \nn
\\
& \leq 
   C_k \IntRRS b(\cos\theta) |v - v_\ast|^\gamma
         \vint{v} \vint{v_*}^{k-1} \sin^{k-3} \tfrac{\theta}{2} 
         |f_\ast| |g| \, |h'| \dbmu
\\
& \leq
   C_k \norm{g}_{L^1_{1+\gamma}}
  \norm{\vint{v}^k f}_{L^2}
  \norm{h}_{L^2} \,, \nn
\end{align}
where we need $k >\frac{9}{2} +\frac{\gamma}{2} + 2s$. Next, 
\begin{align} \label{bound:Gamma-4}
   \abs{\Gamma_4}
&\leq
     \IntRRS b(\cos\theta) |v - v_\ast|^\gamma
         \abs{\mathfrak{R}_2} |f_\ast| |g| \, |h'| \dbmu \nn
\\
& \leq
     C_k \IntRRS b(\cos\theta) |v - v_\ast|^\gamma
         \vint{v}^{k-2} \vint{v_\ast}^2 \sin^2\tfrac{\theta}{2} |f_\ast| |g| \, |h'| \dbmu
\\
& \leq
  C_k \norm{f}_{L^1_{2+\gamma}} 
  \norm{\vint{v}^k g}_{L^2}
  \norm{h}_{L^2} \,. \nn
\end{align}
Similarly, 
\begin{align} \label{bound:Gamma-5}
   \abs{\Gamma_5}
&\leq
     \IntRRS b(\cos\theta) |v - v_\ast|^\gamma
         \abs{\mathfrak{R}_3} |f_\ast| |g| \, |h'| \dbmu \nn
\\
& \leq
     C_k \IntRRS b(\cos\theta) |v - v_\ast|^\gamma
         \vint{v}^{k-4} \vint{v_\ast}^4 \sin^2\tfrac{\theta}{2} |f_\ast| |g| \, |h'| \dbmu
\\
& \leq
  C_k \norm{f}_{L^1_{4+\gamma}} 
  \norm{\vint{v}^k g}_{L^2}
  \norm{h}_{L^2} \,. \nn
\end{align}
We can also estimate the bound
on  $\Gamma_6$ directly as
\begin{align}
   \abs{\Gamma_6}
\leq 
   C_k \norm{f}_{L^1_\gamma}
  \norm{\vint{v}^k g}_{L^2_{\gamma/2}}
  \norm{h}_{L^2_{\gamma/2}}  \,.
\end{align}
To estimate $\Gamma_1$, we rewrite $\omega$ as
\begin{align*}
   \omega 
= \tilde \omega \cos\tfrac{\theta}{2}
   + \frac{v' - v_\ast}{|v' - v_\ast|} \sin\tfrac{\theta}{2} \,,
\end{align*}
where $\tilde \omega = (v' - v)/|v' - v|$. Note that $\tilde \omega \perp (v' - v_\ast)$. Accordingly, 
\begin{align*}
   \Gamma_1
&= k\IntRRS b(\cos\theta) |v - v_\ast|^\gamma
         \vpran{\langle v \rangle^{k-2}  |v-v_*| \big (v_\ast \cdot \tilde \omega \big)\cos^{k} \tfrac{\theta}{2}  \sin\tfrac{\theta}{2} 
} f_\ast g \, h' \dbmu
\\
& \quad \,
  + k\IntRRS b(\cos\theta) |v - v_\ast|^\gamma
         \vpran{\langle v \rangle^{k-2}  |v-v_*| \vpran{v_\ast \cdot \frac{v'-v_\ast}{|v' - v_\ast|}} \cos^{k-1} \tfrac{\theta}{2}  \sin^2 \tfrac{\theta}{2} 
} f_\ast g \, h' \dbmu
\\
& \Denote \Gamma_{1, 1} + \Gamma_{1,2} \,.
\end{align*}
The second term $\Gamma_{1,2}$ is obviously bounded by
\begin{align} \label{bound:Gamma-1-2}
  \abs{\Gamma_{1,2}}
\leq 
  C_k \norm{f}_{L^1_{2+\gamma}}
  \norm{\vint{v}^k g}_{L^2}
  \norm{h}_{L^2} \,.
\end{align}
To estimate $\Gamma_{1,1}$, we 
consider  the cases when $0 < s < 1/2$ and $1/2 \leq s < 1$
separately. In the case of mild singularity when $0 < s < 1/2$, we can directly bound $\Gamma_{1,1}$ by
\begin{align*} 
  \abs{\Gamma_{1,1}}
\leq
  C_k \norm{f}_{L^1_{2+\gamma}}
  \norm{\vint{v}^k g}_{L^2}
  \norm{h}_{L^2} \,.
\end{align*}
Therefore, when $0 < s < 1/2$, we have
\begin{align} \label{bound:Gamma-1-mild}
  \abs{\Gamma_1}
\leq
  C_k \norm{f}_{L^1_{2+\gamma}}
  \norm{\vint{v}^k g}_{L^2}
  \norm{h}_{L^2} \,,
\qquad
  s \in (0, 1/2) \,.
\end{align}
To treat the strong singularity when $1/2 \leq s < 1$, we denote $G(v) = g(v) \vint{v}^{k-2}$ and separate $\Gamma_{1,1}$ such that
\begin{align*}
   \Gamma_{1,1}
&= k \IntRRS b(\cos\theta) |v - v_\ast|^{1+\gamma}
         \big (v_\ast \cdot \tilde \omega \big)\cos^{k} \tfrac{\theta}{2}  \sin  \tfrac{\theta}{2} f_\ast G' \, h' \dbmu
\\
& \quad \,
  + k \IntRRS b(\cos\theta) |v - v_\ast|^{1+\gamma}
         \big (v_\ast \cdot \tilde \omega \big)\cos^{k} \tfrac{\theta}{2}  \sin  \tfrac{\theta}{2} f_\ast \vpran{G - G'} \, h' \dbmu
\, \Denote \,
  \Gamma^{(1)}_{1,1} + \Gamma^{(2)}_{1,1} \,.
\end{align*}
One key  observation 
in this decomposition is that  $\Gamma^{(1)}_{1,1} = 0$. Indeed, one can make the regular change of variables $v \to v'$ and take the new $v' - v_\ast$ as the north pole. Then 
\begin{align*}
   \Gamma^{(1)}_{1,1}
= k \IntRRS b(\cos\theta)\frac{1}{\cos^{4+\gamma}\tfrac{\theta}{2}} |v' - v_\ast|^{1+\gamma}
         \big (v_\ast \cdot \tilde \omega \big)\cos^{k} \tfrac{\theta}{2}  \sin  \tfrac{\theta}{2} f_\ast G' \, h' \sin\theta \dtheta {\rm \, d}\phi \dv' \dv_\ast,
\end{align*}
where $\tilde \omega = (\cos\phi, \sin\phi, 0)$. Thus formally the integration in $\phi$ gives that $\Gamma^{(1)}_{1,1} = 0$. This can be made rigorous by first truncating the singularity of $b$ in $\theta$ and then passing the limit
of truncation. Hence, if $1/2 \leq s < 1$, then
\begin{align*}
   \abs{\Gamma_{1,1}} = \abs{\Gamma^{(2)}_{1,1}}
&\leq
   C_k \IntRRS b(\cos\theta) |v - v_\ast|^{1+\gamma}
         \sin\tfrac{\theta}{2} \vpran{\vint{v_\ast} |f_\ast|} \abs{G - G'} \, |h'| 
\\
&\leq
   C_k \IntRRS b(\cos\theta) |v - v_\ast|^{1+\frac{\gamma}{2}} |v' - v_\ast|^{\gamma/2}
         \sin\tfrac{\theta}{2} \vpran{\vint{v_\ast} |f_\ast|} \abs{G - G'} \, |h'| \dbmu.
\end{align*}
Let $\Eps > 0$ be determined later. Then
\begin{align} \label{bound:Gamma-1-1-strong-1}
   \abs{\Gamma_{1,1}} 
&\leq
  C_k 
 \vpran{\IntRRS b(\cos\theta) |v - v_\ast|^{2+\gamma} 
         \theta^{2-2s-\Eps} \vpran{\vint{v_\ast} |f_\ast|} \abs{G - G'}^2 \dbmu}^{1/2}  \nn
\\
& \quad \quad \,
  \times \vpran{\IntRRS b(\cos\theta)  
         \theta^{2s+\Eps} \vpran{\vint{v_\ast}^{1+\gamma} |f_\ast|}  
         \vpran{\vint{v'}^{\gamma} |h'|^2} \dbmu}^{1/2}  \nn
\\
& \leq 
  C_k \norm{f}^{1/2}_{L^1_{1+\gamma}}
  \norm{h}_{L^2_{\gamma/2}}
  \vpran{\underbrace{\IntRRS b(\cos\theta) |v - v_\ast|^{2+\gamma} 
         \theta^{2-2s-\Eps} \vpran{\vint{v_\ast} |f_\ast|} \abs{G - G'}^2 \dbmu}_{\Gamma_{1,1}^{(3)}}}^{1/2}.
\end{align}
To bound the last factor in~\eqref{bound:Gamma-1-1-strong-1}, we write
\begin{align*}
   |G - G'|^2 = \vpran{(G')^2 - G^2} + 2 G (G - G') \,.
\end{align*}
Hence,
\begin{align*}
   \Gamma^{(3)}_{1,1}
&= \IntRRS b(\cos\theta) |v - v_\ast|^{2+\gamma} 
         \theta^{2-2s-\Eps} \vpran{\vint{v_\ast} |f_\ast|} \vpran{(G')^2 - G^2} \dbmu
\\
& \quad \,
  + 2 \IntRRS b(\cos\theta) |v - v_\ast|^{2+\gamma} 
         \theta^{2-2s-\Eps} \vpran{\vint{v_\ast} |f_\ast|} G (G - G') \dbmu
\\
& \leq 
  C \norm{f}_{L^1_{3+\gamma}}
  \norm{\vint{v}^{k-2} g}^2_{L^2_{1+\gamma/2}}
+ 2 \int_{\R^3} Q_{\tilde b}(\vint{v} f, \,\, G) G \dv
\\
& \leq
  C \norm{f}_{L^1_{3+\gamma}}
  \norm{\vint{v}^{k} g}^2_{L^2_{\gamma/2}}
+ 2 \int_{\R^3} Q_{\tilde b}(\vint{v} f, \,\, G) G \dv,
\end{align*}
where $Q_{\tilde b}$ denotes the bilinear operator with the 
cross section $\tilde b = b(\cos\theta) \theta^{2-2s-\Eps} |v - v_\ast|^{2+\gamma}$. Hence, the singularity is given by 
\begin{align} \label{def:s-prime}
   \theta \tilde b \sim \frac{1}{\theta^{1+2s'}}\,,
\qquad
  s' = 2s - 1 + \frac{\Eps}{2} \,.
\end{align}
We choose $\Eps > 0$ such that $s' < s$, that is,
\begin{align*}
    0 < \Eps < 2(1-s) \,.
\end{align*}
By the trilinear estimate given in Proposition~\ref{prop:trilinear} for $Q_{\tilde b}$, we have
\begin{align*}
   \abs{\int_{\R^3} Q_{\tilde b}(\vint{v} f, \,\, G) G \dv}
\leq
  C \norm{\vint{v} f}_{L^1_{2+\gamma+2s'} \cap L^2}
  \norm{G}^2_{H^{s'}_{\frac{2+\gamma}{2} + s'}}
\leq
  C \norm{\vint{v} f}_{L^1_{2+\gamma+2s'} \cap L^2}
  \norm{\vint{v}^k g}^2_{H^{s'}_{\gamma'/2}} ,
\end{align*}
where the weight $\gamma'$ is given by
\begin{align} \label{def:gamma-prime}
   \frac{\gamma'}{2}
= \frac{2+\gamma}{2} + s' -2
= \frac{\gamma}{2} + (s' -1)
< \frac{\gamma}{2} \,.
\end{align}
Altogether we have
\begin{align*}
   \abs{\Gamma^{(3)}_{1,1}}
\leq
   C \norm{f}_{L^1_{3+\gamma+2s} \cap L^2}
  \norm{\vint{v}^k g}^2_{H^{s'}_{\gamma'/2}}  \,,
\end{align*}
which, by~\eqref{bound:Gamma-1-1-strong-1}, further gives
\begin{align} \label{bound:Gamma-1-1-strong}
   \abs{\Gamma_{1,1}} 
&\leq
  C_k \norm{f}_{L^1_{3+\gamma+2s} \cap L^2}
  \norm{\vint{v}^k g}_{H^{s'}_{\gamma'/2}}
  \norm{h}_{L^2_{\gamma/2}} \,,
\end{align}
where $s', \gamma'$ are defined in~\eqref{def:s-prime} and~\eqref{def:gamma-prime} respectively. Combining the estimates in~\eqref{bound:Gamma-2}-\eqref{bound:Gamma-1-mild} and \eqref{bound:Gamma-1-1-strong}, we obtain the desired estimate in~\eqref{bound:commutator-weight}.  
\end{proof}

We are now ready to show a key coercivity estimate for $Q(F, f)$
stated in     
\begin{prop} \label{prop:coercivity}
Suppose $F = \mu + g$ with $F$ satisfying the conditions in Proposition~\ref{prop:coercivity-1}. For $k >\frac{9}{2} +\frac{\gamma}{2} + 2s$, we have
\begin{align} \label{estimate:coercivity}
& \quad \,
   \int_{\T^3}\int_{\R^3} Q(F, f) \, f \, \vint{v}^{2k} \dv\dx \nn
\\
&\leq
  - \frac{\gamma_0}{2} \int_{\T^3} \norm{\vint{v}^k f}_{L^2_{\gamma/2}}^2 \dx
  - \vpran{\frac{c_0}{4} \delta_2 - C_k \sup_{\T^3}\norm{g}_{L^1_{3+\gamma+2s} \cap L^2}} \int_{\T^3} \norm{\vint{v}^k f}^2_{H^s_{\gamma/2}} \dx \nn
\\
& \quad \,
  + C_k \int_{\T^3} \norm{\vint{v}^k f}^2_{L^2} \dx
  + C_k \int_{\T^3} \norm{\vint{v}^k g}_{L^2}
   \norm{\vint{v}^k f}_{L^2}
   \norm{f}_{L^1_{1+\gamma}} \dx
\\
& \quad \, 
  +  \vpran{\int_{\Ss^2} b(\cos\theta) \sin^{k-\frac{3}{2}-\frac{\gamma}{2}}\tfrac{\theta}{2}  \dsigma} 
    \int_{\T^3} 
    \norm{f}_{L^1_\gamma} \norm{\vint{v}^k g}_{L^2_{\gamma/2}}
     \norm{\vint{v}^k f}_{L^2_{\gamma/2}} \dx  \nn
 \,,
\end{align}
where $\gamma_0$ is defined in~\eqref{def:gamma-0}, $c_0$ is the coefficient in Proposition~\ref{prop:coercivity-1}, and $\delta_2$ is a small enough constant (which may depend on $k$).
\end{prop}
\begin{proof}
We will give two different estimates on $\int_{\T^3}\int_{\R^3} Q(F, f) \, f \, \vint{v}^{2k} \dv\dx$. The first one contains dissipation
in terms of  $\norm{\vint{v}^k f}_{L^2_{\gamma/2}}$ while the second one contains dissipation of $\norm{\vint{v}^k f}_{H^s_{\gamma/2}}$.
First, by the definition of $Q$, we have
\begin{align*}
   \int_{\R^3}  Q(F, f) f \vint{v}^{2k} \dv
& = \IntRRS 
       b(\cos\theta) |v - v_\ast|^\gamma 
       \vpran{F_\ast' f' - F_\ast f} f \vint{v}^{2k} \dbmu
\\
& = \IntRRS 
       b(\cos\theta) |v - v_\ast|^\gamma 
       F_\ast f \vpran{f' \vint{v'}^{2k} - f \vint{v}^{2k}} \dbmu
\\
& = \IntRRS 
       b(\cos\theta) |v - v_\ast|^\gamma 
       F_\ast \vpran{f  f' \vint{v'}^{2k} - |f|^2 \vint{v}^{2k}} \dbmu
\\
& \leq \IntRRS 
       b(\cos\theta) |v - v_\ast|^\gamma 
       F_\ast \vpran{|f| \vpran{|f'| \vint{v'}^{k}} \vint{v'}^{k} - |f|^2 \vint{v}^{2k}}  \,.
\end{align*}
Hence, 
\begin{align} \label{def:T-1-T-2}
    \int_{\R^3}  Q(F, f) f \vint{v}^{2k} \dv \nn
& \leq \IntRRS 
       b(\cos\theta) |v - v_\ast|^\gamma 
       F_\ast \vpran{|f| \vpran{|f'| \vint{v'}^k} \vint{v}^{k} \cos^k \tfrac{\theta}{2} - |f|^2 \vint{v}^{2k}} 
       \dbmu \nn
\\
& \quad \,
    + \IntRRS 
       b(\cos\theta) |v - v_\ast|^\gamma 
       F_\ast |f| |f'| \vint{v'}^k
       \vpran{\vint{v'}^k - \vint{v}^{k} \cos^k \tfrac{\theta}{2}} 
       \dbmu \nn
\\& 
\Denote  
     T_1 + T_2 \,.
\end{align}
We estimate $T_1 $ and $ T_2$ separately. Firstly, 
$T_1$ is a dissipative term.  Indeed, by Cauchy-Schwarz, 
\begin{align*}
   |f| \vpran{|f'| \vint{v'}^k} \vint{v}^{k} \cos^k \tfrac{\theta}{2} - |f|^2 \vint{v}^{2k}
\leq 
   \frac{1}{2}
      \vpran{\vpran{|f'| \vint{v'}^k}^2 \cos^{2k} \tfrac{\theta}{2}
      - |f|^2 \vint{v}^{2k}} \,.
\end{align*}
Therefore, using a regular change of variables, we have
\begin{align} \label{bound:T-1-explicit}
   T_1 
& \leq 
  \frac{1}{2} \IntRRS
    b(\cos\theta) |v - v_\ast|^\gamma
    F_\ast \vpran{\vpran{|f'| \vint{v'}^k}^2 \cos^{2k} \tfrac{\theta}{2}
      - |f|^2 \vint{v}^{2k}} \dbmu \nn
\\
& = \frac{1}{2} \IntRRS
        b(\cos\theta) |v - v_\ast|^\gamma F_\ast |f|^2 \vint{v}^{2k}
        \vpran{\cos^{2k - 3 - \gamma} \tfrac{\theta}{2} - 1} \dbmu \,.
\end{align}
Let $0< \gamma_1 < \gamma_2$ be the coefficients such that 
\begin{align*}
   \gamma_1 \vint{v}^\gamma
\leq
   \int_{\R^3} |v - v_\ast|^\gamma \mu_\ast \dv_\ast
\leq
  \gamma_2 \vint{v}^\gamma \,. 
\end{align*}
Denote $\gamma_0$ as the constant given by 
\begin{align} \label{def:gamma-0}
    \gamma_0 = -\frac{\gamma_1}{2} \int_{\Ss^2} b(\cos\theta)
  \vpran{\cos^{2k - 3 - \gamma} \tfrac{\theta}{2} - 1} \dsigma \,.
\end{align}
Note that for $k \ge  \frac{5+\gamma}{2}$, the constant $\gamma_0$ has a strict lower bound that is independent of $k$. Hence, 
\begin{align} \label{bound:T-1}
  \int_{\T^3} T_1 \dx 
\leq 
  -\vpran{\gamma_0
  - C_k \sup_{\T^3} \norm{g}_{L^1_\gamma(\dv)}}
  \norm{\vint{v}^k f}^2_{L^2_{\gamma/2}(\dx\dv)} \,.
\end{align}

The bound of the second term $T_2$ can be obtained by a direct application of Proposition~\ref{prop:trilinear-weight}. We note that $T_2$ only contains $\Gamma_1 \sim \Gamma_5$ in Proposition~\ref{prop:trilinear-weight} since the difference in $T_2$ is $\vint{v'}^k - \vint{v}^{k} \cos^k \tfrac{\theta}{2}$ instead of $\vint{v'}^k - \vint{v}^{k}$. Hence, using the bounds for $\Gamma_1 \sim \Gamma_5$ in Proposition~\ref{prop:trilinear-weight}, we have 
\begin{align*}
   T_2
&\leq
  C_k \norm{f}_{L^1_\gamma} \norm{\vint{v}^k f}_{L^2_{\gamma/2}}
  + \vpran{\int_{\Ss^2} b(\cos\theta) \sin^{k-\frac{3}{2}-\frac{\gamma}{2}}\tfrac{\theta}{2} \dsigma} \norm{f}_{L^1_\gamma} \norm{\vint{v}^k g}_{L^2_{\gamma/2}}
     \norm{\vint{v}^k f}_{L^2_{\gamma/2}}
\\
& \quad \,
  + C_k \norm{f}_{L^1_{1+\gamma}} 
      \norm{\vint{v}^k g}_{L^2}
      \norm{\vint{v}^k f}_{L^2}
   + C_k \norm{\vint{v}^k f}^2_{L^2}
   + C_k \norm{g}_{L^1_{4+\gamma}}
      \norm{\vint{v}^k f}^2_{L^2}
\\
& \quad \,
   + C_k \norm{\vint{v}^k f}_{H^{s'}_{\gamma'/2}}
  \norm{\vint{v}^k f}_{L^2_{\gamma/2}}
   + C_k \norm{g}_{L^1_{3+\gamma+2s} \cap L^2}
  \norm{\vint{v}^k f}_{H^{s'}_{\gamma'/2}}
  \norm{\vint{v}^k f}_{L^2_{\gamma/2}} \,.
\end{align*}
By the interpolation of $H^{s'}_{\gamma'/2}$ between $L^2$ and $H^s_{\gamma/2}$, we have if $1/2 \leq s < 1$, then
\begin{align} \label{bound:T-2}
   T_2
&\leq
  C_k \norm{f}_{L^1_\gamma} \norm{\vint{v}^k f}_{H^s_{\gamma/2}}
  + \vpran{\int_{\Ss^2} b(\cos\theta) \sin^{k-\frac{3}{2}-\frac{\gamma}{2}}\tfrac{\theta}{2} \dsigma} \norm{f}_{L^1_\gamma} \norm{\vint{v}^k g}_{L^2_{\gamma/2}}
     \norm{\vint{v}^k f}_{L^2_{\gamma/2}} \nn
\\
& \quad \,
  + C_k \norm{f}_{L^1_{1+\gamma}} 
      \norm{\vint{v}^k g}_{L^2}
      \norm{\vint{v}^k f}_{L^2}
   + C_k \norm{\vint{v}^k f}^2_{L^2}
   + C_k \norm{g}_{L^1_{4+\gamma}}
      \norm{\vint{v}^k f}^2_{L^2} 
\\
&\leq
   \vpran{\delta_1 + C_k \norm{g}_{L^1_{3+\gamma+2s}}} \norm{\vint{v}^k f}^2_{H^s_{\gamma/2}}
  + C_k \norm{f}_{L^1_{1+\gamma}} 
      \norm{\vint{v}^k g}_{L^2}
      \norm{\vint{v}^k f}_{L^2}
   + C_k \norm{\vint{v}^k f}^2_{L^2} \nn
\\
& \quad \,
  + 
\vpran{\int_{\Ss^2} b(\cos\theta) \sin^{k-\frac{3}{2}-\frac{\gamma}{2}}\tfrac{\theta}{2}  \dsigma} \norm{f}_{L^1_\gamma} \norm{\vint{v}^k g}_{L^2_{\gamma/2}}
     \norm{\vint{v}^k f}_{L^2_{\gamma/2}}  \,. \nn
\end{align}
Combining~\eqref{bound:T-1} and~\eqref{bound:T-2} gives 
\begin{align} \label{est:coercivity-1}
& \quad \,
   \int_{\T^3}\int_{\R^3} Q(F, f) \, f \, \vint{v}^{2k} \dv\dx    \nn
\\
&\leq
  - \gamma_0 \norm{\vint{v}^k f}_{L^2_{\gamma/2}(\dx\dv)}^2
  + \vpran{\delta_1 + C_k \sup_{\T^3}\norm{g}_{L^1_{3+\gamma+2s} \cap L^2}} \norm{\vint{v}^k f}^2_{H^s_{\gamma/2}}
  + C_k \norm{\vint{v}^k f}^2_{L^2_{x,v}}
 \nn
\\
& \quad \,
  + \vpran{\int_{\Ss^2} b(\cos\theta) \sin^{k-\frac{3}{2}-\frac{\gamma}{2}}\tfrac{\theta}{2}  \dsigma} 
    \int_{\T^3} 
    \norm{f}_{L^1_\gamma} \norm{\vint{v}^k g}_{L^2_{\gamma/2}}
     \norm{\vint{v}^k f}_{L^2_{\gamma/2}} \dx  
\\
& \quad \, 
  + C_k \int_{\T^3} \norm{\vint{v}^k g}_{L^2}
   \norm{\vint{v}^k f}_{L^2}
   \norm{f}_{L^1_{1+\gamma}} \dx \nn
 \,.
\end{align}
Next, we give  the second estimate on $\int_{\T^3}\int_{\R^3} Q(F, f) \, f \, \vint{v}^{2k} \dv\dx$ by firstly rewriting it as
\begin{align*}
  \int_{\T^3}\int_{\R^3} Q(F, f) \, f \, \vint{v}^{2k} \dv\dx
&= \int_{\T^3}\int_{\R^3} Q(F, \, \vint{v}^k f) \, \vint{v}^{k} f \dv\dx
\\
& \quad \,
  + \int_{\T^3}\int_{\R^3} 
      \vpran{\vint{v}^k Q(F, f) - Q(F, \vint{v}^k f)} \, \vint{v}^k f \dv\dx
\\
& \Denote T_3 + T_4 \,.
\end{align*}
Applying Proposition~\ref{prop:coercivity-1} to $T_3$ yields
\begin{align} \label{bound:T-3}
   T_3 
\leq 
   -c_0 \int_{\T^3} \norm{\vint{v}^k f}^2_{H^s_{\gamma/2}} \dx
   + C_1 \int_{\T^3} \norm{\vint{v}^k f}^2_{L^2_{\gamma/2}} \dx \,.
\end{align}
Note that  the second term $T_4$ has the form
\begin{align*}
   T_4
=   \IntTRRS b(\cos\theta) |v - v_\ast|^\gamma
         \vpran{\vint{v'}^{k} - \vint{v}^{k}} F_\ast f \, f' \vint{v'}^k \dbmu \dx.
\end{align*}
Applying the commutator estimate in Proposition~\ref{prop:trilinear-weight} to $T_4$ gives
\begin{align} \label{bound:T-4}
   T_4 
&\leq  
     \vpran{\int_{\Ss^2} b(\cos\theta) \sin^{k-\frac{3}{2}-\frac{\gamma}{2}}\tfrac{\theta}{2} \dsigma}
    \int_{\T^3} \norm{f}_{L^1_\gamma}
    \norm{\vint{v}^k F}_{L^2_{\gamma/2}} 
   \norm{\vint{v}^k f}_{L^2_{\gamma/2}} \dx \nn
\\
& \quad \,
   + C_k \int_{\T^3} \norm{f}_{L^1_{1+\gamma}}
    \norm{\vint{v}^k F}_{L^2} 
   \norm{\vint{v}^k f}_{L^2} \dx 
   + C_k \int_{\T^3}
      \norm{F}_{L^1_{4+\gamma}} \norm{\vint{v}^k f}^2_{L^2} \dx  \nn
\\
& \quad \,
   + C_k \int_{\T^3} \norm{F}_{L^1_\gamma}
  \norm{\vint{v}^k f}^2_{L^2_{\gamma/2}} \dx
  + C_k \int_{\T^3} \norm{F}_{L^1_{3+\gamma+2s} \cap L^2}
  \norm{\vint{v}^k f}_{H^{s'}_{\gamma'/2}}
  \norm{\vint{v}^k f}_{L^2_{\gamma/2}} \dx \,, 
\\
& \leq
   \vpran{\frac{c_0}{2} + C_k \sup_{\T^3} \norm{g}_{L^1_{3+\gamma+2s}}}\int_{\T^3} \norm{\vint{v}^k f}^2_{H^s_{\gamma/2}} \dx
   + C_k \int_{\T^3} \norm{\vint{v}^k f}^2_{L^2_{\gamma/2}} \dx \,. \nn
\end{align}
Combining~\eqref{bound:T-3} and~\eqref{bound:T-4}, we obtain
\begin{align} \label{est:coercivity-2}
& \quad \,
     \int_{\T^3}\int_{\R^3} Q(F, f) \, f \, \vint{v}^{2k} \dv\dx    \nn
\\
&\leq
   -\vpran{\frac{c_0}{2} - C_k \sup_{\T^3} \norm{g}_{L^1_{3+\gamma+2s}}}\int_{\T^3} \norm{\vint{v}^k f}^2_{H^s_{\gamma/2}} \dx
   + C_k \int_{\T^3} \norm{\vint{v}^k f}^2_{L^2_{\gamma/2}} \dx \,. 
\\
& \quad \,
  + \vpran{\int_{\Ss^2} b(\cos\theta) \sin^{k-\frac{3}{2}-\frac{\gamma}{2}}\tfrac{\theta}{2}  \dsigma}
    \int_{\T^3} \norm{f}_{L^1_\gamma}
    \norm{\vint{v}^k g}_{L^2_{\gamma/2}} 
   \norm{\vint{v}^k f}_{L^2_{\gamma/2}} \dx \nn
\\
& \quad \,
   + C_k \int_{\T^3} \norm{f}_{L^1_{1+\gamma}}
    \norm{\vint{v}^k g}_{L^2} 
   \norm{\vint{v}^k f}_{L^2} \dx \,. \nn
\end{align}
Let $\delta_2 > 0$ be a small number  to be determined. Multiply $\delta_2$ to~\eqref{est:coercivity-2} and add it  to~\eqref{est:coercivity-1}. This gives
\begin{align*}
& \quad \,
     \int_{\T^3}\int_{\R^3} Q(F, f) \, f \, \vint{v}^{2k} \dv\dx    \nn
\\
&\leq
  - \vpran{\frac{c_0}{2} \delta_2 -\delta_1 - C_k \sup_{\T^3}\norm{g}_{L^1_{3+\gamma+2s} \cap L^2}} \int_{\T^3} \norm{\vint{v}^k f}^2_{H^s_{\gamma/2}} \dx
\\
& \quad \,
  - \vpran{\gamma_0 - C_k \delta_2} \int_{\T^3} \norm{\vint{v}^k f}_{L^2_{\gamma/2}}^2 \dx
  + C_k \int_{\T^3} \norm{\vint{v}^k f}^2_{L^2} \dx
 \nn
\\
& \quad \,
  + \vpran{\int_{\Ss^2} b(\cos\theta)\sin^{k-\frac{3}{2}-\frac{\gamma}{2}}\tfrac{\theta}{2}  \dsigma} 
    \int_{\T^3} 
    \norm{f}_{L^1_\gamma} \norm{\vint{v}^k g}_{L^2_{\gamma/2}}
     \norm{\vint{v}^k f}_{L^2_{\gamma/2}} \dx  
\\
& \quad \, 
  + C_k \int_{\T^3} \norm{\vint{v}^k g}_{L^2}
   \norm{\vint{v}^k f}_{L^2}
   \norm{f}_{L^1_{1+\gamma}} \dx \nn
 \,.
\end{align*}
Then the dissipation given in the  inequality~\eqref{estimate:coercivity} 
follows from the fact  by first taking $\delta_2$ small enough such that $C_k \delta_2 < \frac{\gamma_0}{2}$ and then taking $\delta_1 > 0$ small enough such that $\delta_1 < \frac{c_0}{4} \delta_2$. 
\end{proof}

\begin{rmk}
We keep the second term on the right hand side of the inequality in Proposition~\ref{prop:coercivity} in the current form since in later sections we may apply the supremum in $x \in \T^3$ to either the $g$-term or the $f$-term depending on the need. 
\end{rmk}

Now we state the proposition for the bound of $Q(g, \mu)$.
\begin{prop} \label{prop:bound-Q-M}
Let  $k >\frac{9}{2} +\frac{\gamma}{2} + 2s$. Then
\begin{align*}
   \int_{\T^3} \int_{\R^3} Q(g, \mu) f \, \vint{v}^{2k} \dv\dx
&\leq 
      \vpran{\int_{\Ss^2} b(\cos\theta) \sin^{k-\frac{3}{2}-\frac{\gamma}{2}}\tfrac{\theta}{2} \dsigma}
    \int_{\T^3} \norm{\vint{v}^k g}_{L^2_{\gamma/2}} 
   \norm{\vint{v}^k f}_{L^2_{\gamma/2}} \dx
\\
& \quad \,
   + C_k 
    \int_{\T^3} \norm{\vint{v}^k g}_{L^2} 
   \norm{\vint{v}^k f}_{L^2_{\gamma/2}} \dx \,.
\end{align*}
\end{prop}
\begin{proof}
First of all, 
\begin{align*}
   \int_{\T^3} \int_{\R^3} Q(g, \mu) f \, \vint{v}^{2k} \dv \dx
&= \IntTRRS b(\cos\theta) |v - v_\ast|^\gamma
   g_\ast \mu \vpran{f' \vint{v'}^{2k} - f \vint{v}^{2k}} \dbmu\dx
\\
& \hspace{-0.8cm} 
  = \IntTRRS b(\cos\theta) |v - v_\ast|^\gamma
   g_\ast \mu \vint{v}^k \vpran{f' \vint{v'}^k - f \vint{v}^k}  \dbmu\dx
\\
& \hspace{-0.8cm} \quad \, 
    + \IntTRRS b(\cos\theta) |v - v_\ast|^\gamma
   g_\ast \mu \vpran{f' \vint{v'}^k}\vpran{\vint{v'}^k - \vint{v}^k} \dbmu \dx
\\
& \hspace{-0.8cm} \Denote T_5 + T_6 \,.
\end{align*}
Note that 
\begin{align*}
   T_5 = \int_{\T^3} \vpran{Q(g, \mu \vint{v}^k), \,\, f \vint{v}^k} \dx \,.
\end{align*}
Taking $m = \gamma/2$ and $\sigma = s$ in the trilinear estimate in Proposition~\ref{prop:trilinear}, we have
\begin{align} \label{bound:T-5}
    \abs{T_5} 
\leq 
   C \int_{\T^3} 
   \vpran{\norm{g}_{L^1_{\gamma + 2s} \cap L^2(\dv)}} \norm{\vint{v}^k f}_{L^2_v} \dx
\leq
   C \norm{\vint{v}^k g}_{L^2_{x, v}}\norm{\vint{v}^k f}_{L^2_{x, v}} \,,
\quad
   k > 2 + 2s + \gamma \,.
\end{align}
Applying Proposition~\ref{prop:trilinear-weight} to $T_6$, we obtain
\begin{align} \label{bound:T-6}
   T_6
&\leq
    \vpran{\int_{\Ss^2} b(\cos\theta) \sin^{k-\frac{3}{2}-\frac{\gamma}{2}}\tfrac{\theta}{2}  \dsigma}
    \norm{\vint{v}^k g}_{L^2_{\gamma/2}} 
   \norm{\vint{v}^k f}_{L^2_{\gamma/2}}
   + C_k 
    \norm{\vint{v}^k g}_{L^2} 
   \norm{\vint{v}^k f}_{L^2_{\gamma/2}} \,.
\end{align}
The desired bound is then obtained by combining~\eqref{bound:T-5} and ~\eqref{bound:T-6}.
\end{proof}

The next proposition is about the bounds on the
commutators with respect to the spatial derivatives. 
\begin{prop} \label{Prop:commute-deriv-x}
Let $\alpha$ be any multi-index such that $|\alpha| = 2$.
Suppose $l \geq 2 + \frac{6}{m_0}$ with $m_0$ being the exponent in~\eqref{def:W}. Let $F = \mu + f$. Then  
\begin{align} \label{bound:commut-deriv-x}
& \quad \,
 \abs{ \int_{\T^3} \int_{\R^3}
    \vpran{\del^\alpha_x Q(F, g) - Q \vpran{F, \, \del^\alpha_x g}}
    W^{2(l - |\alpha|)} \del^\alpha_x h \dx \dv} \nn
\\
& \leq 
  C_l \norm{f}_{Y_l} 
  \vpran{\sum_{\alpha}  \norm{W^{l - |\alpha|} \del^\alpha_x g}_{L^2 (\dx; H^s_{\gamma/2}(\dv))}}
   \vpran{\sum_{\alpha}  \norm{W^{l - |\alpha|} \del^\alpha_x h}_{L^2 (\dx; H^s_{\gamma/2}(\dv))}} \nn
\\
& 
  + C_l \norm{g}_{Y_l} 
  \vpran{\sum_{\alpha}  \norm{W^{l - |\alpha|} \del^\alpha_x f}_{L^2 (\dx; H^s_{\gamma/2}(\dv))}}
   \vpran{\sum_{\alpha}  \norm{W^{l - |\alpha|} \del^\alpha_x h}_{L^2 (\dx; H^s_{\gamma/2}(\dv))}} \,.
\end{align}
\end{prop}
\begin{proof}
By the Lebniz rule for the bilinear operator, the commutator satisfies
\begin{align*}
   \del^\alpha_x Q(F, g) - Q \vpran{F, \, \del^\alpha_x g}
&= \sum_{|\alpha_1| \neq 0}
     C^{\alpha_1, \alpha_2}
     Q \vpran{\del^{\alpha_1}_x F, \, \del^{\alpha_2}_x g} 
= \sum_{|\alpha_1| \neq 0}
     C^{\alpha_1, \alpha_2}
     Q \vpran{\del^{\alpha_1}_x f, \, \del^{\alpha_2}_x g} \,.
\end{align*}
For each $(\alpha_1, \alpha_2) \neq (0, 2)$, we have
\begin{align*}
& \quad \,
   \abs{\int_{\R^3} Q \vpran{\del^{\alpha_1}_x f, \, \del^{\alpha_2}_x g} 
   \vpran{W^{2\vpran{l - |\alpha|}} \del^\alpha_x h} \dv}
\\
& \leq
    \abs{\int_{\R^3} \vpran{W^{l - |\alpha|}Q \vpran{\del^{\alpha_1}_x f, \, \del^{\alpha_2}_x g} 
                 - Q \vpran{\del^{\alpha_1}_x f, \, W^{l - |\alpha|}\del^{\alpha_2}_x g} }
   \vpran{W^{l - |\alpha|} \del^\alpha_x h} \dv}
\\
& \quad \,
   + \abs{\int_{\R^3} Q \vpran{\del^{\alpha_1}_x f, \, W^{l - |\alpha|}\del^{\alpha_2}_x g} 
   \vpran{W^{l - |\alpha|} \del^\alpha_x h} \dv}
\,\, \Denote  T_{7,1} + T_{7,2} \,.
\end{align*}
By the trilinear estimate given in Proposition~\ref{prop:trilinear} with $(m, \sigma) = (0, 0)$, we have
\begin{align*}
  T_{7, 2}
&\leq
    C_l \norm{\del^{\alpha_1}_x f}_{L^1_{\gamma + 2s} \cap L^2}
     \norm{W^{l - |\alpha|}\del^{\alpha_2}_x g}_{H^{s}_{\gamma/2 + 2s}}
     \norm{W^{l - |\alpha|} \del^\alpha_x h}_{H^s_{\gamma/2}} 
\\
& \leq
   C_l \norm{\vint{v}^{4+\gamma}\del^{\alpha_1}_x f}_{L^2_v}
     \norm{W^{l - |\alpha|}\del^{\alpha_2}_x g}_{H^{s}_{\gamma/2 + 2s}}
     \norm{W^{l - |\alpha|} \del^\alpha_x h}_{H^s_{\gamma/2}} \,.
\end{align*}
To bound $\int_{\T^3} T_{7, 2} \dx$, we consider
 the two cases: $|\alpha_1| = |\alpha_2| = 1$ and $\alpha_2 = 0$. If $|\alpha_1| = |\alpha_2| = 1$, then let 
\begin{align*}
   q = \frac{3}{\frac{1}{2} + \frac{2s}{m_0}} \,,
\qquad
   \frac{2}{p} + \frac{2}{q} = 1\,.
\end{align*}
Note that in this case,  $q > 3$ and $p \in (2, 6)$ because $m_0 > 4s$. Recall that in $\R^3$, we have the Sobolev embeddings
\begin{align*}
   H^1(\R^3) \hookrightarrow H^{\frac{3}{2} - \frac{3}{p}}(\R^3)
   \hookrightarrow L^p(\R^3) \,,
\qquad
   H^{\frac{3}{2} - \frac{3}{q}}(\R^3) \hookrightarrow L^q(\R^3)
\end{align*}
with $\frac{3}{2} - \frac{3}{p} = 1 - \frac{2s}{m_0}$. Hence,
\begin{align} \label{bound:T-7-2}
\begin{split}
& \quad \,
   \int_{\T^3} \norm{\vint{v}^{4+\gamma}\del^{\alpha_1}_x f}^2_{L^2_v}
     \norm{W^{l - |\alpha|}\del^{\alpha_2}_x g}^2_{H^{s}_{\gamma/2 + 2s}} \dx 
\\
& = \int_{\T^3} \norm{\vint{v}^{4+\gamma}\del^{\alpha_1}_x f}^2_{L^2_v}
     \norm{W^{l - 1 - (1 - 2s/m_0)}\del^{\alpha_2}_x g}^2_{H^{s}_{\gamma/2}} \dx
\\
& \leq
   \vpran{\int_{\T^3} \norm{\vint{v}^{4+\gamma}\del^{\alpha_1}_x f}^p_{L^2_v} \dx}^{\frac{2}{p}}
   \vpran{\int_{\T^3} \norm{W^{l - 1 - (1 - 2s/m_0)}\del^{\alpha_2}_x g}^q_{H^{s}_{\gamma/2}} \dx}^{\frac{2}{q}} 
\\
& \leq 
   C \norm{\vint{v}^{4+\gamma} \vint{D_x}^2 f}^2_{L^2_{x, v}}
   \vpran{\int_{\T^3} \norm{W^{l - (2 - 2s/m_0)} \vint{D_x}^{2 - 2s/m_0} g}^2_{H^{s}_{\gamma/2}} \dx}  
\\
&\leq 
  C \norm{f}^2_{Y_l} \sum_{\alpha}
  \norm{W^{l - |\alpha|} \del^\alpha_x g}_{L^2 (\dx; H^s_{\gamma/2}(\dv))}^2 \,.
\end{split}
\end{align}
The bound for $T_{7, 2}$ with $(|\alpha_1|, |\alpha_2|) = (2, 0)$ also follows from the Sobolev embedding. In this case, we have $2 - \frac{2s}{m_0} > 1/2$, which implies that 
\begin{align*}
   H^{2 - \frac{2s}{m_0}}(\R^d) \hookrightarrow L^\infty(\R^3) \,.
\end{align*}
Hence,
\begin{align*}
  \sup_{\T^3} \norm{W^{l - |\alpha|}\del^{\alpha_2}_x g}^2_{H^{s}_{\gamma/2 + 2s}}
= \sup_{\T^3} \norm{W^{l - (2 - \frac{2s}{m_0})}  g}^2_{H^{s}_{\gamma/2}}
&\leq 
   \int_{\T^3} \norm{W^{l - (2 - \frac{2s}{m_0})} \vint{D_x}^{(2 - \frac{2s}{m_0})}  g}^2_{H^{s}_{\gamma/2}} \dx \\
& \leq 
   \sum_\alpha \norm{W^{l - |\alpha|} \del^\alpha_x g}^2_{L^2(\dx; H^s_{\gamma/2}(\dv))} \,.
\end{align*}
Therefore, the bound in~\eqref{bound:T-7-2} also holds when $(\alpha_1, \alpha_2) = (2, 0)$. Applying such bound, we obtain that
\begin{align} \label{bound:T-7-2}
  \int_{\T^3} T_{7, 2} \dx
&\leq
  C \norm{f}_{Y_l} 
  \vpran{\sum_{\alpha}  \norm{W^{l - |\alpha|} \del^\alpha_x g}_{L^2 (\dx; H^s_{\gamma/2}(\dv))}}
   \vpran{\sum_{\alpha}  \norm{W^{l - |\alpha|} \del^\alpha_x h}_{L^2 (\dx; H^s_{\gamma/2}(\dv))}} \,.
\end{align}

By the definition of $Q$, the term $T_{7,1}$ has the form
\begin{align*}
   T_{7,1}
=   \abs{\IntRRS b(\cos\theta) |v - v_\ast|^\gamma
         \vpran{W^{l - |\alpha|}(v') - W^{l - |\alpha|}(v)} 
         \vpran{\del^{\alpha_1}_x f} 
         \vpran{\del^{\alpha_2}_x g} \, 
         \vpran{W^{l - |\alpha|} \del_x^\alpha h}  \dbmu}.
\end{align*}
By the commutator estimate in Proposition~\ref{prop:trilinear-weight}, we have
\begin{align*}
    T_{7, 1}
&\leq
   C_l  \norm{\del^{\alpha_1}_x f}_{L^1_{3+\gamma+2s} \cap L^2}
   \norm{W^{l - |\alpha|}
                   \del^{\alpha_2}_x g}_{H^{s'}_{\gamma'/2}(\dv)}
   \norm{W^{l - |\alpha|}
             \del^{\alpha}_x h}_{L^2_{\gamma/2}(\dv)}
\\
& \quad \,
   + C_l 
   \norm{\vint{v}^{1+\gamma} \del^{\alpha_2}_x g}_{L^1_v}
   \norm{W^{l - |\alpha|}
                \del^{\alpha_1}_x f}_{L^2_{\gamma/2}(\dv)}
   \norm{W^{l - |\alpha|}
             \del^{\alpha}_x h}_{L^2_{\gamma/2}(\dv)} \,.
\end{align*}
The upper bound of $T_{7,1}$ is derived in a similar way as that for $T_{7,2}$ by Sobolev embeddings in $\R^3_x$.  Therefore, we have
\begin{align} \label{bound:T-7-1}
\begin{split}
  \int_{\T^3} T_{7, 1} \dx
&\leq
  C_l \norm{f}_{Y_l} 
  \vpran{\sum_{\alpha}  \norm{W^{l - |\alpha|} \del^\alpha_x g}_{L^2 (\dx; H^s_{\gamma/2}(\dv))}}
   \vpran{\sum_{\alpha}  \norm{W^{l - |\alpha|} \del^\alpha_x h}_{L^2 (\dx; H^s_{\gamma/2}(\dv))}} 
\\
& \quad \,  
  + C_l \norm{g}_{Y_l} 
  \vpran{\sum_{\alpha}  \norm{W^{l - |\alpha|} \del^\alpha_x f}_{L^2 (\dx; H^s_{\gamma/2}(\dv))}}
   \vpran{\sum_{\alpha}  \norm{W^{l - |\alpha|} \del^\alpha_x h}_{L^2 (\dx; H^s_{\gamma/2}(\dv))}} \,.
\end{split}
\end{align}
The estimate  in~\eqref{bound:commut-deriv-x} is then obtained by combining ~\eqref{bound:T-7-2} with~\eqref{bound:T-7-1}. 
\end{proof}


\section{Spectral properties of $\CalL$ in $L^{2}_{x,v}$}
In this section, we establish the spectral analysis of the linearized operator $\CalL$ defined in~\eqref{def:L}. This will play a key role in controlling the linear growth of the nonlinear equation when performing energy estimates. 

\subsection{Spectral Analysis of Linearized operator in Gaussian-weighted $L^{2}_{x,v}$}
First we study the spectrum of the operator $L^{(\mu)}: \CalD(L^{(\mu)})\rightarrow L^{2}_{v}$, defined on a dense subset of $\CalD(L^{(\mu)}) \subseteq L^{2}_{v}$, where
\begin{align*}
L^{(\mu)}(h) :&= \mu^{-1/2}\Big( Q(\mu^{1/2}\,h , \mu) + Q(\mu, \mu^{1/2}\,h)\Big)\\
& = \mu^{1/2}\int_{\mathbb{R}^{3}}|u|^{\gamma}\,b(\widehat{u}\cdot{\sigma})\mu_{*}\int_{\mathbb{S}^{2}}\bigg(\frac{h'}{\mu'^{1/2}} + \frac{h'_{*}}{\mu'^{1/2}_{*}} - \frac{h}{\mu^{1/2}} -\frac{h_{*}}{\mu^{1/2}_{*}} \bigg)\text{d}\sigma\text{d}v_{*}\,,
\end{align*}
where $\mu$ is the normalized global Maxwellian. The kernel of $L^{(\mu)}$ in $L^{2}_{v}$ is given by
\begin{equation*}
\text{Ker}(L^{(\mu)}) 
= \text{Span}\big\{\sqrt{\mu}, \,\, v \sqrt{\mu}, \,\, |v|^{2} \sqrt{\mu} \big\}\,.
\end{equation*}
Thus, the generators of the kernel are in the Schwartz space $\mathcal{S}$.  
The projection onto $\text{Ker}(L^{(\mu)})$ is defined as
\begin{equation*}
\pi(h)\Denote\sum_{\varphi\in\text{Ker}(\mathcal{L}^{(\mu)})}\bigg(\int_{\mathbb{R}^{3}}h\,\varphi\,\text{d}v\bigg)\,\varphi\,.
\end{equation*}
Let us first address the decomposition of $L^{(\mu)}$ which is based on truncations of small and large velocities, and grazing angles.  This decomposition is a bit different from the classical decomposition made in the spectral analysis in the cutoff case.  Recall that the scattering kernel $b$ satisfies~\eqref{assump:b-gamma}. 
We use the decomposition
\begin{equation}\label{DAK1}
b\vpran{\cos\theta} = b\vpran{\cos\theta} \big(\text{1}_{|\sin(\theta)| \geq \varepsilon} + \text{1}_{|\sin(\theta)| < \varepsilon}\big)\Denote b_{1}\vpran{\cos\theta} + b_{2}\vpran{\cos\theta}\,.
\end{equation}
For the kinetic potential write $|\cdot|^{\gamma}\Denote\Phi_{1}+\Phi_{2}$, with $\gamma\in(0,1]$, where
\begin{equation}\label{DAK2}
\Phi_{1}(|u|)\Denote |u|^{\gamma}{\chi_{\delta\leq |u|\leq \delta^{-1}}}\,,\qquad \Phi_{2}(|u|)\Denote |u|^{\gamma}\big(1-{\chi_{\delta\leq |u|\leq \delta^{-1}}}\big)\,.
\end{equation}
Here ${\chi_{\delta\leq |u|\leq \delta^{-1}}}$ is a smooth version of the indicator $\One_{\delta\leq |u|\leq \delta^{-1}}$.  Also, denote $L^{(\mu)}$ by $L^{(\mu)}_{\Phi,b}$ to emphasize the kernel dependence and then decompose it as
\begin{align}\label{Decomposition}
\begin{split}
L^{(\mu)}_{\Phi,b} 
&=L^{(\mu)}_{\Phi,b_{1}} + L^{(\mu)}_{\Phi,b_{2}} 
= L^{(\mu)}_{\Phi_{1},b_{1}} + L^{(\mu)}_{\Phi_{2},b_{1}} + L^{(\mu)}_{\Phi,b_{2}}
\\
& = \Big(L^{(\mu)}_{\Phi_{1},b_{1}} + \mu^{-1/2}Q^{-}_{\Phi_{1},b_1}\big(\mu,\mu^{1/2}\,h\big) \Big)
+ \Big(L^{(\mu)}_{\Phi_{2},b_{1}} - \mu^{-1/2}Q^{-}_{\Phi_{1},b_1}\big(\mu,\mu^{1/2}\,h\big) + L^{(\mu)}_{\Phi,b_{2}}\Big) 
\\
& \Denote \mathcal{K}_{\delta,\varepsilon} - \Lambda_{\delta,\varepsilon} \,.
\end{split}
\end{align}
The operators $\mathcal{K}_{\delta,\varepsilon}$ and $\Lambda_{\delta,\varepsilon}$ are self-adjoint in $L^{2}_{v}$ since $L^{(\mu)}$ is self-adjoint in $L^{2}_{v}$ for any reasonable kinetic kernel $\Phi(u)b(\cos\theta)$ (see \cite[Chapter 7]{CIP}) and $\mu^{-1/2}Q^{-}\big(\mu,\mu^{1/2}\,h\big)$ is a multiplication operator.  The operator $\Lambda_{\delta, \Eps}$ include all the singular features of $L^{(\mu)}$ in terms of tails and regularization.  
%
%
%

The linearization that we make in this subsection is $f = \mu + \mu^{1/2}\,h$.  
The full equation for $h$ reads
\begin{equation}\label{Boltzmann-equation}
\partial_{t}h = \mu^{-1/2}Q(\mu^{1/2}\,h, \mu^{1/2}\,h) + {L}^{(\mu)}(h) - v\cdot\nabla_{x}h\,.
\end{equation}
In this way, we want to study the $L^{2}_{x,v}$ spectral properties of the operator
\begin{equation*}
{L}^{(\mu)}(h) - v\cdot\nabla_{x}h\,.
\end{equation*}
A spectral gap in $H^{1}_{x,v}$ was found for \cite{MN} for this operator in the cutoff case using a combination of spectral theory and energy estimates.  The proof follows after checking some structural conditions and \textit{a priori} estimates satisfied by ${L}^{(\mu)}$.  This approach does not seem to apply directly to the non-cutoff case.
Here we give a more ``perturbation-type'' argument that works in both cutoff and non-cutoff cases.   
\subsubsection{Dissipative part}  Let us prove that for $\delta>0$ and $\varepsilon>0$ sufficiently small, the operator $-\big(\Lambda_{\delta,\varepsilon} + v\cdot\nabla_{x}\big)$ is dissipative in $L^{2}_{x,v}$.  The operator $\Lambda_{\delta,\varepsilon}$ is composed of two singular parts such that $\Lambda_{\delta,\varepsilon} = \Lambda_{1} + \Lambda_{2}$, where $\Lambda_1$ is related to the growth in velocity (tails)
\begin{equation*}
- \Lambda_{1} \Denote {L}^{(\mu)}_{\Phi_{2},b_{1}} - \mu^{-1/2}Q^{-}_{\Phi_{1},b_1}\big(\mu,\mu^{1/2}\,h\big)\,,
\end{equation*}
and $\Lambda_2$ is related to regularity
\begin{equation*}
- \Lambda_{2}\Denote {L}^{(\mu)}_{\Phi,b_{2}}\,.
\end{equation*}
\begin{proposition}[Singular part $\Lambda_{2}$]\label{pro-lambda2}
There exist constants $c>0$ and $C>0$ depending only on mass and energy of $\mu$, and $\kappa_0>0$ depending only on $\Phi=|\cdot|^{\gamma}$, such that for any $s\in(0,1)$ and $\varepsilon\in(0,1/5]$, we have
\begin{equation*}
\langle {L}^{(\mu)}_{\Phi,b_{2}}(h),h\rangle \leq -c\kappa_0\sum_{g\in\{h^{\pm}\}}\Big\|\widehat{\langle \cdot \rangle^{\gamma/2} g}(\xi)\,|\xi|^{s} \textbf{1}\big\{|\xi|\geq\tfrac{1}{\varepsilon}\big\}\Big\|^{2}_{L^{2}_{\xi}} + C\| 
\theta^2\, b_{2}\|_{L^{1}_{\theta}}\|\langle{v}\rangle^{\gamma/2} h \|^{2}_{L^{2}_{v}}\,,
\end{equation*}
where $h^\pm$ are the positive and negative parts of $h$ respectively. We remark that the constants $c,\,C$ and $\kappa_0$ are independent of $\varepsilon$.
\end{proposition}
\begin{proof}
Note that
\begin{equation}\label{p1-lambda2-e1}
\langle {L}^{(\mu)}_{\Phi,b_{2}}(h),h\rangle= \langle\mu^{-1/2}Q_{\Phi,b_{2}}(\mu,\mu^{1/2}h),h\rangle+\langle\mu^{-1/2}Q_{\Phi,b_{2}}(\mu^{1/2}h,\mu),h\rangle\,.
\end{equation}
For the first term in the right side of \eqref{p1-lambda2-e1} it follows
\begin{align}\label{p1-lambda2-e2}
\langle\mu^{-1/2}Q_{\Phi,b_{2}}(\mu,\mu^{1/2}h),h\rangle 
& = \int_{\mathbb{R}^{6}}\int_{\mathbb{S}^{2}}\Phi(u)b_{2}(\cos\theta)\mu(v_{*})\mu(v)H(v)\big(H(v') - H(v)\big)  \nn
\\
&= -\tfrac{1}{2}\int_{\mathbb{R}^{6}}\int_{\mathbb{S}^{2}}\Phi(u)b_{2}(\cos\theta)\mu(v_{*})\mu(v)\big(H(v') - H(v)\big)^{2}
\\
& \quad \,
  +\tfrac{1}{2}\int_{\mathbb{R}^{6}}\int_{\mathbb{S}^{2}}\Phi(u)b_{2}(\cos\theta)\mu(v_{*})\mu(v)\big(H(v')^{2} - H(v)^{2}\big)\,, \nn
\end{align}
where $u = v - v_\ast$ and $H(v) = \mu^{-1/2}(v)h(v)$.  Since $\mu\mu_{*}=\mu'\mu'_{*}$, the last term in the right side of \eqref{p1-lambda2-e2} is zero.  Using the technique of proof of \cite[Proposition 2.1]{AMUXY2012JFA} and Lemma \ref{A1}, we have
\begin{align}\label{p1-lambda1-e2.5}
\begin{split}
& \quad \,
 -\tfrac{1}{2}\int_{\mathbb{R}^{6}}\int_{\mathbb{S}^{2}}\Phi(u)b_{2}(\cos\theta)\mu(v_{*})\mu(v)\big(H(v') - H(v)\big)^{2}
 \\
&\leq -c\kappa_0\sum_{g\in\{h^{\pm}\}}\Big\|\widehat{\langle \cdot \rangle^{\gamma/2}g}(\xi)\,|\xi|^{s} \textbf{1}\big\{|\xi|\geq\tfrac{1}{\varepsilon}\big\}\Big\|^{2}_{L^{2}_{\xi}} + C\|\theta^{2} b_{2}\|_{L^{1}_{\theta}}\|\langle{v}\rangle^{\gamma/2} h \|^{2}_{L^{2}_{v}}\,,
\end{split}
\end{align}
where the constants $c, C>0$ depend only on mass and energy of $\mu$, and $\kappa_0>0$ only on the potential $\Phi$.  For the second term in \eqref{p1-lambda2-e1} we can use  \cite[Lemma 2.15]{AMUXY2012JFA} with $B$ replaced by $\Phi b_2$, so that we obtain
\begin{align}\label{p1-lambda1-e4}
\begin{split}
\langle\mu^{-1/2}Q_{\Phi,b_{2}}(\mu^{1/2}h,\mu),h\rangle  &\leq C\|\theta^{2}b_{2}\|_{L^{1}_{\theta}}\|\mu^{1/10^3}h\|^2_{L^{2}_{v}}
\,.
\end{split}
\end{align}
The proposition follows from \eqref{p1-lambda1-e2.5} and \eqref{p1-lambda1-e4}. 
\end{proof}

\begin{proposition}[Singular part $\Lambda_1$]\label{pro-lambda1}
For every $\varepsilon>0$ there exist constants $c>0$ and $C>0$, depending only on the mass and energy of $\mu$, such that for any $s\in(0,1)$, $\varepsilon\in(0,1)$, and $\delta^{\gamma}\in(0,\frac{c}{2C}\varepsilon^{2\gamma})$, it follows that
\begin{equation*}
\langle {L}^{(\mu)}_{\Phi_{2},b_{1}}h - \mu^{-1/2}Q^{-}_{\Phi_{1},b_1}\big(\mu,\mu^{1/2}\,h\big),h\rangle\leq - \tfrac{c}{2}\|b_{1}\|_{L^{1}_{\theta}}\|\langle v \rangle^{\gamma/2}h\|^{2}_{L^{2}_{v}}\,.
\end{equation*}
The constants are independent of both $\delta>0$ and $\varepsilon>0$.
\end{proposition}
\begin{proof}
Note that
\begin{align*}
-\Lambda_{1}h = {L}^{(\mu)}_{\Phi_{2},b_{1}}h - \mu^{-1/2}Q^{-}_{\Phi_{1},b_1}\big(\mu,\mu^{1/2}\,h\big)
&=\mu^{-1/2}Q^{+}_{\Phi_{2},b_{1}}(\mu^{1/2}h,\mu) - \mu^{-1/2}Q^{-}_{\Phi_{2},b_{1}}(\mu^{1/2}h,\mu)\\
& \quad \,
+\mu^{-1/2}Q^{+}_{\Phi_{2},b_{1}}(\mu,\mu^{1/2}h) -\mu^{-1/2}Q^{-}_{\Phi,b_{1}}(\mu,\mu^{1/2}h)\,.
\end{align*}
The first three terms in the right side are treated similarly. Let us proceed with one of them and leave the other two to 
the reader. 
Note that
\begin{equation}\label{p2-lambda2-e0.9}
\langle\mu^{-1/2}Q^{+}_{\Phi_{2},b_{1}}(\mu,\mu^{1/2}h),h\rangle = \int_{\mathbb{R}^{6}}\int_{\mathbb{S}^{2}} \Phi_{2}(u)b_{1}(\cos\theta)\mu^{1/2}(v_{*}) h(v) \mu^{1/2}(v'_{*})h(v')\,.
\end{equation}
Since $b_{1}$ is supported in $|\sin(\theta)|\geq\varepsilon$ one has
\begin{equation*}
|v'_{*}|\geq |u||\sin\frac{\theta}{2}| - |v_*|\geq 
\frac{\varepsilon}{2} |v|  -  2|v_*| \,, 
\end{equation*}
thus,
\begin{equation*}
\mu^{1/32}(\varepsilon v/2 )\geq \mu^{1/16}(v'_*)\mu^{1/4}(v_{*})\,.
\end{equation*}
Plugging this inequality in \eqref{p2-lambda2-e0.9} and recalling that $\Phi_{2}$ is supported on $\{|u|\leq\delta\}\cup\{|u|\geq\delta^{-1}\}$ (so that $\Phi_{2}(|u|)\leq \delta^{\gamma}\langle v_{*}\rangle^{2\gamma}\langle v \rangle^{2\gamma}$), one concludes that the right side of \eqref{p2-lambda2-e0.9} is controlled by
\begin{align*}
&\int_{\mathbb{R}^{6}}\int_{\mathbb{S}^{2}} \Phi_{2}(u)b_{1}(\cos\theta)\mu^{1/4}(v_{*})\mu^{1/32}(\varepsilon v/2) h(v) h(v')\\
&\hspace{.5cm}\leq \delta^{\gamma}\int_{\mathbb{R}^{6}}\int_{\mathbb{S}^{2}}b_{1}(\cos\theta)\langle v_{*}\rangle^{2\gamma}\mu^{1/4}(v_{*})\langle v \rangle^{2\gamma}\mu^{1/32}(\varepsilon v/2 ) h(v) h(v')\\
&\hspace{1cm}\leq\frac{\delta^{\gamma}}{\varepsilon^{2\gamma}}\sup_{x\in\mathbb{R}^{3}}\langle x \rangle^{2\gamma}\mu^{1/32}(x/2)\int_{\mathbb{R}^{6}}\int_{\mathbb{S}^{2}}b_{1}(\cos\theta)\langle v_{*}\rangle^{2\gamma}\mu^{1/4}(v_{*}) h(v) h(v')\,.
\end{align*}
As a consequence,
\begin{equation}\label{p2-lambda2-e1}
\langle\mu^{-1/2}Q^{+}_{\Phi_{2},b_{1}}(\mu,\mu^{1/2}h),h\rangle\leq \frac{C\delta^{\gamma}}{\varepsilon^{2\gamma}}\langle Q^{+}_{1,b_{1}}(\langle \cdot \rangle^{2\gamma}\mu^{1/4},h),h\rangle\leq \frac{C\delta^{\gamma}}{\varepsilon^{2\gamma}}\|b_{1}\|_{L^{1}_{\theta}}\|\langle \cdot \rangle^{2\gamma}\mu^{1/4}\|_{L^{1}_{v}}\|h\|^{2}_{L^{2}_{v}}\,.
\end{equation}
Now, for the last term, it follows readily
\begin{equation}\label{p2-lambda2-e2}
\langle\mu^{-1/2}Q^{-}_{\Phi,b_{1}}(\mu,\mu^{1/2}h),h\rangle \geq c\|b_{1}\|_{L^{1}_{\theta}}\|\langle v\rangle^{\gamma/2} h\|^{2}_{L^{2}_{v}}\,, 
\end{equation}
with $c>0$ depending only on mass and energy of $\mu$.  The result follows from \eqref{p2-lambda2-e1} and \eqref{p2-lambda2-e2}.
\end{proof}
\begin{theorem}\label{t1-dissipative}
Let $h\in H^{s}_{\gamma/2}(\dv)$ with $s\in(0,1)$.  There exist constants $c>0$, $C>0$, $c_o>0$, and $\varepsilon_o>0$ depending only on mass and energy of $\mu$ such that for any $\varepsilon\in(0,\varepsilon_o]$ and $\delta^{\gamma}\in(0,\frac{c}{2C}\varepsilon^{2\gamma})$, the operator $-\Lambda_{\delta,\varepsilon}$ satisfies the dissipative estimate
\begin{equation*}
\langle -\Lambda_{\delta,\varepsilon}(h),h)\rangle \leq -c_o\,\kappa_0\|\langle v \rangle^{\gamma/2}h\|^{2}_{H^{s}_{v}}\,.
\end{equation*}
The constant $\kappa_0>0$ was introduced in Proposition \ref{pro-lambda2}.
\end{theorem}
\begin{proof}
Using Propositions \ref{pro-lambda2} and \ref{pro-lambda1} we have that
\begin{align*}
- \langle \Lambda_{\delta,\varepsilon}(h),h)\rangle \leq -c\kappa_0\sum_{g\in\{h^{\pm}\}}\Big\|\widehat{\langle \cdot \rangle^{\gamma/2}g}&(\xi)\,|\xi|^{s} \textbf{1}\big\{|\xi|\geq\tfrac{1}{\varepsilon}\big\}\Big\|^{2}_{L^{2}_{\xi}} - \tilde{c}\,\|b_{1}\|_{L^{1}_{\theta}}\|\langle v \rangle^{\gamma/2}h\|^{2}_{L^{2}_{v}}\\
&+ C\|\theta^2 b_{2}\|_{L^{1}_{\theta}}\|\langle{v}\rangle^{\gamma/2} h \|^{2}_{L^{2}_{v}}\,.
\end{align*}
Note that $\|b_{1}\|_{L^{1}_{\theta}}\sim\kappa_0\varepsilon^{-2s}$ and $\|\theta^2 b_{2}\|_{L^{1}_{\theta}}\sim\kappa_0\varepsilon^{2-2s}$.  Thus, we may choose any $\varepsilon$ such that
\begin{equation*}
\varepsilon \leq \min\Big\{ \Big(\tfrac{\tilde{c}}{2c}\Big)^{\frac{1}{2s}}, \Big(\tfrac{\tilde{c}}{C}\Big)^{\frac{1}{2-2s}} \Big\} \Denote \varepsilon_o\,,
\end{equation*}
and obtain
\begin{align*}
- \langle \Lambda_{\delta,\varepsilon}(h),h)\rangle &\leq -c\kappa_0\bigg(\sum_{g\in\{h^{\pm}\}}\Big\|\widehat{\langle \cdot \rangle^{\gamma/2} g}(\xi)\,|\xi|^{s} \textbf{1}\big\{|\xi|\geq\tfrac{1}{\varepsilon}\big\}\Big\|^{2}_{L^{2}_{\xi}} + \|\langle v \rangle^{\gamma/2}h\|^{2}_{L^{2}_{v}}\bigg)\\
&\leq - 2c_o\kappa_0\sum_{g\in\{h^{\pm}\}}\|\langle v \rangle^{\gamma/2}g\|^{2}_{H^{s}_{v}}\leq - c_o\kappa_0\|\langle v \rangle^{\gamma/2}h\|^{2}_{H^{s}_{v}}\,.
\end{align*}
This proves the result after using Lemma \ref{A2} in the last inequality. 
\end{proof}
In the sequel, we fix $\varepsilon=\varepsilon_o$ and $\delta_o : =\Big(\frac{c}{2C}\Big)^{1/\gamma}\varepsilon^{2}_o$.  We also set $\Lambda\Denote\Lambda_{\delta_o,\varepsilon_o}$ and denote the dissipative operator as
\begin{equation*}
\mathcal{L}^{(\mu)}_{D}\Denote -\big(\Lambda + v\cdot\nabla_{x}\big).
\end{equation*} 
This operator is closed in $L^{2}_{x,v}$. 
\begin{proposition}\label{c1-dissipative}
The spectrum of $\mathcal{L}^{(\mu)}_{D}$, as operator on $L^{2}_{x,v}$, lies in $\{z\in\mathbb{C}:\mathfrak{R}_e{z}\leq-c_o\kappa_0\}$.
\end{proposition}
\begin{proof}
Note that the domain $\mathcal{D}(\mathcal{L}^{(\mu)}_{D})$ is dense in ${L^{2}_{x,v}}$, for instance, it contains $\mathcal{C}^{1}_{x,v}-$functions with strong velocity decay.  We now prove the existence and uniqueness of the problem
\begin{equation}\label{c1-p1}
\big(-\mathcal{L}^{(\mu)}_{D} + \lambda\,\textbf{1}\big)u=f\in L^{2}_{x,v}\,,\quad \lambda\in\mathbb{C}\,.
\end{equation}
Writing $u=u_{R} + i\,u_{I}$, $f=f_{R} + i\,f_{I}$, and $\lambda=\lambda_{R} + i\,\lambda_{I}$, problem \eqref{c1-p1} is equivalent to the 2-system of real valued problems
\begin{equation}\label{c1-s-p1}
\Bigg( \big(-\mathcal{L}^{(\mu)}_{D}
+ \lambda_{R}\big)\textbf{1}_{2\times2} + \lambda_{I}\,
\bigg[\begin{array}{cc}
0 & -1 \\
1 & 0
\end{array}\bigg]\Bigg)
\bigg[\begin{array}{c}
u_{R}\\
u_{I}
\end{array}\bigg] =
\bigg[\begin{array}{c}
f_{R}\\
f_{I}
\end{array}\bigg]\,.
\end{equation}
We start perturbing this problem to
\begin{equation}\label{c1-s-p2}
\Bigg( \big(-\mathcal{L}^{(\mu)}_{D} - \epsilon\, \mathcal{L}_{P}
+ \lambda_{R}\big)\textbf{1}_{2\times2} + \lambda_{I}\,
\bigg[\begin{array}{cc}
0 & -1 \\
1 & 0
\end{array}\bigg]\Bigg)
\bigg[\begin{array}{c}
u_{R}\\
u_{I}
\end{array}\bigg] =
\bigg[\begin{array}{c}
f_{R}\\
f_{I}
\end{array}\bigg]\,,
\end{equation}
where
\begin{equation*}
\mathcal{L}_{P}:= \Big(\nabla_{v}\cdot \langle v \rangle^{\gamma+2}\nabla_{v} + \langle v \rangle^{\gamma+2}\Delta_{x} - \langle v \rangle^{\gamma+2}\Big)\textbf{1}_{2\times2}\,.
\end{equation*}
This leads us to introduce the bilinear form $B^{\epsilon}[\cdot,\cdot]: \mathcal{H}\times\mathcal{H} \rightarrow \mathbb{R}$ with Hilbert space given by
\begin{equation*}
\mathcal{H}:=H^{1}_{x,v}\big(\langle v \rangle^{\gamma/2+1}\big)\times H^{1}_{x,v}\big(\langle v \rangle^{\gamma/2+1}\big)\,,
\end{equation*}
and, using the definition of $\mathcal{L}^{(\mu)}_{D}$ and $\mathcal{L}_{P}$, weak formulation
\begin{align*}
&B^{\epsilon}[u,w] := \int_{\mathbb{T}^{3}}\int_{\mathbb{R}^{3}}\bigg( \big(\Lambda u\big)\cdot w  - u\cdot\big(v\cdot\nabla_{x}w\big) \\
&\hspace{-0.5cm}+ \epsilon\,\langle v \rangle^{\gamma+2}\Big(\nabla_{v}u\cdot\nabla_{v}w + \nabla_{x}u\cdot\nabla_{x}w + u\cdot w\Big) + \lambda_{R}\,u\cdot w + \lambda_{I}\,
\bigg[\begin{array}{cc}
0 & -1 \\
1 & 0
\end{array}\bigg]u\cdot w\bigg)\text{d}v\text{d}x\,.
\end{align*}  
Thanks to Proposition \ref{prop:trilinear} and the arguments given previously in this section, it follows that
\begin{equation*}
\big| B^{\epsilon}[u,w] \big| \leq (C(\mu) + |\lambda| + \epsilon)\|u\|_{\mathcal{H}}\|w\|_{\mathcal{H}}\,.
\end{equation*}
In addition, thanks to Theorem \ref{t1-dissipative}, as long as $\lambda_{R} + c_o\kappa_0>0$ it follows that
\begin{equation}\label{p2.5}
B^{\epsilon}[u,u] \geq (c_o\kappa_0 + \lambda_{R})\int_{\mathbb{T}^{3}}\|\langle v \rangle^{\gamma/2}u\|^{2}_{H^{s}_{v}\times H^{s}_{v}}\text{d}x + \epsilon\,\|u\|^{2}_{\mathcal{H}}\geq \epsilon\,\|u\|^{2}_{\mathcal{H}}\,.
\end{equation}
Note that the antisymmetric term related to $\lambda_{I}$ vanishes.  As a consequence, invoking Lax-Milgram theorem, for any $f$ in the dual of $\mathcal{H}$ (in particular, for any $f\in L^{2}_{x,v}\times L^{2}_{x,v}$) one has a unique $u\in\mathcal{H}$ such that
\begin{equation*}
B^{\epsilon}[u,w] = \langle f, w \rangle\,,\qquad \forall\,w\in \mathcal{H}\,,\quad \forall\,\epsilon>0\,.
\end{equation*}
This provides existence and uniqueness for problem \eqref{c1-s-p2} as long as $\lambda_{R} + c_o\kappa_0>0$.  Furthermore, using the first inequality in estimate \eqref{p2.5}, one concludes that for $f\in L^{2}_{x,v}\times L^{2}_{x,v}$ the weak solution to problem \eqref{c1-s-p2} satisfies 
\begin{align}\label{p3}
\begin{split}
(c_o\kappa_0&+\lambda_{R})\int_{\mathbb{T}^{3}}\|\langle v \rangle^{\gamma/2}u\|^{2}_{H^{s}_{v}\times H^{s}_{v}}\text{d}x \leq B^{\epsilon}[u,u] = \langle f,u\rangle \leq \|f\|_{L^{2}_{x,v}\times L^{2}_{x,v}}\|u\|_{L^{2}_{x,v}\times L^{2}_{x,v}}\,,\\
&\text{that is}\,,\quad \Bigg(\int_{\mathbb{T}^{3}}\|\langle v \rangle^{\gamma/2}u\|^{2}_{H^{s}_{v}\times H^{s}_{v}}\text{d}x\Bigg)^{1/2} \leq (c_o\kappa_0+\lambda_{R})^{-1}\|f\|_{L^{2}_{x,v}\times L^{2}_{x,v}}\,.
\end{split}
\end{align} 
Now for a fixed $f\in L^{2}_{x,v}\times L^{2}_{x,v}$ take the sequence of solutions $\{u^{\epsilon}\}$ with $\epsilon\rightarrow0$ to problem \eqref{c1-s-p2}.  By previous estimate, there exists (up to a subsequence) a weak limit $u_{f}\in L^{2}_{x,v}\times L^{2}_{x,v}$.  Clearly, such limit satisfies problem \eqref{c1-s-p1} in the sense of distributions\footnote{We note here that each term in the evaluation $\mathcal{L}^{(\mu)}_{D}\textbf{1}_{2\times 2}u_{f}$ is not, in general, an $L^{2}_{x,v}$ function.  However, one has $\mathcal{L}^{(\mu)}_{D}\textbf{1}_{2\times 2}u_{f} = - f + \lambda_{R}u_{f} +\Big[\begin{array}{cc} 0 & -1\\1 & 0 \end{array}\Big]u_{f} \in L^{2}_{x,v}\times L^{2}_{x,v}$.} with estimate \eqref{p3}.  Furthermore, \textit{any} solution to \eqref{c1-s-p1} in $L^{2}_{x,v}\times L^{2}_{x,v}$ satisfies estimate \eqref{p3}.  Therefore, by linearity, solutions are unique in this space. 
Additionally, estimate \eqref{p3} shows that $\mathcal{D}(\mathcal{L}^{(\mu)}_{D})\subset H^{0,s}_{x,v}(\langle v \rangle^{\gamma/2})\subset{L^{2}_{x,v}}$.  This proves that any $\lambda\in\mathbb{C}$ such that $\lambda_{R}>-c_o\kappa_0$ belongs to the range of $\mathcal{L}^{(\mu)}_{D}$.
\end{proof}
\begin{remark} By the same token, the spectrum of $-\Lambda$, as operator on $L^{2}_{v}$, lies in $\{z\in\mathbb{C}:\mathfrak{R}_e{z}\leq-c_o\kappa_0\}$.  Since $\Lambda$ is self-adojint we conclude that $\text{Spectrum}(-\Lambda)\subset(-\infty,-c_o\kappa_0]$.
\end{remark}
\subsubsection{Localization of the spectrum.}
We know that $\mathcal{K}\Denote\mathcal{K}_{\delta_o,\varepsilon_o}$ is continuous in $L^{2}_{v}$, that is,
\begin{equation*}
\|\mathcal{K}(h)\|_{L^{2}_{v}}\leq C(\delta_{o},\varepsilon_o)\|h\|_{L^{2}_{v}}\,,
\end{equation*}
with $C(\delta_o,\varepsilon_o)$ depending only on mass and energy of $\mu$.  Let us prove $\mathcal{K}$ is $\Lambda$-compact by taking a sequence $\{h_{n}\}\subseteq\mathcal{D}(\Lambda)\subseteq L^{2}_{v}$ such that both $\{h_{n}\}$ and $\{\Lambda h_{n}\}$ are bounded in $L^{2}_{v}$.  Then, by Theorem \ref{t1-dissipative}
\begin{equation*}
c_{o}\kappa_0\|\langle v \rangle^{\gamma/2}h_{n}\|^{2}_{H^{s}_{v}}\leq \langle\Lambda h_{n},h_{n}\rangle\leq\|\Lambda h_{n}\|_{L^{2}_{v}}\|h_n\|_{L^{2}_{v}}\,.
\end{equation*}
Thus, $\{h_{n}\}$ is compact in $L^{2}_{v}$.  Using Weyl's theorem for stability of essential spectrum under relative compact (self-adjoint) perturbations we just proved the following result.
\begin{corollary}\label{c2-compact}
The essential spectrum of ${L}^{(\mu)} = \mathcal{K} - \Lambda$, as an operator in $L^{2}_{v}$, lies in $(-\infty,-c_o\kappa_0]$.  In particular, if $0\in\text{Spectrum}({L}^{(\mu)})$, it will be a discrete eigenvalue and, thus, the kernel of ${L}^{(\mu)}$ will be finite dimensional.
\end{corollary}
Recall that the Dirichlet form of ${L}^{(\mu)}$ is non positive, $\langle {L}^{(\mu)}h,h\rangle\leq0$.  This implies, since ${L}^{(\mu)}$ is self adjoint, that the discrete spectrum of ${L}^{(\mu)}$ lies in $(-\infty,0]$.  As a consequence, the restriction of ${L}^{(\mu)}$ to $L^{2}_{v}\setminus\text{Ker}({L}^{(\mu)})$, denoted by ${L}^{(\mu)}_o$, is invertible with inverse $\Big({L^{(\mu)}_o}\Big)^{-1}$ and with domain $\mathcal{D} \Big(\Big({L^{(\mu)}_o\Big)}^{-1}\Big)=L^{2}_{v}\setminus \text{Ker}(L^{(\mu)})$.  This observation together with Corollary \ref{c2-compact} show that, in fact, ${L}^{(\mu)}_o$ has a spectral gap (denoted by $\lambda_o>0$)
\begin{equation*}
\langle L^{(\mu)}_o h , h \rangle \leq -\lambda_{o}\|h\|^{2}_{L^{2}_{v}}\,,\qquad h\in L^{2}_{v}\setminus\text{Ker}({L}^{(\mu)})\,.
\end{equation*}
This leads to the following additional feature in the spectrum.
\begin{proposition}\label{p1-compact}
The eigenvectors associated to the eigenvalues of ${L}^{(\mu)}$ form a basis in $L^{2}_{v}$.
\end{proposition} 
\begin{proof}
Note that $\vpran{{L^{(\mu)}_o}}^{-1}$ is compact.  Indeed, take $\{g_{n}\}$ a bounded sequence in $L^{2}_{v}$ and set $h_{n}= \vpran{{L^{(\mu)}_o}}^{-1} g_n$.  Then, thanks to Theorem \ref{t1-dissipative} and the continuity of $\mathcal{K}$, it follows that
\begin{equation*}
- \langle g_n,h_{n}\rangle = -\langle {L}^{(\mu)}_o(h_{n}),h_{n}\rangle = \langle\Lambda(h_{n}),h_{n}\rangle - \langle\mathcal{K}(h_{n}),h_{n}\rangle \geq c_o\kappa_0\|h_{n}\|^{2}_{H^{s}_{\gamma/2,v}} - C\|h_{n}\|^{2}_{L^{2}_{v}}\,.
\end{equation*}
We conclude by Cauchy-Schwarz inequality that
\begin{equation}\label{p1-compact-e1}
c_o\kappa_0\| h_{n} \|^{2}_{H^{s}_{\gamma/2,v}} \leq \|h_{n}\|_{L^{2}_{v}}\|g_{n}\|_{L^{2}_{v}} + C\|h_{n}\|^{2}_{L^{2}_{v}} \leq \tilde{C}\big(\|g_{n}\|^{2}_{L^{2}_{v}} + \|h_{n}\|^{2}_{L^{2}_{v}}\big)\,.
\end{equation}
Furthermore, since $L^{(\mu)}_{o}$ has spectral gap in $L^{2}_{v}\setminus\text{Ker}({L}^{(\mu)})$, one has that
\begin{equation*}
- \langle g_n,h_{n}\rangle = -\langle {L}^{(\mu)}_o(h_{n}),h_{n}\rangle \geq \lambda_{o}\|h_{n}\|^{2}_{L^{2}_{v}}\,.
\end{equation*}
As a consequence, again by Cauchy-Schwarz inequality,
\begin{equation}\label{p1-compact-e1.5}
\|h_{n}\|^{2}_{L^{2}_{v}}\leq \frac{1}{\lambda^{2}_{o}}\|g_{n}\|^{2}_{L^{2}_{v}}\,.
\end{equation}
Gathering the estimates \eqref{p1-compact-e1} and \eqref{p1-compact-e1.5} lead to
\begin{equation*}
\| h_{n} \|^{2}_{H^{s}_{\gamma/2}(\dv)} \leq \frac{\tilde{C}}{c_o\kappa_0}\Big(1+\frac{1}{\lambda^{2}_{o}}\Big)\|g_{n}\|^{2}_{L^{2}_{v}}\,,
\end{equation*}
which proves that $\vpran{{L^{(\mu)}_o}}^{-1}$ is compact as an operator onto $L^{2}_{v} \setminus \text{Ker}(L^{(\mu)})$.  Being the inverse of a self adjoint operator, it is also self adjoint.  Therefore, by the spectral theorem for compact self adjoint operators, its eigenvectors, and hence the eigenvectors of ${L}^{(\mu)}_{o}$, form a basis of $L^{2}_{v} \setminus \text{Ker}(L^{(\mu)})$. Together with the eigenvectors of the null space one obtains a basis of $L^2_v$ composed of the eigenvectors of $L^{(\mu)}$.
\end{proof}
\begin{remark} Recall that the discrete spectrum of compact operators accumulate at $0$.  As a consequence, the proof of Proposition \ref{p1-compact} shows a difference between the discrete spectrum of $\mathcal{L}^{(\mu)}$ in the cutoff and non cutoff cases.  In the later, the discrete spectrum decreases up to $-\infty$.
\end{remark}
In this final part of the section we localize the spectrum of the operator
\begin{equation*}
{L}^{(\mu)} - v\cdot\nabla_{x} = \mathcal{K} + \mathcal{L}^{(\mu)}_{D} \,.
\end{equation*}
\begin{lemma}\label{l1-compact}
The operator $\mathcal{K}$ is relative compact with respect to $\mathcal{L}^{(\mu)}_{D}$.
\end{lemma}
\begin{proof}
Take a sequence $\{h_{n}\}\subset\mathcal{D}(\mathcal{L}^{(\mu)}_{D})\subset L^{2}_{x,v}$ such that both $\{h_{n}\}$ and $\{\mathcal{L}^{(\mu)}_{D}(h_{n})\}$ are bounded in $L^{2}_{x,v}$.  Then, by the Divergence theorem and Theorem \ref{t1-dissipative}
\begin{align*}
c_{o}\kappa_0\int_{\T^3}\|\langle v \rangle^{\gamma/2}h_{n}\|^{2}_{H^{s}_{v}}
\leq 
\int_{\T^3}\langle\mathcal{L}^{(\mu)}_{D}(h_{n}),h_{n}\rangle\leq \int_{\T^3}\|\mathcal{L}^{(\mu)}_{D}(h_{n})\|_{L^{2}_{v}}\|h_n\|_{L^{2}_{v}}
\leq 
\|\mathcal{L}^{(\mu)}_{D}(h_{n})\|_{L^{2}_{x,v}}\|h_n\|_{L^{2}_{x,v}} \,.
\end{align*}
As a consequence, using \cite[Proposition 1.1]{B}\footnote{This proposition is shown for $x\in\mathbb{R}^{3}$.  The same argument applies for $x\in\mathbb{T}^{3}$ using Fourier series instead of Fourier transform.}, it also follows that
\begin{align*}
\|(-\Delta_{x})^{\frac{s}{1+s}}h_{n}\|^{2}_{L^{2}_{x,v}}
\leq 
C_{d,s}&\|(-\Delta_{v})^{\frac{s}{2}}h_{n}\|^{\frac{2}{1+s}}_{L^{2}_{x,v}}\|\mathcal{L}^{(\mu)}_{D}(h_{n})\|^{\frac{2s}{1+s}}_{L^{2}_{x,v}}
\leq 
C_{d,s}(c_o\kappa_0)^{-\frac{1}{1+s}}\|\mathcal{L}^{(\mu)}_{D}(h_{n})\|^{\frac{1+2s}{1+s}}_{L^{2}_{x,v}}\|h_n\|^{\frac{1}{1+s}}_{L^{2}_{x,v}}\,.
\end{align*}
Thus,
\begin{equation*}
\sup_{n}\Big\{\int_{\T^3}\|\langle v \rangle^{\gamma/2}h_{n}\|^{2}_{H^{s}_{v}} + \|(-\Delta_{x})^{\frac{s}{1+s}}h_{n}\|^{2}_{L^{2}_{x,v}}\Big\}
\leq
C \vpran{\|\mathcal{L}^{(\mu)}_{D}(h_{n})\|_{L^{2}_{x,v}}^2 + \|h_n\|_{L^{2}_{x,v}}^2}\,,
\end{equation*}
which implies that $\{h_{n}\}$ is compact in $L^{2}(\mathbb{T}^{3}\times\mathbb{R}^{3})$.  That is, $\mathcal{K}$ is $\mathcal{L}^{(\mu)}_{D}$-compact.
\end{proof}
\begin{proposition}\label{p2-compact}
The essential spectrum of ${L}^{(\mu)} - v\cdot\nabla_{x}$ lies in $\{z\in\mathbb{C}:\mathfrak{R}_{e}\leq-c_o\kappa_0\}$.  Furthermore, the set $\{z\in\mathbb{C}:\mathfrak{R}_{e}z>-c_o\kappa_0\}$ is contained in the resolvent of ${L}^{(\mu)} - v\cdot\nabla_{x}$ except, possibly, for countably many eigenvalues. 
\end{proposition}
\begin{proof}
We use a similar argument given in \cite[proof of Proposition 3.4]{Mo} using relative compact perturbations in Banach spaces.  More precisely, one uses \cite[Chapter IV - Theorem 5.35 and footnote]{Kato} that asserts that, given Lemma \ref{l1-compact}, $\mathcal{K} + \mathcal{L}^{(\mu)}_{D}$ and $\mathcal{L}^{(\mu)}_{D}$ have the same complementary of the Fredholm domain.  Using Corollary \ref{c1-dissipative}, this implies that
\begin{equation}\label{p2-compact-e1}
\text{Complementary of the Fredholm domain of } {L}^{(\mu)} - v\cdot\nabla_{x}\subset\{z\in\mathbb{C}:\mathfrak{R}_{e}z\leq -c_o\kappa_0\}\,.
\end{equation}
Now, the Fredholm set is composed by a countable number of connected open set in which
\begin{align*}
\text{nul}(z):&=\text{dimension of null space of } T-z\,,\\
\text{def}(z):&=\text{codimension of the range of } T-z
\end{align*}
are finite and constant, and, a countable set of isolated values points (the eigenvalues).  It is known that the boundary of each of these components belong to the complementary of the Fredholm domain.  As a consequence, given \eqref{p2-compact-e1}, the intersection of the Fredholm set and $\{z\in\mathbb{C}:\mathfrak{R}_{e}z>-c_o\kappa_0\}$ is composed of, only, one component and a countable number of eigenvalues.  Since $(a,\infty)$, for any $a\geq\|\mathcal{K}\|_{2}$, belongs to the resolvent of ${L}^{(\mu)}-v\cdot\nabla_{x}$ one concludes that this component is part of the resolvent, that is, $\text{nul}(z)=0$ and $\text{def}(z)=0$ in $\{z\in\mathbb{C}:\mathfrak{R}_{e}z>-c_o\kappa_0\}$ except for a countably many eigenvalues.  Thus, the essential spectrum lies in $\{z\in\mathbb{C}:\mathfrak{R}_{e}\leq-c_o\kappa_0\}$ and the set $\{z\in\mathbb{C}:\mathfrak{R}_{e}z>-c_o\kappa_0\}$ is contained in the resolvent of ${L}^{(\mu)} - v\cdot\nabla_{x}$ except for a countably many eigenvalues. 
\end{proof}
\begin{theorem}\label{t2-localization}
The operator ${L}^{(\mu)} - v\cdot\nabla_{x}$, as an operator in $L^{2}_{x,v}$, has essential spectrum localized in $\{z\in\mathbb{C}:\mathfrak{R}_{e}z\leq-c_o\kappa_0\}$.  Furthermore, its eigenpairs are identical to those of ${L}^{(\mu)}$ as an operator in $L^{2}_{v}$.
\end{theorem}
\begin{proof}
It remains to prove that the eigenpairs of ${L}^{(\mu)} - v\cdot\nabla_{x}$ and ${L}^{(\mu)}$ are identical.  Take first an $(\lambda,\varphi(v))$ eigenpair of $\mathcal{L}^{(\mu)}$.  Then,
\begin{equation*}
\big({L}^{(\mu)} - v\cdot\nabla_{x}\big)\varphi(v) = {L}^{(\mu)}(\varphi(v)) = \lambda\varphi(v)\,.
\end{equation*}
Therefore, $(\lambda,\varphi(v))$ is also an eigepair of ${L}^{(\mu)} - v\cdot\nabla_{x}$.  Now, take $(\lambda,\varphi(x,v))$ an eigenpair of ${L}^{(\mu)} - v\cdot\nabla_{x}$.  Since the set of eigenvectors $\{\varphi_{i}\}$ of ${L}^{(\mu)}$ form a base in $L^{2}_{v}$ by Proposition \ref{p1-compact}, we can write the separation of variables
\begin{equation*}
\varphi(x,v)=\sum_{i\in\mathbb{N}}a_{i}(x)\,\varphi_{i}(v)\,.
\end{equation*}
Plugging this expression into the equation
\begin{equation*}
\big({L}^{(\mu)} - v\cdot\nabla_{x}\big)\varphi(x,v) = \lambda\,\varphi(x,v)\,,
\end{equation*}
we conclude that
\begin{equation}\label{t2-localization-e1}
a_{i}(x)(\lambda_{i}-\lambda) = v\cdot\nabla_{x}a_{i}(x)\,,\qquad \forall\,v\in\mathbb{R}^{3}\,,\quad i\in\mathbb{N}\,,
\end{equation}
for eigenpairs $(\lambda_{i},\varphi_{i})$ of $\mathcal{L}^{(\mu)}$.  For the set $\mathcal{I}_{1}=\{i\in\mathbb{N}:\lambda_{i}\neq\lambda\}$, the right side of \eqref{t2-localization-e1} is a function of velocity and the left side is not.  We conclude that $a_{i}(x)=a_{i}=0$ for every $i\in\mathcal{I}_{1}$.  Note that $\mathcal{I}^{c}_{1}\neq\emptyset$, otherwise $\varphi(x,v)=0$.  In $\mathcal{I}^{c}_{1}$ we conclude immediately that $\lambda$ is eigenvalue of $\mathcal{L}^{(\mu)}$, $a_{i}(x)=a_{i}$ for any $i\in\mathcal{I}^{c}_{1}$, and
\begin{equation*}
\varphi(x,v)=\sum_{i\in\mathcal{I}^{c}_{1}}a_{i}\,\varphi_{i}(v) \Denote\varphi(v) = \text{eigenvector of } {L}^{(\mu)} \text{ associated to } \lambda\,.  \tag*{\qedhere}
\end{equation*}
\end{proof}

\subsection{Localization of the spectrum in polynomially weighted $L^{2}_{x,v}$}
In this section we want to ``enlarge'' the localization of the spectrum of the linearized Boltzmann operator from the space $E = L^{2}(\mu^{-1/2},\mathbb{T}^{3}\times\mathbb{R}^{3})$ to the space $\mathcal{E} = L^{2}(\langle v \rangle^{k},\mathbb{T}^{3}\times\mathbb{R}^{3})$ with $k\geq2$.  The idea of the enlargement of space in the Boltzmann context was introduced in \cite{Mo} to study rate of convergence of the homogeneous problem.  In fact, we will use a later development \cite[Theorem 2.1]{GMM} to facilitate the discussion, although, the argument could be accomplished with classical perturbation theory, as done in \cite{Mo}.  Let us first introduce the operators we work with in this section  
\begin{equation*}
{L}(h)\Denote Q(\mu,h) + Q(h,\mu)\,,
\end{equation*}
which is the operator that naturally appears after the linearization $f=\mu+h$ in the nonlinear problem.  We will consider ${L}$ as a closed operator in $L^{2}_{v}(\langle v \rangle^{k},\mathbb{R}^{3})$, with $k\geq2$.  The final objective is then to study the spectral properties in $L^{2}_{x,v}(\langle v \rangle^{k},\mathbb{T}^{3}\times\mathbb{R}^{3})$ of the (closed) operator
\begin{equation*}
{L} - v\cdot\nabla_{x}\,.
\end{equation*}
Given \cite[Theorem 2.1 and Remark 2.2 (1)]{GMM}, we will be able to localize the spectrum in the larger space $\mathcal{E}$ by knowing the following:\\
(i) The localization of the spectrum of ${L} - v\cdot\nabla_{x}$ in the smaller space $E$.\\
(ii) The operator decomposes as  ${L}=\mathcal{A} - \mathcal{B}$ where $\mathcal{B}$ is a (closed) dissipative operator and $\mathcal{A}:\mathcal{E}\rightarrow E$ is bounded.\\
Regarding item (i), this is exactly what we did in previous section section.  Regarding the decomposition in (ii), we use the analogous decomposition \eqref{Decomposition} adding the advection operator
\begin{align}\label{Decomposition1}
\begin{split}
{L}_{\Phi,b} - v\cdot\nabla_{x}
&= {L}_{\Phi,b_{1}} + {L}_{\Phi,b_{2}} - v\cdot\nabla_{x}
={L}_{\Phi_{1},b_{1}} + {L}_{\Phi_{2},b_{1}} + {L}_{\Phi,b_{2}} - v\cdot\nabla_{x}
\\
&=\Big({L}_{\Phi_{1},b_{1}} + Q^{-}_{\Phi_{1},b_1}\big(\mu, h\big) \Big)
    + \Big({L}_{\Phi_{2},b_{1}} - Q^{-}_{\Phi_{1},b_1}\big(\mu,h\big) + {L}_{\Phi,b_{2}} - v\cdot\nabla_{x}\Big) \Denote \mathcal{A}_{\delta,\varepsilon} - \mathcal{B}_{\delta,\varepsilon} \,.
\end{split}
\end{align}

\subsubsection{Dissipative part}  We already pointed out in the previous section that $\mathcal{B}_{\delta,\varepsilon}=\mathcal{B}_{1}+\mathcal{B}_{2}$ involves all singular part of the operator and decomposes in the tail associated component
\begin{equation*}
-\mathcal{B}_{1} \Denote {L}_{\Phi_{2},b_{1}} - Q^{-}_{\Phi_{1},b_1}\big(\mu,h\big)\,,
\end{equation*}
and the regularity associated component
\begin{equation*}
-\mathcal{B}_{2} \Denote {L}_{\Phi,b_{2}} - v\cdot\nabla_{x}\,.
\end{equation*}
Let us proceed, as we did previously, and prove that $\mathcal{B}$ is indeed a dissipative operator for suitable choices of $\delta>0$ and $\varepsilon>0$ depending only on the mass an energy of $\mu$.
\begin{proposition}[Singular part $L_{\Phi,b_{2}}$]\label{pro-B2}
Let $s \in (0,1)$.  For any $k > \frac{9}{2} + \tfrac{\gamma}{2} + 2s$ there exist constants $c>0$, depending only on mass and energy of $\mu$, and $C_k>0$  such that 
\begin{align*}
\langle L_{\Phi,b_{2}}(h),h\langle v \rangle^{2k}\rangle\leq -c\kappa_0\Big\|\widehat{\langle \cdot \rangle^{\gamma/2+k} h}(\xi)\,|\xi|^{s} \textbf{1}\big\{|\xi|\geq\tfrac{1}{\varepsilon}\big\}\Big\|^{2}_{L^{2}_{\xi}} + C_{k}\,\|\theta^{2} b_{2}\|_{L^{1}_{\theta}}\|\langle{v}\rangle^{\gamma/2+k} h \|^{2}_{L^{2}_{v}}\,.
\end{align*}
The constants are independent of $\varepsilon$.
\end{proposition}
\begin{proof}
Let us compute
\begin{align*}
\langle L_{\Phi,b_{2}}(h)&,h\langle \cdot \rangle^{2k}\rangle 
= \langle { Q}_{\Phi,b_{2}}\big(\mu,h\big),h\langle \cdot \rangle^{2k}\rangle
 + \langle {Q}_{\Phi,b_{2}}\big(h, \mu \big),h\langle \cdot \rangle^{2k}\rangle\\
&= \int_{\mathbb{R}^{6}}\int_{\mathbb{S}^{2}}\Phi(u)b_{2}\vpran{\cos\theta}\mu(v_{*})h(v)\big(h(v')\langle v'\rangle^{2k} - h(v)\langle v \rangle^{2k} \big)\text{d}\sigma\text{d}v_{*}\text{d}v
\\
&+\int_{\mathbb{R}^{6}}\int_{\mathbb{S}^{2}}\Phi(u)b_{2}\vpran{\cos\theta}h(v_{*})\mu(v)\big(h(v')\langle v'\rangle^{2k} - h(v)\langle v \rangle^{2k} \big)\text{d}\sigma\text{d}v_{*}\text{d}v\,.
\end{align*}
Let us consider each of these integral on the right side separately.  For the first integral,
\begin{align}\label{1st-integral}
\langle { Q}_{\Phi,b_{2}}\big(\mu,h\big),h\langle \cdot \rangle^{2k}\rangle 
&=\int_{\mathbb{R}^{6}}\int_{\mathbb{S}^{2}}\Phi(u)b_{2}\vpran{\cos\theta}\mu(v_{*})h(v)h(v')\langle v'\rangle^{k}\big(\langle v'\rangle^{k} - \langle v \rangle^{k} \big)\text{d}\sigma\text{d}v_{*}\text{d}v \nn
\\
& \quad \, 
+ \int_{\mathbb{R}^{6}}\int_{\mathbb{S}^{2}}\Phi(u)b_{2}\vpran{\cos\theta}\mu(v_{*})h(v)\langle v \rangle^{k}\big(h(v')\langle v'\rangle^{k} - h(v)\langle v \rangle^{k} \big)\text{d}\sigma\text{d}v_{*}\text{d}v \nn \\
&= 
\Gamma^{b_2}(\mu, h, h\langle \cdot \rangle^k ) + \langle Q_{\Phi,b_{2}}\big(\mu,h\langle \cdot \rangle^{k}\big),h\langle \cdot \rangle^{k}\rangle 
\,,
\end{align}
where $\Gamma^{b_2}$ is defined by \eqref{for-section-4-2} with $b$ replaced by $b_2$. 
Now, a similar (and simpler) argument given in the proof of Proposition \ref{pro-lambda2}, recalling \cite[Proposition 2.1]{AMUXY2012JFA}, shows
\begin{align}\label{pro-B2-e1}
\begin{split}
 &\langle {Q}_{\Phi,b_{2}}\big(\mu,h\langle \cdot \rangle^{k}\big),h\langle \cdot \rangle^{k}\rangle\\
& \leq -c\kappa_0\Big\|\widehat{\langle \cdot \rangle^{\gamma/2+k} h}(\xi)\,|\xi|^{s} \textbf{1}\big\{|\xi|\geq\tfrac{1}{\varepsilon}\big\}\Big\|^{2}_{L^{2}_{\xi}} + C\| \theta^2 b_{2}\|_{L^{1}_{\theta}}\|\langle{v}\rangle^{\gamma/2+k} h \|^{2}_{L^{2}_{v}}\,.
\end{split}
\end{align}
The estimation for $\Gamma^{b_2}(\mu, h, h\langle \cdot \rangle^k)$ is almost the same as in the proof of 
Proposition \ref{prop:trilinear-weight}. If one splits 
\[
\Gamma^{b_2}(\mu, h, h\langle \cdot \rangle^k) = \Gamma^{b_2}_{1,1} + \Gamma^{b_2}_{1,2} + \sum_{m=2}^6 \Gamma^{b_2}_m,
\]
by the same way, then it follows from the proof of Proposition \ref{prop:trilinear-weight} that
\begin{align}\label{easy-part}
|\Gamma^{b_2}_{1,2}| + \sum_{m=2}^6 |\Gamma^{b_2}_m|  \le C_k \|\theta^2 b_2\|_{L^1_\theta}\|\langle v\rangle^{\gamma/2 +k} h\|^2_{L^2_v}\,.
\end{align}
By the same observation as for $\Gamma_{1,1}$ in the proof of Proposition \ref{prop:trilinear-weight}, 
for any $0<s<1$ we have 
\begin{align*}
\Gamma^{b_2}_{1,1}(\mu, h, h\langle \cdot \rangle^k)  = 
k \IntRRS b_2(\cos\theta) |v - v_\ast|^{1+\gamma}
         \big (v_\ast \cdot \tilde \omega \big)\cos^{k} \tfrac{\theta}{2}  \sin  \tfrac{\theta}{2} \, \mu_\ast 
\vpran{\frac{H}{\langle v \rangle^2}  - \frac{H'}{\langle v' \rangle^2}    } \, H'\dbmu\,,
\end{align*}
where $H = h \langle v \rangle^k$. Therefore
\begin{align*}
&\Gamma^{b_2}_{1,1}(\mu, h, h\langle \cdot \rangle^k)  = 
k \IntRRS b_2(\cos\theta) \frac{|v - v_\ast|^{1+\gamma}}{\langle v \rangle^2}
         \big (v_\ast \cdot \tilde \omega \big)\cos^{k} \tfrac{\theta}{2}  \sin  \tfrac{\theta}{2} \, \mu_\ast 
\vpran{H- H' } \, H'\dbmu\\
&\qquad \qquad + k \IntRRS b_2(\cos\theta) |v - v_\ast|^{1+\gamma}\vpran{\frac{1}{\langle v \rangle^2}  - \frac{1}{\langle v' \rangle^2}    } 
         \big (v_\ast \cdot \tilde \omega \big)\cos^{k} \tfrac{\theta}{2}  \sin  \tfrac{\theta}{2} \, \mu_\ast 
\, |H'|^2\dbmu\\
& \Denote \Gamma^{b_2}_{1,1, main}(\mu, h, h\langle \cdot \rangle^k) + \Gamma^{b_2}_{1,1,rest}(\mu, h, h\langle \cdot \rangle^k) \,.
\end{align*}
Since it follows from the mean value theorem that
\begin{align*}
 |v - v_\ast|^{1+\gamma}\vpran{\frac{1}{\langle v \rangle^2}  - \frac{1}{\langle v' \rangle^2} }
\le (\sqrt 2)^3 \langle v_* \rangle^3 \sin \tfrac{\theta}{2} \,,
\end{align*}
$|\Gamma^{b_2}_{1,1,rest}|$ is estimated by
$C_k \|\theta^2 b_2\|_{L^1_\theta}\|\langle v\rangle^{\gamma/2 +k} h\|^2_{L^2_v}$.
It follows from the Cauchy-Schwarz inequality that
\begin{align}\label{main-prop-4-6}
&\left| \Gamma^{b_2}_{1,1, main}(\mu, h, h\langle \cdot \rangle^k) \right| 
 \leq
  C_k 
 \vpran{\IntRRS b_2(\cos\theta) |v - v_\ast|^{\gamma} 
        \mu_\ast \abs{H - H'}^2 \dbmu}^{1/2}  \nn
\\
& \qquad \qquad \qquad  \,
  \times \vpran{\IntRRS b_2(\cos\theta)  \frac{|v-v_*|^{\gamma+2}\langle v_*\rangle^2}{\langle v \rangle^4}
         \theta^{2} \mu_*
         |H'|^2 \dbmu}^{1/2}  \nn
\\
& \quad \leq 
  C_k \left( -2 \langle {Q}_{\Phi,b_{2}}\big(\mu,h\langle \cdot \rangle^{k}\big),h\langle \cdot \rangle^{k}\rangle 
+\|\theta^2 b_2\|_{L^1_\theta} \|\langle v \rangle^{\gamma/2 +k-1} h\|^2_{L^2_v}\right)^{1/2}
 \left( \|\theta^2 b_2\|_{L^1_\theta} \|\langle v \rangle^{\gamma/2 +k} h\|^2_{L^2_v}\right)^{1/2}    \,, 
\end{align}
where we used the formula
\[
\IntRRS b_2(\cos\theta) \Phi
        \mu_\ast \abs{H - H'}^2 \dbmu = -2 \langle {Q}_{\Phi,b_{2}}\big(\mu, H \big), H \rangle 
+ \IntRRS b_2(\cos\theta) \Phi
        \mu_\ast \vpran{{H'}^2 - H^2}\dbmu
\]
and the cancellation lemma in \cite{ADVW2000}. Summing up the above estimates we obtain 

\begin{align}\label{first}
\langle {Q}_{\Phi,b_{2}}\big(\mu,h\big),h\langle \cdot \rangle^{2k}\rangle
\leq -c\kappa_0\Big\|\widehat{\langle \cdot \rangle^{\gamma/2+k} h}(\xi)\,|\xi|^{s} \textbf{1}\big\{|\xi|\geq\tfrac{1}{\varepsilon}\big\}\Big\|^{2}_{L^{2}_{\xi}} + C_k\|\theta^2 b_{2}\|_{L^{1}_{\theta}}\|\langle{v}\rangle^{\gamma/2+k} h \|^{2}_{L^{2}_{v}}\,.
\end{align}

Let us move to the second integral 
 \begin{align*}
\langle {Q}_{\Phi,b_{2}}\big(h, \mu \big), h\langle \cdot \rangle^{2k}\rangle
&= \int_{\mathbb{R}^{6}}\int_{\mathbb{S}^{2}}\Phi(u)b_{2}\vpran{\cos\theta}h(v_*)\mu(v)h(v')\langle v'\rangle^{k}\big(\langle v'\rangle^{k}
- \langle v \rangle^{k} \big)\text{d}\sigma\text{d}v_{*}\text{d}v\\
& \quad \, 
+ \int_{\mathbb{R}^{6}}\int_{\mathbb{S}^{2}}\Phi(u)b_{2}\vpran{\cos\theta} h(v_{*})\mu(v)\langle v \rangle^{k}\big(h(v')\langle v'\rangle^{k} - h(v)\langle v \rangle^{k} \big)\text{d}\sigma\text{d}v_{*}\text{d}v \nn \\
&= 
\Gamma^{b_2}(h, \mu, h\langle \cdot \rangle^k ) + \langle Q_{\Phi,b_{2}}\big(h,\mu\langle \cdot \rangle^{k}\big),h\langle \cdot \rangle^{k}\rangle 
\,.
\end{align*}
As for $\Gamma^{b_2}(h, \mu, h\langle \cdot \rangle^k )$, we need only to consider 
\begin{align*}
\Gamma^{b_2}_{1,1}(h, \mu, h\langle \cdot \rangle^k)  = 
k \IntRRS b_2(\cos\theta) |v - v_\ast|^{1+\gamma}
         \big (v_\ast \cdot \tilde \omega \big)\cos^{k} \tfrac{\theta}{2}  \sin  \tfrac{\theta}{2} \, h_\ast 
\vpran{\mu \langle v \rangle^{k-2}  - \mu' \langle v' \rangle^{k-2}    } \, h'\langle v' \rangle^k \dbmu\,,
\end{align*}
because the other terms of the decomposition has the same bound as the right hand side of \eqref{easy-part}.
Since it follows from the mean value theorem  that 
\begin{align*}
|v - v_\ast|^{1+\gamma} |\mu \langle v \rangle^{k-2}  - \mu' \langle v' \rangle^{k-2} |
&\le C_k \int_0^1 |v'_\tau - v_*|^{2+\gamma} \mu(v'_\tau)^{1/2} d \tau \sin \tfrac{\theta}{2}, \enskip v'_\tau = v_* + \tau (v'-v_*) \\
&\le C_k \langle v_* \rangle^{2+\gamma}\int_0^1\mu(v'_\tau)^{1/4} d\tau  \sin \tfrac{\theta}{2}, 
\end{align*}
we have 
\[
|\Gamma^{b_2}_{1,1}(h, \mu, h\langle \cdot \rangle^k) | \le C_k \|\theta^2 b_{2}\|_{L^{1}_{\theta}}\|\langle{v}\rangle^{\gamma/2+k} h \|^{2}_{L^{2}_{v}}\,.
\]
If we put $\tilde \mu = \langle v \rangle^k \mu$, then 
\begin{align*}
\langle Q_{\Phi,b_{2}}\big(h,\mu\langle \cdot \rangle^{k}\big),h\langle \cdot \rangle^{k}\rangle
&= \int_{\mathbb{R}^{6}}\int_{\mathbb{S}^{2}}\Phi(u)b_{2}\vpran{\cos\theta} h(v_{*}) 
(\tilde \mu(v) - \tilde \mu(v')) H' \dbmu\\
& \quad \,
+ \int_{\mathbb{R}^{6}}\int_{\mathbb{S}^{2}}\Phi(u)b_{2}\vpran{\cos\theta} h(v_{*}) 
(\tilde \mu(v')H'  - \tilde \mu(v) H) \dbmu\,.
\end{align*}
Using  the Taylor expansion 
\[
\tilde \mu(v) - \tilde \mu(v') = \nabla \tilde \mu(v') \cdot (v-v') + 
\int_0^1 (1-\tau) \nabla \otimes \nabla \mu(v'_\tau) d\tau (v-v')^2
\]
to the first term and applying  the cancellation lemma to the second term, we obtain that 
\begin{align*}
|\langle Q_{\Phi,b_{2}}\big(h,\mu\langle \cdot \rangle^{k}\big),h\langle \cdot \rangle^{k}\rangle|
\le 
C_k \|\theta^2 b_{2}\|_{L^{1}_{\theta}}\|\langle{v}\rangle^{\gamma/2+k} h \|^{2}_{L^{2}_{v}}\,.  \tag*{\qedhere}
\end{align*}
\end{proof}
\begin{proposition}[Singular part $\mathcal{B}_1$]\label{pro-B1}
There exist constant $c>0$ depending only on the mass and energy of $\mu$, such that for any $s\in(0,1)$, $k > 4 + \tfrac{3+\gamma}{2}$, $\varepsilon\in(0,1)$ and $\delta^{\gamma/2}:=\delta(k,\varepsilon)^{\gamma/2}\in(0,C_{k, \nu}\,\varepsilon^{2s})$, it follows that for any $0 < \nu \le 1$, 
\begin{align*}
& \quad \,
\langle L_{\Phi_{2},b_{1}}h - Q^{-}_{\Phi_{1},b_{1}}\big(\mu,h\big),h\langle v \rangle^{2k}\rangle
\\
&\leq - \tfrac{c}{4}\,\|(1-\cos^{2k-(3+\gamma)}(\theta/2) - 2(1+\nu) \sin^{k-\frac{3+\gamma}{2}}(\theta/2))b_{1}\|_{L^{1}_{\theta}}\|\langle v \rangle^{k+\gamma/2}h\|^{2}_{L^{2}_{v}}\,.
\end{align*}
In addition to mass and energy of $\mu$ and $k$, the constant $C_{k}$ also depends on $b,s,\gamma$.
\end{proposition}
Before starting with the proof observe that for $0 < \theta \le \pi/2$ and $k > 4 + \tfrac{d+\gamma}{2}$,
\begin{align*}
1-\cos(\theta/2)^{2k-(3+\gamma)} -
2(1+\nu)\sin^{k-\frac{3+\gamma}{2}}(\theta/2) &\ge \sin^2(\theta/2)\Big( 1- 2(1+\nu)\sin^{k-2-\frac{3+\gamma}{2}}(\theta/2)\Big)\\
& \ge \sin^2(\theta/2)\left( 1- \frac{2(1+\nu)}{2^{\frac{2k-7-\gamma}{4}}}\right) >0.
\end{align*}
\begin{proof}
The proof is similar to those for Propositions~\ref{prop:coercivity} and~\ref{prop:bound-Q-M}.  However, since the coefficients involved need to be more explicit, we show the full details here.
Note that
\begin{equation*}
L_{\Phi_{2},b_{1}}h - Q^{-}_{\Phi_{1},b_{1}}\big(\mu,h\big) = Q_{\Phi_{2},b_{1}}(h,\mu) + Q^{+}_{\Phi_{2},b_{1}}(\mu,h) - Q^{-}_{\Phi,b_{1}}(\mu,h)\,.
\end{equation*}
Let us first control the term $Q_{\Phi_{2},b_{1}}(h,\mu)$.  As before,
\begin{align*}
  \langle Q_{\Phi_{2},b_1}(h,\mu),h\langle v \rangle^{2k}\rangle 
&= \int_{\mathbb{R}^{6}}\int_{\mathbb{S}^{2}}\Phi_{2}(u)b_{1}\vpran{\cos(\theta)}\big)h(v_*)\mu(v)\big(h(v')\langle v'\rangle^{2k} - h(v)\langle v \rangle^{2k} \big)\text{d}\sigma\text{d}v_{*}\text{d}v 
\\
& = \int_{\mathbb{R}^{6}}\int_{\mathbb{S}^{2}}\Phi_{2}(u)b_{1}\vpran{\cos(\theta)}h(v_*)\mu(v)h(v')\langle v'\rangle^{k}\big(\langle v'\rangle^{k} - \langle v \rangle^{k} \big)\text{d}\sigma\text{d}v_{*}\text{d}v
\\
& \quad \, 
  + \int_{\mathbb{R}^{6}}\int_{\mathbb{S}^{2}}\Phi_{2}(u)b_{1}\vpran{\cos(\theta)}h(v_{*})\mu(v)\langle v \rangle^{k}\big(h(v')\langle v'\rangle^{k} - h(v)\langle v \rangle^{k} \big)\text{d}\sigma\text{d}v_{*}\text{d}v\,.
\end{align*}
Since $\Phi_{2}(u)\leq \delta^{\gamma}\langle u \rangle^{2\gamma}$, the second integral in the right side is controlled by
\begin{align}\label{pro-B1-e1}
\begin{split}
& \quad \,
\int_{\mathbb{R}^{6}}\int_{\mathbb{S}^{2}}\Phi_{2}(u)b_{1}\big(\cos(\theta)\big)\big|h(v_{*})\big|\mu(v)\langle v \rangle^{k}\Big|h(v')\langle v'\rangle^{k} - h(v)\langle v \rangle^{k} \Big|\text{d}\sigma\text{d}v_{*}\text{d}v\\
&\leq \delta^{\gamma}C\|b_{1}\|_{L^{1}_{\theta}}\|h\langle\cdot\rangle^{2\gamma}\|_{L^{1}_{v}}\|\mu\langle\cdot\rangle^{k+3\gamma/2}\|_{L^{2}_{v}}\|h\langle\cdot\rangle^{k+\gamma/2}\|_{L^{2}_{v}}\\
&\leq \delta^{\gamma}C_{k}\|b_{1}\|_{L^{1}_{\theta}}\|h\langle\cdot\rangle^{k+\gamma/2}\|^{2}_{L^{2}_{v}}\,,\qquad \qquad k>1+\frac{3+\gamma}{2}\,.
\end{split}
\end{align}
For the first integral, we use the following estimate; for any $0<\nu \le 1$ there exists a $C_{k,\nu} >0$ such that 
\begin{align}\label{bound:diff-v-k-original}
\Big|\langle v'\rangle^{k} - \langle v \rangle^{k} \Big|
&\leq 
(1+\nu)|u|^{k}+C_{k, \nu}\sin(\theta/2)\langle v \rangle^{k}\langle v_{*}\rangle \nn \\
&\leq (1+\nu)\sin^{k}(\theta/2)\langle v_{*}\rangle^{k}+C_{k, \nu}\sin(\theta/2)\langle v \rangle^{k}\langle v_{*}\rangle\,.
\end{align}
Indeed, for any $m >0$  we have 
\begin{equation}\label{ineq:convex}
(1+ X)^m \le (1+\nu) X^m + \vpran{1 + \frac{1}{(1+\nu)^{1/m}-1}}^m, \enskip \mbox{for any $X>0$},
\end{equation}
because $(1+ X)^m > (1+\nu)X^m$ implies $X < 1/((1+\nu)^{1/m} -1)$. 
If we put  $v'_\tau = v + \tau(v'-v)$ for $\tau \in [0,1]$, then it follows from the mean value theorem that
\begin{align*}
\abs{\vint{v}^{k} - \vint{v'}^{k}}
& \leq k \int_0^1 \vint{v'_\tau}^{k-1}d\tau|v'-v|
\leq k \int_0^1 \big( 1 + |v'_\tau|\big)^{k-1}d\tau|v'-v|\\
&\leq k \int_0^1 \big( \sqrt{2}\vint{v}  + \tau|v'-v|\big)^{k-1}d\tau|v'-v|.
\end{align*}
Apply \eqref{ineq:convex} with $m =k-1$ and $X = \tau|v'-v|/(\sqrt{2}\vint{v})$. Then 
\begin{align*}
\abs{\vint{v}^{k} - \vint{v'}^{k}} &\le 
(1+\nu)k\int_0^1 \tau^{k-1} d\tau |v-v'|^k + k\vpran{1 + \frac{1}{(1+\nu)^{1/(k-1)}-1}}^{k-1}\vpran{\sqrt{2}\vint{v}}^{k-1}|v-v'|\\
&=(1+\nu)\abs{v - v'}^{k} + C_{k,\nu} \abs{v-v'} \vint{v}^{{k}-1}\,,
\end{align*}
which proves \eqref{bound:diff-v-k-original}.  The first integral is 
controlled by 
\begin{multline*}
\int_{\mathbb{R}^{6}}\int_{\mathbb{S}^{2}}\Phi_{2}(u)b_{1}\big(\cos(\theta)\big)\big|h(v_*)\big|\mu(v)\big|h(v')\big|\langle v'\rangle^{k}\Big|\langle v'\rangle^{k} - \langle v \rangle^{k} \Big|\text{d}\sigma\text{d}v_{*}\text{d}v\\
\leq (1+\nu)\int_{\mathbb{R}^{6}}\int_{\mathbb{S}^{2}}\Phi_{2}(u)\sin(\theta/2)^{k}b_{1}\big(\cos(\theta)\big)\big|h(v_*)\big|\langle v_{*}\rangle^{k}\mu(v)\big|h(v')\big|\langle v'\rangle^{k}\text{d}\sigma\text{d}v_{*}\text{d}v\\
+C_{k,\nu}\,\delta^{\gamma/2}\int_{\mathbb{R}^{6}}\int_{\mathbb{S}^{2}}\sin(\theta/2)b_{1}\big(\cos(\theta)\big)\big|h(v_*)\big|\langle v_{*}\rangle^{1+3\gamma/2}\mu(v)\langle v \rangle^{k+3\gamma/2}\big|h(v')\big|\langle v'\rangle^{k}\text{d}\sigma\text{d}v_{*}\text{d}v\,.
\end{multline*}
The last integral in this inequality is controlled, for any $k>1+\frac{3+2\gamma}{2}$, as
\begin{align}\label{pro-B1-e2}
\begin{split}
& \quad \,
C_{k,\nu}\,\delta^{\gamma/2}\int_{\mathbb{R}^{6}}\int_{\mathbb{S}^{2}}\sin(\theta/2)b_{1}\big(\cos(\theta)\big)\big|h(v_*)\big|\langle v_{*}\rangle^{1+3\gamma/2}\mu(v)\langle v \rangle^{k+3\gamma/2}\big|h(v')\big|\langle v'\rangle^{k}\text{d}\sigma\text{d}v_{*}\text{d}v\\
&
\leq \delta^{\gamma/2}C_{k,\nu}\|b_{1}\|_{L^{1}_{\theta}}\|h\langle\cdot\rangle^{1+3\gamma/2}\|_{L^{1}_{v}}
\|h\langle\cdot\rangle^{k+\gamma/2}\|_{L^{2}_{v}}
\leq \delta^{\gamma/2}C_{k,\nu}\|b_{1}\|_{L^{1}_{\theta}}\|h\langle\cdot\rangle^{k+\gamma/2}\|^{2}_{L^{2}_{v}}\,.
\end{split}
\end{align}
For the first integral, one uses Cauchy-Schwarz inequality
\begin{equation*}
h(v_{*})\langle v_{*} \rangle^{k}h(v')\langle v' \rangle^{k}\leq \frac{|h(v_{*})|^{2}\langle v_{*} \rangle^{2k}}{2\sin(\theta/2)^{\frac{3+\gamma}{2}}} + \frac{1}{2}|h(v')|^{2}\langle v' \rangle^{2k}\sin(\theta/2)^{\frac{3+\gamma}{2}}\,,
\end{equation*}
the bound $\Phi_{2}(u)\leq |u|^{\gamma}$, and the change of variables $v'\rightarrow v_{*}$ in the second of the above terms to conclude that
\begin{align}\label{pro-B1-e2.5}
\begin{split}
& \quad \,
\int_{\mathbb{R}^{6}}\int_{\mathbb{S}^{2}}\Phi_{2}(u)\sin(\theta/2)^{k}b_{1}\big(\cos(\theta)\big)\big|h(v_*)\big|\langle v_{*}\rangle^{k}\mu(v)\big|h(v')\big|\langle v'\rangle^{k}\text{d}\sigma\text{d}v_{*}\text{d}v
\\
&\leq \int_{\mathbb{R}^{6}}\int_{\mathbb{S}^{2}}|u|^{\gamma}\sin(\theta/2)^{k-\frac{3+\gamma}{2}}b_{1}\big(\cos(\theta)\big)\big|h(v_*)\big|^{2}\langle v_{*}\rangle^{2k}\mu(v)\text{d}\sigma\text{d}v_{*}\text{d}v\,.
\end{split}
\end{align}
Let us move now to the term $Q^{+}_{\Phi_{2},b_{1}}(\mu,h)$.  We use the ideas of \cite[Prop. 2.1 of page 131]{DeMo}  which give us
\begin{align*}
& \quad \,
\langle Q^{+}_{\Phi_{2},b_1}(\mu,h),h\langle v \rangle^{2k}\rangle\\
&\leq \tfrac{1}{2}\int_{\mathbb{R}^{6}}\int_{\mathbb{S}^{2}}\Phi_{2}\Big(\frac{u}{\cos(\theta/2)}\Big)b_{1}\big(\cos(\theta)\big)\cos^{-3}(\theta/2)\,\nu^{-1}(\theta)\,\mu(v_*)\big|h(v)\big|^{2}\langle v \rangle^{2k}\text{d}\sigma\text{d}v_{*}\text{d}v\\
&\hspace{+.4cm}+\tfrac{1}{2}\int_{\mathbb{R}^{6}}\int_{\mathbb{S}^{2}}\Phi_{2}(u)b_{1}\big(\cos(\theta)\big)\nu(\theta)\mu(v_*)\big|h(v)\big|^{2}\langle v' \rangle^{2k}\text{d}\sigma\text{d}v_{*}\text{d}v\\
&\leq \tfrac{1}{2}\int_{\mathbb{R}^{6}}\int_{\mathbb{S}^{2}}\Phi(u)b_{1}\big(\cos(\theta)\big)\cos^{-3-\gamma}(\theta/2)\,\nu^{-1}(\theta)\,\mu(v_*)\big|h(v)\big|^{2}\langle v \rangle^{2k}\text{d}\sigma\text{d}v_{*}\text{d}v\\
&\hspace{+.4cm}+\tfrac{1}{2}\int_{\mathbb{R}^{6}}\int_{\mathbb{S}^{2}}\Phi_{2}(u)b_{1}\big(\cos(\theta)\big)\nu(\theta)\mu(v_*)\big|h(v)\big|^{2}\langle v' \rangle^{2k}\text{d}\sigma\text{d}v_{*}\text{d}v
\end{align*}
for any $\nu(\theta)>0$.  As a consequence, 
\begin{align}\label{pro-B1-e3}
\begin{split}
\langle Q^{+}_{\Phi_{2},b_{1}}&(\mu,h) - Q^{-}_{\Phi,b_{1}}(\mu,h), h \langle v \rangle^{2k}\rangle\leq\tfrac{1}{2}\int_{\mathbb{R}^{6}}\int_{\mathbb{S}^{2}}\Phi(u)b_{1}\big(\cos(\theta)\big)\\
&\times\big(\cos^{-3-\gamma}(\theta/2)\nu(\theta)^{-1} - 1\big)\mu(v_{*})\,|h(v)|^{2}\langle v \rangle^{2k} \text{d}\sigma\text{d}v_{*}\text{d}v\\
&\hspace{-1cm}+\tfrac{1}{2}\int_{\mathbb{R}^{6}}\int_{\mathbb{S}^{2}}\Phi_{2}(u)b_{1}\big(\cos(\theta)\big)\mu(v_{*})\,|h(v)|^{2}\big(\langle v' \rangle^{2k}\nu(\theta) - \langle v \rangle^{2k}\big)\text{d}\sigma\text{d}v_{*}\text{d}v\,.
\end{split}
\end{align}
At this point one chooses $\nu(\theta):=\cos^{-2k}(\theta/2)$.  For the second integral in \eqref{pro-B1-e3}, one uses the classical formula for $\omega\in\mathbb{S}^{1}$ with $\omega\perp u$,
\begin{equation*}
|v'|^{2}=|v|^{2}\cos^{2}(\theta/2) + |v_{*}|^{2}\sin^{2}(\theta/2)+2\cos(\theta/2)\sin(\theta/2)|u| v_{*}\cdot\omega\,.
\end{equation*}
Thus, for any $k\geq 2$ it follows from the Taylor expansion of the second order that  
\begin{align*}
\langle v' \rangle^{2k} & - \cos(\theta/2)^{2k}\langle v \rangle^{2k} 
= k \Big(\langle v \rangle^2 \cos^2 (\theta/2)\Big)^{k-1} \Big( \langle v_*\rangle^2 \sin^2 (\theta/2) + 
2\cos(\theta/2)\sin(\theta/2)|u| v_{*}\cdot\omega\Big) \\
& + k(k-1) \int_0^1 (1-t) \Big(
\langle v \rangle^2 \cos^2(\theta/2) + t \big(\langle v_*\rangle^2 \sin^2 (\theta/2) + 
2\cos(\theta/2)\sin(\theta/2)|u| v_{*}\cdot\omega\big) \Big)^{k-2}dt \\
&\qquad \qquad \times \Big( \langle v_*\rangle^2 \sin^2 (\theta/2) + 
2\cos(\theta/2)\sin(\theta/2)|u| v_{*}\cdot\omega\Big)^2\\
&
= 2k \langle v \rangle^{2k-2} \cos^{2k-1} (\theta/2)\sin(\theta/2)|u| v_{*}\cdot\omega
\pm c_{k}\sin^{2}(\theta/2)\langle v_{*}\rangle^{2k}\langle v \rangle^{2(k-1)}\,.
\end{align*}
Therefore,
\begin{equation*}
\bigg|\int_{\mathbb{S}^{1}}\Big(\langle v' \rangle^{2k} - \cos(\theta/2)^{2k}\langle v \rangle^{2k}\Big)\text{d}\omega\bigg|\leq c_{k}\sin^{2}(\theta/2)\langle v_{*}\rangle^{2k}\langle v \rangle^{2(k-1)}\,.
\end{equation*}
In this way, the second integral in \eqref{pro-B1-e3} is estimated by
\begin{align}\label{pro-B1-e4}
\begin{split}
&\int_{\mathbb{R}^{6}}\int_{\mathbb{S}^{2}}\Phi_{2}(u)\frac{b_{1}\big(\cos(\theta)\big)}{\cos(\theta/2)^{2k}}\Big(\langle v' \rangle^{2k} - \cos(\theta/2)^{2k}\langle v \rangle^{2k}\Big)\mu(v_*)\big|h(v)\big|^{2}\text{d}\sigma\text{d}v_{*}\text{d}v\\
&\leq c_{k}\delta^{\gamma}\Big\|\frac{b_{1}\sin^{2}(\theta/2)}{\cos(\theta/2)^{2k}}\Big\|_{L^{1}_{\theta}}\|\mu\langle \cdot \rangle^{2(k+\gamma)}\|_{L^{1}_{v}}\|h\langle v\rangle^{k}\|^{2}_{L^{2}_{v}}\leq C_{k}\delta^{\gamma}\|b_{1}\sin^{2}(\theta/2)\|_{L^{1}_{\theta}}\|h\langle v \rangle^{k}\|^{2}_{L^{2}_{v}}\,.
\end{split}
\end{align}
Gathering \eqref{pro-B1-e1},\eqref{pro-B1-e2},\eqref{pro-B1-e2.5},\eqref{pro-B1-e3} and \eqref{pro-B1-e4} one gets,
\begin{align*}
\langle & L_{\Phi_{2},b_{1}}h - Q^{-}_{\Phi_{1},b_1}\big(\mu,h\big),h\langle v \rangle^{2k}\rangle \\
&\leq -\tfrac{1}{2}\|\big(1-\cos^{2k-(3+\gamma)}(\theta/2) - 2(1+\nu)\sin^{k-\frac{3+\gamma}{2}}(\theta/2)\big)b_{1}\|_{L^{1}_{\theta}}\int_{\mathbb{R}^{3}}|h(v)|^{2}\langle v \rangle^{2k}\bigg(\int_{\mathbb{R}^{3}}\mu(v_{*})|u|^{\gamma}\text{d}v_{*}\bigg)\text{d}v\\
&\hspace{2cm} + \delta^{\gamma/2}C_{k, \nu }\|b_{1}\|_{L^{1}_{\theta}}\|h\langle \cdot \rangle^{k+\gamma/2}\|^{2}_{L^{2}_{v}}\,,\qquad k>2+\frac{3+\gamma}{2}\,.
\end{align*}
Now, one has the estimates
\begin{equation*}
\int_{\mathbb{R}^{3}}\mu(v_{*})|u|^{\gamma}\text{d}v_{*}\geq c\,\langle v \rangle^{\gamma}\,,\qquad \|b_{1}\|_{L^{1}_{\theta}}\sim\frac{\kappa_0}{2s\,\varepsilon^{2s}}\,,
\end{equation*}
as a consequence, the proposition follows taking any $\delta:=\delta(k,\varepsilon)>0$ such that
\begin{equation*}
\frac{\delta^{\gamma/2}C_{k, \nu }\kappa_0}{2s\,\varepsilon^{2s}}\leq \tfrac{c}{4}\|\big(1-\cos(\theta/2)^{2k-(3+\gamma)} -
2(1+\nu) \sin^{k-\frac{3+\gamma}{2}}(\theta/2)\big)b_{1}\|_{L^{1}_{\theta}}\,.  \tag*{\qedhere}
\end{equation*} 
\end{proof}
\begin{theorem}\label{t1-dissipative-L2}
Let $k > \frac{9}{2} + \tfrac{\gamma}{2} + 2s$ and $h\in H^{s}_{k+\gamma/2}(\dv)$ with $s\in (0,1)$.  There exist constants $c_o>0$, $\varepsilon_o>0$ depending on the mass and energy of $\mu$, the scattering kernel $b$ and $k>0$, such that for any $\delta^{\gamma/2}:=\delta(k,\varepsilon_o)^{\gamma/2}\in(0,C_{k}\,\varepsilon^{2s}_o)$ with $C_{k}>0$ given in the Proposition \ref{pro-B1}, the operator $-\mathcal{B}_{1} + L_{\Phi,b_2}$ (with $\varepsilon=\varepsilon_o$) satisfies the dissipative estimate
\begin{equation}\label{t1-e1}
\big\langle (-\mathcal{B}_{1} + L_{\Phi,b_2})(h),h\langle v\rangle^{2k} \big\rangle \leq -c_o\,\kappa_0\|\langle v \rangle^{k+\gamma/2}h\|^{2}_{H^{s}_{v}}\,.
\end{equation}
Furthermore, it follows from~\eqref{t1-e1} that
\begin{equation}\label{t1-e2}
\big\langle -\mathcal{B}_{\delta,\varepsilon_o}h,h\langle v\rangle^{2k} \big\rangle_{x,v} \leq -c_o\,\kappa_0\|\langle v \rangle^{k+\gamma/2}h\|^{2}_{H^{s}_{x,v}}\,.
\end{equation}
\end{theorem}
\begin{proof}
Using Propositions \ref{pro-B2} and \ref{pro-B1} one has
\begin{align*}
\big\langle (-\mathcal{B}_{1}& + L_{\Phi,b_2})(h),h\langle v\rangle^{2k} \big\rangle\leq -c_{1}\kappa_0\Big\|\widehat{\langle \cdot \rangle^{\gamma/2+k} h}(\xi)\,|\xi|^{s} \textbf{1}\big\{|\xi|\geq\tfrac{1}{\varepsilon}\big\}\Big\|^{2}_{L^{2}_{\xi}}
\\
&+ \Big(C_{k}\|\theta^2 b_{2}\|_{L^{1}_{\theta}} - c_{2}\|(1-\cos(\theta/2)^{2k-(3+\gamma)} - 2(1+\nu)\sin^{k-\frac{3+\gamma}{2}}(\theta/2))b_{1}\|_{L^{1}_{\theta}}\Big)\|\langle v \rangle^{k+\gamma/2}h\|^{2}_{L^{2}_{v}}
\end{align*}
for positive constants $c_{1}$, $c_{2}$ that depend only on the mass and energy of $\mu$ and positive constant $C_{k}$ that depends additionally on $k$.  Since $\| \theta^2 b_{2,\varepsilon}\|_{L^{1}_{\theta}}\sim \frac{\kappa_0\,
\varepsilon^{2-2s}}{2-2s}$, one can choose $\varepsilon = \varepsilon_{o}>0$ sufficiently small such that
\begin{equation*}
C_{k}\| \theta^2 b_{2,\varepsilon_o}\|_{L^{1}_{\theta}}\leq \tfrac{c_{2}}{2}\|\big(1-\cos(\theta/2)^{2k - (3+\gamma)} - 2(1+\nu)\sin(\theta/2)^{k-\frac{3+\gamma}{2}}\big)b_{1,\varepsilon_o}\|_{L^{1}_{\theta}}\,.
\end{equation*}
Clearly $\varepsilon_o$ has the aforementioned dependence on the parameters.  The choice
\begin{equation*}
c_o\Denote\min\Big\{ c_{1}, \tfrac{c_{2}\varepsilon^{2s}_{o}}{4\kappa_0}\|\big(1-\cos(\theta/2)^{2k - (3+\gamma)} - 2(1+\nu)\sin(\theta/2)^{k-\frac{3+\gamma}{2}}\big)b_{1,\varepsilon_o}\|_{L^{1}_{\theta}}\Big\}
\end{equation*}
proves the first statement \eqref{t1-e1}.  Using the divergence theorem one has $\langle -v\cdot\nabla_{x}h,h\langle v \rangle^{2k}\rangle_{x,v}=0$, which proves the second statement \eqref{t1-e2}.
\end{proof}
\subsubsection{Bounded part}  Let us consider the operator
\begin{equation*}
\mathcal{A}_{\delta,\varepsilon}h= L_{\Phi_{1},b_{1}}h+Q^{-}_{\Phi_{1},b_{1}}(\mu,h) = Q_{\Phi_{1},b_{1}}(h,\mu) + Q^{+}_{\Phi_{1},b_{1}}(\mu,h)\,.
\end{equation*}
The following bound for $\mathcal{A}_{\delta,\varepsilon}$ holds:
\begin{proposition}\label{pro-bounded-L2}
For any $\delta>0$, $\varepsilon>0$, $k>\frac{3}{2}$ the operator $\mathcal{A}_{\delta,\varepsilon}:L^{2}(\langle v\rangle^{k},\mathbb{R}^{3})\rightarrow L^{2}(\mu^{-1/2}(v),\mathbb{R}^{3})$ is a bounded operator with norm estimated as
\begin{equation*}
\|\mathcal{A}_{\delta,\varepsilon}\|\leq 3\,\delta^{-\gamma}e^{\frac{3}{4}\delta^{-2}}\|b_{1}\|_{L^{1}_{\theta}}\max\big\{\|\mu^{1/4}\|_{L^{1}},\|\mu^{1/4}\|_{L^{2}}\big\}\,.
\end{equation*}
In particular, the operator $\mathcal{A}_{\delta,\varepsilon}:L^{2}\big(\langle v\rangle^{k},\mathbb{T}^{3}\times\mathbb{R}^{3}\big)\rightarrow L^{2}\big(\mu^{-1/2}(v),\mathbb{T}^{3}\times\mathbb{R}^{3}\big)$ is bounded with the same norm.
\end{proposition}%
\begin{proof}
Noticing that Cauchy Schwarz inequality implies that for any $\epsilon>0$
\begin{align*}
|v'|^{2} = |v - u^{-}|^{2}&\leq \big(1+\epsilon^{-1}\big)|v|^{2} + \big(1+\epsilon\big)|u|^{2}\,,\\
|v'|^{2} = |v_{*} + u^{+}|^{2}&\leq \big(1+\epsilon^{-1}\big)|v_{*}|^{2} + \big(1+\epsilon\big)|u|^{2}\,,\qquad u^{\pm} := \tfrac{1}{2}\big(u\pm|u|\sigma\big)\,,
\end{align*}
one can choose $\epsilon=2$ to prove that $|v'|^{2}\leq \frac{3}{2}\min\{|v|^{2},|v_{*}|^{2}\}+3|u|^{2}$.  Therefore,
\begin{equation*}
\mu(v)\mu^{-1/2}(v')\leq\mu^{1/4}(v)e^{3/4|u|^{2}}\quad \text{and}\quad \mu(v_{*})\mu^{-1/2}(v')\leq\mu^{1/4}(v_{*})e^{3/4|u|^{2}}.
\end{equation*}
From here, one can readily prove the control  
\begin{align}\label{prop4.7-e1}
\begin{split}
\Big|Q^{+}_{\Phi_{1},b_{1}}(h,\mu)\Big|\mu^{-1/2}&\leq Q^{+}_{e^{3/4|u|^{2}}\Phi_{1},b_{1}}(|h|,\mu^{1/4})\,,\\
\Big|Q^{+}_{\Phi_{1},b_{1}}(\mu,h)\Big|\mu^{-1/2}&\leq Q^{+}_{e^{3/4|u|^{2}}\Phi_{1},b_{1}}(\mu^{1/4},|h|)\,.
\end{split}
\end{align}
Indeed, take an arbitrary $\varphi\geq0$ and compute using this estimate
\begin{align*}
\int_{\mathbb{R}^{3}}\Big|&Q^{+}_{\Phi_{1},b_{1}}(h,\mu)\Big|\mu^{-1/2}\varphi\,\text{d}v\leq \int_{\mathbb{R}^{6}}\int_{\mathbb{S}^{2}}\Phi_{1}(|u|)b_{1}(\cos(\theta))\big|h(v_{*})\big|\mu(v)\mu^{-1/2}(v')\varphi(v')\text{d}\sigma\text{d}v_{*}\text{d}v\\
&\leq\int_{\mathbb{R}^{6}}\int_{\mathbb{S}^{2}}e^{3/4|u|^{2}}\Phi_{1}(|u|)b_{1}(\cos(\theta))\big|h(v_{*})\big|\mu^{1/4}(v)\varphi(v')\text{d}\sigma\text{d}v_{*}\text{d}v=\int_{\mathbb{R}^{3}}Q^{+}_{e^{3/4|u|^{2}}\Phi_{1},b_{1}}(|h|,\mu^{1/4})\varphi\,\text{d}v\,.
\end{align*}
Since $\varphi$ is arbitrary, the first estimate in \eqref{prop4.7-e1} follows.  A similar argument gives the second estimate.  As a consequence, using Young's inequality for the gain collision operator in estimate \eqref{prop4.7-e1} it follows that
\begin{align*}
\big\|Q^{+}_{\Phi_{1},b_{1}}&(h,\mu)\mu^{-1/2}\big\|_{L^{2}_{v}}\leq C\delta^{-\gamma}e^{\frac{3}{4}\delta^{-2}}\|b_{1}\|_{L^{1}_{\theta}}\|h\|_{L^{1}_{v}}\leq C\delta^{-\gamma}e^{\frac{3}{4}\delta^{-2}}\|b_{1}\|_{L^{1}_{\theta}}\|h\langle v \rangle^{k}\|_{L^{2}_{v}}\\
&\text {and }\qquad \big\|Q^{+}_{\Phi_{1},b_{1}}(\mu,h)\mu^{-1/2}\big\|_{L^{2}_{v}}\leq C\delta^{-\gamma}e^{\frac{3}{4}\delta^{-2}}\|b_{1}\|_{L^{1}_{\theta}}\|h\|_{L^{2}_{v}}\,,
\end{align*}
with $C=\max\{\|\mu^{1/4}\|_{L^{1}},\|\mu^{1/4}\|_{L^{2}}\}$ and $k>\frac{3}{2}$.  Furthermore,
\begin{equation*}
\big\|Q^{-}_{\Phi_{1},b_{1}}(h,\mu)\mu^{-1/2}\big\|_{L^{2}_{v}}\leq C\delta^{-\gamma}\|b_{1}\|_{L^{1}_{\theta}}\|h\|_{L^{1}_{v}}\leq C\delta^{-\gamma}\|b_{1}\|_{L^{1}_{\theta}}\|h\langle v\rangle^{k}\|_{L^{2}_{v}}\,,\quad k>\tfrac{3}{2}\,,
\end{equation*}
for the same aforementioned constant $C$.
\end{proof}
\subsubsection{Enlargement of the Spectrum} We are in position now to extend Theorem \ref{t2-localization} to the larger space $L^{2}(\langle v\rangle^{k},\mathbb{T}^{3}\times\mathbb{R}^{3})$.
\begin{theorem}\label{tenlargement}
Let $s\in(0,1)$.  The operator $\mathcal{L} = {L} - v\cdot\nabla_{x}$ defined on $L^{2}(\langle v\rangle^{k},\mathbb{T}^{3}\times\mathbb{R}^{3})$, with $k > \frac{9}{2} + \tfrac{\gamma}{2}+2s$, has essential spectrum localized in $\{z\in\mathbb{C}:\mathfrak{R}_{e}z\leq-c_o\kappa_0\}$.  Furthermore, its eigenpairs are identical to those of $\mathcal{L}^{(\mu)}$ (as an operator in $L^{2}(\mu^{-1/2}(v),\mathbb{R}^{3})$) in $\{z\in\mathbb{C}:\mathfrak{R}_{e}z>-c_o\kappa_0\}$.  Thus, $\CalL$ has the same spectral gap as $\CalL^{(\mu)}$ and its null space is given by
\begin{equation*}
\text{Null}(\mathcal{L}) = \text{Span}\{\mu, \, v\mu, \, |v|^{2}\mu\}\,.
\end{equation*}
\end{theorem}
\begin{proof}
Set the spaces $\mathcal{E}=L^{2}(\langle v\rangle^{k},\mathbb{T}^{3}\times\mathbb{R}^{3})$ and $E=L^{2}(\mu^{-1/2}(v),\mathbb{T}^{3}\times\mathbb{R}^{3})$.  Note that $E\subset\mathcal{E}$.  Theorem \ref{t2-localization} gives us the localization of $\mathcal{L}$ as an operator in $E$.  Using Theorem \ref{t1-dissipative-L2} and Proposition \ref{pro-bounded-L2}, we know that we can decompose $\mathcal{L} = \mathcal{A} - \mathcal{B}$ with $\mathcal{B}:\mathcal{E}\rightarrow\mathcal{E}$ closed and dissipative and $\mathcal{A}:\mathcal{E}\rightarrow E$ bounded (for a suitable choice of the parameters $\delta>0$ and $\varepsilon>0$).  This fulfills the hypothesis \textbf{(H1)} and \textbf{(H2)}\footnote{Note that the fact that $\mathcal{A}$ is bounded from the large space $\mathcal{E}$ to the small space $E$ ensures both \textbf{(H2) (ii)} and \textbf{(iii)}.} of \cite[Theorem 2.1]{GMM} which implies the result.
\end{proof}

A direct consequence of Theorem \ref{tenlargement} and \cite[Theorem 2.1]{GMM} is 
\begin{corollary}
Let $s\in(0,1)$.  The operator $\mathcal{L}$ generates a strongly continuous semigroup $e^{\mathcal{L}t}$ in $\mathcal{E}=L^{2}(\langle v\rangle^{k},\mathbb{T}^{3}\times\mathbb{R}^{3})$, for any $k > \frac{9}{2} + \tfrac{\gamma}{2} + 2s$. Moreover, if $h(t)$ is the solution to the initial value problem
\begin{equation*}
\frac{\text{d}h}{\text{d}t} = \mathcal{L} h\,,\quad h_o\in\mathcal{E}\,,
\end{equation*}
that is, $h(t)=e^{\mathcal{L}t}h_o$, then
\begin{equation*}
\|h(t) - \pi h_o\|_{\mathcal{E}}\leq C e^{-\lambda t}\|(1-\pi) h_o\|_{\mathcal{E}}\,.
\end{equation*}
Here $\pi$ is the projection onto $\text{Null}(\mathcal{L})$ such that
\begin{align*}
\pi g\Denote \sum_{\varphi\in\{1,v_{1},\cdots,v_{d}, |v|^{2}\}}\bigg(\int_{\mathbb{T}^{3}\times\mathbb{R}^{3}}g\,\varphi\,\text{d}v\text{dx}\bigg)\varphi\,\mu
\end{align*}
and $\lambda>0$ is the spectral gap of $\mathcal{L}^{(\mu)}$ as an operator in $L^{2}(\mu^{-1/2}(v),\mathbb{R}^{3})$.

\end{corollary}

\section{Regularization of  $\CalL$} \label{sec:regularization}
Recall the  linearized operator $\CalL$ is
\begin{align*} 
   \CalL h = -v \cdot \nabla_x h + Lh \,,
\qquad
    L h = Q(h, \mu) + Q(\mu, h) \,.
\end{align*}
In this section, we will show the regularization of $\CalL$ in both $x, v$. The main result is
\begin{thm} \label{thm:linear-regularization}
Let $\CalL$  be the linearized operator and let $k_0 \in \RR$ satisfy
\begin{align}\label{global-existence-number}
k_0 >\frac{5\gamma+37 }{2} , 
\end{align}
so that the spectral gap of $\CalL$ holds on the space $L^2(\vint{v}^{k_0} \dv\dx)$. Let $h$ be the solution to the linear equation
\begin{align} \label{eq:linearize}
   \del_t h = \CalL h \,,
\qquad
   h|_{t=0} = h^{in}(x, v) \,,
\end{align}
where 
 the Fourier transform of $h^{in}$ in both $x, v$ satisfies
\begin{align*}
   \sum_{\ell \in \ZZ^3} \int_{\R^3_v} \abs{\frac{1}{\vint{\xi}^s + \vint{\ell}^{\frac{s}{2s+1}}} \widehat{\vint{v}^{k_0} h^{in}_\ell}(\xi)}^2 \dxi
< \infty \,.
\end{align*}
Then for any $t > 0$, we have $h(t, \cdot, \cdot) \in L^2(\vint{v}^{k_0} \dv\dx)$ with the bound
\begin{align*}
    \iint_{\T^3 \times \R^3_v}
     \vint{v}^{2 k_0} |h(t,x, v)|^2 \dx\dv
\leq
   C_{k_0} \sum_{\ell \in \ZZ^3} \int_{\R^3_v} \abs{\frac{1}{\vint{\xi}^s + \vint{\ell}^{\frac{s}{2s+1}}} \widehat{\vint{v}^{k_0} h^{in}_\ell}(\xi)}^2 \dxi \,.
\end{align*}
\end{thm}

The proof of Theorem~\ref{thm:linear-regularization} relies on various commutator estimates related to the collision operator. These estimates are the subjects of the following subsections.

\subsection{Definition of $\CalM$} 
The regularization of $\CalL$ will be shown by applying a Fourier multiplier $\CalM$ to the linearized equation~\eqref{eq:linearize}. To define the operator $\CalM$, we use the Fourier series with respect to $x$ and write 
\[
h(t,x,v) = \sum_{\ell \in \ZZ^3} e^{-2\pi i \ell\cdot x} h_\ell (t,v).
\]
Thus equation~\eqref{eq:linearize} is reduced to
\begin{align}\label{fourier}
    (\del_t   + v \cdot (-2\pi \ell i)h_\ell  = L h_\ell \,,
\qquad
   h_\ell|_{t=0} = h^{in}_\ell(v) \,.
\end{align}
Let $\delta >0$ be small to  be chosen later. For a fixed $T > 0$ and any $t \in [0,T]$, 
we define a symbol of $\CalM(t, \ell, D_v)$ as
\begin{align} \label{def:CalM}
\CalM(t,\ell, \xi)  &= \Big(1 + \delta \int_t^T \La \xi +2\pi (t-\rho)\ell \Ra^{2s} d\rho \Big)^{-1/2 -\varepsilon}
=  \Big(1 + \delta \int_0^{T-t} \La \xi - 2\pi \tau \ell \Ra^{2s} \dtau \Big)^{-1/2 -\varepsilon}
\end{align}
with $0 < \varepsilon  < \frac{1-s}{2s}$.
For brevity, we write $2\pi \ell = \eta$ and $\CalM = \CalM(t, \eta, \xi)$ sometimes in the following.  

The basic estimate for the symbol $\CalM$ is
\begin{lem} \label{lem:CalM}
Let $\CalM(t,\eta, \xi)$ be the symbol defined in~\eqref{def:CalM} with $\eta= 2\pi \ell$. Then for any $\alpha \in \ZZ^3$ there exists a constant $C$ that only depends on $s, \Eps, \alpha$ such that
\begin{align*}
    \abs{\frac{\nabla_\xi^\alpha  \CalM(t, \eta, \xi)}{\CalM(t, \eta, \xi)} }
\leq C
\left\{ \begin{array}{ll}
 ( \la \xi\ra +(T- t) |\eta |)^{-1} & \text{if $s <1/2$,}\\
 ( \la \xi \ra  + (T-t) |\eta| )^{-\min\{|\alpha|, 2\}}& \text{if $s >1/2$,}\\
 ( \la \xi \ra +(T- t )|\eta| )^{-\min\{|\alpha|, 2-\varepsilon_1\}}, \forall \, 0<  \varepsilon_1 < 1, & \text{if $s=1/2$,}
\end{array} \right.
\end{align*}
\end{lem}
\begin{proof}
Note that for $\alpha \ne 0$ we have 
\begin{align}\label{simple-form}
\left|\partial_\xi^\alpha \Big(\langle \xi\rangle^{2s} \Big)\right| \le C_\alpha \langle \xi\rangle^{2s-|\alpha|} 
\le C_\alpha \left\{ \begin{array}{ll}
 \langle \xi\rangle^{2s-1} & \text{if $s <1/2$,}\\
 \langle \xi\rangle^{2s-\min\{|\alpha|, 2\}}& \text{if $s >1/2$,}\\
\langle \xi\rangle^{1-\min\{|\alpha|, 2-\varepsilon_1\}}, \forall \, 0< \varepsilon_1 < 1, & \text{if $s=1/2$.}
\end{array}
\right.
\end{align}
By Corollary\ref{appendix-cor-A} and Lemma~\ref{lem:integral-bound-negative} in the appendix, we have
\begin{align*}
  \abs{\frac{\nabla_\xi \CalM}{\CalM}}
&= \abs{\vpran{\frac{1}{2} + \Eps} \frac{ \delta \int_0^{T-t}\nabla_\xi \vpran{\La \xi - \tau \eta \Ra^{2s} }d\tau}
{1 + \delta \int_0^{T-t} \La \xi - \tau \eta \Ra^{2s} d\tau}}\\
& \le C_s  \frac{\delta(T-t)\{\la \xi\ra + (T-t)|\eta|\}^{2s-1} } {1+\delta(T-t)\{\la \xi\ra +(T-t)|\eta|\}^{2s}} \le C_s
\vpran{\la \xi\ra + (T-t)|\eta|}^{-1}\,. 
\end{align*}
Furthermore, we have 
\begin{align*}
&\frac{\partial_{\xi_j} \nabla_\xi \CalM}{\CalM} = \frac{\partial_{\xi_j} \CalM}{\CalM}\frac{\nabla_\xi \CalM}{\CalM}\\
&\quad -\vpran{\frac{1}{2} +\Eps}\left [\frac{
\delta \int_0^{T-t} \partial_{\xi_j} \nabla_\xi \vpran{\La \xi - \tau \eta \Ra^{2s} }d\tau}
{1 + \delta \int_0^{T-t} \La \xi - \tau \eta \Ra^{2s} d\tau}  -\frac{
\delta \int_0^{T-t} \partial_{\xi_j}  \vpran{\La \xi - \tau \eta \Ra^{2s} }d\tau}
{1 + \delta \int_0^{T-t} \La \xi - \tau \eta \Ra^{2s} d\tau}
\frac{
\delta \int_0^{T-t} \nabla_\xi \vpran{\La \xi - \tau \eta \Ra^{2s} }d\tau}
{1 + \delta \int_0^{T-t} \La \xi - \tau \eta \Ra^{2s} d\tau} \right]\,.
\end{align*}
Applying Corollary\ref{appendix-cor-A} and Lemma~\ref{lem:integral-bound-negative}, in view of 
\eqref{simple-form},  we obtain the inequalities for the case of the second order derivatives. 
The cases for higher order derivatives can be obtained inductively.
\end{proof}

%

We are mainly concerned with the commutator estimate of $\CalM$ with the collision operator $Q$. The result will be shown by dividing the collision kernel in $Q$ into the bounded and unbouned 
domains in terms of $\abs{v - v_\ast}$ . More precisely, let $0 \leq \chi_{R}(r) \leq 1$ be a smooth cutoff function such that
\begin{align*}
     \chi_R(r) 
   = \begin{cases}
      1, & 0 \leq r \leq R \,, \\
      0, & r > 2R \,,
      \end{cases}
\end{align*}
and $\chi_R$ satisfies that for any $k \in \NN$
\begin{align} \label{property:chi-R}
     \abs{D^{k} \chi_R(r)} \leq C \vint{r}^{-k} \,,
\end{align}
where $C$ is independent of $R$. Such $\chi_R$ exists, for example, by rescaling a smooth cutoff function supported on $[0, 2]$. 
Denote
\begin{align} \label{def:Phi-R-bar-R}
    \Phi_{R} (v - v_\ast) = |v - v_\ast|^\gamma \chi_{R} (|v - v_\ast|) \,, 
\qquad
    \Phi_{\bar R} (v - v_\ast) 
    = |v - v_\ast|^\gamma \vpran{1 - \chi_{R} (|v - v_\ast|)} \,.
\end{align}
Decompose $Q$ such that
\begin{align} \label{def:Q-R-bar-R}
  Q(f, g)
&= Q_R(f, g) + Q_{\bar R}(f, g) \nn
\\
& = \int_{\R^3} \int_{\Ss^2} b(\cos \theta) \Phi_{R} (f'_\ast g' - f_\ast g) \dsigma \dv_\ast
+ \int_{\R^3} \int_{\Ss^2} b(\cos \theta) \Phi_{\bar R} (f'_\ast g' - f_\ast g) \dsigma \dv_\ast \,,
\end{align}
where the choice of the constant $R > 0$ will be specified later. 

\subsection{Commutator estimate for $Q_R$} 
The first commutator estimate is for $Q_R$ with $\CalM$. The estimate is proved in a similar way as in \cite[Proposition 3.4]{AMUXY2012JFA}.
\begin{prop}\label{singular-part} 
Let $R > 0$ and $Q_R$ be defined as in~\eqref{def:Q-R-bar-R}. Then there exists a constant $C_R$ such that
  \[\left | \Big( \CalM Q_R(f,g) - Q_R(f, \CalM g), h \Big)\right|
\leq 
C_R \|f\|_{L^1} \|\CalM g\|_{H^{s'} } \|h\|_{H^{s'}}
\]
for any $s'$ satisfying 
\begin{align*}
&\mbox{$s' \ge (s -1/2)^+$} \qquad \qquad \mbox{if  $s\ne 1/2$ \,,}
\\
&\mbox{$s' > 0 $} \hspace{2.9cm} \mbox{if $s=1/2$}.
\end{align*}
\end{prop}
\begin{proof} 
We use a decomposition 
\begin{align*}
{\bf 1} = {\bf 1}_{\La \xi_* \Ra\geq \sqrt 2 |\xi|} + 
{\bf 1}_{\La \xi_*\Ra \leq |\xi |/2 } + {\bf 1}_{\sqrt 2 |\xi| \geq \La \xi_*\Ra \geq | \xi|/2 } \,.
\end{align*}
The following property holds in each of the following subdomains:
\begin{equation}\label{equivalence-relation-spatiall-homo}
\left \{ \begin{array}{ll}
\La \xi \Ra \lesssim \la \xi_* \ra \sim \la \xi-\xi_*\ra,   &\mbox{on supp ${\bf 1}_{\la \xi_* \ra\geq \sqrt 2 |\xi|}$},\\
\la \xi \ra \sim \la \xi-\xi_*\ra,   &\mbox{on supp ${\bf 1}_{\la \xi_*\ra \leq |\xi |/2 } $}, \\
\la \xi \ra \sim \la \xi_* \ra \gtrsim   \la \xi-\xi_*\ra,  &
\mbox{on supp ${\bf 1}_{\sqrt 2 |\xi| \geq \la \xi_*\ra \geq | \xi|/2 }$\,.}
\end{array}
\right.
\end{equation}
Noting the Ukai formula given in Lemma~\ref{lem:Ukai},
we put 
\[
F(x;a) = \frac{x}{\Big(1+ \delta( a x^{2s} + a^{2s+1} |\eta|^{2s})\Big)^{\frac{1}{2}+\varepsilon}}.
\]
It is easy to check that
$F(x;a)$
is an increasing function in $x \in (0,\infty)$ because of  $\varepsilon < \frac{1-s}{2s}$. 
Therefore, it follows from the first formula of \eqref{equivalence-relation-spatiall-homo} that,
on supp ${\bf 1}_{\la \xi_* \ra\geq \sqrt 2 |\xi|}$,
we have 
\begin{align*}
 \La \xi \Ra \CalM(t,\xi,\eta)  \sim F(\La \xi \Ra ; T-t) \le F(\La \xi_* \Ra ; T-t)  \lesssim
  \La \xi_* \Ra \CalM(t,\xi_* ,\eta) \sim \La \xi_* \Ra \CalM(t, \xi- \xi_* ,\eta).
\end{align*}
Hence,
\begin{align} \label{I-part}
  \CalM(t,\xi,\eta)  
\lesssim 
   \frac{\La \xi_* \Ra }{\La \xi \Ra }\CalM(t, \xi- \xi_* ,\eta) \,,
\qquad
   \text{if $\vint{\xi_\ast} \geq \sqrt{2} |\xi|$.}
\end{align}
Combining the mean value theorem, Lemma~\ref{lem:CalM} and the second formula in~\eqref{equivalence-relation-spatiall-homo}  $\mbox{on supp ${\bf 1}_{\la \xi_*\ra \leq |\xi |/2 } $} $, we have
\begin{align}\label{II-part}
\left|\CalM (t, \xi, \eta) - \CalM (t, \xi-\xi_*, \eta)\right| &\le
\int_0^1 |\big(\nabla_\xi  \CalM \big) (t, \xi - \tau \xi_*, \eta) )|d\tau |\xi_*| \notag \\
&  \lesssim  
\CalM(t, \xi-\xi_*,\eta) \displaystyle \frac{\la \xi_*\ra}{\la \xi \ra} \,,
\qquad \qquad 
\text{if $\vint{\xi_\ast} \leq |\xi|/2$} \,.
\end{align}
Here and in what follows, we abbreviate $\CalM (\xi) = \CalM(t,\xi,\eta)$ to show its dependence on
$\xi$.
On supp ${\bf 1}_{\la \xi_*\ra \leq |\xi |/2 } $ we have   $\la \xi \ra \sim \la \xi - \xi_*\ra$ and hence 
\[
|\CalM (\xi) - \CalM (\xi-\xi_*) |\lesssim \CalM(\xi -\xi_*)\,,
\]
which together with \eqref{II-part} gives 
\begin{equation}\label{II-part-bis}
|\CalM (\xi) - \CalM (\xi-\xi_*) |\lesssim \frac{\la \xi_* \ra^\kappa}{\la \xi \ra^{\kappa}} \CalM(\xi -\xi_*)
\end{equation}
for any $0 \le \kappa \le 1$ and $\vint{\xi_\ast} \leq |\xi |/2$.

On supp ${\bf 1}_{\sqrt{2}|\xi|\geq\la \xi_*\ra \geq |\xi |/2 } $, by means of Proposition~\ref{lem:Ukai} and the third formula from \eqref{equivalence-relation-spatiall-homo},  we have 
\begin{align}\label{III-part}
\CalM (\xi) \sim \CalM (\xi_*) \sim
\CalM(\xi-\xi_*) \frac{1 + \delta\Big((T-t) \la \xi -\xi_*\ra^{2s} + (T-t)^{2s+1} |\eta|^{2s} \Big)}
{1 + \delta\Big((T-t) \la \xi_* \ra^{2s} + (T-t)^{2s+1} |\eta|^{2s} \Big)}
\lesssim
   \CalM (\xi-\xi_*).
\end{align}
It follows from \eqref{I-part}, \eqref{II-part-bis} and \eqref{III-part} that
\begin{align}\label{M-inequality}
\left|\CalM (\xi) - \CalM (\xi-\xi_*)\right|\leq &\CalM (\xi-\xi_*)\left\{ \Big(
\frac{\la \xi_*\ra}{\la
 \xi \ra }\Big) {\bf 1}_{\la \xi_* \ra \ge \sqrt 2 |\xi|} 
+  {\bf 1}_{   \sqrt 2 |\xi| >\la \xi_* \ra \ge   |\xi|/2  }
+ \frac{\la \xi_*\ra^\kappa}{\la \xi \ra^\kappa}{\bf 1}_{|\xi|/2>\la \xi_* \ra }\right\}\,, 
\end{align}
which corresponds to \cite[(3.4)]{AMUXY2012JFA}.

We shall follow an almost same procedure as in the proof of \cite[Proposition 3.4]{AMUXY2012JFA}. By using the Bobylev formula {}from
the Appendix of \cite{ADVW2000}, we have
\begin{align*}
( Q_R(f, g), h ) =& \iiint_{ \RR^3\times\RR^3\times\Ss^2} b \Big({\frac{\xi}{ | \xi |}} \cdot \sigma \Big) [ \hat\Phi_R (\xi_* - \xi^- ) - \hat\Phi_R (\xi_* ) ]
\hat f (\xi_* ) \hat g(\xi - \xi_* ) \overline{{\hat h} (\xi )} \dxi \dxi_*d\sigma \,,
\end{align*}
where $\xi^-=\frac 1 2 (\xi-|\xi|\sigma)$.
Therefore,
\begin{align*}
 \Big (\CalM(D) \, Q_R (f,  g) -  Q_R (f,  \CalM (D)\, g)  , h \Big )
= &\iiint b \Big({\frac{\xi}{ | \xi |}} \cdot \sigma \Big) [ \hat\Phi_R (\xi_* - \xi^- ) - \hat\Phi_R (\xi_* ) ] \\
&\quad \times
\Big(\CalM (\xi) - \CalM (\xi-\xi_*)\Big)
\hat f (\xi_* ) \hat g(\xi - \xi_* ) \overline{{\hat h} (\xi )} \dxi \dxi_*d\sigma \\
= & \iiint_{ | \xi^- | \leq {1\over 2} \la \xi_*\ra }  \cdots\,\, \dxi \dxi_*d\sigma
+ \iiint_{ | \xi^- | \geq {1\over 2} \la \xi_*\ra } \cdots\,\, \dxi \dxi_*d\sigma \,\\
=& \CalA_1(f,g,h)  +  \CalA_2(f,g,h) \,\,.
\end{align*}
For $\CalA_1$, we use the Taylor expansion of $\hat\Phi_R$ of
order $2$ to have
$$
\CalA_1 = \CalA_{1,1} (f,g,h) +\CalA_{1,2} (f,g,h),
$$
where
\begin{align*}
\CalA_{1,1} &= \iiint b\,\, \xi^-\cdot (\nabla\hat\Phi_R)( \xi_*)
{\bf 1}_{ | \xi^- | \leq {1\over 2} \la \xi_*\ra }
\Big(\CalM  (\xi) - \CalM (\xi-\xi_*)\Big)
\hat f (\xi_* ) \hat g(\xi - \xi_* ) \bar{\hat h} (\xi ) \dxi \dxi_*d\sigma,
\end{align*}
and $\CalA_{1,2} (f,g,h)$ is the remaining term corresponding to the second order  Taylor expansion of $\hat\Phi_R$.

We first consider $\CalA_{1,1}$.
By writing
\[
\xi^- = \frac{|\xi|}{2}\left(\Big(\frac{\xi}{|\xi|}\cdot \sigma\Big)\frac{\xi }{|\xi|}-\sigma\right)
+ \left(1- \Big(\frac{\xi}{|\xi|}\cdot \sigma\Big)\right)\frac{\xi}{2},
\]
we see that the integral corresponding to the first term on the right hand side vanishes because of the symmetry
on $\Ss^2$.
Hence, we have
\[
\CalA_{1,1}= \iint_{\RR^6} K(\xi, \xi_*) \Big(\CalM (\xi) - \CalM (\xi-\xi_*)\Big)
\hat f (\xi_* ) \hat g(\xi - \xi_* ) \bar{\hat h} (\xi ) \dxi \dxi_* \,,
\]
where
\[
K(\xi,\xi_*) = \int_{\Ss^2}
 b \Big({\xi\over{ | \xi |}} \cdot \sigma \Big)
\left(1- \Big(\frac{\xi}{|\xi|}\cdot \sigma\Big)\right)\frac{\xi}{2}\cdot
(\nabla\hat\Phi_R)( \xi_*)
{\bf 1}_{ | \xi^- | \leq {1\over 2} \la \xi_*\ra } d \sigma \,.
\]
Note that $| \nabla \hat\Phi_R (\xi_*) | \lesssim \frac{1}{\la
\xi_*\ra^{3+\gamma +1}}$, {}from the Appendix of \cite{AMUXY2012JFA}. If $\sqrt 2
|\xi| \leq \la \xi_* \ra$, then $\sin\frac{\theta}{2} $ $| \xi|= |\xi^-| \leq \la \xi_* \ra/2$
because $0 \leq \theta \leq \pi/2$, and we have
\begin{align*}
|K(\xi,\xi_*)| &\lesssim \int_0^{\pi/2} \theta^{1-2s} d \theta\frac{ \la \xi\ra}{\la \xi_*\ra^{3+\gamma +1}}
\lesssim \frac{1  }{\la \xi_*\ra^{3+\gamma}}
\frac{\la \xi \ra}{\la \xi_*\ra} \,.
\end{align*}
On the other hand, if $\sqrt 2 |\xi| \geq \la \xi_* \ra$, then
\begin{align*}|K(\xi,\xi_*)| &\lesssim \int_0^{\pi\la \xi_*\ra /(2|\xi|)} \theta^{1-2s} d \theta\frac{ \la \xi\ra}{\la \xi_*\ra^{3+\gamma +1}}
\lesssim \frac{1  }{\la \xi_*\ra^{3+\gamma}}\left(
\frac{\la \xi \ra}{\la \xi_*\ra}\right)^{2s-1}\,.
\end{align*}
Hence,  we obtain
\begin{align}\label{later-use1}
|K(\xi,\xi_*)| &\lesssim \frac{1  }{\la \xi_*\ra^{3+\gamma}}\left\{
\left( \frac{\la \xi \ra}{\la \xi_*\ra}\right){\bf 1}_{ \la \xi_*
\ra \geq
\sqrt 2 |\xi| } +{\bf 1}_{ \sqrt 2 |\xi| \geq  \la \xi_* \ra \geq
|\xi|/2} + \left( \frac{\la \xi \ra}{\la \xi_*\ra}\right)^{2s-1}
{\bf 1}_{ |\xi|/2 \ge \la \xi_* \ra }\right\}\,. \notag
\end{align}
Similar to $\CalA_{1,1}$,
we  can also write
\[
\CalA_{1,2}= \iint_{\RR^6} \tilde K(\xi, \xi_*) \Big(\CalM (\xi) - \CalM (\xi-\xi_*)\Big)
\hat f (\xi_* ) \hat g(\xi - \xi_* ) \bar{\hat h} (\xi ) \dxi \dxi_* \,,
\]
where
\[
\tilde K(\xi,\xi_*) = \int_{\Ss^2}
 b \Big({\xi\over{ | \xi |}} \cdot \sigma \Big)
\int^1_0(1-\tau)  (\nabla^2\hat\Phi_R) (\xi_* -\tau\xi^- ) : \vpran{\xi^- \otimes\xi^-}
{\bf 1}_{ | \xi^- | \leq {1\over 2} \la \xi_*\ra } d\tau  d \sigma\,.
\]
Again {}from the Appendix of \cite{AMUXY2012JFA}, we have
$$
| (\nabla^2\hat\Phi_R) (\xi_* -\tau\xi^- ) | \lesssim {1\over{\la  \xi_* -\tau \xi^-\ra^{3+\gamma +2}}}
\lesssim
 {1\over{\la \xi_*\ra^{3+\gamma +2}}},
$$
because $|\xi^-| \leq \la \xi_*\ra/2$. This leads to
\begin{align}\label{later-use2}
|\tilde K(\xi,\xi_*)| &\lesssim \frac{1  }{\la
\xi_*\ra^{3+\gamma}}\left\{ \left( \frac{\la \xi \ra}{\la
\xi_*\ra}\right)^2 {\bf 1}_{ \la \xi_* \ra \geq \sqrt 2 |\xi|
}  +{\bf 1}_{ \sqrt 2 |\xi| \geq  \la \xi_* \ra \geq
|\xi|/2} + \left( \frac{\la \xi \ra}{\la \xi_*\ra}\right)^{2s} {\bf
1}_{ |\xi|/2 \ge \la \xi_* \ra }\right\}\,. \notag
\end{align}
It follows {}from  \eqref{M-inequality}, \eqref{later-use1} and \eqref{later-use2} that
we have 
$$
|\CalA_{1}|\lesssim  |\CalA_{1,1}| + |\CalA_{1,2}|
\lesssim A_1+ A_2 + A_3,
$$
where
\begin{equation}\label{A_1}
A_1=\iint_{\RR^6}
\left |\frac{\hat f(\xi_*)}{\la \xi_*\ra^{3+\gamma}}\right | \left |\CalM (\xi-\xi_*)\hat g(\xi -\xi_*)\right | |\hat h(\xi)|
{\bf 1}_{ \la \xi_* \ra \geq \sqrt 2 |\xi| }\dxi_* \dxi\, ,
\end{equation}
and
\begin{align*}
A_2=&
\iint_{\RR^6}
\left |\frac{\hat f(\xi_*)}{\la \xi_*\ra^{3+\gamma}}\right |  \left |\CalM (\xi-\xi_*)\hat g(\xi -\xi_*)\right | |\hat h(\xi)|
 {\bf 1}_{   \sqrt 2 |\xi| >\la \xi_* \ra \ge   |\xi|/2}\dxi_*\dxi\, ,
\\
A_3=&\iint_{\RR^6}
\left |\frac{\hat f(\xi_*)}{\la \xi_*\ra^{3+\gamma}}\right | \left |\CalM (\xi-\xi_*)\hat g(\xi -\xi_*)\right | |\hat h(\xi)|
\left(  \frac{\la \xi\ra^{2s - \kappa}}{\la \xi_* \ra^{2s-\kappa} } \right)  {\bf 1}_{|\xi|/2>\la \xi_* \ra } \dxi_* \dxi\, .
\end{align*}
For $\gamma >0$, we can obtain 
\begin{align*}
&\left| A_1\right|^2 \lesssim \|\hat f \|^2_{L^\infty}
\left(\int_{\RR^3} \frac{\dxi_*}{ \la \xi_*\ra^{3+ \gamma}}
\int_{\RR^3_\xi} |\hat h(\xi)|^2\dxi\right)\\
&\hspace{2cm} \times \left(\int_{\RR^3}  \frac{\dxi}{ \la \xi \ra^{3+ \gamma}}
\int_{\RR^3} \left(\frac{\la \xi\ra}{\la \xi_*\ra}\right)^{3+\gamma}
{\bf 1}_{ \la \xi_* \ra \geq \sqrt 2 |\xi| } |\CalM (\xi-\xi_*)\hat g ( \xi -\xi_*)|^2 \dxi_*
\right)\notag \\
& \qquad
\lesssim
\|f\|_{L^1}^2 \|\CalM g\|_{L^2}^2 \|h\|_{L^2}^2\,. \notag
\end{align*}

Noticing the third formula of \eqref{equivalence-relation-spatiall-homo}, we get
\begin{align*}
\left| A_2 \right|^2 \lesssim&
\left(\int_{\RR^3}\frac{|\hat f(\xi_*)|^2 \dxi_*}{
\la \xi_*\ra^{6+2\gamma}}
\int_{\la \xi -\xi_*\ra \lesssim \la \xi_* \ra} 
\dxi \right) \left( \iint_{\RR^6}|\CalM(\xi-\xi_*)\hat g ( \xi -\xi_*)|^2 |\hat h(\xi)|^2 \dxi \dxi_* \right) \\
& \lesssim
\int_{\RR^3} \frac{|\hat f(\xi_*)|^2} {
\la \xi_* \ra^{3+2\gamma} }
\dxi_* \|\CalM g\|^2_{L^2} \|h\|_{L^2}^2
\lesssim
\|f\|_{L^1}^2 \|\CalM g\|_{L^2}^2 \|h\|_{L^2}^2\,.
\end{align*}

Setting $\hat G(\xi) = \la \xi \ra^{s'} \CalM(\xi) \hat g(\xi)$ and $\hat H(\xi) =\la \xi \ra^{s'} \hat h(\xi)$ with 
$s' = s - \kappa/2$, 
we have when $\kappa \le 2s$,
\begin{align*}
&\left| A_3 \right|^2 \lesssim
\left(\int_{\RR^3} \frac{|\hat f(\xi_*)| \dxi_*}{ \la \xi_*\ra^{3+ \gamma + 2s-\kappa}}
\int_{\RR^3} |\hat H(\xi)|^2\dxi\right)
\left(\int_{\RR^3}  \frac{|\hat f(\xi_*)|\dxi_*}{ \la \xi_* \ra^{3+ \gamma +2s-\kappa}}
\int_{\RR^3} 
{\bf 1}_{ |\xi|/2 \ge \la \xi_* \ra } |\hat G ( \xi -\xi_*)|^2 \dxi
\right)\notag \\
& \qquad
\lesssim
\|f\|_{L^1}^2 \|\CalM g\|^2_{H^{s'} }\|h\|^2_{H^{s'}}\,.  \notag
\end{align*}
The above three estimates yield the desired estimate for $\CalA_1(f,g,h)$.

Next consider $\CalA_2(f,g,h)$. Write $\CalA_2(f,g,h)$ as
\begin{align*}
\CalA_2 &=  \iiint b \Big({\xi\over{ | \xi |}} \cdot \sigma \Big) {\bf 1}_{ | \xi^- | \ge {1\over 2}\la \xi_*\ra }
\hat\Phi_R (\xi_* - \xi^- ) \cdots 
\dxi \dxi_*d\sigma \\
&\quad - \iiint b \Big({\xi\over{ | \xi |}} \cdot
 \sigma \Big){\bf 1}_{ | \xi^- | \ge {1\over 2}\la \xi_*\ra } \hat\Phi_R (\xi_* )
\cdots 
 \dxi \dxi_*d\sigma \\
&= \CalA_{2,1}(f,g,h) - \CalA_{2,2}(f,g,h)\,.
\end{align*}
Since $|\xi^-|= |\xi| \sin(\theta/2)$ $\geq$
$\la \xi_*\ra/2$ and $\theta \in [0,\pi/2]$, we have $\sqrt 2 |\xi| \geq
\la \xi_*\ra$. Write
\[
\CalA_{2,j}= \iint_{\RR^6} K_j(\xi, \xi_*) \Big(\CalM (\xi) - \CalM (\xi-\xi_*)\Big)
\hat f (\xi_* ) \hat g(\xi - \xi_* ) \bar{\hat h} (\xi ) \dxi \dxi_* \,.
\]
Then by $\abs{\hat\Phi_R(\xi)} \lesssim \frac{1}{\vint{\xi_\ast}^{3+\gamma}}$, we have
\begin{align*}
   |K_2(\xi, \xi_*)| &= \left|\int  b \Big({\xi\over{ | \xi |}} \cdot \sigma \Big)\hat\Phi_R(\xi_*) {\bf 1}_{ | \xi^- | \ge {1\over 2}\la \xi_*\ra } d\sigma\right|
\lesssim  
   {1\over{\la \xi_* \ra^{3+\gamma }}} \frac{\la  \xi\ra^{2s} }{\la \xi_*\ra^{2s}}{\bf 1}_{\sqrt 2 |\xi| \geq \la \xi_* \ra} \notag 
\\
&\lesssim
\frac{1  }{\la \xi_*\ra^{3+\gamma}}\left\{
{\bf 1}_{ \sqrt 2 |\xi| \geq  \la \xi_* \ra \geq |\xi|/2}
+
\left(
\frac{\la \xi \ra}{\la \xi_*\ra}\right)^{2s}
{\bf 1}_{ |\xi|/2 \ge \la \xi_* \ra }\right\}  \notag \,,
\end{align*}
which shows the desired estimate for $\CalA_{2,2}$ in exactly the same way as
the estimation on $A_2$ and $A_3$.

As for $\CalA_{2,1}$, it suffices to work under the condition
 $|\xi_* \cdot \xi^-| \ge \frac1 2 |\xi^-|^2$.
In fact, on the complement of this
set, we have
 $|\xi_* -\xi^-| > | \xi_*|$, and $\hat\Phi_R(\xi_*-\xi^-)$ is
the same as $\hat\Phi_R(\xi_*)$.
Therefore, we can consider $\CalA_{2,1,p}$ in which  $K_1(\xi, \xi_*)$ is replaced by
\[
K_{1,p}(\xi,\xi_*) = \int_{\Ss^2}
 b \Big({\xi\over{ | \xi |}} \cdot \sigma \Big)
\hat\Phi_R ( \xi_*-\xi^-)
{\bf 1}_{ | \xi^- | \geq {1\over 2} \la \xi_*\ra }{\bf 1}_{| \xi_* \,\cdot\,\xi^-| \ge {1\over 2} | \xi^-|^2} d \sigma \,.
\]
By writing
\[
{\bf 1}= {\bf 1}_{\la \xi_* \ra \geq |\xi|/2} 
+  {\bf 1}_{\la \xi_* \ra < |\xi|/2},
\]
we decompose respectively
\begin{align*}
\CalA_{2,1,p}
=
B_1+ B_2\,.
\end{align*}
On the set of integration in $K_{1,p}$, we have $\la \xi_* -\xi^- \ra \lesssim \,
\la \xi_* \ra$, because $| \xi^- | \lesssim | \xi_*|$
by $| \xi^-|^2 \le 2 | \xi_* \cdot\xi ^-| \lesssim |\xi^-|\, | \xi_*|$.
Furthermore, on the set for $B_1$ 
we have $\la \xi \ra \sim \la \xi_* \ra$,
so that
$\la \xi_* -\xi^- \ra \lesssim \ \la \xi \ra$ and $b\,\, {\bf 1}_{ | \xi^- | \ge {1\over 2} \la \xi_*\ra } {\bf 1}_{\la \xi_* \ra \geq |\xi|/2}$
is bounded.
By means of \eqref{M-inequality} and Cauchy-Schwarz inequality, we have
\begin{align*}
|B_1| ^2  
&\lesssim \|f\|^2_{L^1} \iiint {
|\hat\Phi_R (\xi_* - \xi^-) |   }
|\CalM (\xi -\xi_*)\hat g(\xi -\xi_*)|^2 d \sigma \dxi \dxi_* \\
& \quad \,
\times \iiint {
|\hat\Phi_R (\xi_* - \xi^-) |   }
 |{ \hat h} (\xi ) |^2 d\sigma \dxi \dxi_*
 \\
&\lesssim  \|f\|^2_{L^1}  \|\CalM g\|^2_{L^2}
\|h\|_{L^2}^2\,.
\end{align*}
Here, 
noting  $\xi -\xi_* = \xi^+ -u$ with
$u = \xi_* - \xi^-$, 
it holds for the first integral factor
\begin{align*}
\iiint {
|\hat\Phi_R (\xi_* - \xi^-) |   }
|\CalM (\xi -\xi_*)\hat g(\xi -\xi_*)|^2 d \sigma \dxi \dxi_*
\lesssim \int_{\Ss^2}\left( \iint \la u \ra^{-3-\gamma} 
|\CalM (\xi^+ -u)\hat g(\xi^+ -u)|^2 \dxi \dxi_*\right) d \sigma\,,
\end{align*}
we have used 
the change of variables
$(\xi, \xi_*) \rightarrow (\xi^+, u)$ whose Jacobian is
\begin{align*}
&\Big|\frac{\partial(\xi^+,u)}{\partial (\xi, \xi_*)}\Big|=\Big|\frac{\partial \xi^+}{\partial \xi} \Big|=\frac{ \Big|I+ \frac{\xi}{|\xi|}\otimes
\sigma\Big|} {8}
 =\frac{|1+ \frac{\xi}{|\xi|}\cdot\sigma|}{8}=\frac{\cos^2
(\theta/2)}{4}\ge \frac{1}{8}, \qquad \theta\in [0, \, \pi/2].  \notag
\end{align*}

 

On the  set of the  integration for  $B_2$,  we recall  $\la \xi \ra \sim \la \xi - \xi_*\ra$ and $\eqref{II-part-bis}$ with $\kappa \in [0,1]$.
Setting
$\hat G(\xi) =  \CalM(\xi) \hat g(\xi)$, 
we have 
\begin{align*}
|B_2 | ^2  \lesssim& \|f\|_{L^1}^2\iiint b\,\, {\bf 1}_{ | \xi^- | \ge {1\over 2} \la \xi_*\ra}\frac{
|\hat\Phi_R (\xi_* - \xi^-) | \la \xi_* \ra^\kappa }
{\la \xi\ra^\kappa } |\hat G(\xi -\xi_*)|^2 d \sigma \dxi \dxi_* \\
& \times \iiint b\,\, {\bf 1}_{ | \xi^- | \ge {1\over 2} \la \xi_*\ra}\frac{
|\hat\Phi_R (\xi_* - \xi^-) | \la \xi_* \ra^\kappa }
{\la \xi\ra^\kappa } |{\hat h} (\xi ) |^2 d\sigma \dxi \dxi_*\,.
\end{align*}
We use the change of variables  $\xi_*  \rightarrow u= \xi_* -\xi^-$.
Note that $| \xi ^-| \ge {1\over 2} \la u +\xi^-\ra $ implies  $|\xi^-| \geq \la u\ra/\sqrt {10}$,
and that
\[
\la \xi_* \ra^\kappa \lesssim \la \xi_* - \xi^- \ra^\kappa + |\xi|^\kappa \sin^\kappa \theta/2\,.
\]
If $\kappa -2s\ne 0$,  then we have
\begin{align*}
& \quad \,
\iint  b\,\, {\bf 1}_{ | \xi^- | \ge {1\over 2} \la \xi_*\ra }  \frac{
|\hat\Phi_R (\xi_* - \xi^-) | \la \xi_* \ra}
{\la \xi\ra }d\sigma \dxi_*
\\
&\lesssim
\int \frac{{\bf 1}_{\la u\ra \lesssim |\xi|}}{\la u \ra^{3+\gamma}} 
\Big( \frac {\la u \ra^\kappa}
{\la \xi \ra^\kappa } \int  b\, {\bf 1}_{ | \xi^- |  \gtrsim \la u \ra } d\sigma +
\int  b(\cos \theta) \sin^\kappa (\theta/2)  {\bf 1}_{ | \xi^- |  \gtrsim \la u \ra } d\sigma
  \Big)du
\\
&\lesssim \int \frac{1}{\la u \ra^{3+\gamma} } \left( \frac {\la u \ra}
{\la \xi \ra}\right)^{ \kappa-2s} du \lesssim \la \xi \ra^{(2s-\kappa)^+}\,,
\end{align*}
which yields 
\[
|B_2| \lesssim \|f\|_{L^1} \|\CalM g\|_{H^{(s-\kappa/2)^+}}\| h\|_{H^{(s-\kappa/2)^+}},
\]
for any $\kappa \in [0,1]$ with $\kappa\neq 2s$.
Summing above estimates we complete the proof of  the proposition.
\end{proof}

The second commutator estimate is for $Q_R(\mu, g)$ with the weight $\vint{v}^k$. 
\begin{prop} \label{prop:commutator-R}
Let $k, R > 0$ and $Q_R$ be defined as in~\eqref{def:Q-R-bar-R}. Then
\begin{align*}
   \abs{\vpran{\vint{v}^k Q_R(\mu, g) - Q_R(\mu, \vint{v}^k g), \,\, h}}
\leq 
   C_{k,R} \min\big\{\norm{g}_{L^2} \norm{h}_{H^{s'}}, \, \,
   \norm{g}_{H^{s'}} \norm{h}_{L^2} \big\}  \,,
\end{align*}
where $C_{k, R}$ may depend on $k, R$ and 
$s' =0$ if $s < 1/2$, $s'>2s- 1$ if $s\ge 1/2$.
\end{prop}
\begin{proof}
By the definition of $Q_R$, we have
\begin{align*}
   \vpran{\vint{v}^k Q_R(\mu, g) - Q_R(\mu, \vint{v}^k g), \,\, h}
&= \int_{\R^3} \int_{\R^3} \int_{\Ss^2}
    b(\cos\theta) |v - v_\ast|^\gamma \chi_R(|v - v_\ast|)
    \mu_\ast g \vpran{\vint{v'}^k h' - \vint{v}^k h}
\\
& \quad \,
   - \int_{\R^3} \int_{\R^3} \int_{\Ss^2}
    b(\cos\theta) |v - v_\ast|^\gamma \chi_R(|v - v_\ast|)
    \mu_\ast g \vint{v}^k \vpran{h' - h}
\\
& = \int_{\R^3} \int_{\R^3} \int_{\Ss^2}
    b(\cos\theta) |v - v_\ast|^\gamma \chi_R(|v - v_\ast|)
    \vpran{\vint{v'}^k - \vint{v}^k} \mu_\ast g h' \,.
\end{align*}
Note that by the cutoff function on $|v-v_\ast|$, we have
\begin{equation}\label{cutoff-equi}
\la v \ra \sim \la v_* \ra \sim \la v' \ra \sim \la v'_* \ra\,,
\end{equation}
where the equivalence constants may depend on $R$. 
It follows from Lemma \ref{lem:diff-v-k} 
that
\begin{align} \label{ineq:v-v-k}
\vint{v'}^k - \vint{v}^k
= k \la v\ra^{k-2} |v-v_*|(v_*\cdot \omega)\cos^{k-1} \dfrac{\theta}{2} \sin 
\dfrac{\theta}{2} 
+ \mathfrak{R}\,,
\end{align}
where $|\mathfrak{R}| \le C_{k,R} \la v_*\ra^k \sin^2 \dfrac{\theta}{2}$ and $\Ss^2 \ni \omega \perp (v-v_*)$. As
in the proof of 
Proposition \ref{prop:trilinear-weight}, one can replace
$\omega$ by $\tilde \omega \in \Ss^2 $ with
$\tilde \omega \perp (v'-v_*)$. 
The term coming from $\mathfrak{R}$ can be  easily estimated by
$C_{k, R} \norm{g}_{L^2} \norm{h}_{L^2}$. In the case $0< s <1/2$,
we have the same bound for the term coming from 
the first term on the right hand side of \eqref{ineq:v-v-k}.
On the other hand, when $s \ge 1/2$
one can write 
\begin{align*}
&\int_{\R^3} \int_{\R^3} \int_{\Ss^2}
    b(\cos\theta) |v - v_\ast|^\gamma \chi_R(|v - v_\ast|)
 k \la v\ra^{k-2} |v-v_*|(v_*\cdot \omega)\cos^{k-1} \dfrac{\theta}{2} \sin 
\dfrac{\theta}{2}   
\mu_\ast g h' d\sigma dv dv_* 
\\
&= \int_{\R^3} \int_{\R^3} \int_{\Ss^2}
\cdots \mu_\ast g h d\sigma dv dv_* + 
\int_{\R^3} \int_{\R^3} \int_{\Ss^2}
\cdots \mu_\ast \la v \ra g \left(\frac{h'}{\la v'\ra}-\frac{h}{\la v \ra}\right) d\sigma dv dv_* 
\\
&  \quad \,
+ \int_{\R^3} \int_{\R^3} \int_{\Ss^2}
\cdots \mu_\ast g  \big(\la v'\ra -\la v\ra\big)\frac{h'}{\la v'\ra} d\sigma dv dv_*\,.
\end{align*}
The first term on the right hand side vanishes because
of the symmetry on $\Ss^2$. The third is estimated by 
$C\|g\|_{L^2} 
\|h\|_{L^2}$.  It follows from 
Cauchy-Schwarz inequality that the second is estimated by
\begin{align*}
&C_{k,R}
\left(\int_{\R^3} \int_{\R^3} \int_{\Ss^2}
b(\cos \theta) \theta^{2s+\epsilon} \la v_* \ra^{k}\mu_* g^2
d\sigma dv dv_* \right)^{1/2} \\
& \qquad 
\times 
\left(\int_{\R^3} \int_{\R^3} \int_{\Ss^2}
b(\cos \theta) \theta^{2-(2s+\epsilon)}
\la v_* \ra^{k}\mu_*
\left(\frac{h'}{\la v' \ra} -\frac{h}{\la v \ra} \right)^2d\sigma dv dv_* \right)^{1/2}\\
&\le C_{k,R} \|g\|_{L^2} 
\|\la v \ra^{-1} h\|_{H^{2s-1+\epsilon}_{2s-1+\epsilon}}\,,\enskip 
\mbox{for any $\epsilon >0$},
\end{align*}
because of \cite[Proposition 2.2]{AMUXY2012JFA}.
Similarly,  we have another bound since $\omega$ in
\eqref{ineq:v-v-k} 
can be replaced by $\tilde \omega$. Indeed, it suffices to
write
\begin{align*}
\la v \ra^{k-2}gh' = g' (\la \cdot \ra^{k-2} h)'
+\left(\frac{g}{\la v \ra} -\frac{g'}{\la v' \ra} \right) (\la \cdot \ra^{k-1} h)' + 
\left(\la v \ra^{k-1} - \la v' \ra^{k-1}\right)\frac{g}{\la v \ra} h'
\end{align*}
and to use the symmetry on $\Ss^2$ with the north pole
$v'-v_*$ for the first term
and 
Lemma  \ref{lem:diff-v-k} for the third term again. This completes the proof of the proposition.
\end{proof}

We also have the following direct bound on $Q_R(f, \mu)$.
\begin{prop} \label{prop:bound-Q-R}
Let $k, R > 0$ and $Q_R$ be defined as in~\eqref{def:Q-R-bar-R}. Then 
\begin{align*}
   \abs{\vpran{\vint{v}^k Q_R (f, \mu), \,\, h}}
\leq   C_{k,R} \norm{f}_{L^2} \norm{h}_{L^2} \,,
\end{align*}
where $C_{k, R}$ may depend on $k, R$.
\end{prop}
\begin{proof}
By the definition of $Q_R$, we have
\begin{align} \label{eq:v-k-Q-R}
   \vpran{\vint{v}^k Q_R (f, \mu), \,\, h}
& = \int_{\R^3} \int_{\R^3} \int_{\Ss^2}
     b(\cos\theta) |v - v_\ast|^\gamma \chi_R(|v - v_\ast|)
     f_\ast \mu \vpran{\vint{v'}^k h' - \vint{v}^k h} \nn
\\
& = \int_{\R^3} \int_{\R^3} \int_{\Ss^2}
     b(\cos\theta) |v - v_\ast|^\gamma \chi_R(|v - v_\ast|)
     f_\ast \mu \vpran{\vint{v'}^k - \vint{v}^k} h' \nn
\\
& \quad \,
     + \int_{\R^3} \int_{\R^3} \int_{\Ss^2}
     b(\cos\theta) |v - v_\ast|^\gamma \chi_R(|v - v_\ast|)
     f_\ast \mu \vint{v}^k \vpran{h' - h} \nn \\
&\Denote \mathcal{I} + \mathcal{II} \,.
\end{align}
For the first term $\mathcal{I}$ 
we use \eqref{ineq:v-v-k} with $\omega$ replaced by $\tilde \omega$ and split $\mathcal{I}=  \mathcal{I}_{main} + 
\mathcal{I}_{\mathfrak{R}}$.
In view of \eqref{cutoff-equi},
we have the following bound  for the term $\mathcal{I}_{\mathfrak{R}}$
coming from $\mathfrak{R}$, 
\begin{align*}
|\mathcal{I}_{\mathfrak{R}}| \le C_{k,R} \int_{\R^3} \int_{\R^3} \int_{\Ss^2}
b(\cos \theta) \theta^2 \la v_* \ra^{-2}| f_* |
\la v \ra^{k+2} \mu| h'|d\sigma dv dv_*
\le C_{k,R}\|f\|_{L^2} \|h\|_{L^2}.
\end{align*}
The main term is written as
\begin{align*}
&\mathcal{I}_{main} = k \int_{\R^3} \int_{\R^3} \int_{\Ss^2}
    b(\cos\theta) |v - v_\ast|^{\gamma+1} \chi_R(|v - v_\ast|)
(v_*\cdot \tilde \omega)\cos^{k-1} \dfrac{\theta}{2} \sin 
\dfrac{\theta}{2}   
f_\ast \big(\la v\ra^{k-2} \mu \big) h' d\sigma dv dv_* \\
&=k \int_{\R^3} \int_{\R^3} \int_{\Ss^2}
    b(\cos\theta) |v - v_\ast|^{\gamma+1} \chi_R(|v - v_\ast|)
(v_*\cdot \tilde \omega)\cos^{k-1} \dfrac{\theta}{2} \sin 
\dfrac{\theta}{2}   
f_\ast \Big( \big(\la v\ra^{k-2} \mu \big)- \big(\la v'\ra^{k-2} \mu' \big)\Big) h' d\sigma dv dv_* 
\end{align*}
because of the symmetry on $\Ss^2$ with the north pole
$v'-v_*$. Since it follows from the mean value theorem that 
\[
\chi_R(|v-v_*|)\left| \big(\la v\ra^{k-2} \mu \big)- \big(\la v'\ra^{k-2} \mu' \big)\right| \le C_{k,R} \int_0^1\la v_* \ra^{-10}\mu(v+ \tau (v' -v))^{1/2} d\tau
\sin \dfrac{\theta}{2},
\]
we have 
\begin{align*}
|\mathcal{I}_{main} | & \le C_{k,R}  \int_0^1\int_{\R^3}\int_{\R^3} \int_{\Ss^2}
    b(\cos\theta)\theta^2\frac{|f_*|}{\la v_*\ra^2}
\mu(v+ \tau (v' -v))^{1/2} |h'| d\sigma dv dv_* d\tau
\le C_{k,R}\|f\|_{L^2} \|h\|_{L^2}.
\end{align*}
To bound the second term $\mathcal{II}$ on the right hand side of~\eqref{eq:v-k-Q-R}, we 
use Proposition 2.1 in \cite{AMUXY2011}  
to have
\begin{align*} 
|\mathcal{II} |& =
   \abs{\int_{\R^3} \int_{\R^3} \int_{\Ss^2}
     b(\cos\theta) |v - v_\ast|^\gamma \chi_R(|v - v_\ast|)
     f_\ast \mu \vint{v}^k \vpran{h' - h}}  \nn
\\
& 
= \abs{\vpran{Q_R(f, \mu \vint{v}^k), \,\, h}}
\leq
   C_{k,R} \norm{f}_{L^2} \norm{h}_{L^2} \,.
\end{align*}
Then the  desired bound is a  combination of above three estimates.
\end{proof}



\subsection{Commutator estimate for $Q_{\bar R}$}

Now we consider the commutator between $\CalM(t,D_v, \eta)$ and the regular part of the kinetic 
factor $\Phi_{\bar R}(v-v_*)$ defined in~\eqref{def:Phi-R-bar-R}.  First, by the property of $\chi_R$ in~\eqref{property:chi-R}, we have
\begin{align} \label{bound:deriv-Phi-bar-R}
   \Phi_{\bar R} (z) \in C^\infty(\RR^3),   
\qquad
   |\del_z^\alpha \Phi_{\bar R}(z)| 
\leq 
   C \la z \ra^{\gamma- |\alpha|}\,,  
\end{align}
where $C$ is independent of $R$. From now on, we use the notation $\lesssim$ only when the constant involved is independent of $R, \delta, t, T$.

\begin{lem}\label{Lemma 2.1} Let $\gamma \le 1$ and denote 
\begin{align} \label{def:Phi-ast}
   \Phi_*(v)=\Phi_{\bar R}(v-v_*) \,,
\end{align} 
regarding $v_*$ as a parameter.
Denote
\begin{align*}
[\CalM(D_v), \Phi_*(v)]=\CalM(D_v)\Phi_*(v)-\Phi_*(v) M(D_v) \,.
\end{align*}
Then for any real number $r \in \R$ there
exists a constant $C = C(r) >0$ independent of $\delta, R >0$ and $0 \le t \le T$
such that the following estimate holds for any $f(v) \in {\mathcal
S} (\R^3)$.
\begin{equation}\label{2.1}
\|[\CalM(D_v), \Phi_*(v)]f\|_{{H^r}}\leq
C \|\CalM(D_v)f\|_{H^{r-1}} .
\end{equation}
In other words, we can write $[\CalM(D_v), \Phi_*(v)] = a(v,D_v) \La D_v \Ra^{-1}\CalM(D_v)$  with
$a(v,\xi)$ belonging to the symbol class $S^0_{0,0}$ uniformly with respect to $v_*$,  and in addition
to  $\delta, R, T, t$. 
\end{lem}
\begin{proof}
 The lemma follows from the calculus of
pseudo-differential operators ( for example, see Kumano-go
\cite{Kuma1982}). For the sake of completeness and for the later use in
some  arguments, we give a brief proof. It follows from
Theorem 3.1 of \cite{Kuma1982} that for any $N \in \N$,
\begin{equation}\label{2.2}
\CalM(D_v)\Phi_* = \Phi_* \CalM(D_v) + \sum_{1 \leq |\alpha| <
N} \frac{1}{\alpha!} \Phi_{*,(\alpha)} \CalM^{(\alpha)}(D_v) +
r_N(v,D_v),
\end{equation}
where $\Phi_{*,(\alpha)} (v)= D_v^{\alpha} \Phi_*(v)$,
$\CalM^{(\alpha)}(\xi)= \partial_{\xi}^{\alpha}\CalM_{\delta}(\xi)$
and
\begin{align*}
r_N(v,\xi) = N \sum_{|\alpha| = N} \int_0^1
\frac{(1-\tau)^{N-1}}{\alpha !} r_{N,\tau,\alpha}(v,\xi) d
\tau.\end{align*}
 Here
\begin{align*}
r_{N,\tau,\alpha}(v,\xi) = \mbox{Os}- \int \int e^{-iz\cdot \zeta}
\CalM^{(\alpha)}(\xi+ \tau \zeta) \Phi_{*,(\alpha)}(v+z)
\frac{dzd\zeta}{(2\pi)^3},
\end{align*} where ``Os-" means the
oscillating integral (\cite{Kuma1982}). Fix  $N \ge 3$. Since $\Phi_{*,(\alpha)} $ satisfies~\eqref{bound:deriv-Phi-bar-R} ($|\alpha| \ne 0$) 
and it follows from Lemma \ref{lem:CalM} that
\begin{align}\label{M-estimate-Weiran}
|\CalM^{(\alpha)}(\xi)| \leq C_{\alpha} \la \xi \ra^{-1}
\CalM(\xi)
\end{align}
with a constant  $C_{\alpha}>0$ depending only on $\alpha$, the
estimate (\ref{2.1}) with $r=0$ is obvious if we show that
$r_{N,\tau,\alpha}(v,\xi)/ \CalM(\xi) $ belongs to the symbol class
$S_{0,0}^{-1}$  uniformly with respect to $\tau \in [0,1]$. That
is, cf. \cite{Kuma1982},
\begin{equation}\label{2.3}
|D_v^{\beta} \partial_{\xi}^{\beta'}\Big(r_{N,\tau,\alpha}(v,\xi)/\CalM(\xi)\Big)| \leq
C_{\beta, \beta'} \la \xi \ra^{-1},
\end{equation}
because it follows from the product formula of pseudodifferential operators that
\[
Op(\Big(r_{N,\tau,\alpha}(v,\xi)/\CalM(\xi)\Big))Op\big(\CalM(\xi)\big) = r_{N,\tau,\alpha}(v,D_v).
\]
In view of \eqref{M-estimate-Weiran}, it suffices to show 
\begin{equation}\label{2.3-bis}
|D_v^{\beta} \partial_{\xi}^{\beta'}r_{N,\tau,\alpha}(v,\xi)| \leq
C_{\beta, \beta'} \la \xi \ra^{-1}\CalM(\xi),
\end{equation}
instead of \eqref{2.3}.
We only consider \eqref{2.3-bis} with $\beta = \beta'=0$ because the
proof for the general case is similar by taking the derivatives of
the integrand. Firstly, using the elementary identities
\[e^{-iz\cdot\zeta} =\la \zeta \ra^{-2l}(1-\Delta_z)^l e^{-iz\cdot\zeta}, \quad\quad
e^{-iz\zeta} = \la z \ra^{-2m}(1-\Delta_{\zeta})^m e^{-iz\cdot \zeta},
\]
we have, for $l, m \in \N$ with $l \ge 4, m \ge 2$,
\begin{equation}\label{remainder0}
 \begin{split}
r_{N,\tau,\alpha}(v,\xi)=& \displaystyle \int \displaystyle \Big(
\int e^{-i z \cdot \zeta} \la z \ra^{-2m} (1-\Delta_{\zeta})^m
\{
 \la \zeta \ra^{-2l}
(1-\Delta_{z})^l   \CalM^{(\alpha)}(\xi+ \tau \zeta)
\Phi_{*,(\alpha)}(v+z) \}  \frac{d\zeta}{(2\pi)^3}\Big) dz
\\
 =&
\displaystyle \int \displaystyle \{(1-\Delta_{z})^l
\Phi_{*,(\alpha)}(v+z)\}
\Big(  \int e^{-i z\cdot \zeta}
(1-\Delta_{\zeta})^m\{
\la \zeta \ra^{-2l}
  \CalM^{(\alpha)}(\xi+ \tau \zeta) \}  \frac{d\zeta}{(2\pi)^3}\Big)
\frac{dz}{\la z \ra^{2m}  }
\\
=&\displaystyle \int \displaystyle
\{(1-\Delta_{z})^l  \Phi_{*,(\alpha)}(v+z)\}
\Big(  \int_{|\zeta| \leq \frac{\la \xi
\ra}{2}}  \{\cdots \} \frac{d\zeta}{(2\pi)^3} +\int_{|\zeta| \geq
\frac{ \la \xi \ra}{2}}
 \{\cdots \}  \frac{d\zeta}{(2\pi)^3}\Big)
\frac{dz}{ \la z \ra^{2m}  }\\
 :=&\displaystyle \int \displaystyle
\{(1-\Delta_{z})^l  \Phi_{*,(\alpha)}(v+z)\} \Big( I_1(\xi;z) +
I_2(\xi,z) \Big) \frac{dz}{\la z\ra^{2m} } .
\end{split}
\end{equation}
Since $\la \xi \ra$ and $\la \xi+ \tau \zeta \ra$ are equivalent in
$I_1$, it follows that
\begin{align*}
|I_1| \leq C \la \xi \ra^{-1}\CalM(\xi).
\end{align*}
Moreover,  the same bound for $|I_2|$ holds because
\begin{align*}
|I_2| &\lesssim \int_{|\zeta| \geq
\frac{ \la \xi \ra}{2}}  \frac{\la \zeta \ra^{-1}\la \xi+\tau \zeta \ra^{-1}}
{\la \zeta \ra^{\ell-1} \Big((1 + \delta(T-t)\big(\la \xi+\tau \zeta \ra^{2s} + (T-t)^{2s} |\eta|^{2s}\big) \Big)^{ 1/2 + \varepsilon}}
\frac{d\zeta}{\la \zeta \ra^\ell}\\
&\lesssim \int_{|\zeta| \geq
\frac{ \la \xi \ra}{2}}  \frac{\la \xi \ra^{-1}}
{ \Big((1 + \delta(T-t)\big(\la \zeta \ra^{2(\ell-1)/(1+2\varepsilon)}} + (T-t)^{2s} |\eta|^{2s}\big) \Big)^{1/2 + \varepsilon}
\frac{d\zeta}{\la \zeta \ra^\ell} \lesssim \la \xi \ra^{-1}\CalM(\xi).
\end{align*}
These estimates lead us
to \eqref{2.3-bis}.  The estimate \eqref{2.1} with $r\ne 0$ also follows
by considering the expansion formula with $\la D_v \ra^r$ similar to
\eqref{2.2}.
\end{proof}
\begin{cor}\label{cor-234}
Let $\phi_{\ast, k}$ be defined as in~\eqref{def:phi-ast-k}
and $\Phi_{*,k}(v) = \Phi_{\bar R}(v-v_*)\phi_{*,k}(v)$. 
Then we have 
\begin{equation}\label{adjoint}
[\CalM(D_v), \Phi_{*,k}(v)] = A_{*,k}(v,D_v) \CalM(D_v)\la D_v \ra^{-1} \,,
\end{equation}
where $2^k\la v-v_*\ra^{-\gamma} A_{*,k}(v,D_v)$ is a pseudo-differential operator with a symbol belonging to $S^0_{0,0}$ and its 
operator norm from $L^2$ to $L^2$ is uniformly bounded with respect to $k, v_*, \delta, t, T, R$.  
Furthermore, for any real $a, b$ and $c$ satisfying $a+\gamma-1 = b+c$,
there exists a constant $C >0$ depending only on $a,b$ and $c$  such that 
\begin{equation}\label{com-imp-1}
\|\la v-v_*\ra^{a} A_{*,k}(v,D_v) f\|_{L^2}  \le  C 2^{kb} \|\la v-v_*\ra^c f\|_{L^2}. 
\end{equation}
\end{cor}

\begin{proof}
It suffices to note in the similar formula as in \eqref{remainder0} that for any real $b$ we have
\begin{align*}
  |\Phi_{*,k,(\alpha)}(v+z)| 
\leq 
   C \la v-v_* +z\ra^{\gamma-|\alpha| + b}2^{-kb}
\leq 
  C \la v-v_*\ra^{\gamma-|\alpha| + b}2^{-kb} \la z\ra^{|\gamma-|\alpha| + b|},
\end{align*}
where $C$ is independent of $R$. The power of $\la z \ra$ can be handled by the factor $\la z \ra^{2m}$
by taking a sufficiently large $m$. 
\end{proof}

\begin{cor}\label{cor-5-7} For any $k \in \RR$ there exists an $L^2$-bounded operator $A(v, D_v)$ whose symbol belongs
to $S^{-1}_{0,0}$ uniformly with respect to $\eta, \delta, \Eps, T-t$
such that 
\begin{align*}
&\la v \ra^k [\CalM, \la v \ra^{-k}] = [\la v \ra^k , \CalM] \la v \ra^{-k} = A(v, D_v) \CalM, \\
& 
\mbox{with}
\enskip \|\la v \ra ( 1+ |D_v|^2 + |T-t|^2 |\eta|^2)^{1/2}  A(v,D_v) g\|_{L^2_v} \le C \|g\|_{L^2_v} \enskip \mbox{for} \enskip g \in \mathcal{S}(\RR^3)\,, 
\end{align*}
where $C >0$ is independent of $\eta, \delta, \Eps, T-t$.

\end{cor}
\begin{proof}
The proof is almost the same as the one of Lemma \ref{Lemma 2.1} with $\Phi_*$ replaced by $\la v \ra^{-k}$. We thus omit the details.
\end{proof}

\begin{prop}\label{reg-commute} 
{Let $0<s<1$ and $0< \gamma \le1$.  For any $a \in \RR$
we have 
\begin{align}\label{desired-cut-part}
&\notag \left | \Big( \CalM Q_{\bar R}(f,g) - Q_{\bar R}(f, \CalM g), h \Big)\right| \\
&\lesssim 
\delta^{s'/(2s)} \|f\|^{1/2}_{L^1_{2|a| +\gamma+2}} \|h\|_{H^{s'+\varepsilon'}_{-a+\gamma/2}} \Big\{ \|f\|_{L^1_{2|a|+\gamma+2} }\|\mathcal{M}g\|^2_{L^2_{a+ \gamma/2}} 
\\
&\quad \notag + 
\int_{\RR^6\times \Ss^2} b(\cdot)| f_*| \la v_*\ra^{2|a|+ \gamma+2 }\Big| (\la \cdot \ra^{a+\gamma/2} \mathcal{M}g)(v')  
- (\la \cdot \ra^{a+ \gamma/2} \mathcal{M}g)(v) \Big|^2d \sigma dvdv_* \Big\}^{1/2}\\
&\quad \notag + R^{\max \{s-1, \gamma/2-1\}}  \Big\{
\|f\|_{L^1_{2|a|+\gamma+2}}\| \mathcal{M} g \|^2_{L^2_{a+\gamma/2}}\\
&\notag 
\qquad \qquad +\int_{\RR^6\times \Ss^2} b(\cdot)| f_*| \la v_*\ra^{2|a|+\gamma+2} |\big (\la \cdot \ra^{a+\gamma/2}\mathcal{M}g\big)(v')- 
\big( \la \cdot \ra^{a+\gamma/2} \mathcal{M}g\big) (v)|^2 d\sigma dv dv_* \Big \}^{1/2}\\
&\notag \qquad \times 
\Big\{ \|f\|_{L^1_{2|a|+\gamma+2}}\|h\|^2_{L^2_{-a+\gamma/2}} 
\\
&\notag \qquad \qquad + 
\int_{\RR^6\times \Ss^2} b(\cdot) |f_*| \la v_*\ra^{2|a|+ \gamma+2 }\Big| (\la \cdot \ra^{-a+\gamma/2} h )(v')  
- (\la \cdot \ra^{-a+\gamma/2} h )(v) \Big|^2d \sigma dvdv_* \Big\}^{1/2} \,,
\end{align}
where $0 < s' < s$ is arbitrary, and $\varepsilon'$ is non-negative and can be chosen as zero if $s \ne 1/2$. }

\end{prop}

\begin{proof}
For the proof 
we introduce the Littlewood-Paley
decomposition in $\R^3_v$ as follows:
\begin{equation*} 
\sum_{k=0}^{\infty} \phi_k(v) =1\,, \enskip \enskip  \phi_k(v) = \phi(2^{-k} v) \enskip
\mbox{for} \enskip k \geq 1 \enskip \mbox{with} \enskip 0 \leq
\phi_0, \phi \in C_0^{\infty}(\R^3),
\end{equation*}
and
\begin{equation*}
\text{supp}~\phi_0 \subset \{|v|<2\}, \qquad\quad \text{supp}~\phi
\subset \{ 1 <|v|< 3\}.
\end{equation*}
Take also $\tilde \phi_0$ and $\tilde \phi \in C_0^\infty$ such that
\begin{align*}
\tilde \phi_0 = 1 \ \ \text{on} \ \  \{|v|\leq 2\},& \qquad \qquad
\quad \text{supp}~\tilde{\phi}_0 \subset \{|v|<3\},
\\
\tilde \phi = 1 \ \ \text{on} \ \  \{  1/2 \leq |v|\leq 3\},& \qquad
\quad \text{supp}~\tilde{\phi} \subset \{1/3 <|v|<4\}.
\end{align*}
Furthermore, we assume that all these functions are radial. It
follows from the equivalence $|v'-v_*| \leq |v-v_*| \leq \sqrt{2}
|v'-v_*|$ that
\begin{equation}\label{B0}
\tilde \phi_k(v'-v_*) \phi_k(v-v_*) = \phi_k(v-v_*)=\tilde
\phi_k(v-v_*) \phi_k(v-v_*),\enskip k \geq 0.
\end{equation}
Write 
\begin{align} \label{def:phi-ast-k}
   \phi_{*,k}(v) = \phi_k(v-v_*), 
\qquad 
   \tilde \phi_{*,k}(v)= \tilde \phi(v-v_*) \,,
\qquad 
   \Phi_{*,k}(v) = \Phi_*(v)\phi_{*,k}(v) \,.
\end{align}
By the definition of $\Phi_\ast = \Phi_{\overline{R}}(v-v_*)$, 
we have 
\begin{align}\label{B1}
\Phi_*(v) = \sum_{2^k \geq R}^\infty \Phi_*(v)\phi_{*,k}(v).
\end{align}
If we set $\Phi_{*,k}(v):=\Phi_*(v) \phi_{*,k}(v)$, then for any real number $r$ we have
\begin{align}\label{B2}
\left|D_v^{\alpha} \Phi_{*,k}(v)\right| 
\leq C \la v-v_* \ra^{\gamma-|\alpha| +r} 2^{-kr} \tilde \phi_{*,k}(v),
\end{align}
where $C$ is independent of $R$.
It follows from \eqref{B0} and \eqref{B1} that
\begin{align*}
(Q_{\bar R}(f, g),h)& =\sum_{k=1}^\infty
\displaystyle\int_{\R^{6}}\displaystyle\int_{\Ss^{2}}b\left(\frac{(v-v_*)\cdot \sigma}{|v-v_*|}\right)
f(v_*)\Phi_{*,k}(v) g(v)
\left\{\overline{\tilde \phi_{*,k} (v^\prime)h(v')}-\overline{\tilde \phi_{*,k}(v)h(v)}\right\}d\sigma d v_*
d v.
\end{align*}
Therefore, writing $f_* = f(v_*), g =g(v), h' = h(v')$ and so on, we have
\begin{align*}
&\Big(\CalM(D_v)Q_{\bar R} (f,\, g) - Q_{\bar R}(f,\, \CalM(D_v)g),\,\, h\Big)\\
&= \sum_{2^k \ge R}^\infty \int_{\RR^{6}}
\int_{\Ss^{2}} f_* b\left(\frac{v-v_*}{|v-v_*|}\cdot \sigma\right)
\Phi_{*,k}g \Big(
\overline{\tilde \phi^{\prime}_{*,k} \big(\CalM \, h\big)^{\prime}}- \overline{\tilde \phi_{*,k}\big(\CalM\, h\big)}
\Big) d\sigma d v_* d v
\\
&\qquad -\sum_{2^k \ge R}^\infty \int_{\RR^{6}}
\int_{\Ss^{2}} f_* b\left(\frac{v-v_*}{|v-v_*|}\cdot \sigma\right)
\Phi_{*,k} \big(\CalM g\big)  \Big(
\overline{\tilde \phi^{\prime}_{*,k} h^{\prime}}-\overline{\tilde \phi_{*,k}h}
\Big) d\sigma d v_* d v
\\
&=\sum_{2^k \ge R}^\infty \Big[\int_{\RR^{6}}
\int_{\Ss^{2}} f_* b\left(\frac{v-v_*}{|v-v_*|}\cdot \sigma\right)
 \Big(\Phi_{*,k}g
\overline{ \big(\CalM \tilde \phi_{*,k} \, h\big)^{\prime}} -\big(\CalM \Phi_{*,k}g)\big)
\overline{ \big(\tilde \phi_{*,k} \, h\big)^{\prime}}\Big)d\sigma d v_* d v\\
&\quad \quad + \int_{\RR^{6}}\int_{\Ss^{2}} f_* b\left(\frac{v-v_*}{|v-v_*|}\cdot \sigma\right)
\Phi_{*,k}g \Big(
\overline{\big([\tilde \phi_{*,k} , \CalM ]\, h\big)^{\prime}}- \overline{[\tilde \phi_{*,k}, \CalM] h\big)}
\Big) d\sigma d v_* d v\\
&\quad \quad - \int_{\RR^{6}}
\int_{\Ss^{2}} f_* b\left(\frac{v-v_*}{|v-v_*|}\cdot \sigma\right)
\big ([\Phi_{*,k}, \CalM ]g\big)  \Big(
\overline{\big(\tilde \phi_{*,k} h\big)^{\prime}}-\overline{\tilde \phi_{*,k}h} 
\Big) d\sigma d v_* d v\Big]\\
& := \sum_{2^k \ge R}^\infty \Big[ {\mathcal I}_k + {\mathcal {II}}_k -  {\mathcal {III}}_k\Big] \,.
\end{align*}
\underline{\it Bound of $\CalI_k$} If we  set $G_{*,k} =\Phi_{*,k}g$ and $h_{*,k} =\tilde \phi_{*,k} h$, then it follows from the Bobylev formula that
\begin{align*}
{\mathcal I}_k = \int_{\RR^{6}}\int_{\Ss^{2}} f_* b\left(\frac{\xi}{|\xi|}\cdot \sigma\right)
\Big(\CalM(\xi) -\CalM(\xi^+)\Big) e^{iv_*\cdot\xi^+} \hat G_{*,k}(\xi^+) \overline{e^{iv_*\cdot \xi} \hat h_{*,k}(\xi)}
d\sigma d \xi \frac{dv_*}{(2\pi)^3}.
\end{align*}
Indeed, this follows from the substitution of 
\[
\CalM h_{*,k}(v') = \int_{\RR^3}  e^{iv'\cdot \xi} \CalM(\xi) \hat h_{*,k}(\xi)\frac{d \xi}{(2\pi)^3}, \enskip h_{*,k}(v') = \int_{\RR^3} e^{iv'\cdot \xi} 
\hat h_{*,k}(\xi)\frac{d \xi}{(2\pi)^3}, 
\]
and the exchange of $\frac{v-v_*}{|v - v_*|}\cdot \sigma$ and $\frac{\xi}{|\xi|}\cdot \sigma$ on the integrand. 
We decompose ${\mathcal I}_k$ into
\begin{align*}
{\mathcal I}_k &= \int_{\RR^3} f_* \bigg \{\int_{\RR^3\times \Ss^2}
b\left(\frac{\xi}{|\xi|}\cdot \sigma\right)\frac{\CalM(\xi) -\CalM(\xi^+)}{\CalM(\xi)} e^{iv_*\cdot\xi} 
\CalM(\xi) \hat G_{*,k}(\xi)\overline{e^{iv_*\cdot \xi} \hat h_{*,k}(\xi)}
\frac{d\sigma d \xi }{(2\pi)^3} 
\\
& \quad \,
+ \int_{\RR^3\times \Ss^2}b\left(\frac{\xi}{|\xi|}\cdot \sigma\right)
 \Big(e^{iv_*\cdot\xi^+} 
\CalM(\xi^+) \hat G_{*,k}(\xi^+)- e^{iv_*\cdot\xi} 
\CalM(\xi) \hat G_{*,k}(\xi)\Big)\\
& \hspace{5cm} \times \frac{\CalM(\xi) -\CalM(\xi^+)}{\CalM(\xi^+)}\overline{e^{iv_*\cdot \xi} 
\hat h_{*,k}(\xi)}
\frac{d\sigma d \xi   }{(2\pi)^3}
\\
& \quad \,
+ \int_{\RR^3\times \Ss^2}b\left(\frac{\xi}{|\xi|}\cdot \sigma\right)
 \frac{(\CalM(\xi) -\CalM(\xi^+))^2}{\CalM(\xi^+)\CalM(\xi)} e^{iv_*\cdot\xi} 
\CalM(\xi) \hat G_{*,k}(\xi) \overline{e^{iv_*\cdot \xi} 
\hat h_{*,k}(\xi)}
 \frac{d\sigma d \xi }{(2\pi)^3} \bigg\} dv_*\\
&\Denote \int f_* \Big\{ {\mathcal{I}}_k^{(1)}(v_*) + \mathcal{I}_k^{(2)}(v_*) + \mathcal{I}_k^{(3)}(v_*)\Big\} dv_*\,.
\end{align*}
First 
we estimate $\mathcal{I}_k^{(1)}(v_*)$ in the case $1/2 < s<1$ by using the Taylor expansion 
\begin{align*}
\CalM(\xi) -\CalM(\xi^+)
=\nabla \CalM(\xi) \cdot \xi^- - \int_0^1(1-\tau) \Big(\nabla\otimes \nabla \mathcal{M}(\xi - \tau\xi^-)\Big)d\tau  {\xi^-}\cdot {\xi^-} \,.
\end{align*}
Noting 
\[
\xi^- = \frac{|\xi|}{2}\left(\Big(\frac{\xi}{|\xi|}\cdot \sigma\Big)\frac{\xi }{|\xi|}-\sigma\right)
+ \left(1- \Big(\frac{\xi}{|\xi|}\cdot \sigma\Big)\right)\frac{\xi}{2},
\]
we see that the integral corresponding to the first term on the right hand side vanishes because of the symmetry
on $\Ss^2$. By means of the Cauchy-Schwarz inequality, it follows from Lemma  \ref{lem:CalM} that 
\begin{align}\label{estimate-1}
&|\mathcal{I}_k^{(1)}(v_*)| \lesssim \|\mathcal{M} 2^{k(a-\gamma/2)} \Phi_{*,k} g\|_{L^2_v} \|2^{k(-a+\gamma/2)} h_{*,k}\|_{L^2_v}\\
&\lesssim 
\la v_* \ra^{2|a|+\gamma} \left[
\|\tilde \phi_{*,k} \la v \ra^{a+\gamma/2} \mathcal{M} g\|_{L^2_v} 
+ 2^{-k}\|\la v \ra^{a+\gamma/2} 
 \mathcal{M} g\|_{L^2_v}
 \right]\|\tilde \phi_{*,k} \la v\ra^{-a+\gamma/2} h\|_{L^2_v},\notag 
\end{align}
where {we have used Lemma 5.2 and its proof}. 
When $s=1/2$, in view of Lemma \ref{lem:CalM}  we see that \eqref{estimate-1} holds with $L^2$ norm of 
one of the factor replaced by $H^{\varepsilon}(\RR^3_v)$. 
If $0<s<1/2$, then one can get
\eqref{estimate-1} directly by means of the mean value theorem instead of the Taylor expansion of the second order. The  mean value theorem can be applied to $\mathcal{I}_k^{(3)}(v_*)$ and we get the same bound as the right-hand side of \eqref{estimate-1} for $\mathcal{I}_k^{(3)}(v_*)$.

We  consider $\mathcal{I}_k^{(2)}(v_*)$. Since it follows from the same manipulation as the Bobylev formula
(see \cite[Proposition 1]{ADVW2000} and its proof) that
\begin{align*}
& \quad \,
\int_{\RR^3\times \Ss^2} b\left(\frac{\xi}{|\xi|}\cdot \sigma\right)
 \Big|e^{iv_*\cdot\xi^+} 
\CalM(\xi^+) \hat G_{*,k}(\xi^+)- e^{iv_*\cdot\xi} 
\CalM(\xi) \hat G_{*,k}(\xi)\Big|^2 \frac{d\sigma d \xi }{(2\pi)^3}
\\
&= \int_{\RR^3\times \Ss^2}b\left(\frac{v-v_*}{|v-v_*|}\cdot \sigma\right)
 \Big|
(\CalM G_{*,k})(v')-  
(\CalM  G_{*,k})(v) \Big|^2 d\sigma dv\,,
\end{align*}
we have 
\begin{align}\label{estimate-2-pre}
|{\mathcal I}_k^{(2)}(v_*)| \lesssim \left(
\int_{\RR^3\times \Ss^2}b\left(\cdot \right)2^{k(2a -\gamma)} 
 \Big|
(\CalM G_{*,k})(v')-  
(\CalM  G_{*,k})(v) \Big|^2 d\sigma dv\right)^{1/2} \|2^{k(-a+\gamma/2)} h_{*,k}\|_{L^2_v}\,.
\end{align}
Note that 
\begin{align*}
\Big|
(\CalM G_{*,k})(v')-  
(\CalM  G_{*,k})(v) \Big|^2 &\le 2 \Big|\Phi_{*,k}(v') (\mathcal{M}g)(v')  -\Phi_{*,k}(v) (\mathcal{M}g)(v) \Big|^2 
\\
& \quad \, 
+ 2\Big|\big([\mathcal{M},  \Phi_{*,k}]g\big)(v') - \big([\mathcal{M},  \Phi_{*,k}]g\big)(v) \Big|^2 .
\end{align*}
If we put $\widetilde \Phi_{*,k}(v) = 2^{k(a-\gamma/2)} \Phi_{*,k}(v) \la v \ra^{-(a+\gamma/2)}$ then 
$\widetilde \Phi_{*,k}(v) \lesssim \la v_*\ra^{|a|+ \gamma/2} \tilde \phi_{*,k}(v)$ and 
\begin{align*}
\big|\widetilde \Phi_{*,k}(v') - \widetilde \Phi_{*,k}(v)\big| \lesssim  \la v_*\ra^{|a|+ \gamma/2+1} \tilde \phi_{*,k}(v') \sin\theta/2.
\end{align*}
Therefore,
\begin{align*}
& \quad \,
2^{k(2a-\gamma)}
\Big|\Phi_{*,k}(v') (\mathcal{M}g)(v')  -\Phi_{*,k}(v) (\mathcal{M}g)(v) \Big|^2
\\
&\lesssim \la v_*\ra^{2|a| +\gamma} \tilde \phi_{*,k}(v)^2 \Big| (\la \cdot \ra^{a+\gamma/2} \mathcal{M}g)(v')  
- (\la \cdot \ra^{a+\gamma/2} \mathcal{M}g)(v) \Big|^2
\\
&\quad \, 
+ \theta^2 \la v_*\ra^{2|a|+\gamma+2} \tilde \phi_{*,k}(v')^2 \Big| (\la \cdot \ra^{a+\gamma/2} \mathcal{M}g)(v') |^2\,,
\end{align*}
and 
\begin{align*}
& \quad \, \int_{\RR^3\times \Ss^2}b\left(\cdot \right)2^{k(2a-\gamma)} 
\Big|\Phi_{*,k}(v') (\mathcal{M}g)(v')  -\Phi_{*,k}(v) (\mathcal{M}g)(v) \Big|^2 d \sigma dv
\\
& \lesssim \la v_*\ra^{2|a|+\gamma+2} \left\{\int_{\RR^3\times \Ss^2}b\left(\cdot \right)\tilde \phi^{\,\,2}_{*,k}(v) \Big| (\la \cdot \ra^{a+\gamma/2} \mathcal{M}g)(v')  
- (\la \cdot \ra^{a+\gamma/2} \mathcal{M}g)(v) \Big|^2d \sigma dv\right.\\
&\qquad \qquad \qquad \qquad  + \left. 
\int_{\RR^3} \tilde \phi^{\,\, 2}_{*,k}(v) \Big| (\la \cdot \ra^{a+\gamma/2} \mathcal{M}g)(v) |^2dv \right\}.
\end{align*}

In view of Corollary \ref{cor-234}, we put $[\mathcal{M},  \Phi_{*,k}]g = A_{*,k} \la D_v \ra^{-1} \mathcal{M}g
= A_{*,k}\tilde g$. Dividing the interval $[0,\pi/2]$ into $[0, 2^{-k}], [2^{-k}, \pi/2]$ we have 
\begin{align}\label{taihen}
&\int_{\RR^3\times \Ss^2}b\left(\cdot \right)2^{k(2a-\gamma)} 
\Big|\big([\mathcal{M},  \Phi_{*,k}]g\big)(v') - \big([\mathcal{M},  \Phi_{*,k}]g\big)(v) \Big|^2 
d \sigma dv \nn \\
&\qquad \qquad \lesssim  2^{k(2a-\gamma +2s)} 
\int_{\RR^3} \Big( \Big| \nabla_v (A_{*,k} \tilde g)\Big|^2 + \Big| A_{*,k} \tilde g\Big|^2 \Big) dv \\
&\qquad \qquad  \lesssim  2^{-2k(1-s)}\int_{\RR^3} |\la v-v_* \ra^{a+\gamma/2} \mathcal{M}g|^2 dv 
\lesssim 2^{-2k(1-s)}\la v_* \ra^{2|a|+ \gamma} \int_{\RR^3} |(\la \cdot \ra^{a+\gamma/2} \mathcal{M}g)(v)|^2 dv , \nn
\end{align}
where we have used \eqref{com-imp-1} , and the fact that $\nabla_v (A_{*,k} \tilde g) = (\nabla_v A_{*,k}) \tilde g +
 A_{*,k} (\nabla_v \tilde g)$
and the regular change of variable 
$v \rightarrow v + \tau(v'-v)$ for $\tau \in [0,1]$.

Above two estimates together with \eqref{estimate-2-pre} lead us  to
\begin{align}\label{estimate-2}
|{\mathcal I}_k^{(2)}(v_*)| &\lesssim \la v_* \ra^{2|a| +\gamma+2} 
\|\tilde \phi_{*,k} \la v\ra^{-a+\gamma/2} h\|_{L^2_v}\notag \\
&\quad \times 
\left[ \Big( 
\int_{\RR^3\times \Ss^2}b\left(\cdot \right)\tilde \phi_{*,k}(v)^2 \Big| (\la \cdot \ra^{a+\gamma/2} \mathcal{M}g)(v')  
- (\la \cdot \ra^{a+\gamma/2} \mathcal{M}g)(v) \Big|^2d \sigma dv\Big)^{1/2}\right. \notag \\
&+
\Big( \int_{\RR^3} \tilde \phi_{*,k}(v)^2 \Big| (\la \cdot \ra^{a+\gamma/2} \mathcal{M}g)(v) |^2dv\Big)^{1/2}
+\left. 
\Big(\int_{\RR^3} 2^{-2k(1-s)} \Big| (\la \cdot \ra^{a+\gamma/2}  \mathcal{M}g)(v) |^2dv \Big)^{1/2}
\right]\,.
\end{align}
Summing up estimates for ${\mathcal I}_k^{(j)}(v_*)$, $j=1,2,3$, we have 
\begin{align*}
\sum_{2^k \ge R} \mathcal{I}_k \lesssim &
\|f\|_{L^1_{2|a|+\gamma+2} }^{1/2} \|h\|_{L^2_{-a +\gamma/2}} \left\{ \|f\|_{L^1_{2|a| +\gamma+2}}^{1/2} \|\mathcal{M}g\|_{L^2_{a+\gamma/2}} +\right.\\
&\left.
\left( \int_{\RR^6\times \Ss^2} b(\cdot) |f_*| \la v_*\ra^{2|a| +\gamma+2} \Big| (\la \cdot \ra^{a+\gamma/2} \mathcal{M}g)(v')  
- (\la \cdot \ra^{a+\gamma/2} \mathcal{M}g)(v) \Big|^2d \sigma dvdv_* \right)^{1/2}\right\},
\end{align*}
provided that $s \ne 1/2$. When $s=1/2$, the term $ \|f\|_{L^1_{2|a|+\gamma+2}}^{1/2} \|\mathcal{M}g\|_{L^2_{a+\gamma/2}} $
is replaced by $ \|f\|_{L^1_{|a|+\gamma+2}}^{1/2} \|\mathcal{M}g\|_{H^{\varepsilon'}_{a+\gamma/2}} $.

In order to have a small factor we go back to above procedure.
All ${\mathcal I}_k^{(j)}(v_*)$ contain a factor $\mathcal{M}(\xi^+) - \mathcal{M}(\xi)$.
Note that for any $0< s' <s$ we have 
\begin{align*}
\left|\frac{\mathcal{M}(\xi^+) - \mathcal{M}(\xi)}{ \mathcal{M}(\xi)}\right|
&\lesssim \frac{\delta(T-t) (1+ |\xi| + (T-t) |\eta| )^{2s-1} |\xi|\theta}{1+\delta(T-t) ( 1+ |\xi| + (T-t) |\eta| )^{2s}}
\le (\delta (T-t))^{s'/(2s)}|\xi|^{s'}\theta 
\end{align*}
because of $\sup_{X \in [0,\infty)} X^{(2s-s')/2s}/ (1+ X) < 1$, and moreover 
\begin{align*}
\left|\frac{1}{\mathcal{M}(\xi)}\int_0^1(1-\tau) \Big(\nabla\otimes \nabla \mathcal{M}(\xi - \tau\xi^-)\Big)d\tau  {\xi^-}\cdot {\xi^-}\right|
\lesssim (\delta (T-t))^{s'/(2s)}|\xi|^{s'}\theta^2\,, 
\end{align*}
with a convention that $|\xi|^{s'}$ is replaced by $|\xi|^{s'+ \varepsilon'}$ when $s=1/2$. 

Using the above observation we can show that
\begin{align*}
\sum_{2^k \ge R} \mathcal{I}_k \lesssim & \,
\delta^{s'/(2s)} \|f\|_{L^1_{2|a|+\gamma+2}}^{1/2} \|h\|_{H^{s'+\varepsilon'}_{-a+\gamma/2}} \Big\{ \|f\|_{L^1_{2|a|+ \gamma+2}}^{1/2} 
\|\mathcal{M}g\|_{L^2_{a+\gamma/2}} 
\\
&+ 
\left( \int_{\RR^6\times \Ss^2} b(\cdot) |f_*| \la v_*\ra^{2|a| +\gamma +2}
\Big| (\la \cdot \ra^{a+\gamma/2} \mathcal{M}g)(v')  
- (\la \cdot \ra^{a+\gamma/2} \mathcal{M}g)(v) \Big|^2d \sigma dvdv_* \right)^{1/2}\Big\},
\end{align*}
where $\varepsilon' >0$ can be chosen to be zero if $s \ne 1/2$.

\medskip
\noindent
 \underline{\it Bound of $\CalI\CalI_k$} Put $H_{*,k}(v) = [\tilde \phi_{*,k}, \CalM] h$. Then as for ${\mathcal I}_k$, we have 
\begin{align*}
{\mathcal {II}}_k &= \int_{\RR^{3}} f_* \Big( \int_{\RR^3} \int_{\Ss^{2}} b\left(\frac{\xi}{|\xi|}\cdot \sigma\right)
\Big( e^{iv_*\cdot\xi^+} \hat G_{*,k}(\xi^+)  - e^{iv_*\cdot\xi} \hat G_{*,k}(\xi) \Big)\overline{e^{iv_*\cdot \xi} \hat H_{*,k}(\xi)}
d\sigma d \xi\Big)  \frac{dv_*}{(2\pi)^3}.
\end{align*}
We decompose this into 
\begin{align*}
{\mathcal {II}}_k &= \int_{\RR^3} f_* \left \{\int_{\RR^3\times \Ss^2}
b\left(\frac{\xi}{|\xi|}\cdot \sigma\right)\frac{\CalM(\xi) -\CalM(\xi^+)}{\CalM(\xi)} e^{iv_*\cdot\xi} 
\CalM(\xi) \hat G_{*,k}(\xi)\overline{e^{iv_*\cdot \xi} \frac{\hat H_{*,k}(\xi)}{\mathcal{M}(\xi)}}
\frac{d\sigma d \xi }{(2\pi)^3} \right. 
\\
& \quad \,
+ \int_{\RR^3\times \Ss^2}b\left(\frac{\xi}{|\xi|}\cdot \sigma\right)
 \Big(e^{iv_*\cdot\xi^+} 
\CalM(\xi^+) \hat G_{*,k}(\xi^+)- e^{iv_*\cdot\xi} 
\CalM(\xi) \hat G_{*,k}(\xi)\Big)
\\
& \hspace{5cm} \times \frac{\CalM(\xi)-\CalM(\xi^+) }{\CalM(\xi^+)}\overline{e^{iv_*\cdot \xi} 
\frac{\hat H_{*,k}(\xi)}{\mathcal{M}(\xi)}}
\frac{d\sigma d \xi   }{(2\pi)^3}
\\
& \quad \, 
{+ \int_{\RR^3\times \Ss^2}b\left(\frac{\xi}{|\xi|}\cdot \sigma\right)
 \Big(e^{iv_*\cdot\xi^+} 
\CalM(\xi^+) \hat G_{*,k}(\xi^+)- e^{iv_*\cdot\xi} 
\CalM(\xi) \hat G_{*,k}(\xi)\Big)}
\\
& \hspace{5cm} {  \times \overline{e^{iv_*\cdot \xi} 
\frac{\hat H_{*,k}(\xi)}{\mathcal{M}(\xi)}}
\frac{d\sigma d \xi   }{(2\pi)^3}}
\\
& \quad \,\left. + \int_{\RR^3\times \Ss^2}b\left(\frac{\xi}{|\xi|}\cdot \sigma\right)
 \frac{(\CalM(\xi) -\CalM(\xi^+))^2}{\CalM(\xi^+)\CalM(\xi)} e^{iv_*\cdot\xi} 
\CalM(\xi) \hat G_{*,k}(\xi) \overline{e^{iv_*\cdot \xi} 
\frac{\hat H_{*,k}(\xi)}{\mathcal{M}(\xi)}}
 \frac{d\sigma d \xi }{(2\pi)^3} \right\} dv_*\\
&\Denote \int f_* \Big\{ {\mathcal{II}}_k^{(1)}(v_*) + \mathcal{II}_k^{(2)}(v_*) + {\mathcal{II}_k^{(3)}(v_*)} + \mathcal{II}_k^{(4)}(v_*)\Big\} dv_*\,.
\end{align*}
By means of  Corollary \ref{cor-234} (with $\gamma =0$), we have 
\begin{align*}
   H_{*,k}(v) = -([\tilde \phi_{*,k}, \CalM] )^* h 
= -\la D_v \ra^{-1} \mathcal{M} (A_{*,k}(v,D_v))^* h
\end{align*}
and 
\[
2^{-ka} \|\la \xi \ra^\varepsilon \mathcal{M}(\xi)^{-1} \hat H_{*,k}(\xi)\|_{L^2_\xi} \lesssim 2^{-ka} 
\|\la D \ra^{-1+ \varepsilon}  (A_{*,k}(v,D_v))^* h \|_{L^2_v} \lesssim 2^{-k} \|\la v -v_*\ra^{-a}
h\|_{L^2_v}\,. 
\] 
Similar to the argument for  $\mathcal{I}_k^{(1)}$,  we can obtain
\begin{align}\label{estimate-2-1}
|\mathcal{II}_k^{(1)}(v_*)| 
&\lesssim 
\la v_* \ra^{2|a| +\gamma/2} \left[
\|\tilde \phi_{*,k} \la v \ra^{a+\gamma/2} \mathcal{M} g\|_{L^2_v} 
+ 2^{-k}\|\la v \ra^{a+\gamma/2} 
 \mathcal{M} g\|_{L^2_v}
 \right]\times 2^{k(\gamma/2-1)}\|\la v \ra^{-a} h\|_{L^2_v}\,,
\end{align}
whose right hand side becomes  obviously the upper bound for $|\mathcal{II}_k^{(4)}(v_*)|$ too. 
By the almost same procedure as for $\mathcal{I}_k^{(2)}(v_*)$, 
we also have 
\begin{align}\label{estimate-2-2}
|{\mathcal{II}}_k^{(2)}(v_*)| &\lesssim \la v_* \ra^{|a|+2}2^{k(\gamma/2-1)} 
\|\la v\ra^{-a} h\|_{L^2_v}\notag \\
&\times 
\left[ \Big( 
\int_{\RR^3\times \Ss^2}b\left(\cdot \right)\tilde \phi_{*,k}(v)^2 \Big| (\la \cdot \ra^{a+\gamma/2} \mathcal{M}g)(v')  
- (\la \cdot \ra^{a+\gamma/2} \mathcal{M}g)(v) \Big|^2d \sigma dv\Big)^{1/2}\right. \notag \\
&+
\Big( \int_{\RR^3} \tilde \phi_{*,k}(v)^2 \Big| (\la \cdot \ra^{a+\gamma/2} \mathcal{M}g)(v) |^2dv\Big)^{1/2}
+\left. 
\Big(\int_{\RR^3} 2^{-2k(1-s)} \Big| (\la \cdot \ra^{a+\gamma/2} \mathcal{M}g)(v) |^2dv \Big)^{1/2}
\right]\,.
\end{align}

Putting 
\begin{align*}
&\tilde G_k(v;v_*) = 2^{k(a-\gamma/2)} \mathcal{M}(D_v) \Phi(|v-v_*|) \phi_k(v-v_*) g(v),\\
&\tilde H_k(v;v_*)  = \mathcal{M}^{-1}(D_v)[ \tilde \phi_k(v-v_*), \mathcal{M}(D_v)]2^{k(-a+\gamma/2)}  h(v)\,,
\end{align*}
we consider ${\mathcal{II}}_k^{(3)}(v_*)$ by first writing it  as 
\begin{align*}
{\mathcal{II}}_k^{(3)}(v_*)
&= \int_{\RR^3 \times \Ss^2} b(\cdot) \tilde G_k(v;v_*) \overline{
\big (\tilde H_k(v';v_*) - \tilde H_k(v;v_*) \big )}  d\sigma dv\\
&= \int_{\RR^3 \times \Ss^2} b(\cdot)\big( \tilde G_k(v;v_*)- 
\tilde G_k(v';v_*)\big)
 \overline{
\big (\tilde H_k(v';v_*) - \tilde H_k(v;v_*) \big )}  d\sigma dv\\
&\quad+\int_{\RR^3 \times \Ss^2} b(\cdot) \tilde G_k(v';v_*) \overline{
\big (\tilde H_k(v';v_*) - \tilde H_k(w;v_*) \big )}  d\sigma dv\\
&\quad + \int_{\RR^3 \times \Ss^2} b(\cdot) \tilde G_k(v;v_*) \overline{
\big (\tilde H_k(w;v_*) - \tilde H_k(v;v_*) \big )}  d\sigma dv\\
&\quad + \int_{\RR^3 \times \Ss^2} b(\cdot)\big( \tilde G_k(v';v_*)- 
\tilde G_k(v;v_*)\big)
 \overline{
\big (\tilde H_k(w;v_*) - \tilde H_k(v;v_*) \big )}  d\sigma dv\\
&\Denote S^k(v_*) + M_0^k(v_*) + R^k_1(v_*) + R^k_2(v_*), 
\end{align*}
where {(see Figure 1 of [AMUXY10-2013])}
\begin{align*}
&w = v_* + \Big(\cos^2 \frac{\theta}{2}\Big) (v-v_*) = \frac{v'+v_*}{2} + \frac{|v'-v_*|}{2} \omega, \\
&\cos \theta = \frac{v-v_*}{|v-v_*|} \cdot \sigma = \frac{v'-v_*}{|v'-v_*|} \cdot \omega.
\end{align*}
It follows from the Cauchy-Schwarz inequality that
\begin{align*}
  |S^k(v_*)|^2 
\le 
  \vpran{\int_{\RR^3 \times \Ss^2} b(\cdot) \left|\tilde G_k(v;v_*)-\tilde G_k(v';v_*)\right|^2
d\sigma dv}
\vpran{\int_{\RR^3 \times \Ss^2} b(\cdot) \left|
\tilde H_k(v';v_*) - \tilde H_k(v;v_*) \right| ^2d\sigma dv} \,.
\end{align*}
In the estimation for ${\mathcal I}_k^{(2)}(v_*)$,  we have already shown 
\begin{align*}
& \quad \,
\int_{\RR^3 \times \Ss^2} b(\cdot) \left|\tilde G_k(v;v_*)-\tilde G_k(v';v_*)\right|^2
d\sigma dv \\
&\lesssim  \la v_*\ra^{\gamma+2|a|+2}\Big\{
\int_{\RR^3\times \Ss^2}b\left(\cdot \right)\tilde \phi_{*,k}(v)^2 \Big| (\la \cdot \ra^{a+\gamma/2} \mathcal{M}g)(v')  
- (\la \cdot \ra^{a+\gamma/2} \mathcal{M}g)(v) \Big|^2d \sigma dv \notag \\
&\notag \quad  + 
\int_{\RR^3} \tilde \phi_{*,k}(v)^2 \Big| (\la \cdot \ra^{a+\gamma/2} \mathcal{M}g)(v) |^2dv
+
\int_{\RR^3} 2^{-2k(1-s)} \Big| (\la \cdot \ra^{a+ \gamma/2} \mathcal{M}g)(v) |^2dv \Big\}
\,.
\end{align*}
Similar argument as for \eqref{taihen} yields
\begin{align*}
\int_{\RR^3 \times \Ss^2} b(\cdot) \left|
\tilde H_k(v';v_*) - \tilde H_k(v;v_*) \right| ^2d\sigma dv 
\lesssim 2^{-2k(1-s)} \la v_*\ra^{2|a| + \gamma+2 } 
\int_{\RR^3}  \Big| \la v \ra^{-a+ \gamma/2} h(v) \Big|^2dv \,.
\end{align*}
Therefore, we have
\begin{align}\label{sym-II}
&\sum_{2^k \ge R} \int_{\RR^3} f(v_*)|S^k(v_*)| dv_* \notag \\
&\quad \lesssim R^{s-1} \|f\|^{1/2}_{L^1_{\gamma+2|a|+2} }\|\la v \ra^{-a +\gamma/2} h\|_{L^2}
\Big(  \|f\|_{L^1_{\gamma+2|a|+2} }\|\mathcal{M}g\|_{L^2_{a+\gamma/2}}^2 \\
&\qquad + \int_{\RR^6\times \Ss^2} b(\cdot) |f_* |\la v_*\ra^{\gamma+2|a|+2}
\Big| (\la \cdot \ra^{a+\gamma/2} \mathcal{M}g)(v')  
- (\la \cdot \ra^{a+\gamma/2} \mathcal{M}g)(v) \Big|^2d\sigma dv dv_* \Big)^{1/2}\,. \nn
\end{align}
The same observation as for $M_1$ in the proof of Lemma B.4  shows 
\begin{align}\label{main-part}
&M_0^k(v_*) = -{\mathcal{II}}_k^{(3)}(v_*) + R_3^k(v_*), \\
&R_3^k(v_*)= \int_{\RR^3 \times \Ss^2} b(\cdot) \tilde G_k(v';v_*) \overline{
\big (\tilde H_k(v;v_*) - \tilde H_k(v';v_*) \big )}\left(\frac{1}{\cos^3(\theta/2) }-1\right) d\sigma dv\,. \notag
\end{align}
The term coming from $R^k_3$ is bounded as
\begin{align}\label{estimate-2-3}
\sum_{2^k \ge R} \int_{\RR^3} f(v_*)|R_3^k(v_*)| dv_* \lesssim R^{\gamma/2-1}\|f\|_{L^1_{\gamma/2+2|a|+2}} \|\mathcal{M}g\|_{L^2_{a+\gamma/2}}
\|\la v \ra^{-a} h\|_{L^2}\,.
\end{align}
We consider $R_1^k(v_*)$, by taking the change of variables $v \rightarrow z+ v_*$. Putting 
\begin{align*}
&\tilde G_k(z+v_* ;v_*) = 2^{k(a-\gamma/2)} \mathcal{M}(D_z) \Phi(|z|) \phi_k(z) g(z+v_*) \Denote G_{*,k}^0(z),\\
&\tilde H_k(z+v_* ;v_*)  = \mathcal{M}^{-1}(D_z)[ \tilde \phi_k(z), \mathcal{M}(D_z)]2^{k(-a+\gamma/2)}  h(z+v_*) \Denote H_{*,k}^0(z)\,,
\end{align*}
we have 
\begin{align*}
R_1^k(v_*) &= \int_0^{2\pi} b(\cos \theta) \sin \theta \int_{\RR^3} G_{*,k}(z) \overline{\Big(
H_{*,k}^0(z \cos^2(\theta/2) )  - H_{*,k}^0(z )\Big)}dz d\theta\\
&= \int_0^{2\pi} b(\cos \theta) \sin \theta \int_{\RR^3} \widehat G_{*,k}(\xi) \overline{
\Big(\frac{1}{\big(\cos^2(\theta/2)\big)^3 }\widehat H_{*,k}^0(\frac{\xi }{\cos^2(\theta/2)} )  -\widehat  H_{*,k}^0(\xi)\Big)}\frac{d\xi}{(2\pi)^3} d\theta\\
&=\int_0^{2\pi} b(\cos \theta) \sin \theta \left( \frac{1}{\big(\cos^2(\theta/2)\big)^3 }-1\right)
\int_{\RR^3} \widehat G_{*,k}(\xi) \overline{
\widehat H_{*,k}^0(\frac{\xi }{\cos^2(\theta/2)} )}\frac{d\xi}{(2\pi)^3} d\theta\\
&\quad + \int_0^{2\pi} b(\cos \theta) \sin \theta \int_{\RR^3} \widehat G_{*,k}(\xi)\overline{
\Big(\widehat H_{*,k}^0(\frac{\xi }{\cos^2(\theta/2)} )  -\widehat  H_{*,k}^0(\xi)\Big)}\frac{d\xi}{(2\pi)^3} d\theta\\
&\Denote
R_{1,1}^k(v_*) + R_{1,2}^k(v_*).
\end{align*}
We have
\begin{align}\label{estimate-2-4}
\sum_{2^k \ge R} \int_{\RR^3} f(v_*)|R_{1,1}^k(v_*)| dv_* \lesssim R^{\gamma/2-1}\|f\|_{L^1_{\gamma/2+2|a|+2}} 
\|\mathcal{M}g\|_{L^2_{a+\gamma/2}}
\|\la v\ra^{-a} h\|_{L^2}\,.
\end{align}
Write
\begin{align*}
R_{1,2}^k(v_*)& =  \int_{\RR^3} \widehat G_{*,k}(\xi) \Big(\int_0^{2^{-k/2} \la \xi \ra^{-1/2}}b(\cos \theta) \sin \theta\tan^2 \frac{\theta}{2}\\
& \qquad \qquad \qquad \qquad \qquad \qquad \times  \overline{
\Big(\int_0^1 \xi \cdot \nabla_\xi \widehat H_{*,k}^0\big(\xi (1+ \tau \tan^2 \frac{\theta}{2})\big)d\theta d\tau \Big)}  \frac{d\xi}{(2\pi)^3}\\
&+ \int_{\RR^3} \widehat G_{*,k}(\xi) \Big(\int_{2^{-k/2} \la \xi \ra^{-1/2}}^{\pi/2} b(\cos \theta) \sin \theta \overline{
\Big(\widehat H_{*,k}^0(\frac{\xi }{\cos^2(\theta/2)} )  -\widehat  H_{*,k}^0(\xi)\Big)} d\theta\Big)  \frac{d\xi}{(2\pi)^3}\\
&\Denote B^{(1)}_k(v_*) + B^{(2)}_k(v_*). 
\end{align*}
By Cauchy-Schwarz inequality,  we have 
\begin{align*}
|B^{(1)}_k(v_*)|^2 &\lesssim \int_{\RR^3} \la \xi \ra^{1-s} |\widehat G^0_{*,k}(\xi)|^2 \int_0^{2^{-k/2} \la \xi \ra^{-1/2}}b(\cos \theta) \sin \theta\tan^2 \frac{\theta}{2}d\theta d\xi \\
&\qquad  \times \int_{\RR^3} \la \xi \ra^{1+s} |\nabla_\xi \widehat H^0_{*,k}(\xi)|^2 \int_0^{2^{-k/2} \la \xi \ra^{-1/2}}b(\cos \theta) \sin \theta\tan^2 \frac{\theta}{2}d\theta d\xi \,,
\end{align*}
where, in the second factor, we have used the change of variable 
\[
\xi \rightarrow (1+\tau \tan^2 \frac{\theta}{2} )\xi
\]
after exchanging $d\theta d\xi$ by $d\xi d\theta$. Therefore,
\begin{align*}
|B^{(1)}_k(v_*)|^2 &\lesssim \left(\int_{\RR^3}  |\widehat G^0_{*,k}(\xi)|^2 d\xi \right)
\left(\int_{\RR^3} 2^{2ks}\la \xi \ra^{2s} |\nabla_\xi 2^{-k} \widehat H^0_{*,k}(\xi)|^2  d\xi\right)\,.
\end{align*}
On the other hand,
\begin{align*}
|B^{(2)}_k(v_*)|^2 &\lesssim \left(\int_{\RR^3}  |\widehat G^0_{*,k}(\xi)|^2 d\xi \right)
\left(\int_{\RR^3} 2^{2ks}\la \xi \ra^{2s} | \widehat H^0_{*,k}(\xi)|^2  d\xi\right)\,.
\end{align*}
Both estimates lead us to
\begin{align}\label{estimate-2-5}
\sum_{2^k \ge R} \int_{\RR^3} f(v_*)|R_{1,2}^k(v_*)| dv_* \lesssim R^{s-1}\|f\|_{L^1_{\gamma+2|a|}} \|\mathcal{M}g\|_{L^2_{a+\gamma/2}}
\|\la D\ra^{s-1} \la v \ra^{-a} h\|_{L^2_{\gamma/2}}\,.
\end{align}

As for $R_2^k(v_*)$, it follows from Cauchy-Schwarz inequality that
\begin{align*}
|R_2^k(v_*)|^2 &\le \left(
\int_{\RR^3 \times \Ss^2} b(\cdot)\left| \tilde G_k(v';v_*)- 
\tilde G_k(v;v_*)\right|^2 d\sigma dv\right)
 \left(
\int_{\RR^3 \times \Ss^2} b(\cdot)\left| 
\tilde H_k(w;v_*) - \tilde H_k(v;v_*) \right|^2 d\sigma dv\right).
\end{align*}
The first factor of the right hand side is exactly the same as the one of \eqref{estimate-2-pre}.
In view of   the proof of Lemma B.4 for $M_1$, the second factor is bounded by 
\begin{align*}
2\Big\{ \int_{\RR^3 \times \Ss^2} &b(\cdot)\left| 
\tilde H_k(w;v_*) - \tilde H_k(v';v_*) \right|^2 d\sigma dv
+  \int_{\RR^3 \times \Ss^2} b(\cdot)\left| 
\tilde H_k(v' ;v_*) - \tilde H_k(v;v_*) \right|^2 d\sigma dv\Big\}\\
&\le 2^{5/2} \int_{\RR^3 \times \Ss^2} b(\cdot)\left| 
\tilde H_k(v' ;v_*) - \tilde H_k(v;v_*) \right|^2 d\sigma dv,
\end{align*}
which concludes
\begin{align}\label{estimate-2-6}
\sum_{2^k \ge R}& \int_{\RR^3} f(v_*)|R_2^k(v_*)| dv_* \lesssim R^{s-1}\|f\|^{1/2}_{L^1_{\gamma+2|a|+2}} \|\la v \ra^{-a+\gamma/2} h\|_{L^2}
\Big\{ \|f\|_{L^1_{\gamma+2|a|+2}}^{1/2} \|\mathcal{M}g\|_{L^2_{a+\gamma/2}} 
\\
&\notag + 
\left( \int_{\RR^6\times \Ss^2} b(\cdot)| f_* |\la v_*\ra^{2|a|+\gamma+2} \Big| (\la \cdot \ra^{a+\gamma/2} \mathcal{M}g)(v')  
- (\la \cdot \ra^{a+ \gamma/2} \mathcal{M}g)(v) \Big|^2d \sigma dvdv_* \right)^{1/2}\Big\}\,.
\end{align}
Summing up estimates from \eqref{sym-II} to \eqref{estimate-2-6} we have 
\begin{align*}
\sum_{2^k \ge R}& \left|\int_{\RR^3} f(v_*) \mathcal{II}_k^{(3)}(v_*) dv_*\right | \lesssim
R^{\max\{s-1, \gamma/2-1\}}
\|f\|^{1/2}_{L^1_{\gamma+2|a|+2}} \|\la v \ra^{-a+\gamma/2} h\|_{L^2}
\Big\{ \|f\|_{L^1_{\gamma+2|a|+2}} \|\mathcal{M}g\|^2_{L^2_{a+\gamma/2}} 
\\
&\notag + 
\int_{\RR^6\times \Ss^2} b(\cdot) |f_* |\la v_*\ra^{2|a|+\gamma+2} \Big| (\la \cdot \ra^{a+\gamma/2} \mathcal{M}g)(v')  
- (\la \cdot \ra^{a+ \gamma/2} \mathcal{M}g)(v) \Big|^2d \sigma dvdv_* \Big\}^{1/2}\,.
\end{align*}
From this together with \eqref{estimate-2-1} and \eqref{estimate-2-2}, we obtain the  upper bound 
for $\displaystyle \sum_{2^k \ge R} \left|\int_{\RR^3} f(v_*) \mathcal{II}_k (v_*) dv_*\right |$.


\noindent
\underline{\it Bound of $\CalI\CalI\CalI_k$} Putting $F_{*,k}(v) = [ \Phi_{*,k}, \CalM] g$, we have  
\begin{align*}
{\mathcal {III}}_k &= \int_{\RR^{3}}\Big( \int_{\RR^3} \int_{\Ss^{2}} f_* b\left(\frac{\xi}{|\xi|}\cdot \sigma\right)
\Big( e^{iv_*\cdot\xi^+} \hat F_{*,k}(\xi^+)  - e^{iv_*\cdot\xi} \hat F_{*,k}(\xi) \Big)\overline{e^{iv_*\cdot \xi} \hat h_{*,k}(\xi)}
d\sigma d \xi\Big)  \frac{dv_*}{(2\pi)^3}\\
&\Denote \int_{\RR^{3}} f(v_*) J^k(v_*)\,
dv_*.
\end{align*}
The estimation for $J_k(v_*)$ is almost same  as for $\mathcal{II}_k^{(3)}(v_*)$. Indeed, we decompose it into
\begin{align*}
J_k(v_*)
&= \int_{\RR^3 \times \Ss^2} b(\cdot) F_{*,k}(v) \overline{
\big ( h_{*,k}(v')  - h_{*,k}(v) \big )}  d\sigma dv\\
&= \int_{\RR^3 \times \Ss^2} b(\cdot)
\big (F_{*,k}(v)-  F_{*,k}(v')\big)\overline{
\big ( h_{*,k}(v')  - h_{*,k}(v) \big )}  d\sigma dv
\\
&\quad+\int_{\RR^3 \times \Ss^2} b(\cdot) F_{*,k}(v')\overline{
\big ( h_{*,k}(v')  - h_{*,k}(w) \big )} d\sigma dv\\
&\quad + \int_{\RR^3 \times \Ss^2} b(\cdot) 
 F_{*,k}(v)\overline{
\big ( h_{*,k}(w)  - h_{*,k}(v) \big )} 
d\sigma dv\\
&\quad + \int_{\RR^3 \times \Ss^2} b(\cdot)\big( 
F_{*,k}(v')-  F_{*,k}(v)
\big)
 \overline{
\big (  h_{*,k}(w)  - h_{*,k}(v) \big )}  d\sigma dv\\
&\Denote \widetilde S^k(v_*) + \widetilde M_0^k(v_*) + \widetilde R^k_1(v_*) + \widetilde R^k_2(v_*). 
\end{align*}
By using a  similar procedure with the role of $g$ and $h$ exchanged, we can show
\begin{align*}
\sum_{2^k \ge R} \left|\int_{\RR^3} f(v_*) J_k(v_*) dv_* \right | &\lesssim
R^{\max\{s-1, \gamma/2-1\}} \|f\|^{1/2}_{L^1_{\gamma+2|a|+2} }\| \mathcal{M} g \|_{L^2_{a+ \gamma/2}}\Big\{
\|f\|_{L^1_{\gamma+2|a|+2}}\|h\|^2_{L^2_{-a+\gamma/2}} 
\\
& +\int_{\RR^6\times \Ss^2} b(\cdot)| f_*|
\la v_*\ra^{\gamma +2|a|+2}|(\la \cdot \ra^{-a+\gamma/2}h)(v')-
(\la \cdot \ra^{-a+\gamma/2} h)(v)|^2 d\sigma dv dv_* 
\Big \}^{1/2}\,.
\end{align*}
And  the proof of \eqref{desired-cut-part} is completed.
\end{proof}

\subsection{Proof of Theorem~\ref{thm:linear-regularization}}

Now we apply Propositions~\ref{singular-part}-\ref{reg-commute} and the spectral gap estimate  to obtain a closed form energy estimate for the linearized equation~\eqref{eq:linearize}.
The statement of the theorem is
\begin{thm} \label{thm:weighted-H-s-L-2}
Let $f$ be the solution to~\eqref{eq:linearize}. Let $k_0$ satisfy
\eqref{global-existence-number} so that the spectral gap theorem holds. Let $\Eps, \delta > 0$ (in the definition of $\CalM$) be small enough. Then there exists $T_0 > 0$  such that
\begin{align*}
   \int^{T_0}_0 
   \norm{\vint{v}^{k_0}  f(t, \cdot, \cdot)}_{L^2_{x,v}}^2 \dt
\leq
  \frac{C}{\Eps  \delta} \sum_{\ell \in \ZZ^3} \int_{\R^3_v} \abs{\frac{1}{\vint{\xi}^s + \vint{\ell}^{\frac{s}{2s+1}}} \hat{\vint{v}^{k_0} h^{in}_\ell}(\xi)}^2 \dxi
\end{align*}
and for any $t \geq T_0$,
\begin{align*}
    \norm{\vint{v}^{k_0}  f(t, \cdot, \cdot)}_{L^2_{x,v}}
\leq
    C\vpran{\frac{1}{\sqrt{T_0 \delta}} + \frac{C_R}{\sqrt{\Eps \delta}}} e^{-\lambda t} 
    \vpran{ \sum_{\ell \in \ZZ^3} \int_{\R^3_v} \abs{\frac{1}{\vint{\xi}^s + \vint{\ell}^{\frac{s}{2s+1}}} \hat{\vint{v}^{k_0} h^{in}_\ell}(\xi)}^2 \dxi}^{1/2} \,.
\end{align*}
Here $\lambda >0$ is the same decay rate as in the spectral gap estimate.  
\end{thm}
\begin{proof}
Let $T_0 < 1$ be small enough whose size will be specified later. First we consider the bound over any finite time interval $[0, T]$ with $T \leq T_0$. The equation for $ f_\ell$ is
\begin{align} \label{eq:v-k-f-l}
   \del_t f_\ell 
   - i \eta \cdot v \, f_\ell
&= Q(\mu, f_\ell)
   + Q(f_\ell, \mu) \,,
\end{align}
where  $f_\ell$ is the $\ell$-th Fourier mode in $x$.
We multiply~\eqref{eq:v-k-f-l} by $\CalM \vpran{\vint{v}^{2k_0} \CalM f_\ell}$, integrate over $\R^3_v$, and then take the real part of the  equation. 
Noting that 
\begin{equation}\label{commute}
   (\del_t - \eta \cdot \nabla_\xi ) \CalM(t,\xi, \eta) 
= \left( \frac{1}{2} +\varepsilon \right) \CalM(t,\xi, \eta)  \frac{ \delta \La \xi \Ra^{2s} }{1 + \delta \int_0^{T-t} \La \xi - \tau \eta \Ra^{2s} \dtau} \,,
\end{equation}
the left hand side gives
\begin{align} \label{eq:LHS-Mf}
    LHS 
 \geq  
    \frac{1}{2} \frac{\rm d}{\dt} \norm{\vint{v}^{k_0} \CalM f_\ell}^2_{L^2_{v} } 
-\Big(\frac{1}{2} + \varepsilon\Big) c_0 \delta 
   \norm{\vint{v}^{k_0} \CalM f_\ell}^2_{H^s_{v} } \,.
\end{align}
The contribution from each term on the right hand side is estimated as follows. 
First, we decompose  $Q$ as $Q_R + Q_{\bar R}$. Putting $h_\ell = \la v \ra^{k_0} \CalM f_\ell$, we consider
\begin{align*}
& \quad \,
\vpran{Q_R(\mu, f_\ell), \CalM \big(\vint{v}^{2k_0} \CalM f_\ell \big) }
-\vpran{Q_R(\mu, \CalM f_\ell),  \vint{v}^{2k_0} \CalM f_\ell } \nn 
\\
&= \vpran{Q_R(\mu, f_\ell), [\CalM, \vint{v}^{k_0}]  h_\ell }
+ \vpran{
\vint{v}^{k_0}Q_R(\mu, f_\ell)-
Q_R(\mu, \vint{v}^{k_0}f_\ell),  \CalM h_\ell }
\\
& +\vpran{\CalM
Q_R(\mu, \vint{v}^{k_0} f_\ell)-
Q_R\Big(\mu, \CalM \big(\vint{v}^{k_0}f_\ell\big)\Big ),   h_\ell } 
+ \vpran{
Q_R\Big(\mu, [\CalM, \vint{v}^{k_0} ]f_\ell \Big ),   h_\ell }\nn \\
&+\vpran{
Q_R(\mu, \vint{v}^{k_0} \CalM f_\ell)-\vint{v}^{k_0}
Q_R\Big(\mu,  \CalM f_\ell \Big ),   h_\ell } \nn 
\\
&\Denote R_1 + R_2 + R_3 + R_4 + R_5\,.\nn 
\end{align*}
By means of Proposition \ref{prop:commutator-R}, we have 
\begin{align*}
|R_2| + |R_5| \le C_{k_0, R} \|f_\ell\|_{L^2(dv)} \|\la v \ra^{k_0} \CalM f_\ell\|_{H^{s'}(dv)}.
\end{align*}
It follows from Proposition \ref{singular-part} that
\begin{align*}
|R_3|  \le C_R \|\la v \ra^{k_0} \CalM f_\ell\|^2_{H^{s'}(dv)}.
\end{align*}
By a similar argument used in   Corollary \ref{cor-5-7}, there exists $A(v,D_v) \in Op(S^{-1}_{0,0})$ such that  
\[ [\CalM, \vint{v}^{k_0} ] = A(v,D_v) \vint{v}^{k_0} \CalM, \enskip \|A(v,D_v) g\|_{H^{2s}(dv)} \lesssim \|g\|_{H^{2s-1}(dv)}\,.
\]
It follows from \cite[Proposition 6.11]{MS2016} that
\begin{align*}
|R_4| \le C_R \|A \la v \ra^{k_0} \CalM f_\ell\|_{H^{2s}(dv)} \|h_\ell\|_{L^2} \le C_{k_0, R} \|\la v \ra^{k_0} \CalM f_\ell\|_{H^{2s-1}(dv)}
\| \vint{v}^{k_0} \CalM f_\ell\|_{L^2(dv)}\,.
\end{align*}
By means of  Corollary \ref{cor-5-7}, we write $[\CalM, \vint{v}^{k_0} ] = \la v\ra^{k_0} A(v, D_v) \CalM$ and 
\begin{align*}
R_1 = \vpran{
\vint{v}^{k_0}Q_R(\mu, f_\ell)-
Q_R(\mu, \vint{v}^{k_0}f_\ell),  A \CalM h_\ell } + \vpran{Q_R(\mu, \vint{v}^{k_0}f_\ell),  A \CalM h_\ell }.
\end{align*} 
Apply Proposition \ref{prop:commutator-R} to the first term and  \cite[Proposition 6.11]{MS2016} to the second one. 
Then 
\[
|R_1| \le C_{k_0, R} \Big( \|f_\ell\|_{L^2(dv)} \|\la v \ra^{k_0} \CalM f_\ell\|^2_{L^2(dv)} +
\|\la v \ra^{k_0} f_\ell\|_{L^2(dv)}\|\la v \ra^{k_0} \CalM f_\ell\|^2_{H^{2s-1}(dv)} \Big).
\]
Summing up estimates for $R_j, j= 1,\cdots,5$, we see that there exists a $0 \le s'' < s$ such that 
\begin{align}\label{Q-cut-main}
\vpran{Q_R(\mu, f_\ell), \CalM \big(\vint{v}^{2k_0} \CalM f_\ell \big) }
&\le \vpran{Q_R(\mu, \CalM f_\ell),  \vint{v}^{2k_0} \CalM f_\ell } \nn \\
&\quad + C_{k_0, R} \vpran{
\|\la v \ra^{k_0} \CalM f_\ell\|^2_{H^{s''}(dv)} + \|\la v \ra^{k_0} f_\ell\|^2_{L^2(dv)}}\,,
\end{align}
where $s''=0$  if $0<s<1/2$. 

Apply Proposition \ref{reg-commute} with $f = \mu, g = f_\ell, h = \vint{v}^{2k_0} \CalM f_\ell, a = k_0$. Then there exists 
a $C_{k_0} >0$ independent of $R >0$ such that
\begin{align}\label{Q-non-cut-main}
&\vpran{Q_{\bar{R}}(\mu, f_\ell), \CalM \big(\vint{v}^{2k_0} \CalM f_\ell \big) }
\le \vpran{Q_{\bar{R}}(\mu, \CalM f_\ell),  \vint{v}^{2k_0} \CalM f_\ell } \nn \\
&\quad 
+ C_{k_0}\left[
\delta^{\tilde s/(2s)} \|\CalM f_\ell \|_{H^{\tilde s+\varepsilon'}_{k_0+\gamma/2}} \Big\{ \|\mathcal{M}f_\ell \|^2_{L^2_{k_0+ \gamma/2}} 
\right. \\
&\quad \qquad \notag + 
\int_{\RR^6\times \Ss^2} b(\cdot) \mu_* \la v_*\ra^{2k_0+ \gamma+2 }\Big| (\la \cdot \ra^{k_0+\gamma/2} \mathcal{M}f_\ell)(v')  
- (\la \cdot \ra^{k_0+ \gamma/2} \mathcal{M}f_\ell)(v) \Big|^2d \sigma dvdv_* \Big\}^{1/2}\\
&\quad \qquad \notag + R^{\max\{s-1, \gamma/2-1\}}  \Big\{
\| \mathcal{M} f_\ell \|^2_{L^2_{k_0+\gamma/2}}\\
&\notag 
\qquad \qquad \left.  +\int_{\RR^6\times \Ss^2} b(\cdot) \mu_* \la v_*\ra^{2k_0+\gamma+2}
 |\big (\la \cdot \ra^{k_0+\gamma/2}\mathcal{M}f_\ell g\big)(v')- 
\big( \la \cdot \ra^{k_0+\gamma/2} \mathcal{M}f_\ell \big) (v)|^2 d\sigma dv dv_* \Big \}\right]\,,
\end{align}
where $0 < \tilde s < s$ is arbitrary, and  $\varepsilon'$ is non-negative and can be chosen to be zero if $s \ne 1/2$.

By Proposition~\ref{prop:coercivity} with $F= \mu$ and its proof, we notice that 
\begin{align} \label{bound:damp-M-f}
    &\vpran{Q (\mu, \CalM f_\ell), \vint{ v }^{2k_0} \CalM f_\ell}
\leq 
   -c_0 \left\{ \norm{\CalM f_\ell}^2_{H^s_{k_0+\gamma/2}(\dv)} \right. \nn \\
&\qquad \left. +\int_{\RR^6\times \Ss^2} b(\cdot) \mu_* \la v_*\ra^{2k_0+\gamma}
 |\big (\la \cdot \ra^{k_0+\gamma/2}\mathcal{M}f_\ell \big)(v')- 
\big( \la \cdot \ra^{k_0+\gamma/2} \mathcal{M}f_\ell \big) (v)|^2 d\sigma dv dv_* \right\}\\
  & \qquad \qquad \qquad \qquad -\frac{\gamma_0}{2}\norm{\CalM f_\ell}^2_{L^2_{k_0+\gamma/2}(\dv)} 
 + C \norm{\CalM f_\ell}^2_{L^2_{k_0}(\dv)} \,, \nn
\end{align}
where the integral part on the right-hand side is obtained by using Lemma \ref{upper-F-g-g} and \ref{coer-J-1}.

Choosing a small $\delta>0$ and a large $R$, by means of
\eqref{Q-cut-main}, \eqref{Q-non-cut-main} and
\eqref{bound:damp-M-f}, we obtain for any $\epsilon' >0$
\begin{align}\label{main-coer-0}
&\vpran{Q(\mu, f_\ell), \CalM \big(\vint{v}^{2k_0} \CalM f_\ell \big) }
\nn \\
&\le
-\frac{c_0}{2} \norm{\CalM f_\ell}^2_{H^s_{k_0+\gamma/2}(\dv)}
-\frac{\gamma_0-\epsilon'}{2}\norm{\CalM f_\ell}^2_{L^2_{k_0+\gamma/2}(\dv)} 
 + C \norm{ f_\ell}^2_{L^2_{k_0}(\dv)},
\end{align}
where it follows from \eqref{def:gamma-0} that
\begin{align}\label{gamma-0}
\gamma_0 > \left(\int_{\Ss^2} b(\cos \theta)\sin^2 \dfrac{\theta}{2}d\sigma \right)\frac{\int_{\RR^3} |v-v_*|^\gamma \mu_* d v_*}{
\la v \ra^\gamma}\,.
\end{align}
Notice that
\begin{align}\label{lower-value}
 \gamma_1 \Denote \min_v\frac{ \int |v-v_*|^\gamma \mu(v_*) dv_*} {\langle v \rangle^\gamma} >  2^{-\gamma-7} \frac{28}{3\sqrt{2\pi}} \, ,
\end{align}
because if $|v| \ge 2$, then 
\begin{align*}
&\frac{ \int |v-v_*|^\gamma \mu(v_*) dv_*} {\langle v \rangle^\gamma}  \ge 
2^{-\gamma} \int_{|v_*| \le 1/2} \mu(v_*) dv_*
\ge 2^{-\gamma}(2\pi)^{-3/2}e^{-2^{-3}} \frac{4\pi}{3} 2^{-3} > 2^{-\gamma-2}\left(1-  \frac{1}{8} \right) \frac{1}{3\sqrt{2\pi}}
\,,
\end{align*}
where we have used $|v-v_*| \ge \dfrac{|v|}{2} + 1 - |v_*|$. Moreover
if $|v| <2$ then for any $0< \epsilon'' <1$ 
\begin{align*}
\frac{ \int |v-v_*|^\gamma \mu(v_*) dv_*} {\langle v \rangle^\gamma}  
&\ge 5^{-\gamma/2}
\vpran{\epsilon''}^\gamma \int_{|v-v_*|\ge \epsilon''} \mu(v_*) dv_*
\ge 5^{-\gamma/2} \vpran{\epsilon''}^\gamma \vpran{
1-  \int_{|v_*|\le \epsilon''} \mu(v_*) dv_*} 
\\
&\ge  5^{-\gamma/2} \vpran{\epsilon''}^\gamma
\vpran{1- (2\pi)^{-3/2} \frac{4\pi}{3} \vpran{\epsilon''}^3}
\ge 2^{-4}
\end{align*}
where the last inequality follows by choosing $\epsilon'' = \sqrt{5}/8$.

It follows from Corollary \ref{cor-5-7} that $\CalM \vint{v}^{k_0} = \vint{v}^{k_0} \CalM + [ \CalM, \vint{v}^{k_0} ]
= \vint{v}^{k_0}(Id + A)\CalM$ and 
\begin{align}\label{sub-QR}
\left|\vpran{Q_R ( f_\ell, \mu ), \CalM \big( \vint{ v }^{2k_0} \CalM f_\ell\big) }\right|
& =\left|\vpran{Q_R ( f_\ell, \mu ), \vint{v}^{k_0} (Id + A) \CalM \big(\vint{v}^{k_0} \CalM f_\ell\big) }\right| \nn \\
&\le C_{k_0, R} \|f_\ell \|_{L^2(dv)} \|\vint{v}^{k_0} \CalM f_\ell\|_{L^2(dv)}\,,
\end{align}
where we have used Proposition \ref{prop:bound-Q-R} in the last inequality.

{We now consider 
\begin{align}\label{missing-term}
&\Big(Q_{\bar R}(f_\ell, \mu), \mathcal{M} \langle v\rangle^{2k_0} \mathcal{M} f_\ell \Big) -
\Big(Q_{\bar R}(f_\ell, \langle v\rangle^{k_0}\mu), \langle v\rangle^{-k_0}\mathcal{M} \langle v\rangle^{2k_0} \mathcal{M} f_\ell \Big) \nn \\
&= \IntRRS b \Phi_{\overline{ R}} f_{\ell,*} \mu \Big(\langle v' \rangle^{k_0} - \langle v \rangle^{k_0} \Big)    \overline{\Big(
\langle v \rangle^{-k_0}  \mathcal{M} \langle v\rangle^{2k_0} \mathcal{M} f_\ell \Big)' }dv dv_* d\sigma.
\end{align}
It follows from Proposition 
\ref{prop:trilinear} (\cite[Theorem 2.1]{AMUXY2010})
that
\[
\left |\Big(Q_{\bar R}(f_\ell, \langle v\rangle^{k_0}\mu), \langle v\rangle^{-k_0}\mathcal{M} \langle v\rangle^{2k_0} \mathcal{M} f_\ell \Big)\right|
\le C_{k_0, R}  \|f_\ell \|_{L^1_{\gamma+2s} }\| \langle v\rangle^{k_0}\mu\|_{H^{2s}_{\gamma+2s} }\|\langle v\rangle^{k_0} \mathcal{M} f_\ell \|_{L^2}.
\]

Note that
\begin{align*}
\langle v \rangle^{-k_0}  \mathcal{M} \langle v\rangle^{k_0} = \mathcal{M} -[\mathcal{M} , \langle v\rangle^{-k_0} ]
 \langle v \rangle^{k_0} = \CalM + \big( \langle v \rangle^{k_0} [\mathcal{M} , \langle v\rangle^{-k_0} ]\big)^* \,.
\end{align*}
Apply Corollary \ref{cor-5-7} to the second term. Then we have 
$
\langle v \rangle^{-k_0}  \mathcal{M} \langle v\rangle^{k_0} = \CalM (Id + A^*)$ with $A \in Op(S^0_{0,0})$. 
Therefore, if we  put  $g_\ell = ( Id + A^*)\langle v\rangle^{k_0} \mathcal{M} f_\ell$, then the right hand side of 
\eqref{missing-term} is written as 
\begin{align*}
&\IntRRS b\Big(\frac{v-v_*}{|v-v_*|}\cdot \sigma\Big)  \Phi_{\overline{ R}}(|v-v_*|) f_{\ell,*} \, \mu \Big(\langle v' \rangle^{k_0} - \langle v \rangle^{k_0} \Big)\overline {\Big(
 \mathcal{M} g_\ell \Big)' }dv dv_* d\sigma.
\end{align*}
We apply Lemma \ref{lem:diff-v-k} to the decomposition of the factor $\langle v' \rangle^{k_0} - \langle v \rangle^{k_0} $ with $2k =k_0$. 
To estimate   the part coming from the first term of the right hand side of \eqref{precise}, we rewrite
 $\omega$ as
\begin{align*}
   \omega 
= \tilde \omega \cos\tfrac{\theta}{2}
   + \frac{v' - v_\ast}{|v' - v_\ast|} \sin\tfrac{\theta}{2} \,,
\end{align*}
with $\tilde \omega = (v' - v)/|v' - v|$ satisfying $\tilde \omega \perp (v' - v_\ast)$. Then, by the almost same procedure for $\Gamma_1$
in the proof of Proposition \ref{prop:trilinear-weight}, one can 
show that  it is bounded by a constant times 
\[
\|f_\ell\|_{L^1_{3+\gamma+2s}(dv) \cap L^2(dv)}
\|\CalM g_\ell\|_{L^2(dv)} \le C_{k_0} 
\|\la v \ra^{k_0} f_\ell\|_{L^2(dv)} \||\la v\ra^{k_0}\mathcal{M} f_\ell\|_{L^2(dv)}\,.
\]
The part coming from $\mathfrak{R}_1$ is estimated by
\begin{align*}
&\int_{\RR^3} \la v \ra^{\gamma+1} \mu(v) \vpran{ \IntRS b\Big(\frac{v-v_*}{|v-v_*|}\cdot \sigma\Big)\la v_* \ra^{k_0 +\gamma -1}|f_{\ell, *}|\sin^{k_0-3} \dfrac{\theta}{2}
\abs{\overline {\Big(
 \mathcal{M} g_\ell \Big)' }}dv_* d\sigma }dv \\
&\le C_{k_0} \|\la v\ra^{k_0} f_\ell\|_{L^2(dv)} \|\la v\ra^{k_0}\mathcal{M} f_\ell\|_{L^2(dv)} \,,
\end{align*}
where we have used the singular change of variable $v' \rightarrow v_*$ and the $L^2$ boundedness of 
$\CalM (Id + A^*)$.
The parts coming from $\mathfrak{R}_2, \mathfrak{R}_3$ are easily estimated by $\| f_\ell\|_{L^1_{\gamma+4}(dv)} \|\la v\ra^{k_0}\mathcal{M} f_\ell\|_{L^2(dv)}$.
It remains to estimate the part coming from the second term of \eqref{precise}, that is, 
\begin{align}\label{main-missing}
&\IntRRS b\Big(\frac{v-v_*}{|v-v_*|}\cdot \sigma\Big)  \Phi_{\overline{ R}}(|v-v_*|)f_{\ell,*} \, \mu(v)
\la v_* \ra^{k_0} \sin^{k_0} \dfrac{\theta}{2}
 \overline {\Big(
 \mathcal{M} g_\ell \Big)' }dv dv_* d\sigma \nn \\
&=\int_{\RR^3}\mu(v) \left(\iint_{\RR^3  \times \Ss^2} 
 b\Big(\frac{v-v_*}{|v-v_*|}\cdot \sigma\Big) 
 \sin^{k_0} \dfrac{\theta}{2}
F(v_*; v)
 \overline {\left \la \frac{v'-v} {\sin{\theta}/{2}} \right \ra^{\gamma/2}\Big(
 \mathcal{M} g_\ell \Big)' }dv_* d\sigma \right) dv \nn \\
& = \int_{\RR^3}\mu(v) \left(\iint_{\RR^3  \times \Ss^2} 
 b\Big(\frac{v-v_*}{|v-v_*|}\cdot \sigma\Big) 
 \sin^{k_0} \dfrac{\theta}{2}
F(v_*; v)
 \overline {
\mathcal{M}(D_{v'})\Big( \left \la \frac{v'-v} {\sin{\theta}/{2}} \right \ra^{\gamma/2}g_\ell(v') \Big)}dv_* d\sigma \right) dv\\
& \quad - \int_{\RR^3}\mu(v) \left(\iint_{\RR^3  \times \Ss^2} 
 b\Big(\frac{v-v_*}{|v-v_*|}\cdot \sigma\Big) 
 \sin^{k_0} \dfrac{\theta}{2}
F(v_*; v)
 \overline {\Big
[\mathcal{M}(D_{v'}),\left \la \frac{v'-v} {\sin{\theta}/{2}} \right \ra^{\gamma/2}\Big] g_\ell(v') }dv_* d\sigma \right) dv \nn 
 \\
& \Denote
\Gamma_{\ell,1}  + \Gamma_{\ell , 2}\,, \nn
\end{align}
where  $\theta$ is determined from  $\cos \theta =
\frac{v-v_*}{|v-v_*|}\cdot \sigma$ and 
\begin{align*}
&F(v_*; v) =\frac{\Phi_{\overline{ R}}(|v-v_*|)}{\la v -v_*\ra^{\gamma/2}}f_{\ell,*} \, 
\la v_* \ra^{k_0}.
\end{align*}
Since for $\alpha \ne 0$ we have 
\begin{align*}
\left |D_z^\alpha \left \la \frac{z} {\sin{\theta}/{2}} \right\ra^{\gamma/2}
\right| 
\le C_\alpha (\sin{\theta}/{2})^{-|\alpha|}\left \la \frac{z} {\sin{\theta}/{2}} \right \ra^{\gamma/2-|\alpha|}
\le C_\alpha (\sin{\theta}/{2})^{-|\alpha|}\la z \ra^{\gamma/2-|\alpha|}\,,
\end{align*}
the singular change of variable $v_* \rightarrow v'$  and the Calder\`on-Vaillancourt theorem yield 
\begin{align*}
|\Gamma_{\ell , 2}|&
\le 
\int_{\RR^3} \mu(v) \la v \ra^\gamma 
\left( \int_{\Ss^2}
b(\cos \theta) \sin^{k_0-3} \dfrac{\theta}{2}\left ( 
\int_{\RR^3}|f_{\ell, *}\la v_* \ra^{k_0} |^2 dv_*\right)^{1/2}
\right. \\
&\times \left. 
\left(\int_{\RR^3}\Big
|\frac{\Phi_{\overline{ R}}(|v'-v|/\sin \theta/2)}{\la v \ra^\gamma \la (v' -v)/\sin \theta/2 \ra^{\gamma/2}}
\Big
[\mathcal{M}(D_{v'}),\left \la \frac{v'-v} {\sin{\theta}/{2}} \right \ra^{\gamma/2}\Big] g_\ell(v') \Big|^2 dv' \right)^{1/2}
d\sigma  \right)   dv \\
& \le C_{R,k_0} \int_{\RR^3} \mu(v) \la v \ra^\gamma dv
\int_{\Ss^2} b(\cos \theta) \sin^{k_0-3-5}\frac{\theta}{2}
d\sigma \|f_\ell \|_{L^2(dv)} \|\vint{v}^{k_0} \CalM f_\ell\|_{L^2(dv)}.
\end{align*}
Put 
\begin{align*}
G(w; v, \theta) =
\left \la \frac{w-v} {\sin{\theta}/{2}} \right \ra^{\gamma/2}g_\ell(w
).\end{align*}
Then the exchange of $\xi \cdot \sigma/|\xi|$ and $(v-v_*)\cdot \sigma/|v-v_*|$ as similar as in the derivation of  the Bobylev formula 
leads us to
\begin{align*}
\Gamma_{\ell,1} &=\int_{\RR^3} \mu(v) \left(
{ \int_{\Ss^2} \int_{\R^3} 
 b\Big(\frac{v-v_*}{|v-v_*|}\cdot \sigma\Big)  F(v_*; v) \overline{\int_{\R^3} e^{i v' \cdot \xi} 
\mathcal{M}(\xi, \ell)\widehat {G(\xi, v, \theta)}\frac{d\xi}{(2\pi)^3}}\sin^{k_0}\frac{\theta}{2} 
dv_*  d \sigma} \right)dv
\\
&= \int_{\RR^3} \mu(v)
\left(\int_{\Ss^2} \int_{\R^3} b\Big(\frac{\xi}{|\xi|}\cdot \sigma\Big) \widehat{F(\xi^-;v)}
\mathcal{M}(\xi, \ell) \overline {e^{i v\cdot \xi^+} \widehat {G(\xi, v, \theta)}}
\sin^{k_0}\frac{\theta}{2} \frac{d\xi d \sigma}{(2\pi)^3}\right)dv\\
&\Denote \int_{\RR^3} \mu(v)\la v \ra^\gamma  K(v) dv \, ,
\end{align*}
where $\theta$ on the second  formula is determined by $\cos \theta = \xi \cdot \sigma/|\xi|$. 
It follows from Cauchy-Schwarz inequality and the singular change of variable $\xi \rightarrow \xi^-$ that
\begin{align*}
|K(v)|^2 &\le \int_{\Ss^2} \int_{\R^3} b\Big(\frac{\xi}{|\xi|}\cdot \sigma\Big)| \mathcal{M}(\xi^-, \ell)\la v \ra^{-\gamma/2}\widehat{F(\xi^-;v)}|^2 
\sin^{k_0-\gamma/2 + 3/2} \frac{\theta}{2} \frac{d\xi d \sigma}{(2\pi)^3}\\
& \qquad \times \int_{\Ss^2} \int_{\R^3} b\Big(\frac{\xi}{|\xi|}\cdot \sigma\Big)
\left|
\frac{\mathcal{M}(\xi, \ell)}{\mathcal{M}(\xi^-, \ell)}\la v \ra^{-\gamma/2}\widehat {G(\xi, v, \theta)} \right|^2 \sin^{k_0 +\gamma/2 -3/2} \frac{\theta}{2} \frac{d\xi d \sigma}{(2\pi)^3}\\
&= (2\pi)^2 \int_0^{\pi/2} b(\cos \theta) \sin^{k_0-\gamma /2-
3/2}
 \frac{\theta}{2} \sin\theta d\theta
\|\mathcal{M} \la v \ra^{-\gamma/2}F_k(\cdot,;v)\|^2_{L^2} \\
& \qquad \qquad \times 
\int_0^{\pi/2} b(\cos \theta) \sin^{k_0+\gamma /2-
3/2}
 \frac{\theta}{2} \sin\theta 
 \|\la v \ra^{-\gamma/2} G(\cdot , v, \theta)\|^2_{L^2} d \theta
\\
&\le\left( \int_{\Ss^2} b(\cos \theta) \sin^{k_0-\gamma /2-
3/2}
 \frac{\theta}{2} d\sigma
\left(      
\|\mathcal{M}f_\ell \|^2_{L^2_{k_0+ \gamma/2}} + C_{R, k_0}
\|f_\ell \|^2_{L^2_{k_0}} 
\right) \right)^2\,,
\end{align*}
because 
$\la v \ra^{-\gamma/2}
\left \la \frac{w-v} {\sin{\theta}/{2}} \right \ra^{\gamma/2}
\le \frac{1}{\sin^{\gamma/2} {\theta}/{2}}\la w\ra^{\gamma/2}
$.
Noting $|\sin{\theta}/{2}|\le 2^{-1/2}$ for $\theta \in [0,\pi/2]$, 
we have 
\[
|\Gamma_{\ell, 1}| \le
2^{-(2k_0-\gamma-5)/4}\int_{\RR^3} \mu(v)\la v \ra^\gamma dv 
\int_{\Ss^2} b(\cos \theta) \sin^2 \dfrac{\theta}{2 }d \sigma 
\left(      
\|\mathcal{M}f_\ell \|^2_{L^2_{k_0+ \gamma/2}} + C_{R, k_0}
\|f_\ell \|^2_{L^2_{k_0}} 
\right)\,.
\]
As a consequence, if we define
\begin{align}\label{gamma-2}
\gamma_3 \Denote 2^{-(2k_0-\gamma-5)/4}\int_{\RR^3} \mu(v)\la v \ra^\gamma dv 
\int_{\Ss^2} b(\cos \theta) \sin^2 \dfrac{\theta}{2 }d \sigma,
\end{align}
then we have 
\begin{align}\label{missing-conclusion}
\left|\Big(Q_{\bar R}(f_\ell, \mu), \mathcal{M} \langle v\rangle^{2k_0} \mathcal{M} f_\ell \Big) \right|
\le 
\gamma_3
\|\mathcal{M} f_\ell\|^2_{L^2_{k_0 + \gamma/2}(dv)} + C_{R, k_0} \| f_\ell\|^2_{L^2_{k_0}(dv)} \,,
\end{align}
}
where we notice that

\begin{align}\label{upper-number}
\int \la v \ra^\gamma  \mu(v)dv < \int (1+ \sum_{j=1}^3|v_j|) \mu(v)dv 
= 1+ \frac{6}{\sqrt{2\pi}} < \frac{26}{3\sqrt{2\pi}}\,.
\end{align}
Combine \eqref{eq:LHS-Mf} with 
\eqref{main-coer-0}, \eqref{sub-QR} and 
\eqref{missing-conclusion}.  If $\gamma_0/2> \gamma_3$,
which is verified for $k_0$ satisfying \eqref{global-existence-number} in view of \eqref{lower-value} and \eqref{upper-number},
then we have 
\begin{align*}
       \frac{\rm d}{\dt} \norm{ \vpran{\vint{v}^{k_0} \CalM f_\ell}}^2_{L^2_v}
\leq
   - \frac{c_0}{2} \norm{ (\vint{v}^{k_0}  \CalM f_\ell)}^2_{H^{s}_v}
   + C_{R} \norm{f_\ell}_{L^2(\vint{v}^{k_0} \dv)}^2
\leq
   C_{R }\norm{\vint{v}^{k_0}  f_\ell}_{L^2_v}^2 \,.
\end{align*}
Integrating $t$ on $[0, T]$, by Corollary \ref{cor-5-7} we get
\begin{align}\label{int-smoothing}
    \norm{\vint{v}^{k_0} f_\ell(T, \cdot)}^2_{L^2_{v}}
&\leq 
   \norm{ \vpran{\vint{v}^{k_0} \CalM(0,D_v, \ell) f_{\ell}^{in}}}^2_{L^2_v}
   + C_{R} \int^T_0 \norm{\vint{v}^{k_0} f_\ell(s, \cdot)}^2_{L^2_{v}} \ds \nn
\\
&\leq 
   \norm{ (Id + A(0, v,D_v, \ell))\CalM(0,D_v, \ell) \vpran{\vint{v}^{k_0} f_{\ell}^{in}}}^2_{L^2_v}
   + C_{R} \int^T_0 \norm{\vint{v}^{k_0} f_\ell(s, \cdot)}^2_{L^2_{v}} \ds 
\\
&\leq C_{k_0}
   \int_{\RR^3_\xi}  \frac{ \La \xi \Ra^{2s}}{(1+ c_0 \delta T \La \xi \Ra^{2s})^{1+2\varepsilon}}
\left | \La \xi \Ra^{-s} \hat {\la v \ra^{k_0} f_\ell}( 0,\xi) \right |^2\dxi 
  + C_R \int^T_0 \norm{\vint{v}^{k_0} f_\ell(s, \cdot)}^2_{L^2_{v}} \ds \,. \nn
\end{align}
Integrating $T$ on $[0, T_0]$, we obtain that
\begin{align*}
    \int_0^{T_0} \norm{\vint{v}^{k_0} f_\ell(T, \cdot)}^2_{L^2_{v}} {\rm d}T
\leq
 \frac{C_{k_0}}{2\Eps c_0 \delta}\norm{\vint{v}^{k_0} f^{in}_\ell}_{H^{-s}_{v}}^2
  + C_R T_0 \int^{T_0}_0 \norm{\vint{v}^{k_0} f_\ell(s, \cdot)}^2_{L^2_{v}} \ds \,.
\end{align*}
Therefore, by choosing $T_0$ small such that $C_R T_0 < 1/2$, we have
\begin{align*}
    \int_0^{T_0} \norm{\vint{v}^{k_0} f_\ell(t, \cdot)}^2_{L^2_{v}} {\rm d}t
\leq
  \frac{C_{k_0}}{\Eps c_0 \delta} \norm{\vint{v}^{k_0} f^{in}_\ell}_{H^{-s}_{v}}^2 \,,
\end{align*}
Adding all the modes $\ell$ then gives
\begin{align*}
    \int_0^{T_0} \norm{\vint{v}^{k_0} f_\ell(t, \cdot)}^2_{L^2_{x,v}} \dt
\leq
  \frac{C_{k_0}}{\Eps c_0 \delta} \norm{\vint{v}^{k_0} f^{in}}_{L^2(\dx;H^{-s}_{v})}^2 \,,
\end{align*}
where $T_0$ is small enough. In particular, this implies at $t=T_0$, it holds that
\begin{align} \label{bound:H-s-L-2-pointwise}
    \norm{\vint{v}^{k_0} f_\ell(T_0, \cdot)}^2_{L^2_{x,v}}
\leq
  \vpran{\frac{C_{k_0}}{T_0 \delta} + \frac{C_R C_{k_0}}{\Eps c_0 \delta}}
  \norm{\vint{v}^{k_0} f^{in}}^2_{L^2(\dx;H^{-s}_{v})} \,.
\end{align}
Similar to \eqref{int-smoothing}, we have 
\begin{align*}
\sum_{\ell \in \ZZ^3} \norm{\vint{v}^{k_0} f_\ell(T, \cdot)}^2_{L^2_{v}}
&\leq C_{k_0}
   \sum_{\ell \in \ZZ^3} \norm{ \CalM(0,D_v, \ell) \vpran{\vint{v}^{k_0} f_{\ell}^{in}}}^2_{L^2_v}
   + C_{R} \int^T_0  \sum_{\ell \in \ZZ^3} \norm{\vint{v}^{k_0} f_\ell(s, \cdot)}^2_{L^2_{v}} \ds 
\\
&\leq C_{k_0}\sum_{\ell \in \ZZ^3}
   \int_{\RR^3_\xi}  \frac{ \La \eta \Ra^{2s/(2s+1)}}{(1+ c_0 \delta T^{2s+1} \La \eta \Ra^{2s})^{1+2\varepsilon}}
\left | \La \eta \Ra^{-s/(2s+1)} \hat {\la v \ra^{k_0} f_\ell}( 0,\xi) \right |^2\dxi \\
& \qquad \qquad \qquad \qquad \qquad 
  + C_R \int^T_0 \sum_{\ell \in \ZZ^3} \norm{\vint{v}^{k_0} f_\ell(s, \cdot)}^2_{L^2_{v}} \ds \,. \nn
\end{align*}
Integrating  $T$ on $[0, T_0]$, we obtain that if  then $C_R T_0 < 1/2$
\begin{align*}
    \int_0^{T_0} \sum_{\ell \in \ZZ^3} \norm{\vint{v}^{k_0} f_\ell(T, \cdot)}^2_{L^2_{v}} {\rm d}T
\leq
 \frac{C_{k_0}(2sc_0\delta+1)} {4s c_0 \delta} \sum_{\ell \in \ZZ^3} \int_{\RR^3_\xi}
\left | \La \eta \Ra^{-s/(2s+1)} \hat {\la v \ra^{k_0} f_\ell}( 0,\xi) \right |^2\dxi 
\end{align*}
because $\int_0^\infty\frac{d\tau}{(1+c_0\delta \tau^{2s+1})^{1+2 \varepsilon}} < 1 +\int_1^\infty \frac{d \tau}{c_0\delta\tau^{2s+1}}$.
Noting 
\[
\vint{\xi}^{-2s} {\bf 1}_{|\xi| \ge |\eta|^{1/(2s+1)} }+ \vint{\eta}^{-2s/(2s+1)} {\bf 1}_{|\xi| <|\eta|^{1/(2s+1)}}
\le \frac{2}{\vint{\xi}^{2s} + \vint{\eta}^{2s/(2s+1)}}
\]
and splitting $\sum_{\ell \in \ZZ^3} \int_{\RR^3_\xi}$ into two regions 
$\{\vint{\xi} \ge \vint{\eta}^{1/(2s+1)} \}$ and $\{\vint{\xi} <  \vint{\eta}^{1/(2s+1)} \}$, we obtain the desired estimates in Theorem \ref{thm:weighted-H-s-L-2}.

Starting from $t = T_0$, we can apply the spectral gap estimate and obtain that
\begin{align*}
   \norm{\vint{v}^{k_0} f_\ell(t, \cdot)}_{L^2_{v}}
\leq
  C \vpran{\frac{1}{\sqrt{\delta  T_0} }+ \frac{1}{\sqrt{\Eps c_0 \delta}}} e^{-\lambda t} 
  \vpran{\sum_{\ell \in \ZZ^3} \int_{\R^3_v} \abs{\frac{1}{\vint{\xi}^s + \vint{\ell}^{\frac{s}{2s+1}}} \hat{h^{in}_\ell}(\xi)}^2 \dxi}^{1/2} ,
\qquad
  t \geq T_0 \,,
\end{align*}
where $T_0$ depends on $R$.
\end{proof}


\section{Energy Estimates} \label{Sec:6}
In this section we close the a-priori estimate based on the estimates  in Sections~\ref{sec:propositions}-\ref{sec:regularization}. The main result is
\begin{thm} \label{thm:a-priori}
Let $f$ be the solution to~\eqref{eq:perturbation} 
with initial data $f_0 \in Y_l$ and $l_0 > \frac{5\gamma + 37}{2}$. Then 
\begin{align} \label{est:full-energy}
& \quad \,
   \frac{1}{2} \frac{\rm d}{\dt} 
   \vpran{\norm{f}_{Y_l}^2
   + 4 C_l \int_0^\infty
      \norm{\Semi(\tau)f(t, \cdot, \cdot)}^2_{L^2(W^{l_0}\dv;H^2(\dx))} \dtau}  \nn
\\
& \leq 
  - \vpran{\frac{\gamma_0}{4} - C_l \norm{f}_{Y_l}}\sum_{\alpha} \int_{\T^3} \norm{W^{l - |\alpha|} \del^\alpha_x f}_{L^2_{\gamma/2}}^2 \dx  
\\
& \quad \,
  - \vpran{\frac{c_0}{4} \delta_2 - C_l \norm{f}_{Y_l}} 
    \sum_{\alpha} \int_{\T^3} \norm{W^{l - |\alpha|} \del^\alpha_x f}^2_{H^s_{\gamma/2}} \dx  
  - C_{l} \norm{f}^2_{L^2(\dv; H^2(\dx))} \,. \nn
\end{align}
\end{thm}
\begin{proof}
The proof is divided into two steps.  

\smallskip

\Ni{\bf Step 1. Evolution of $\norm{f}_{Y_l}$.} For any $|\alpha| = 0, 2$, the equation for $\del^\alpha_x f$ is
\begin{align} \label{eq:deriv-f}
    \vpran{\del_t + v \cdot \nabla_x} \del^\alpha_x f
&= Q \vpran{F, \, \del^\alpha_x f}
   + Q \vpran{\del^\alpha_x f, \, \mu}
    + \vpran{\del^\alpha_x Q(F, f) - Q \vpran{F, \, \del^\alpha_x f}} \nn
\\
& \Denote
    \Gamma_1 + \Gamma_2 + \Gamma_3 \,.
\end{align}
\vspace{-0.05cm}
Multiply $W^{2(l - |\alpha|)} \del^\alpha_x f$ to~\eqref{eq:deriv-f}, integrate in $x, v$, and sum over $|\alpha| = 0, 2$. 
The left hand side gives
\begin{align} \label{term:LHS}
   \frac{1}{2} \frac{\rm d}{\dt}  \sum_{\alpha}
   \int_{\T^3} \int_{\R^3} W^{2\vpran{l - |\alpha|}} 
    \abs{\del^\alpha_x f }^2 \dv\dx
= \frac{1}{2} \frac{\rm d}{\dt}\norm{f}_{Y_l}^2 \,.
\end{align}
By Proposition~\ref{prop:coercivity}, the first term $\Gamma_1$ on the right hand side satisfies
\begin{align} \label{term:T-1}
& \quad \,
   \int_{\T^3} \int_{\R^3} 
     Q \vpran{F, \, \del^\alpha_x f}
     \vpran{\del^\alpha_x f} W^{2\vpran{l - |\alpha|}} \dv\dx \nn
\\
& \leq 
  - \frac{\gamma_0}{2} 
  \int_{\T^3} \norm{W^{l - |\alpha|} \del^\alpha_x f}_{L^2_{\gamma/2}}^2 \dx
  - \vpran{\frac{c_0}{4} \delta_2 - C_l \norm{f}_{Y_l}} 
   \int_{\T^3} \norm{W^{l - |\alpha|} \del^\alpha_x f}^2_{H^s_{\gamma/2}} \dx \nn
\\
& \quad \,
  + C_l \int_{\T^3} \norm{W^{l - |\alpha|} \del^\alpha_x f}^2_{L^2} \dx
  + C_l \int_{\T^3} \norm{W^{l - |\alpha|} f}_{L^2}
   \norm{W^{l - |\alpha|} \del^\alpha_x f}_{L^2}
   \norm{\del^\alpha_x f}_{L^1_{1+\gamma}} \dx
\\
& \quad \, 
  + C_0 
    \int_{\T^3} 
    \norm{\del^\alpha_x f}_{L^1_\gamma} \norm{W^{l - |\alpha|} f}_{L^2_{\gamma/2}}
     \norm{W^{l - |\alpha|} \del^\alpha_x f}_{L^2_{\gamma/2}} \dx  \nn
 \,.
\end{align}
It follows from the proof of Proposition~\ref{prop:coercivity} that
the first term of \eqref{term:T-1} can be replaced by
\begin{align}\label{accurate-lower}
-\frac{1}{2}\vpran{\int_{\Ss^2} b(\cos \theta) \sin^2 \tfrac{\theta}{2}d \sigma}
\int_{\mathbb{T}^3_x}\int_{\RR^3_v}
\left( \int_{\RR^3} \mu(v_*)|v-v_*|^\gamma dv_*\right)
|W^{l - |\alpha|} \del^\alpha_x f(x,v)|^2 dxdv\,.
\end{align}
We treat the two cases $|\alpha| = 0$ and $|\alpha| = 2$ separately. If $|\alpha| = 0$, then the last two terms in~\eqref{term:T-1} satisfy
\begin{align*}
   \int_{\T^3} \norm{W^{l} f}_{L^2}^2
   \norm{f}_{L^1_{1+\gamma}} \dx
+ \int_{\T^3} 
    \norm{f}_{L^1_\gamma} \norm{W^{l} f}_{L^2_{\gamma/2}}^2 \dx
&\leq
  \vpran{\sup_{\T^3} \norm{f}_{L^2_{1+\gamma}}}
  \norm{W^l f}^2_{L^2_{\gamma/2}(\!\dx\dv)}
\\
&\leq
   C \norm{f}_{Y_l} \norm{W^l f}^2_{L^2_{\gamma/2}(\!\dx\dv)}  \,.
\end{align*}
If $|\alpha| = 2$, then the second last term in~\eqref{term:T-1} satisfies
\begin{align*}
  \int_{\T^3} \norm{W^{l-2} f}_{L^2_v}
   \norm{W^{l-2} \del^\alpha_x f}_{L^2_v}
   \norm{\del^\alpha_x f}_{L^1_{1+\gamma}(\!\dv)} \dx
&\leq
  \vpran{\sup_{\T^3} \norm{W^{l-2} f}_{L^2_v}}
  \norm{W^{l-2} \del^\alpha_x f}_{L^2_{x,v}}^2
\\
& \leq 
   \norm{f}_{Y_l} \sum_{\alpha}\norm{W^{l-2} \del^\alpha_x f} ^2_{L^2_{x,v}} \,.
\end{align*}
Similarly, the last term in~\eqref{term:T-1} satisfies
\begin{align*}
   \int_{\T^3} 
    \norm{\del^\alpha_x f}_{L^1_\gamma} \norm{W^{l - 2} f}_{L^2_{\gamma/2}}
     \norm{W^{l - 2} \del^\alpha_x f}_{L^2_{\gamma/2}} \dx 
&\leq
  \vpran{\sup_{\T^3} \norm{W^{l-2} f}_{L^2_{\gamma/2}}}
  \norm{W^{l-2} \del^\alpha_x f}_{L^2_{\gamma/2}(\!\dx\dv)}^2
\\
&\leq
  \norm{f}_{Y_l} \sum_{\alpha}\norm{W^{l-2} \del^\alpha_x f} ^2_{L^2_{\gamma/2}(\!\dx\dv)} \,.
\end{align*}
Hence for both cases we have
\begin{align} \label{term:T-1-1}
& \quad \,
   \int_{\T^3} \int_{\R^3} 
     Q \vpran{F, \, \del^\alpha_x f}
     \vpran{\del^\alpha_x f} W^{2\vpran{l - |\alpha|}} \dv\dx \nn
\\
&\leq
  - \frac{3\gamma_0}{8} \int_{\T^3} \norm{W^{l - |\alpha|} \del^\alpha_x f}_{L^2_{\gamma/2}}^2 \dx
  - \vpran{\frac{c_0}{4} \delta_2 - C_l \norm{f}_{Y_l}} \int_{\T^3} \norm{W^{l - |\alpha|} \del^\alpha_x f}^2_{H^s_{\gamma/2}} \dx   
  + C_{l}  \norm{f}^2_{Y_l} \,.
\end{align}
By Proposition~\ref{prop:bound-Q-M}, the second term $\Gamma_2$ on the right hand side satisfies
\begin{align} \label{term:T-2}
& \quad \,
   \int_{\T^3} \int_{\R^3} 
     Q \vpran{\del^\alpha_x f, \, \mu}
     \vpran{\del^\alpha_x f} W^{2\vpran{l - |\alpha|}} \dv\dx 
\leq
  C_{l} 
  \norm{W^{l - |\alpha|} \del^\alpha_x f}_{L^2_{x,v}}^2
+ C_0 \norm{W^{l - |\alpha|} \del^\alpha_x f}_{L^2_{\gamma/2} (\!\dx\dv)}^2 ,
\end{align}
where 
%
\begin{align*}
   C_0 = \int_{\Ss^2} b(\cos\theta) \sin^{m_0 (l - |\alpha|)- \frac{3+\gamma}{2}} \tfrac{\theta}{2} \dsigma \int_{\RR^3}
\la v \ra^\gamma \mu(v) dv 
\,,
\qquad
   m_0 > \max\{4s, 1\}\,.
\end{align*}
The second term on the right hand side of \eqref{term:T-2}
has the following more accurate estimate
\begin{align}\label{accurate-upper}
\left( \int_{\Ss^2} b(\cos\theta) \sin^{m_0 (l - |\alpha|)- \frac{3+\gamma}{2}} \tfrac{\theta}{2} \dsigma\right)
\int_{\mathbb{T}^3_x}\int_{\RR^3_v}
\left( \int_{\RR^3} \mu(v_*)|v-v_*|^\gamma dv_*\right)
|W^{l - |\alpha|} \del^\alpha_x f(x,v)|^2 dxdv\,.
\end{align}

%
%
%
\smallskip
The bound related to $\Gamma_3$ is obtained by using $g = h = f$ in Proposition~\ref{Prop:commute-deriv-x}, which gives
\begin{align} \label{bound:Gamma-3}
   \abs{\int_{\T^3} \int_{\R^3} \Gamma_3 \, 
   \vpran{W^{2\vpran{l - |\alpha|}} \del^\alpha_x f} \dv\dx}
\leq
   C_{l} \norm{f}_{Y_l} \sum_{\alpha} 
     \norm{W^{\vpran{l - |\alpha|}} \del^\alpha_x f}_{L^2 (\dx; H^s_{\gamma/2}(\dv))}^2 \,.
\end{align}
Combine all the bounds in~\eqref{term:T-1-1}-\eqref{bound:Gamma-3} in considering the accurate version.
If $\ell$ satisfies 
\begin{align*}
\frac{1}{2}\vpran{m_0(\ell-2)-\frac{3+\gamma}{2}-2} \ge 2 \times 2^2
\end{align*}
then we derive that
\begin{align} \label{bound:evolution-1}
  \frac{1}{2} \frac{\rm d}{\dt}\norm{f}_{Y_l}^2
&\leq 
  - \vpran{\frac{\gamma_0}{4} - C_l \norm{f}_{Y_l}}\sum_{\alpha} \int_{\T^3} \norm{W^{l - |\alpha|} \del^\alpha_x f}_{L^2_{\gamma/2}}^2 \dx  \nn
\\
& \quad \,
  - \vpran{\frac{c_0}{4} \delta_2 - C_l \norm{f}_{Y_l}} 
    \sum_{\alpha} \int_{\T^3} \norm{W^{l - |\alpha|} \del^\alpha_x f}^2_{H^s_{\gamma/2}} \dx  
  + C_{l} \norm{f}^2_{Y_l} \,.
\end{align}

\Ni{\bf Step 2. Evolution of the semigroup part.} In the second step, we derive the evolution of the semigroup part $\int_0^\infty \norm{\Semi(\tau)h}^2_{L^2(\dv;H^2(\dx))} \dtau$. 
This part will provide a linear damping term that can be used to control the linear growth term in~\eqref{bound:evolution-1}. Let $f$ be the solution to~\eqref{eq:perturbation}. Then the semigroup term satisfies
\begin{align*}
& \quad \,
   \frac{1}{2} \frac{\rm d}{\dt} 
   \int_0^\infty
      \norm{\Semi(\tau)f(t, \cdot, \cdot)}^2_{L^2(W^{l_0}\dv;H^2(\dx))} \dtau
\\
&= \frac{1}{2} \frac{\rm d}{\dt} 
   \int_0^\infty \int_{\T^3} \int_{\R^3}
      \abs{\Semi(\tau) \vint{D_x}^2 f(t, x, v)}^2 W^{2l_0}
      \dv\dx\dtau
\\
& = \int_0^\infty \int_{\T^3} \int_{\R^3}
      \vpran{\Semi(\tau) \vint{D_x}^2 f(t, x, v)} 
      \vpran{\vint{D_x}^2  \Semi(\tau) \del_t f(t, x, v)} W^{2l_0}\dv\dx\dtau
\\
& = \int_0^\infty \int_{\T^3} \int_{\R^3}
      \vpran{\Semi(\tau) \vint{D_x}^2  f(t, x, v)} 
      \vpran{\vint{D_x}^2 \Semi(\tau) \CalL f(t, x, v)} W^{2l_0}\dv\dx\dtau
\\
& \quad \, 
      + \int_0^\infty \int_{\T^3} \int_{\R^3}
      \vpran{\Semi(\tau) \vint{D_x}^2 f(t, x, v)} 
      \vpran{\vint{D_x}^2 \Semi(\tau) Q(f, f)(t, x, v)} W^{2l_0}\dv\dx\dtau
\\
& \Denote E_1 + E_2  \,.
\end{align*}
We will show that $E_1$ provides a linear damping term and $E_2$ can be well controlled. First, since $\Semi$ is the semigroup generated by $\CalL$, we have
\begin{align*}
    \Semi(\tau) \CalL f(t, x, v)
 = \del_\tau \vpran{\Semi(\tau) f(t, x, v)} \,.
\end{align*}
Therefore, 
\begin{align} \label{eq:E-1}
   E_1
& = \int_0^\infty \int_{\T^3} \int_{\R^3}
      \vpran{\Semi(\tau) \vint{D_x}^2 f(t, x, v)} 
      \del_\tau \vpran{\Semi(\tau) \vint{D_x}^2 f(t, x, v)} W^{2l_0}\dv\dx\dtau \nn
\\
& = \frac{1}{2} \int_{\T^3} \int_{\R^3}  \int_0^\infty
        \del_\tau \vpran{\vpran{\Semi(\tau) \vint{D_x}^2 f(t, x, v)}^2} 
        W^{2l_0} \dtau \dv\dx \nn
\\
&    = - \frac{1}{2} \norm{W^{l_0} f}^2_{L^2(dv; H^2(\dx))} \,.
\end{align}
The bound of $E_2$ is given by the exponential decay and regularization of $\Semi$. By Theorem~\ref{thm:weighted-H-s-L-2}, there exist constants $C$ and $T_0$ such that
\begin{align*}
   \int_0^{T_0} \norm{\Semi(\tau) Q(f, f)}^2_{L^2(W^{l_0}\dv; H^2(\dx))} \dtau
\leq
   C \norm{Q(f,f)}^2_{H^{-s}(W^{l_0}\dv; H^2(\dx))},
\end{align*}
and
\begin{align*}
      \norm{\Semi(\tau) Q(f,f)}_{L^2(W^{l_0}\dv; H^2(\dx))}
\leq
   C e^{-\lambda \tau} \norm{Q(f,f)}_{H^{-s}(W^{l_0}\dv; H^2(\dx))} \,,
\qquad
   \tau \geq T_0 \,.
\end{align*}
Applying these bounds, we have
\begin{align*}
  \abs{E_2}
&\leq
\int_0^{T_0} \norm{\Semi(\tau) f}_{L^2(W^{l_0} \dv; H^2(\dx))}
\norm{\Semi (\tau) Q(f, f)}_{L^2(W^{l_0}\dv; H^2(\dx))} \dtau
\\
& \quad \,
   + \int^\infty_{T_0} \norm{\Semi(\tau) f}_{L^2(W^{l_0} \dv; H^2(\dx))}
\norm{\Semi (\tau) Q(f, f)}_{L^2(W^{l_0}\dv; H^2(\dx))} \dtau
\\
&\leq
  C \int_0^{T_0} 
   \norm{f}_{L^2(W^{l_0} \dv; H^2(\dx))}
  \norm{\Semi (\tau) Q(f, f)}_{L^2(W^{l_0}\dv; H^2(\dx))} \dtau 
\\
& \quad \, 
  + C \int^\infty_{T_0} e^{-\lambda \tau}
     \norm{f}_{L^2(W^{l_0} \dv; H^2(\dx))}
     \norm{Q(f, f)}_{H^{-s}(W^{l_0}\dv; H^2(\dx))} \dtau
\\
& \leq 
  C \norm{f}_{L^2(W^{l_0} \dv; H^2(\dx))}
     \norm{Q(f, f)}_{H^{-s}(W^{l_0}\dv; H^2(\dx))}
\\
& \leq C \sum_{|\alpha| \leq 2} \sum_{\alpha_1 + \alpha_2 = \alpha} C^{\alpha_1, \alpha_2}
   \norm{f}_{Y_l}
    \norm{Q(\del^{\alpha_1}_x f, \, \, \del^{\alpha_2}_x f)}_{H^{-s}(W^{l_0}\dv; L^2(\dx))} \,.
\end{align*}
Using a similar estimate in~\eqref{bound:T-7-2} and the trilinear estimate in Proposition~\ref{prop:trilinear} gives
\begin{align*}
   \int_{\T^3} \norm{Q(\del^{\alpha_1}_x f, \, \, \del^{\alpha_2}_x f)}_{H^{-s}(W^{l_0}\dv)}^2 \dx
&\leq
   C \int_{\T^3} \norm{W^{l_0}\del^{\alpha_1}_x f}^2_{L^1_{\gamma + 2s} \cap L^2(\dv)}
     \norm{W^{l_0}\del^{\alpha_2}_x f}^2_{H^{s}_{\gamma + 2s}(\dv)} \dx  \,.
\\
&\leq
   C \vpran{\sum_{\alpha} \norm{W^{l - |\alpha|}\del^{\alpha}_x f}^2_{L^2(\dx; H^s_{\gamma/2}(\dv))}}^2
\end{align*}
by taking  
\begin{align*}
   \frac{37 + 5\gamma}{2m_0} \leq l_0 <  l - 3 - \frac{2s + \frac{\gamma}{2}}
{m_0} \,.
\end{align*} 
Using such $l_0$ we have the bound of $E_2$ as
\begin{align*}
   \abs{E_2}
\leq
   C \norm{f}_{Y_l} \sum_{\alpha} \norm{W^{l - |\alpha|}\del^{\alpha}_x f}^2_{L^2(\dx; H^s_{\gamma/2}(\dv))} \,,
\end{align*}
which gives
\begin{align} \label{bound:semigroup}
& \quad \,
   \frac{1}{2} \frac{\rm d}{\dt} 
   \int_0^\infty
      \norm{\Semi(\tau)f(t, \cdot, \cdot)}^2_{L^2(W^{l_0}\dv;H^2(\dx))} \dtau \nn
\\
& \leq 
   - \frac{1}{2} \norm{W^{l_0} f}^2_{L^2(dv; H^2(\dx))}
   + C \norm{f}_{Y_l} \sum_{\alpha} \norm{W^{l - |\alpha|}\del^{\alpha}_x f}^2_{L^2(\dx; H^s_{\gamma/2}(\dv))} \,.
\end{align}
Multiplying~\eqref{bound:semigroup} by $4 C_l$ where $C_l$ is the constant in front of the linear growth in~\eqref{bound:evolution-1} and adding the resulting equation with~\eqref{bound:evolution-1}, we obtain the closed energy estimate stated in~\eqref{est:full-energy}.
This estimate extends the local-in-time bound  of $\norm{f}_{Y_l}$ to the global one provided the initial norm is small.
\end{proof}



\section{Local existence and non-negativity}
Recall  the notation
\[
J_1^{\Phi_\gamma}(f) = \iiint b \Phi_\gamma(|v-v_*|) \mu_* \Big( f' - f \Big)^2 dv dv_*d \sigma,
\]
as in~\cite{AMUXY2012JFA}.  It follows from the proof of (\cite{AMUXY2012JFA}, Lemma 2.17) that
if 
$$\CalC_0(F, f) = \iiint b F_* \Big( f' - f \Big)^2 dv dv_*d \sigma \enskip \, \mbox{ for $F \ge 0$
with $\|F\|_{L^1} {\geq} D_0$ and $\|F\|_{L^1_2} + \|F\|_{L \log L} \le E_0$, }$$
then there exist $C, C' >0$ independent of $F$ and $C_F >0$ such that 
\begin{align*}
&\CalC_0(F, f) 
\le C \|F\|_{L^1_{2s}}\Big( J_1^{\Phi_0}(f) + \|f\|^2_{L^2}\Big), 
\\
&\|F\|_{L^1} J_1^{\Phi_0}(f) \le 2 \CalC_0(F, f) + C' \|F\|_{L^1_{2s}} \|f\|_{H^s}^2 \le C_{F} \Big( \CalC_0(F, f) + \|f\|_{L^2}^2\Big), 
\end{align*}
where $F = \mu + f$ and $C_F$ depends only on $D_0$ and $E_0$.

\subsection{Local existence of linear equation and non-negativity}
\begin{lem}[A linear equation]\label{lemma-linear}
There exist some
$C_0 > 1$, $\tau_0>0$, $T_0 >0$, $l_0 > 0$ such that for all $0 < T \le T_0$, $l \geq l_0$, $f_0 \in Y_l$ with 
$\mu + f_0 \ge 0$, $g \in 
L^{\infty}([0,T]; Y_l)$ 
satisfying
\begin{align*}
\|g\|_{L^{\infty}([0,T]; Y_l)} \leq \tau_0,
\enskip \mu + g \ge 0,\enskip \mbox{and} \enskip \int_0^T \sum_{\alpha} \norm{W^{l - |\alpha|}\del^{\alpha}_x g(\tau)}^2_{L^2(\dx; H^s_{\gamma/2}(\dv))}d\tau \, < \tau_0^2 \,,
\end{align*}
the Cauchy problem
\begin{align}\label{l-e-c}\begin{cases}
\partial_t f+v\cdot \nabla_{x}f- Q(\mu+g, f) 
= Q(g, \mu) 
,\\
f|_{t=0}=f_0(x,v),
\end{cases}
\end{align}
admits a weak solution satisfying 
\begin{align*}
   f \in L^{\infty}([0,T]; Y_l) \,,
\qquad
   \vint{v}^{1/2} \nabla_v f \in  L^2(0, T; Y_l) \,,
\qquad
  \mu + f \geq 0
\end{align*}
with the energy bound
\begin{align}\label{energy-important}\notag 
& \quad \,
\|f\|^2_{L^{\infty}([0,T]; Y_l)}
+ \int_0^T\sum_{\alpha} \norm{W^{l - |\alpha|}\del^{\alpha}_x f(\tau)}^2_{L^2(\dx; H^s_{\gamma/2}(\dv))}d\tau 
\\
& 
\leq  2 \Big( \|f_0\|^2_{Y_l} +
\delta_4 \int_0^T\sum_{\alpha} \norm{W^{l - |\alpha|}\del^{\alpha}_x g(\tau)}^2_{L^2(\dx; H^s_{\gamma/2}(\dv))}d\tau
\Big).  
\end{align}
Here $\delta_4 >0$ is defined in~\eqref{def:coeff-delta}. Moreover, 
\begin{align} \label{bound:J-1-gamma}
   J_1^{\Phi_\gamma}(f) 
\lesssim 
   \|f_0\|^2_{Y_l} +
\delta_4 \int_0^T\sum_{\alpha} \norm{W^{l - |\alpha|}\del^{\alpha}_x g(\tau)}^2_{L^2(\dx; H^s_{\gamma/2}(\dv))}d\tau \,.
\end{align}

\end{lem}
\begin{proof}
To be rigorous, we regularize the $f$-equation and consider the following system
\begin{align}\label{l-e-c-reg}
\begin{cases}
\partial_t f_\kappa + v\cdot \nabla_{x} f_\kappa - Q(\mu+g, f_\kappa) 
= Q(g, \mu) - 2\kappa W^6(v) (\lambda_0 \Id - \Delta_v) f_\kappa \,,
\\
f_\kappa|_{t=0}=f_0(x,v),
\end{cases}
\end{align}
where $\lambda_0 > 0$ is large enough such that for any $\psi \in H^2(W^{l+2} \dv \dx)$, it holds that
\begin{align}  \label{coercivity-reg}
   \vpran{2 W^6(v) (\lambda_0 \Id-\Delta_v) \psi, \, \psi}_{Y_l}
\geq
   \sum_{|\alpha| \leq 2} \norm{W^{l - |\alpha| + 3} \del_x^{\alpha} \psi}^2_{L^2(\!\dx; H^1(\!\dv))} \,.
\end{align}
Let $\CalQ$ be the linear operator given by
\begin{align*}
  \mathcal{Q}
= -\partial_t + \vpran{v\cdot \nabla_{x} - Q(\mu+g,\cdot))
   + \kappa W^4(v) ((\lambda_0 \Id-\Delta_v) \cdot)}^\ast \,,
\end{align*}
where the adjoint operator $(\cdot)^*$ is taken with respect to the scalar product 
in $Y_l$. 
Then, 
for all $h \in C^{\infty}([0,T]\times \mathbb{T}^3_x, \mathcal{S}(\RR_{v}^3))$, with $h(T)=0$ and~$0 \leq t \leq T$, and for any $0 < \delta_3 \ll 1$, we have 
\begin{align*}
\textrm{Re}\big(h(t),\mathcal{Q}h(t)\big)_{Y_l} 
& =   -\frac{1}{2}\frac{\rm d}{\dt}\|h\|^2_{Y_l}
 +\textrm{Re}(v\cdot\nabla_{x}h,h)_{Y_l}-\textrm{Re}(Q(\mu+g,h),h)_{Y_l} 
\\
&\quad \,
   + 2 \kappa \vpran{W^6(v) (\lambda_0 \Id - \Delta_v) h, \, h}_{Y_l}\, 
\\ 
&\geq \ -\frac{1}{2}\frac{\rm d}{\dt} \|h(t)\|^2_{Y_l}
+ \vpran{\frac{\gamma_0}{2} - C_{l}  \norm{g}_{Y_l}} 
  \sum_{\alpha} 
  \norm{W^{\vpran{l - |\alpha|}}\del^\alpha_x h}_{L^2_{\gamma/2}(\!\dx\dv)}^2
\\
& \quad \, 
-C_l \norm{h}_{Y_l}\norm{W^l h}_{L^2_{\gamma/2}(\dx\dv)}\norm{W^l g}_{L^2_{\gamma/2}(\dx\dv)}
\\
& \quad \,
  -  C_l \norm{g}_{Y_l}
  \sum_{\alpha} \norm{W^{l - |\alpha|}\del^{\alpha}_x h}^2_{L^2(\dx; H^s_{\gamma/2}(\dv))}
  \\
& \quad \, 
 -  C_l \norm{h}_{Y_l}
  \sum_{\alpha} \norm{W^{l - |\alpha|}\del^{\alpha}_x h}_{L^2(\dx; H^s_{\gamma/2}(\dv))}
\norm{W^{l - |\alpha|}\del^{\alpha}_x g}_{L^2(\dx; H^s_{\gamma/2}(\dv))}
  \\
& \quad \,
 + \frac{c_0}{4} \delta_2 
\sum_{\alpha} \norm{W^{l - |\alpha|}\del^{\alpha}_x h}^2_{L^2(\dx; H^s_{\gamma/2}(\dv))}
- C_l \delta_3 \sum_{\alpha} 
  \norm{W^{\vpran{l - |\alpha|}}\del^\alpha_x h}_{L^2_{\gamma/2}(\dx\dv)}^2  
\\
& \quad \,
  - C_{l}  \norm{h}^2_{Y_l}
  + \kappa  \sum_{\alpha} \norm{W^{l - |\alpha| + 3} \del_x^{\alpha} h}^2_{L^2(\!\dx; H^1(\!\dv))} \,,
\end{align*}
by~\eqref{coercivity-reg} and the similar method as for $\Gamma_1$ and $\Gamma_3$ in the proof of Theorem~\ref{est:full-energy}.
Putting 
\begin{align*}
   a(t) = \sum_{\alpha} \norm{W^{l - |\alpha|}\del^{\alpha}_x g(t)}^2_{L^2(\dx; H^s_{\gamma/2}(\dv))} 
\end{align*}
and taking a sufficiently small $\delta_3$, we get
\begin{align*}
\textrm{Re}\big(h(t),\mathcal{Q}h(t)\big)_{Y_l}
&\ge  
-\frac{1}{2}\frac{d}{dt}\big(\|h(t)\|^2_{Y_l} \big)
-C_l(1+ a(t) ) \|h(t)\|^2_{Y_l} 
+ \kappa  \sum_{\alpha} \norm{W^{l - |\alpha| + 3} \del_x^{\alpha} h}^2_{L^2_xH^1_v}
\\
& \quad \,
+ \Big(\frac{c_0}{4} \delta_2 - C_l \norm{g}_{Y_l}\Big) \sum_{\alpha} \norm{W^{l - |\alpha|}\del^{\alpha}_x h}^2_{L^2(\dx; H^s_{\gamma/2}(\dv))} \,.
\end{align*}
If $\varepsilon_0 < \frac{c_0 \delta_2}{8 C_l}$ and $V(t) = 2 C_l \int_t^T (1+a(\tau))d\tau$ then we get
\begin{align*}
& \quad \,
  -\frac{d}{dt}\Big(e^{V(t)} \|h\|^2_{Y_l}\Big) + \frac{c_0\delta_2}{8}e^{V(t)}
\sum_{\alpha} \norm{W^{l - |\alpha|}\del^{\alpha}_x h}^2_{L^2(\dx; H^s_{\gamma/2}(\dv))}
\\
&\leq 
  2 e^{V(t)} \abs{\big(h(t),\mathcal{Q}h(t)\big)_{Y_l}}
  - \kappa e^{V(t)} \sum_{\alpha} \norm{W^{l - |\alpha| + 3} \del_x^{\alpha} h}^2_{L^2_xH^1_v} \,.
\end{align*}
Since $h(T) =0$, for all $t \in [0,T]$ we have 
\begin{align} \label{bound:energy-Qh}
   \|h(t)\|^2_{Y_l} 
  + \kappa \int_0^T \sum_{\alpha} \norm{W^{l - |\alpha| + 3} \del_x^{\alpha} h(\tau)}^2_{L^2_xH^1_v}  \dtau 
&\leq 
   2 \int_t^T e^{V(\tau)} 
   \abs{\big(h(\tau),\mathcal{Q}h(\tau)\big)_{Y_l}} \dtau  \nn
\\
&\leq 
   2 e^{V(0)} \int_0^T \abs{\big(h(t),\mathcal{Q}h(t)\big)_{Y_l}} \dtau \,.
\end{align}
We estimate the right hand side of the above inequality in two different ways. First, 
\begin{align*}
   \int_0^T \abs{\big(h(t),\mathcal{Q}h(t)\big)_{Y_l}} \dtau
\leq
   \int_0^T \|h\|_{Y_l} \|\mathcal{Q}h\|_{Y_l} \dtau
\leq
  \|h\|_{L^\infty([0, T]; Y_l)} \|\mathcal{Q}h\|_{L^1(0, T; Y_l)} \,,
\end{align*}
which implies 
\begin{equation}\label{injection}
\|h\|_{L^\infty([0,T];Y_l)} 
\le 2 e^{V(0)} \|\mathcal{Q}h\|_{L^1([0,T]; Y_l)},
\end{equation}
and 
\begin{align}\label{injection-2}
\kappa \int_0^T \sum_{\alpha} \norm{W^{l - |\alpha| + 3} \del_x^{\alpha} h(\tau)}^2_{L^2_xH^1_v} \dtau \nn
&\leq 
   2 e^{V(0)}\|h\|_{L^\infty([0,T];Y_l)} \|\mathcal{Q}h\|_{L^1([0,T]; Y_l)}  \nn
\\
&\leq 
   4e^{2V(0)} \|\mathcal{Q}h\|^2_{L^1([0,T]; Y_l)}.
\end{align}
Second, we have
\begin{align*}
   2 e^{V(0)}\int_0^T \abs{\big(h(t),\mathcal{Q}h(t)\big)_{Y_l}} \dtau
&\leq
  \frac{\kappa}{2} \int_0^T \sum_{\alpha} \norm{W^{l - |\alpha| + 3} \del_x^{\alpha} h(\tau)}^2_{L^2_xH^1_v} \dtau
\\
& \quad \,
  + \frac{1}{2\kappa} e^{2V(0)}
     \int_0^T \sum_{\alpha} \norm{W^{l - |\alpha| - 3} \del_x^{\alpha} \CalQ h(\tau)}^2_{L^2_xH^{-1}_v} \dtau \,.
\end{align*}
Therefore, by~\eqref{bound:energy-Qh} we have
\begin{align} \label{injection-3}
   \|h(t)\|^2_{Y_l} 
  + \frac{\kappa}{2} \int_0^T \sum_{\alpha} \norm{W^{l - |\alpha| + 3} \del_x^{\alpha} h(\tau)}^2_{L^2_xH^1_v}  \dtau 
&\leq 
   \frac{1}{2\kappa} e^{2V(0)} 
   \int_0^T
   \sum_{\alpha} \norm{W^{l - |\alpha| - 3} \del_x^{\alpha} \CalQ h(\tau)}^2_{L^2_xH^{-1}_v} \dtau \,.
\end{align}
Denote $L^2([0, T]; \tilde Y_l)$ as the space such that $w \in L^2([0, T]; \tilde Y_l)$ if and only if
\begin{align*}
   \int_0^T\sum_{\alpha} \norm{W^{l - |\alpha| - 3} \del_x^{\alpha} w(\tau)}^2_{L^2_xH^{-1}_v} \dtau < \infty \,.
\end{align*}
Consider the vector subspace
\begin{align*}
\mathbb{W}&=\{w=\mathcal{Q}h : h \in C^{\infty}([0,T]\times \mathbb{T}^3_x, \mathcal{S}(\RR_{v}^3)), \ h(T)=0\} 
\subseteq L^{1}([0,T], Y_l)\cap L^2([0, T]; \tilde Y_l) \,.
\end{align*}
This inclusion holds because it follows from Proposition \ref{prop:trilinear}
that for any $\varphi \in Y_l$ we have 
\begin{align*}
\left |(\varphi, Q(\mu+g,\cdot)^* h)_{Y_l}\right|& = 
\left|\sum_\alpha (Q(\mu+g), \varphi), W^{2(l-|\alpha|)}\del_x^{2\alpha}h)_{L^2(dx dv)}\right|\\
& \lesssim \sum_\alpha \big(1+ \|\la v \ra^2 g\|_{L^2(dx dv)}\big)\|\varphi\|_{L^2_{\gamma/2}(dx dv)} \|W^{2l} \del_x^{2\alpha}h\|_{L^2(\dx; H^{2s}_{\gamma/2}(\dv))}.
\end{align*}
Since $f_0 \in Y_l$, we define 
a linear functional
\begin{align*}
\mathcal{G} \ : \qquad &\mathbb{W} \enskip \longrightarrow \mathbb{C}\\ \notag
w=&\mathcal{Q}h \mapsto (f_0,h(0))_{Y_l} -
(Q(g, \mu), h)_{L^2([0,T]; Y_l)}\,,
\end{align*}
where $h \in  C^{\infty}([0,T]\times \mathbb{T}^3_x, \mathcal{S}(\RR_{v}^3))$, with $h(T)=0$.
According to \eqref{injection}, the operator $\mathcal{Q}$ is injective. The linear functional $\mathcal{G}$ is therefore well defined. It follows from Proposition \ref{prop:bound-Q-M} and \eqref{injection}-\eqref{injection-3} that $\mathcal{G}$ is a continuous linear form on $(\mathbb{W},\|\cdot\|_{L^{1}([0,T]; Y_l) \cap L^2([0, T];\tilde Y_l))}$, due to the estimates that
\begin{align*}
|\mathcal{G}(w)| &\leq \|f_0\|_{Y_l}\|h(0)\|_{Y_l}
+ C_T \displaystyle \sum_\alpha \|W^{l-|\alpha|}\del_x^\alpha g\|_{L^2([0,T]; L^2_{\gamma/2}(dx dv)) }
\|W^{l-|\alpha|}\del_x^\alpha h\|_{L^{2}([0,T]; L^2_{\gamma/2}(dx dv))}
\\
&\le C'_T(g) 
\|\CalQ h\|_{L^1([0,T];Y_l) \cap L^2([0, T]; \tilde Y_l)} 
= C'_T(g) 
\|w\|_{L^1([0,T];Y_l) \cap L^2([0, T]; \tilde Y_l)}\,.
\end{align*}
By using the Hahn-Banach theorem, $\mathcal{G}$ may be extended as a continuous linear form on $L^{1}([0,T]; Y_l) \cap L^2([0, T]; \tilde Y_l)$
with a norm smaller than $C'_T$.
It follows that there exists $f_\kappa \in L^{\infty}([0,T]; Y_l) \cap L^2([0, T]; \tilde Y_l^\ast)$ 
such that
\begin{align*}
\mathcal{G}(w)=\int_0^T(f_\kappa(t),w(t))_{Y_l}dt ,
\qquad
\forall w \in L^{1}([0,T]; Y_l) \,.
\end{align*}
Hence for all $h \in 
C_0^{\infty}((-\infty,T]; C^\infty(\mathbb{T}^3_x; \mathcal{S}(\RR_{v}^3)))$,
\begin{align*}
\mathcal{G}(\mathcal{Q}h)&=\int_0^T(f_\kappa(t),\mathcal{Q}h(t))_{Y_l}dt
=(f_0,h(0))_{Y_l}
- \int_0^T (Q(g(t), \mu)\,, h(t))_{ Y_l}dt\,.
\end{align*}
This shows that $f_\kappa \in L^{\infty}([0,T]; Y_l) \cap L^2([0, T]; \tilde Y_l)$ is a weak solution of the Cauchy problem~\eqref{l-e-c-reg}
because $\displaystyle   \sum_{|\alpha| \le 2} W^{2(l-|\alpha|)} \del_x^{2\alpha} $ is bijective in $C_0^{\infty}((-\infty,T]; C^\infty(\mathbb{T}^3_x; \mathcal{S}(\RR_{v}^3)))$. 
Note that $f_\kappa \in L^2([0, T]; \tilde Y_l^\ast)$ if and only if
\begin{align*}
   \int_0^T \sum_{\alpha} \norm{W^{l - |\alpha| + 3} \del_x^{\alpha} f_\kappa(\tau)}^2_{L^2_xH^{1}_v} \dtau < \infty \,.
\end{align*}
In particular, this implies
\begin{align} \label{reg:f-sufficient}
   \int_0^T \sum_{\alpha} \norm{W^{l - |\alpha|} \del_x^{\alpha} f_\kappa(\tau)}^2_{L^2(!\dx; H^{s}_{\gamma/2 + 2s}(\!\dv))} \dtau < \infty \,.
\end{align}

Equipped with the regularity of $f_\kappa$ in~\eqref{reg:f-sufficient}, we are ready to show the energy bound in \eqref{energy-important}.
Similar to the proof for Theorem~\ref{thm:a-priori}, 
we have 
\begin{align} \label{bound:evolution-1-linear} 
   \frac{\rm d}{\dt} \norm{f_\kappa}_{Y_l}^2
&\le -\vpran{\frac{\gamma_0}{4} - C_{l} \norm{g}_{Y_l}} 
  \sum_{\alpha} 
  \norm{W^{\vpran{l - |\alpha|}}\del^\alpha_x f_\kappa}_{L^2_{\gamma/2}(\dx\dv)}^2
  + C_{l} \norm{f_\kappa}^2_{Y_l} \nn
\\
& \quad \,-\big(\frac{c_0}{4} \delta_2 -C_l \norm{g}_{Y_l} \big) \sum_{\alpha} \norm{W^{l - |\alpha|}\del^{\alpha}_x f_\kappa}^2_{L^2(\dx; H^s_{\gamma/2}(\dv))}  \nn
\\
& \quad \,+ C_l \norm{f}_{Y_l}\sum_{\alpha}    \norm{W^{\vpran{l - |\alpha|}} \del^\alpha_x g}_{L^2 (\dx; H^s_{\gamma/2}(\dv))} 
\norm{W^{\vpran{l - |\alpha|}} \del^\alpha_x f_\kappa}_{L^2 (\dx; H^s_{\gamma/2}(\dv))}
\\
& \quad \,+ \delta_4 \sum_{\alpha} 
     \norm{W^{\vpran{l - |\alpha|}} \del^\alpha_x g}^2_{L^2_{\gamma/2} (\dx \dv))} \,, \nn
\end{align}
where the small coefficient $\delta_4$ is given by
\begin{align} \label{def:coeff-delta}
    \delta_4
= \frac{4}{\gamma_0} \int_{\Ss^2} b(\cos\theta) \sin^{m_0 l - 3 - \gamma/2} \frac{\theta}{2} \dtheta  \,,
\qquad
   m_0 > \max\{4s, 1\} \,.
\end{align}
By the definition of $\gamma_0$ in~\eqref{def:gamma-0}, we have
\begin{align} \label{bound:upper-delta}
    \delta_4
 \leq
    \frac{8}{\gamma_1} 2^{-\frac{l - 5 - \gamma/2}{2}} \,.
\end{align}
Rigorously speaking, due to the transport term $v \cdot \nabla_x f_\kappa$, one should regularize $f_\kappa$ in $x$ as well to justify~\eqref{bound:evolution-1-linear}. The procedures are the same as in the proof of Theorem 1.1 in \cite{AMUXY-KRM-2011} and we refer the reader there for the details. Putting 
\begin{align*}
  E_l(f_\kappa) 
= \int_0^T\sum_{\alpha} \norm{W^{l - |\alpha|}\del^{\alpha}_x f_\kappa(\tau)}^2_{L^2(\dx; H^s_{\gamma/2}(\dv))}d\tau \,,
\end{align*} 
we have 
\begin{align*}
   (1-C_l T) \norm{f_\kappa}^2_{L^\infty([0,T];Y_l) } 
   + \frac{3c_0}{16} \delta_2 E_l(f_\kappa) 
\leq 
   \norm{f_0}_{Y_l}^2 
   + C_l \norm{f_\kappa}_{L^\infty([0,T];Y_l) } E_l(f_\kappa)^{1/2}E_l(g)^{1/2} + \delta_4 E_l(g),
\end{align*}
which shows 
\begin{align*}
  \frac{c_0}{8} \delta_2 E_l(f_\kappa) 
\leq 
  \norm{f_0}_{Y_l}^2 + C_l\norm{f_\kappa}^2_{L^\infty([0,T];Y_l) } E_l(g)  + \delta_4 E_l(g),
\end{align*}
and hence 
\begin{align*}
(1-C_lT -C_l E_l(g)) \norm{f_\kappa}^2_{L^\infty([0,T];Y_l) } 
\leq 
\frac{3}{2} \norm{f_0}_{Y_l}^2 + \frac{3}{2}\delta_4 E_l(g)\,.
\end{align*}
This concludes
\begin{align}\label{iteration-linear}
\norm{f_\kappa}^2_{L^\infty([0,T];Y_l) }  + E_l(f_\kappa) 
\leq  
2\norm{f_0}_{Y_l}^2 + 2\delta_4 E_l(g),
\end{align}
under the assumption that $\tau_0$ and $T_0$ are sufficiently small. Moreover, by Corollary~\ref{est-Q(Fg)g)} and~\eqref{iteration-linear}, we obtain~\eqref{bound:J-1-gamma} by using the term $-(Q(\mu+g,W^{l-|\alpha|}\del_x^\alpha f_\kappa), W^{l-|\alpha|}\del_x^\alpha f_\kappa))_{L^2(\!\dx\dv)}$ to 
control
\[
\int_{\mathbb{T}^3_x} \Big(\iiint B(\mu+g)_* \Big((W^{l-|\alpha|}\del_x^\alpha f_\kappa)' -W^{l-|\alpha|}\del_x^\alpha f_\kappa\Big)^2
dvdv_*d\sigma \Big) \dx \,.
\]
The existence of a weak solution to~\eqref{l-e-c} is then obtained by using the uniform estimate in~\eqref{energy-important} and passing $\kappa \to 0$. 

\bigskip

{\bf Proof of the non-negativity} We put  $G=\mu+g$ and $F=\mu+f$,
and follow the method developed in \cite{AMUXY2011}.
It follows from \eqref{l-e-c} that 
\begin{equation}\label{4.4.3}
\left\{\begin{array}{l} \partial_t F+ v\,\cdot\,\nabla_x
F =Q (G, F), \\ F|_{t=0} = F_0 = \mu +
 f_0\geq 0\, .
\end{array} \right.
\end{equation}
Thanks to \eqref{energy-important}, we have
$\displaystyle 
 \int_0^T \norm{ J_1^{\Phi_0}(W^l
 F(t))}_{L^1(\mathbb{T}^3_x)} dt <\infty\,,
$
and hence, if $F_{\pm} = \pm \max\{ \pm F, 0\}$ then we have 
\[
\int_0^T \norm{ J_1^{\Phi_0}(W^l 
 F_+(t))}_{L^1(\mathbb{T}^3_x)} dt + 
\int_0^T \norm{ J_1^{\Phi_0}(W^l
F_-(t))}_{L^1(\mathbb{T}^3_x)} dt
<\infty,
\]
because for $\tilde F = W^l F$ 
\[
J_1^{\Phi_0}(\tilde F)= J_1^{\Phi_0}( \tilde F_+)
+ J_1^{\Phi_0}(\tilde F_{-}) - 2\iiint b \mu_* \Big( \tilde F'_+ \tilde F_- +  
\tilde F_+ \tilde F'_{-} \Big) dvdv_*d\sigma,
\] 
and the third term is non-negative. 
Take the convex function $\beta (s) = \frac 1 2 (s^- )^2= \frac 1 2
s\,(s^- )  $ with $s^-=\min\{s, 0\}$, and 
notice that
\begin{align*}
\beta_s (F) &: =\left(\frac{d}{ds}\,\,\beta
\right)(F) 
={F_{ -}} \,\in
L^2([0,T]\times \mathbb{T}_x^3 ;H^s_{2+\gamma/2}(\RR_v^3)).
\end{align*}
Let $m$ be sufficiently large but  $W^m \la v \ra^s \le W^l$.
Multiply the first equation of \eqref{4.4.3}
by $\beta_s (F)W^{2m}$ 
$
= F_{-}W^{2m}$ and integrate over
$[0,t] \times \mathbb{T}^3_x \times \RR^3_v$, ($t \in (0,T]$). 
Then, in view of $\beta(F(0)) = F_{0, -}^2/2 =0$ and 
$$
\displaystyle \int_0^t \int_{\mathbb{T}^3_x \times \RR_v^3}{  W^{2m}\,  v\,\cdot\, \nabla_x \enskip (  \beta
(F(\tau))  }dxdv d\tau =0, $$
we have
\begin{align*}
\int_{\mathbb{T}^3_x \times \RR_v^3} \beta ( F(t)) W^{2m}dxdv
&=\int_0^t \left(\int_{\mathbb{T}^3_x \times \RR_v^3}
Q(G(\tau),\, F(\tau) )\,\, \beta_s(F(\tau)) W^{2m} \,dxdv\right) d\tau \,,
\end{align*}
where the right hand side is well defined because 
\[
G \in L^\infty([0,T]\times \mathbb{T}^3_x; L^2_{10}(\RR^3_v)), 
\qquad
W^l F, \, W^l F_{\pm} \in 
\in L^2([0,T]\times \mathbb{T}_x^3 ;H^s_{\gamma/2}(\RR_v^3))\,.
\]
The integrand  on the right hand side is equal to
\begin{align*}
\int_{\mathbb{T}^3_x \times \RR_v^3}   Q(G, F_{-} ) F_{-} W^{2m} \,dxdv
+ \int_{\mathbb{T}^3_x \times \RR_v^3\times \RR^3_{v_*} \times \mathbb{S}^2}    B  G'_* ( F_+)' F_{-}W^{2m}
dvdv_* d\sigma dx .
\end{align*}
From the induction hypothesis, the second term is non-positive. Therefore it follows from 
Proposition \ref{prop:coercivity}  with $l=m$ that
\begin{align*}
\norm{W^m F_{-}(t)}^2_{L^2(dxdv)} &= \int_{\mathbb{T}^3_x \times \RR_v^3} \beta ( W^m F(t))dxdv\\
&\le \int_0^t \left(
\int_{\mathbb{T}^3_x \times \RR_v^3}   Q(G(\tau), F_{-}(\tau) ) F_{-}(\tau)W^{2m}dxdv\right) d\tau\\
&\le -c_0\big(1-C_l \norm{g}_{L^\infty([0,T];Y_l)}\big) \int_0^t
 \norm{W^m F_{-}(\tau)}^2_{L^2_{\gamma/2}(dxdv)} d\tau\\
&\qquad + C_l(1+ \norm{g}_{L^\infty([0,T];Y_l)}
\int_0^t \norm{W^m F_{-}(\tau)}^2_{L^2(dxdv)} d\tau,
\end{align*}
which implies that 
$F= \mu +f  \ge 0$ for $(t,x,v) \in [0,T]\times\mathbb{T}^3_x \times \RR^3_v$. 
\end{proof}

\subsection{Local solution for non-linear equation and its uniqueness}
\begin{thm}[Local Existence]
\label{local-existence}
There exist $\epsilon_0$, $\Eps_1$ and $T>0$ such that  if $ f_0 \in Y_l$ and 
\begin{equation*}
\norm{f_0}_{Y_l} \le \epsilon_0,
\end{equation*}
then the Cauchy problem \eqref{eq:perturbation}
admits a  unique solution 
\begin{align} \label{bound:f-Y-l}
&f \in  L^\infty([0,T];Y_l) \enskip \mbox{satisfying} \enskip \norm{f}_{L^\infty([0,T];Y_l)} \le \epsilon_1 \,,
\end{align}
and
\begin{align} \label{def:A-l}
  E_l(f)  \Denote \int_0^T\sum_{\alpha} \norm{W^{l - |\alpha|}\del^{\alpha}_x f(\tau)}^2_{L^2(\dx; H^s_{\gamma/2}(\dv))}d\tau < \epsilon_1^2
\,.
\end{align}
\end{thm}
\begin{proof}
Consider the sequence of approximate solutions defined by $f^0 = 0$ and
\begin{align}\label{l-e-c-itere-}\begin{cases}
\partial_t f^{n+1} +v\cdot \nabla_{x}f^{n+1} - Q(\mu+f^n , f^{n+1}) = Q(f^n,\mu)
,\\
f^{n+1}|_{t=0}=f_0(x,v). 
\end{cases}
\end{align}
Use Lemma \ref{lemma-linear} with $f= f^{n+1}, g= f^n$ and choose $T, \delta$ sufficiently small. 
Then it follows from \eqref{energy-important} that
\begin{align} \label{bound:f-n}
\norm{f^n}_{L^\infty([0,T];Y_l)} \le \epsilon_1, \enskip 
E_l(f^n) \le \epsilon_1^2,
\end{align}
inductively, if $\epsilon_0$ and $\delta_4$ are taken such that 
\begin{align} \label{cond:size-small}
  2(\epsilon_0^2 + \delta_4 \epsilon_1^2)
\le \epsilon_1^2 \,.
\end{align}  
A sufficient condition for~\eqref{cond:size-small} to hold is by choosing \begin{align*}
   \delta_4 < \frac{1}{2} \,,
\qquad
   \Eps_1 \geq 2 \Eps_0 \,.
\end{align*}
It remains to prove the convergence of the sequence $\{f^n\}$. Setting $w^n 
=f^{n+1}- f^n$, from \eqref{l-e-c-itere-} we have 
\begin{align*}
\partial_t w^{n} +v\cdot \nabla_{x}w^{n} - Q(\mu+f^n , w^n) 
= Q(w^{n-1},\mu + f^n),
\end{align*}
with $w^n |_{t=0} =0$. Repeating the estimates leading to~\eqref{energy-important}, we have
\begin{align} \label{ineq:w-n}
& \quad \,
   \norm{w^n}^2_{L^\infty(0, T; Y_{l'})} + E_{l'}(w^n) \nn
\\
& \leq
   2 \delta_4 E_{l'}(w^{n-1})
   + 2 \sum_\alpha \int_0^T \int_{\T^3} \int_{\R^3}
       \del_x^\alpha Q(w^{n-1}, f^n) \vpran{\del_x^\alpha w^n}
       W^{2(l' - |\alpha|)} \dv\dx\dt \,,
\end{align}
where $E_{l'}$ is defined in~\eqref{def:A-l}. We claim that if we choose 
\begin{align*}
   l' = l -2 \,, 
\qquad
   l \geq 11 + \gamma \,,
\end{align*} 
then for each $|\alpha| \leq 2$, the last term of the right hand side of the above inequality is bounded by
\begin{align} \label{bound:smaller-weight}
& \quad \,
    \abs{\int_0^T \int_{\T^3} \int_{\R^3}
       \del_x^\alpha Q(w^{n-1}, f^n) \vpran{\del_x^\alpha w^n}
       W^{2(l' - |\alpha|)} \dv\dx\dt }  \nn
\\
& \leq
    C_l \vpran{\norm{w^{n-1}}_{L^\infty(0, T; Y_{l'})} + E^{1/2}_{l'}(w^{n-1})}
    \vpran{\norm{f^n}_{L^\infty(0, T); Y_l} + E^{1/2}_{l}(f^n)} E^{1/2}_{l'}(w^n) \,.   
\end{align}
The proof is similar to  the one for Proposition~\ref{prop:bound-Q-M}. Indeed, for any $\alpha_1 + \alpha_2 = \alpha$, 
\begin{align*}
& \quad \,
  \int_{\R^3}
       Q(\del_x^{\alpha_1} w^{n-1}, \del_x^{\alpha_2} f^n) \vpran{\del_x^\alpha w^n}
       W^{2(l' - |\alpha|)} \dv
\\
&= \IntRRS b(\cos\theta) |v - v_\ast|^\gamma
      \vpran{\del_x^{\alpha_1} w^{n-1}(v^\ast)} 
      \vpran{\del_x^{\alpha_2} f^n(v)} 
\\
& \hspace{2.5cm}
   \times
      \vpran{\vpran{\del_x^\alpha w^n (v')}
       W^{2(l' - |\alpha|)}(v') - \vpran{\del_x^\alpha w^n(v)}
       W^{2(l' - |\alpha|)}(v)} \dbmu
\\
& = \IntRRS b(\cos\theta) |v - v_\ast|^\gamma
       \vpran{\del_x^{\alpha_1} w^{n-1}(v^\ast)} 
      \vpran{\del_x^{\alpha_2} f^n(v) W^{l'-|\alpha|}(v)}
\\
& \hspace{2.5cm}
   \times
       \vpran{\vpran{\del_x^\alpha w^n (v')}
       W^{l' - |\alpha|}(v') - \vpran{\del_x^\alpha w^n(v)}
       W^{l' - |\alpha|}(v)} \dbmu
\\
& \quad \, 
   + \IntRRS b(\cos\theta) |v - v_\ast|^\gamma
      \vpran{\del_x^{\alpha_1} w^{n-1}(v^\ast)} 
      \vpran{\del_x^{\alpha_2} f^n(v)}
      \vpran{\vpran{\del_x^\alpha w^n (v')}
       W^{l' - |\alpha|}(v')}
\\
& \hspace{2.5cm}
   \times
     \vpran{W^{l' - |\alpha|}(v') - W^{l' - |\alpha|}(v)} \dbmu        
\\
& \Denote T_8 + T_9 \,.
\end{align*}
Applying the trilinear estimate in Proposition~\ref{prop:trilinear} to $T_8$ gives
\begin{align} \label{bound:T-8}
  \abs{T_8} \dx\dt
&\leq 
   C_{l'} \norm{\del_x^{\alpha_1} w^{n-1}}_{L^1_{\gamma + 2s} \cap L^2}
     \norm{\del_x^{\alpha_2} f^n(v) W^{l'-|\alpha|}}_{H^{s}_{\gamma/2 + 2s}}
     \norm{\del_x^\alpha w^n(v)
       W^{l' - |\alpha|}}_{H^{s}_{\gamma/2}} \nn
\\
& \leq 
   C_{l'} \norm{W^{l' - |\alpha_1|}\del_x^{\alpha_1} w^{n-1}}_{L^2}
     \norm{W^{l - |\alpha_2|}\del_x^{\alpha_2} f^n(v)}_{H^{s}_{\gamma/2}}
     \norm{\del_x^\alpha w^n(v)
       W^{l' - |\alpha|}}_{H^{s}_{\gamma/2}} \,,
\end{align}
if we choose $l'$ such that
\begin{align*}
   \frac{\gamma}{2} + 2s + 4
\leq l'
\leq l - 2s \,,
\qquad
   l \geq 8 + \gamma/2 \,.
\end{align*}
By Proposition~\ref{prop:trilinear-weight}, we bound $T_9$ as
\begin{align}
   \abs{T_9}
&\leq
    C_0
    \norm{\del_x^{\alpha_2} f^n}_{L^1_\gamma}
    \norm{W^{l'-|\alpha|} \del_x^{\alpha_1} w^{n-1}}_{L^2_{\gamma/2}} 
   \norm{W^{l'-|\alpha|} \del_x^\alpha w^n}_{L^2_{\gamma/2}} \nn
\\
& \quad \,
   + C_{l'} \norm{\del_x^{\alpha_2} f^n}_{L^1_{1+\gamma}}
    \norm{W^{l' - |\alpha|} \del_x^{\alpha_1} w^{n-1}}_{L^2} 
   \norm{W^{l'-|\alpha|} \del_x^\alpha w^n}_{L^2} \nn
\\
& \quad \,
   + C_k \norm{\del_x^{\alpha_1} w^{n-1}}_{L^1_{4+\gamma}} \norm{W^{l' - |\alpha|} \del_x^{\alpha_2} f^n}_{L^2} \norm{W^{l'-|\alpha|} \del_x^\alpha w^n}_{L^2} 
\\
& \quad \,
  +   C_{l'} \norm{\del_x^{\alpha_1} w^{n-1}}_{L^1_\gamma}
  \norm{W^{l' - |\alpha|} \del_x^{\alpha_2} f^n}_{L^2_{\gamma/2}}
  \norm{W^{l' - |\alpha|} \del_x^\alpha w^n}_{L^2_{\gamma/2}} \nn
\\
& \quad \,
  + C_{l'} \norm{\del_x^{\alpha_1} w^{n-1}}_{L^1_{3+\gamma+2s} \cap L^2}
  \norm{W^{l' - |\alpha|} \del_x^{\alpha_2} f^n}_{H^{s'}_{\gamma'/2}}
  \norm{W^{l' - |\alpha|} \del_x^\alpha w^n}_{L^2_{\gamma/2}} \nn \,.
\end{align}
Integrating in $t, x$ and using H\"{o}lder's inequality, we have
\begin{align*}
  \int_0^T \int_{\T^3} |T_8| \dx\dt
\leq
  C_{l'} \norm{w^{n-1}}_{L^\infty([0, T]; Y_{l'})}
  E^{1/2}_{l}(f^n) E^{1/2}_{l'}(w^n)  \,.
 \end{align*}
Similarly,  we have the bound for $T_9$ as
\begin{align*}
  \int_0^T \int_{\T^3} |T_9| \dx\dt
\leq
   C_{l'} \vpran{\norm{f^n}_{L^\infty([0, T]; Y_l)}
   + E^{1/2}_l(f^n)} E^{1/2}_{l'}(w^{n-1}) E^{1/2}_{l'}(w^{n}), 
\end{align*}
if we choose $l'$ such that $9 + \gamma \leq l'$. In summary, \eqref{bound:smaller-weight} holds if 
\begin{align*}
   9 + \gamma \leq l' \leq l - 2s \,.
\end{align*}
Applying~\eqref{bound:smaller-weight} and H\"{o}lder's inequality in~\eqref{ineq:w-n}, we have
\begin{align} \label{bound:iteration}
     \norm{w^n}^2_{L^\infty(0, T; Y_{l'})} + E_{l'}(w^n)
&\leq
    \vpran{4 \delta_4 + C_l \vpran{\norm{f^n}^2_{L^\infty([0, T]; Y_l)} + E_l(f^n)}}    
    \vpran{\norm{w^{n-1}}^2_{L^\infty(0, T; Y_{l'})} + E_{l'}(w^{n-1})} \nn
\\
& \leq 
    \vpran{4 \delta_4 + 2 C_l\Eps_1^2}    
    \vpran{\norm{w^{n-1}}^2_{L^\infty(0, T; Y_{l'})} + E_{l'}(w^{n-1})} \,.
\end{align}
Hence, if we choose $\delta, \Eps_1$ small enough such that 
\begin{align} \label{cond:delta-1}
    4\delta_4 + 2\Eps_1^2 < 1 \,,
\end{align}
then the series $\sum_{n=0}^\infty \vpran{\norm{w^{n-1}}^2_{L^\infty(0, T; Y_{l'})} + E_{l'}(w^{n-1})}$ converges. 
With the smallness condition~\eqref{cond:delta-1}, there exists a function $f \in L^\infty(0, T; Y_{l'})$ with $E_{l'}(f) < \infty$ such that
\begin{align} \label{converg:f}
    f^n \to f \quad \text{strongly in $L^\infty(0, T; Y_{l'})$} \,,
\qquad
   E_{l'} (f^n - f)  \to 0 \,. 
\end{align}
Moreover, by~\eqref{bound:f-n} we also have that $f \in L^\infty([0,T];Y_l)$ and 
\begin{align} \label{bound:solution-f}
   \norm{f}_{L^\infty([0,T];Y_l)} \le \epsilon_1, 
\qquad
    E_l(f) \le \epsilon_1^2 \,.
\end{align}
To complete the proof of the local existence of the solution to the nonlinear equation, we only need to show that 
\begin{align} \label{converg:Q}
    Q(\mu + f^n , f^{n+1}) \to Q(\mu + f, f) 
\qquad 
    \text{in $\CalD'([0, T] \times \T^3 \times \R^3)$.}
\end{align}
To this end, let $\phi \in C^\infty_c([0, T] \times \T^3 \times \R^3))$. Then by letting $(\sigma, m) = (-s, 0)$ in Proposition~\ref{prop:trilinear} and using the uniform bounds in ~\eqref{bound:f-n} and~\eqref{bound:solution-f}, we have
\begin{align*}
& \quad \,
   \abs{\int_{\T^3} \int_{\R^3} \vpran{Q(\mu+f^n, f^{n+1}) - Q(\mu+f, f)} \phi \dv\dx}
\\
&\leq
   \abs{\int_{\T^3} \int_{\R^3} Q(f^n - f, f^{n+1}) \, \phi \dv\dx}
  +    \abs{\int_{\T^3} \int_{\R^3} Q(\mu+f, f^{n+1} - f) \, \phi \dv\dx}  
\\
& \leq 
  C_\phi \vpran{\int_{\T^3} \norm{f^n-f}_{L^1_{\gamma+2s} \cap L^2} \norm{f^{n+1}}_{L^2_{\gamma/2 + 2s}} \dx
  + \int_{\T^3} \norm{\mu+f}_{L^1_{\gamma+2s} \cap L^2} \norm{f^{n+1} - f}_{L^2_{\gamma/2 + 2s}}}
\\
& \leq
   C_\phi \vpran{\norm{f^n - f}_{Y_{l'}} + \norm{f^{n+1} - f}_{Y_{l'}}}
   \vpran{\norm{f^n}_{Y_l} + \norm{f}_{Y_l}}
\to 0 
\quad
\text{as $n \to \infty$} \,,
\end{align*}
as long as $l' \geq 2 + \gamma + 2s$. 
Hence~\eqref{converg:Q} holds. The other terms in equation~\eqref{l-e-c-itere-} are all linear, therefore they all converge to the corresponding terms in $f$ in the sense of distribution. We thereby complete the proof of the existence of a weak solution $f$ to~\eqref{eq:perturbation} with the desired bounds in~\eqref{bound:f-Y-l} and~\eqref{def:A-l}. Given the bounds for $f$, the uniqueness follows from the estimate in~\eqref{bound:iteration}. 
\end{proof}

\begin{rmk}
By the definition of $\delta_4$ in the proof of Theorem~\ref{local-existence}, one sufficient condition for $\delta_4 < 1/8$ is
\begin{align*}
   \frac{8}{\gamma_1} 2^{-\frac{l - 5 - \gamma/2}{2}}
< 1/8 \,,
\end{align*}
which gives $l \geq 33$.
\end{rmk}

\appendix


\section{Two lemmas}

In the first part of the  appendix, we give two lemmas that have been used and can be useful for future study.

\begin{lem}[Ukai estimate]\label{lem:Ukai}
For any $\alpha>0$, there exists a constant $c_\alpha >0$ such that
\begin{equation}\label{4.1}
\int^t_0|\xi-s\eta|^\alpha ds \geq c_\alpha (t|\xi|^\alpha +
t^{\alpha+1}|\eta|^\alpha).
\end{equation}
\end{lem}
\noindent{\bf Remark :} If $\alpha=2$, 
estimate follows from a  direct calculation.
The following simple proof in general case is due to Seiji Ukai.

\begin{proof}  Setting $s = t \tau$ and {  $\tilde \eta = t
\eta$}, we see that the estimate is equivalent to
$$
\int^1_0|\xi- \tau \tilde \eta|^\alpha d\tau \geq c_\alpha
(|\xi| ^\alpha+ |\tilde
\eta|^\alpha ).
$$
Since this is trivial when  $\tilde \eta =0$, we may assume
$\tilde \eta \ne 0$.
If $|\xi| < |\tilde \eta|$ then
\begin{eqnarray*}
&&\int^1_0|\xi- \tau \tilde \eta|^\alpha d\tau \geq |\tilde
\eta|^\alpha
\int^{1}_{0} \left| \tau - \frac{|\xi|}{|\tilde \eta|} \right|^\alpha d \tau \\
&=& |\tilde \eta|^\alpha \left \{ \int_{0}^{|\xi|/|\tilde \eta|}
\Big( \frac{|\xi|}{|\tilde \eta|}- \tau \Big)^\alpha d \tau +
\int_{|\xi|/|\tilde \eta|}^1 \Big( \tau - \frac{|\xi|}{|\tilde
\eta|} \Big)^\alpha d \tau \right\}
\\
&\geq&  \frac{|\tilde \eta|^\alpha}{\alpha+1}  \min_{0\leq \theta
\leq 1} ( \theta^{\alpha+1} + (1-\theta )^{\alpha+1})
= \frac{|\tilde \eta|^\alpha}{2^\alpha (\alpha+1)}\\
& \geq& \frac{1}{2^{\alpha+1}(\alpha+1)}(|\xi| ^\alpha+ |\tilde
\eta|^\alpha ).
\end{eqnarray*}
If $|\xi| \geq |\tilde \eta|$ then
\begin{eqnarray*}
\int^1_0|\xi- \tau \tilde \eta|^\alpha d\tau &\geq& |\xi|^\alpha
\int^{1}_0\Big( 1 - \tau \frac{|\tilde \eta|}{|\xi |}\Big)^\alpha
d\tau \geq
 |\xi|^\alpha
\int^{1}_0\big( 1 -  \tau\big)^\alpha d\tau \\
&=& \frac{|\xi|^\alpha}{\alpha+1} \geq \frac{1}{2(\alpha+1)}(|\xi|
^\alpha+ |\tilde \eta|^\alpha ).
\end{eqnarray*}
Hence we obtain (\ref{4.1}).
\end{proof}

\begin{cor}\label{appendix-cor-A}
For any $\alpha >0$, we have 
\begin{equation}\label{A-c}
\int^t_0\langle \xi-s\eta\rangle^\alpha ds \sim t \big( 1+ |\xi|^2 + t^2 |\eta|^2\big )^{\alpha/2}\, .
\end{equation}
\end{cor}
\begin{lem} \label{lem:integral-bound-negative}
For any $0<\beta<1$, there exists a constant $C_\beta >0$ such that
\begin{equation}\label{4.1}
\int^t_0\la \xi-s\eta \ra^{-\beta} ds \leq C_\beta \frac{t }{\big(1+ |\xi|^2 + t^2|\eta|^2\big)^{\beta/2}}.
\end{equation}
\end{lem}
\begin{proof}
Setting $s = t \tau$ and {  $\tilde \eta = t
\eta$} as in the preceding proof, we see that the estimate is equivalent to
$$
\int^1_0 \frac{d\tau}{\la \xi- \tau \tilde \eta \ra^\beta}  \leq C_\beta
\frac{1 }{\big(1+ |\xi|^2 + |\tilde \eta|^2\big)^{\beta/2}}.
$$
Note $1+|\xi-\tau \tilde \eta|^2 \geq 1 + \left( |\xi| - |\tau \tilde \eta|\right)^2$ and put $a= |\xi|, b =|\tilde \eta|$. Then, it suffices to show
\[
\int_0^1 \frac{d\tau}{\big(1+ (a-b\tau)^2\big)^{\beta/2}} \le C_\beta (1+a^2+ b^2)^{-\beta/2},
\]
when $\max (a , b) >1$. If $b \le a/2$ then this holds with $C_\beta = 5^{\beta}$
because $(a-b\tau)^2 \ge a^2/4 \ge (a^2+b^2)/5$.
 If $a/2 < b \le 2a$ then by the change of variable
$u = a-b\tau$, we have
\begin{align*}
&\int_0^1 \frac{d\tau}{\big(1+ (a-b\tau)^2\big)^{\beta/2}}
= \frac{1}{b} \int_{a-b}^a \frac{du}{(1+u^2)^{\beta/2}}\\
&\le\frac{1}{b} \int_{-a}^a \frac{du}{(1+u^2)^{\beta/2}} \le \frac{2}{b}
 \int_{0}^a \frac{du}{u^{\beta}} = \frac{2a^{1-\beta}}{(1-\beta)b} \le \frac{4}{(1-\beta)a^\beta}\,.
\end{align*}
If $b >2a$, then 
\begin{align*}
&\int_0^1 \frac{d\tau}{\big(1+ (a-b\tau)^2\big)^{\beta/2}}
= \frac{1}{b} \int_{a-b}^a \frac{du}{(1+u^2)^{\beta/2}}\\
&\le \frac{2}{b}
 \int_{0}^{b-a} \frac{du}{(1+u^2)^{\beta/2}} \le \frac{2}{b}\int_{0}^b \frac{du}{u^{\beta}}
=\frac{2}{(1-\beta)b^\beta} \,. \qedhere
\end{align*}
\end{proof}


\section{Some estimates on $Q$}

The second part of the appendix is about some estimates on the nonlinear collision operator $Q$.

\begin{lemma}\label{lemma-2.17-JFA-2011}
Let $b$ satisfy $b(\cos\theta) \sin\theta \sim \frac{1}{\theta^{1 + 2s}} $ with $0<s<1$.
Then there exists a constant $C>0$ such that
\begin{align}\label{upper-1}
\iiint b F_* (g'-g)^2 d\sigma dv dv_* \le C \|F\|_{L^1_{2s} }
\Big(J_1^{\Phi_0}(g) + \|g\|^2_{L^2} \Big).
\end{align}
If $F \in L^1$ satisfies $F \ge 0$ and  there exist constants
$D_0, E_0 >0$ such that
\[
\|F\|_{L^1}  \ge D_0 \enskip \mbox{and}\enskip \|F\|_{L^1_2} + \|F\|_{L \log L} \le E_0,
\]
then there exists a $C_F >0$ depending only on $D_0$ and $E_0$ such that 
\begin{align}\label{another-1}
J_1^{\Phi_0}(g) \le  C_{F} \Big( 
\iiint b F_* (g'-g)^2 d\sigma dv dv_* 
+ \|g\|_{L^2}^2\Big).
\end{align}
\end{lemma}
\begin{proof} For the proof of \eqref{upper-1} we may assume $F \ge 0$. 
It follows from \cite[Proposition 1]{ADVW2000} that 
\begin{align}\label{M-J-1}
J_1^{\Phi_0}(g) &=\iiint b M_*(g'-g)^2 d \sigma dv dv_* \notag \\
& = \frac{1}{(2\pi)^3} \iint b\Big(\frac{\xi}{|\xi|}\cdot \sigma \Big) 
\Big \{ \hat M(0) \left|\hat g(\xi) - \hat g(\xi^+)\right |^2 \notag \\
& \quad + 2 Re \Big(\big( \hat M(0) - \hat M(\xi^-) \big) \hat g(\xi^+) \overline{\hat g(\xi)} \Big) \Big\}d \xi d\sigma, 
\end{align}
and 
\begin{align}\label{F-J-1}
\mathcal{C}_0(F,g) &\Denote  \iiint b F_*(g'-g)^2 d \sigma dv dv_* \notag \\
& = \frac{1}{(2\pi)^3} \iint b\Big(\frac{\xi}{|\xi|}\cdot \sigma \Big) 
\Big \{ \hat F(0) \left|\hat g(\xi) - \hat g(\xi^+)\right |^2 \notag \\
& \quad + 2 Re \Big(\big( \hat F(0) - \hat F(\xi^-) \big) \hat g(\xi^+) \overline{\hat g(\xi)} \Big) \Big\}d \xi d\sigma.
\end{align}
Since $\hat F(0) = \|F\|_{L^1}$ and $\hat M(0) = c_0 >0$, we obtain
\begin{align*}
c_0 \mathcal{C}_0(F,g) - \|F\|_{L^1} J_1^{\Phi_0}(g) 
&=  - \frac{2}{(2\pi)^3}\|F\|_{L^1} \iint b\Big(\frac{\xi}{|\xi|}\cdot \sigma \Big) 
 Re \Big(\big( \hat M(0) - \hat M(\xi^-) \big) \hat g(\xi^+) \overline{\hat g(\xi)} \Big)d \xi d\sigma\\
&\quad  + \frac{2c_0}{(2\pi)^3}\iint b\Big(\frac{\xi}{|\xi|}\cdot \sigma \Big) 
 Re \Big(\big( \hat F(0) - \hat F(\xi^-) \big) \hat g(\xi^+) \overline{\hat g(\xi)} \Big)d \xi d\sigma\\
&\Denote   D_1 + D_2 \,. 
\end{align*}
Write
\begin{align*}
  D_2 & = \frac{2c_0}{(2\pi)^3}\Big\{ \int \left| \hat g(\xi)\right|^2  
\Big( \int b\Big(\frac{\xi}{|\xi|}\cdot \sigma \Big) 
 Re \Big(\big( \hat F(0) - \hat F(\xi^-) \big) d\sigma \Big) d \xi \\
& \quad + 
\iint b\Big(\frac{\xi}{|\xi|}\cdot \sigma \Big) 
 Re \Big(\big( \hat F(0) - \hat F(\xi^-) \big) \big( \hat g(\xi^+)- \hat g(\xi) \big)  \overline{\hat g(\xi)} \Big)d \xi d\sigma
\Big\}\\
&\Denote  D_{2,1} +  D_{2,2}. 
\end{align*}
By  the Cauchy-Schwarz inequality, we have 
\begin{align*}
| D_{2,2}|^2 &\lesssim \iint b\Big(\frac{\xi}{|\xi|}\cdot \sigma \Big) \left| \hat F(0) - \hat F(\xi^-) \right|^2 |\hat g(\xi)|^2 d\xi d\sigma
\iint b\Big(\frac{\xi}{|\xi|}\cdot \sigma \Big) \left|\hat g(\xi^+)- \hat g(\xi)\right|^2 d\xi d\sigma \\
&\Denote D_{2,2}^{(1)} \times D_{2,2}^{(2)} \,. 
\end{align*}
Since
\[
|\widehat{F}
(0) - \widehat{F} (\xi^-)|
\leq \int F(v)|1- e^{-iv\cdot \xi^-}|dv,
\]
we have
\begin{align*}
D_{2,2}^{(1)} &\leq \frac{1}{2} \iiint |\widehat { g} (\xi) |^2
F(v) F(w) 
\Big( \int b
\left( \frac{\xi}{|\xi|}\cdot \sigma \right)( |1- e^{-iv\cdot \xi^-}|^2 +
|1- e^{-iw\cdot \xi^-}|^2) d\sigma \Big) dv dw d\xi\\
&\leq
C\| g\|^2_{H^s}\|F\|_{L^1} \|F\|_{L^1_{2s}},
\end{align*}
because
\begin{align*}
\int b
\left( \frac{\xi}{|\xi|}\cdot \sigma \right)|1- e^{-iv\cdot \xi^-}|^2 d\sigma
&\leq C\Big ( \int_0^{(\la v\ra \la \xi \ra)^{-1}}
\theta^{-1-2s} (|v ||\xi|)^2 \theta^2
d\theta +
\int_{(\la v \ra \la \xi \ra)^{-1}}^{\pi/2}
\theta^{-1-2s}
d\theta \Big)\\
&\leq C \la v \ra^{2s} \la \xi \ra^{2s}\,.
\end{align*}
Then we have 
$
|A_{2,1}| \leq C \|g\|_{H^s}^2 \|F\|_{L^1_{2s}}$
because
\begin{align*}
\int b
\left( \frac{\xi}{|\xi|}\cdot \sigma \right) Re\, \Big(\widehat{F}
(0) - \widehat{F} (\xi^-) \Big)  d\sigma
&= \int F(v)  \Big(\int b
\left( \frac{\xi}{|\xi|}\cdot \sigma \right)\big (1- \cos (v\cdot \xi^-)\big)d\sigma
\Big)dv\\
& \leq C \la \xi \ra^{2s}\int F(v) \la v \ra^{2s}  dv \,.
\end{align*}
Since $\widehat {M}(\xi)$
is real valued,
it follows that
\begin{align}\label{later-use}
Re\, \Big(\widehat {M} (0)
- \widehat {M} (\xi^-) \Big) \widehat { g} (\xi^+ )
\overline{\widehat { g} } (\xi )
&= \Big (\int
\big (1- \cos (v\cdot \xi^-)\big)
M(v) dv\Big)\,
Re\,\widehat { g} (\xi^+ )
\overline{\widehat { g} } (\xi )\notag \\
&\lesssim \min\{\la \xi \ra^2 \theta^2, 1\}|\widehat { g} (\xi^+ )
\overline{\widehat { g} } (\xi )|.
\end{align}
Therefore, by using Cauchy-Schwarz inequality and the change of variables
$\xi \rightarrow \xi^+$, 
we obtain
$$
|D_1| \leq C \|F\|_{L^1} \| g\|_{H^s}^2.$$
Furthermore, it follows from \eqref{M-J-1} that
\begin{align*}
D_{2,2}^{(2)} = \iint b \Big(\frac{\xi}{|\xi|}\cdot \sigma \Big)
 | \widehat { g} (\xi ) -
\widehat {g} (\xi^+) |^2 d\xi d \sigma 
\leq 
  C \Big( J_1^{\Phi_0}(g)   + \| g\|_{H^s}^2 \Big)
\,,
\end{align*}
which yields $|D_{2,2}| \leq C\|F\|^{1/2}_{L^1}\|F\|^{1/2}_{L^1_{2s}}
\| g\|_{H^s}\Big( J_1^{\Phi_0}(g)   + \| g\|_{H^s}^2 \Big)^{1/2} $.
Hence
\[
|D_{2}| \leq C\|F\|_{L^1_{2s}}
\| g \|_{H^s}\Big( J_1^{\Phi_0}(g)   + \| g\|_{H^s}^2 \Big)^{1/2}  \,.
\]
Finally, we have 
\begin{align*}
\left|c_0 \mathcal{C}_0(F,g) - \|F\|_{L^1} J_1^{\Phi_0}(g) \right|
\leq C \|F\|_{L^1_{2s}}
\| g \|_{H^s}\Big( J_1^{\Phi_0}(g)   + \| g\|_{H^s}^2 \Big)^{1/2} \le 
\frac{1}{2} \|F\|_{L^1} J_1^{\Phi_0}(g) + C_F\| g\|_{H^s}^2\,,
\end{align*}
which completes the proof of the lemma because it follows from
\cite[Proposition 1 and 2]{ADVW2000} (see also the proof of \cite[Proposition 2.1]{AMUXY-Kyoto-2012}) that 
\begin{align}\label{coer-origin}
\| g\|_{H^s}^2 &\le   C_F \Big( \mathcal{C}_0(F,g)  + \|g\|^2_{L^2}\Big). \notag \qedhere
\end{align}
\end{proof}
The following lemma is essentially the same as \cite[Lemma 3.2]{AMUXY2012JFA}.

\begin{lemma}\label{upper-F-g-g}
Let $0\le \gamma \le 2$ and $0<s<1$. Then for $0 \le F \in L^1_{2s+\gamma}$ we have
\begin{align*}
\iiint B F_* (g'-g)^2d\sigma  dv_* dv\lesssim \|F\|_{L^1_{\max\{2, 2s+\gamma\}}}
\Big(  J_1^{\Phi_0}( \langle v \rangle^{\gamma/2} g) + \|g\|^2_{L^2_{\gamma/2}}\Big).
\end{align*}
\end{lemma}
Since the estimate 
\[
-2 \big (Q(F,g), g \big ) = \iiint B F_* (g'-g)^2d\sigma  dv_* dv + \iiint B F_* (g^2 -{g'}^2)d\sigma  dv_* dv,
\]
holds and the cancellation lemma in \cite[Lemma 1]{ADVW2000} shows that the second term on the right hand side
is bounded above from 
$C \|F\|_{L^1_{\gamma}} \|g\|^2_{L^2_{\gamma/2}}$, we have the following;
\begin{corollary}\label{est-Q(Fg)g)}
\begin{align*}
- \big (Q(F,g), g \big )_{L^2} \lesssim \|F\|_{L^1_{\max\{2, 2s+\gamma\}}}
\Big(  J_1^{\Phi_0}( \langle v \rangle^{\gamma/2} g) + \|g\|^2_{L^2_{\gamma/2}}\Big).
\end{align*}
\end{corollary}

\begin{proof}[Proof of Lemma \ref{upper-F-g-g}]
Since $|v-v_*|^\gamma \lesssim \la v' \ra^\gamma + \la v_* \ra^\gamma$, we have
\begin{align*}
\iiint b(\cos \theta) |v-v_*|^\gamma F_* (g' - g)^2 d\sigma dv dv_* 
& \lesssim
 \iiint b(\cos \theta) F_*\Big(\la v' \ra^{\gamma/2}g'-
\la v \ra^{\gamma/2} g\,\Big)^2
d\sigma dvdv_* \\
& 
+ \iiint b(\cos \theta) \Big(\la v_* \ra^{\gamma} F_*\Big) (g' - g)^2 d\sigma dv dv_*\\
& 
+ \iiint b(\cos \theta)F_*
\Big(
\la v \ra^{\gamma/2} -\la v' \ra^{\gamma/2}\Big)^2 |g|^2
d\sigma dvdv_*
\\
& 
= A^{(1)} + A^{(2)} + A^{(3)}\,.
\end{align*}
By the mean value theorem we have, for a suitable $v'_\tau =
v + \tau (v'-v)$ with $0 < \tau <1$,
\begin{align*}
\Big|\la v \ra^{\gamma/2} -\la v' \ra^{\gamma/2}\Big|
&\leq C_\gamma \la
v'_\tau  \ra^{(\gamma/2 -1)}  |v-v_*|\theta\,\\
&\leq  \sqrt 2  C_\gamma  \la
v'_\tau  -v_* \ra^{(\gamma/2 -1)} \la v_*\ra^{|\gamma/2-1|} \la v'_\tau -v_*\ra \theta\,\\
&\leq  \sqrt 2  C_\gamma \la
v  -v_* \ra^{\gamma/2}\la v_*\ra^{|\gamma/2-1|} \theta 
\le  \sqrt 2  C_\gamma \la v\ra^{\gamma/2} \la v_* \ra^{\max\{1, \gamma -1\}} \theta.
\end{align*}
Therefore, 
we have
\begin{align*}
A^{(3)} &\lesssim \|F\|_{L^1_{2}} \|g \|^2 _{L^2_{\gamma/2}}.
\end{align*}
It follows from 
\eqref{upper-1}
that
\begin{align*}
A^{(1)} \lesssim  \|F\|_{L^1_{2s}}
\Big(  J_1^{\Phi_0}( \langle v \rangle^{\gamma/2} g) + \|g\|^2_{L^2_{\gamma/2}}\Big),
\qquad
A^{(2)} \lesssim  \|F\|_{L^1_{2s+ \gamma }}
\Big(  J_1^{\Phi_0}( g) + \|g\|^2_{L^2}\Big)\,.
\end{align*}
Note that  
\begin{equation}\label{remember-obvious}
J_1^{\Phi_0}(g) \le 2 J_1^{\Phi_0}(\la v\ra^{\gamma/2}g) + C \|g\|_{L^2_{\gamma/2}}
\end{equation}
holds 
because 
\[ (g'-g)^2 \le \la v' \ra^{\gamma}  (g'-g)^2  \le 2 \big( \la v' \ra^{\gamma/2} g' - \la v \ra^{\gamma/2} g\big)^2 + 2
\big(  \la v' \ra^{\gamma/2} - \la v \ra^{\gamma/2}\big)^2 |g|^2.
\]
Then the lemma follows immediately. 
\end{proof}

By  \eqref{another-1} we have the following  opposite bound corresponding to  Lemma \ref{upper-F-g-g}.
\begin{lemma}\label{coer-J-1}
If $F \in L^1$ satisfies $F \ge 0$ and  there exist constants
$D_0, E_0 >0$ such that
\[
\|F\|_{L^1}  \ge D_0 \enskip \mbox{and}\enskip \|F\|_{L^1_{\max\{2, 2s+\gamma\} }} + \|F\|_{L \log L} \le E_0,
\]
then there exists a $C_F >0$ depending only on $D_0$ and $E_0$ such that 
\begin{align}\label{another-2}
J_1^{\Phi_0}(\la v \ra^{\gamma/2} g) &\le  C_{F} \Big( 
\iiint B F_* (g'-g)^2 d\sigma dv dv_* 
+ \|g\|_{L^2_{\gamma/2}}^2\Big),\\
 &\le  C'_{F} \Big( -\big( Q(F,g), g\big)
+ \|g\|_{L^2_{\gamma/2}}^2\Big). \label{another-corollary}
\end{align}
\end{lemma}

\begin{proof}
As in \cite{AMUXY-Kyoto-2012},  for $D_0, E_0 >0$ we put 
\[
\mathcal{U}(D_0, E_0) 
=
\{ F\in L^1_{\max\{2, 2s+\gamma\} }\cap L\log L; F\ge 0,
\|F\|_{L^1}  \ge D_0, \|F\|_{L^1_{\max\{2, 2s+\gamma\} }} + \|F\|_{L \log L} \le E_0\}.
\]
Set  $B(R) = \{v \in \RR^3\,;\, |v| \le R\}$ for $R >0$ and  $B_0(R,r) =
\{ v \in B(R)\,; \, |v-v_0| \ge r\}$ for  a $v_0 \in \RR^3$ and $r \ge 0$.
It follows {}from  the definition of ${\mathcal{U}}(D_0,E_0)$ that there
exist positive constants $R > 1 > r_0$ depending only on $D_0, E_0$ such that
\begin{equation}\label{uni-g}
\mbox{
$g \in {\mathcal{U}}(D_0,E_0)$\,\, \mbox{implies} \,\,$\chi_{B_0(R, r_0)} g \in {\mathcal{U}}(D_0/2,E_0)$
}\,,
\end{equation}
where $\chi_A$ denotes a characteristic function of the set $A \subset \RR^3$.
We denote
\[
\mathcal{C}_\gamma(F,g) = \iiint B F_* (g'-g)^2 d\sigma dv dv_*.
\]
Let $\varphi_R$ be a non-negative  smooth function not greater  than one, which is 1
for $| v| \ge 4R$ and $0$ for $| v| \le {2R}$.
In view of
\[
\frac{\la v \ra}{4} 
\le|v-v_*|\le 2 \la v \ra \enskip
\mbox{on supp $(\chi_{B(R)})_* \varphi_R$}\,,
\]
we have
\begin{align*}
& \quad 
4^{|\gamma|}|v-v_*|^\gamma F_*(g'-g)^2
\ge \big( F \chi_{B(R)}\big)_* \big(\la v \ra^{\gamma/2} \varphi_R\big)^2(g'-g)^2 \\
&\ge \big( F \chi_{B(R)}\big)_* \Big[\frac{1}{2} \Big(\big(\la v \ra^{\gamma/2} \varphi_R g \big)'-
\la v \ra^{\gamma/2} \varphi_R g \Big)^2
 -\Big( \big(\la v \ra^{\gamma/2} \varphi_R\big)'
 - \la v \ra^{\gamma/2} \varphi_R  \Big)^2 {g' }^2 \Big]\,.
\end{align*}
It follows from  the mean value theorem that for a suitable $v'_\tau =
v + \tau (v'-v)$ with $0 < \tau <1$,
\begin{align*}
\Big|\la v \ra^{\gamma/2} -\la v' \ra^{\gamma/2}\Big|
\lesssim  \la v' \ra^{\gamma/2} \la v_* \ra^{\max\{1, \gamma -1\}} \theta.
\end{align*}
Therefore,
we have
\begin{align}\label{coer-out}
\mathcal{C}_\gamma(F,\, g) \ge 2^{-1-2|\gamma|} \mathcal{C}_0 ( F \chi_{B(R)},\, \varphi_R \la v \ra^{\gamma/2}   g     )
- C_R \|F\|_{L^1} \|g\|^2_{L^2_{\gamma/2}},
\end{align}
for a positive constant $C_R \sim R^{\max\{2, 2(\gamma-1)\}}$.
For a set $B(4R)$ we take a finite covering
\[
B(4R) \subset \mathop{\cup}_{v_j \in B(4R)} A_j \,, \enskip A_j = \{ v \in \RR^3\,;\, |v-v_j| \le \frac{r_0}{4}\}\,,
j \in \{1,2, \cdots,  N\}, N =N(R, r_0).
\]
For each $A_j$ we choose a non-negative smooth function $ \varphi_{A_j}$ which is  $1$ on $A_j$
and $0$ on $\{|v-v_j|\ge r_0/2\}$.
Note that
\[
\frac{r_0}{2} 
\le|v-v_*|\le 6R  \enskip
\mbox{on supp $(\chi_{B_j(R,r_0)})_* \varphi_{A_j}$}\,.
\]
Then we  have
\begin{align*}
&|v-v_*|^\gamma F_*(g'-g)^2
\gtrsim r_0^{\gamma} \big( F \chi_{B_j(R,r_0)}\big)_* \varphi_{A_j}^2(g'-g)^2 \\
&\gtrsim
r_0^{\gamma}\big( F \chi_{B_j(R,r_0)}\big)_* 
\left [\frac{1}{2} \Big(\big(\la v \ra^{\gamma/2} \varphi_{A_j} g \big)'-
\la v \ra^{\gamma/2} \varphi_{A_j} g \Big)^2
 -\Big( \big(\la v \ra^{\gamma/2} \varphi_{A_j}\big)'
 - \la v \ra^{\gamma/2} \varphi_{A_j}  \Big)^2 {g' }^2 \right]\,.
\end{align*}
Since $\left|\big(\la v \ra^{\gamma/2} \varphi_{A_j}\big)'
 - \la v \ra^{\gamma/2} \varphi_{A_j} \right| \lesssim R^{\gamma/2
+\max\{1, \gamma -1\} } r_0^{-1} \theta$ if $|v_*|\le R$, we obtain
\begin{align}\label{coer-inner}
\mathcal{C}_\gamma(F,\, g)
&\gtrsim r_0^\gamma
 (\mathcal{C}_0 ( F \chi_{B_j(R,r_0)},\, \varphi_{A_j} \la v \ra^{\gamma/2}   g    )
 - C_{R,r_0}\|F\|_{L^1} \|g\|^2_{L^2}, 
\end{align}
for a positive constant $C_{R,r_0} \sim R^{\gamma
+2 \max\{1, \gamma -1\} } r_0^{\gamma-2}$.
Writing $\varphi_0$ and $\varphi_j , ( j \ge 1) $ instead of $\varphi_R$ and $\varphi_{A_j}$, and
summing up  \eqref{coer-out} and \eqref{coer-inner} we have
\begin{align*}
\mathcal{C}_\gamma(F,\, g)
&\gtrsim r_0^{\gamma} \sum_{j=0}^N \mathcal{C}_0 ( F \chi_{B_j(R,r_0)},\, \varphi_{j} \la v \ra^{\gamma/2}   g    )
- C'_{R, r_0}\|F\|_{L^1} \|g\|^2_{L^2_{\gamma/2}},
\end{align*}
where $\chi_{B_0(R,r_0)} = \chi_{B(R)}$.  Apply \eqref{another-1} to each $\mathcal{C}_0(\cdot, \cdot)$ term. Then
\begin{align*}
\sum_{j=0}^N J_1^{\Phi_0}(\varphi_j \la v \ra^{\gamma/2} g) \le r_0^{-\gamma}  \mathcal{C}_\gamma(F,\, g)
+ C_F \|g\|_{L^2_{\gamma/2}}^2.
\end{align*}
Since for $G = \la v \ra^{\gamma/2} g$ we have
\[
(G'-G)^2 \lesssim \Big(\sum_{j=0}^N {\varphi'_j}^2 \Big)(G'-G)^2
\le 2 \sum_{j=0}^N \Big((\varphi_j G)' - \varphi_j G\Big)^2 + 2 \sum_{j=0}^N
(\varphi'_j - \varphi_j)^2 G^2.
\]
Then we obtain \eqref{another-2} because 
\[
|\varphi'_j - \varphi_j| \le \left|(\nabla \varphi_j)(v'_\tau)\right||v-v_*|\sin \frac{\theta}{2}
\lesssim \left|(\nabla \varphi_j)(v'_\tau)\right||v'_\tau -v_*|\theta \lesssim 
R r_0^{-1} \la v_*\ra \theta. \qedhere
\]
\end{proof}

\begin{remark}\label{rem-equivalence}
Taking $F= M$ in Lemma \ref{upper-F-g-g} and  \ref{coer-J-1} we have the equivalence
\[
J_1^{\Phi_\gamma} (g) \sim J_1^{\Phi_0}(\la v \ra^{\gamma/2}g) , \enskip \mbox{\rm modulo} \enskip
\|g\|^2_{L^2_{\gamma/2}},
\]
which is a sharper version of the formula given above \cite{AMUXY2012JFA}, Lemma 2.17.
\end{remark}

\begin{lemma}\label{upper-F-g-h}
Let $0 \le \gamma \le 2 $ and $0<s<1$. Then we have 
\begin{align*}
\left|\Big(Q(F,g), h\Big)_{L^2}\right| \lesssim \left\{\begin{array}
{l}\displaystyle  \|F\|_{L^2_4} \big (J_1^{\Phi_0}(\langle  v \rangle^{\gamma/2} g) 
+ \|g\|^2_{L^2_{s+\gamma/2}}\big) ^{1/2}\big ( J_1^{\Phi_0}(\langle  v \rangle^{\gamma/2} h) + \|h\|^2_{L^2_{\gamma/2}}\big)^{1/2},\\\\
\displaystyle \|F\|_{L^2_4} \big(J_1^{\Phi_0}(\langle  v \rangle^{\gamma/2} g) 
+ \|g\|^2_{L^2_{\gamma/2}}\big) ^{1/2}\big(J_1^{\Phi_0}(\langle  v \rangle^{ \gamma/2} h )
   + \|h\|^2_{L^2_{s+ \gamma/2}}\big  )^{1/2}.
\end{array} \right. 
\end{align*}
\end{lemma}
\begin{proof}
We recall the decomposition of \eqref{def:Q-R-bar-R} such that
\begin{align*} 
  Q(F, g)
&= Q_R(F, g) + Q_{\bar R}(F, g) \nn
\\
& = \int_{\R^3} \int_{\Ss^2} b(\cos \theta) \Phi_{R} (F'_\ast g' - F_\ast g) \dsigma \dv_\ast
+ \int_{\R^3} \int_{\Ss^2} b(\cos \theta) \Phi_{\bar R} (F'_\ast g' - F_\ast g) \dsigma \dv_\ast \,. 
\end{align*}
It follows from \cite[Proposition 2.1]{AMUXY2011} (see also \cite[Prop.6.11]{MS2016}) 
that 
\begin{align*}
\left|\Big(Q_R(F,g), h\Big)_{L^2}\right| &\lesssim \|F\|_{L^2} \|g\|_{H^s} \|h\|_{H^s}\\
&\lesssim \|F\|_{L^2} \big(J_1^{\Phi_0}(g) 
+ \|g\|^2_{L^2}\big) ^{1/2}\big(J_1^{\Phi_0}( h )
   + \|h\|^2_{L^2}\big  )^{1/2}.
\end{align*}
In view of \eqref{remember-obvious}, it suffices to consider only 
$\Big(Q_{\bar R}(F,g), h\Big)_{L^2}$,  whose estimation is the almost same as in the proof of
\cite[Lemma 3.2]{AMUXY10-2013}. 
We write
\begin{align*}
A&= (Q_{\bar R}(F,g),h)
= \iiint  b(\cos \theta) \Phi_{\bar R} F_* g(v)(h(v')-h(v))dv dv_* d \sigma \\
&= \iiint  b(\cos \theta) \Phi_{\bar R} F_* (g(v) - g( v')) (h( v') -h(v) ) dv dv_* d \sigma\\
&\qquad + \iiint  b(\cos \theta) \Phi_{\bar R} F_* g( v') (h(v') -h(w) ) dv dv_* d \sigma\\
&\qquad + \iiint  b(\cos \theta) \Phi_{\bar R}F_* g(v) (h(w) -h(v) ) dv dv_* d \sigma\\
&\qquad + \iiint  b(\cos \theta) \Phi_{\bar R} F_* (g( v') - g(v))(h(w) -h(v) ) dv dv_* d \sigma\\
& \Denote S + M_0 + R_1 + R_2,
\end{align*}
where (see also Figure 1 in \cite{AMUXY10-2013} given below)
\begin{align*}
&w = v_* + \Big(\cos^2 \frac{\theta}{2}\Big) (v-v_*) = \frac{v'+v_*}{2} + \frac{|v'-v_*|}{2} \omega, \qquad
\cos \theta = \frac{v-v_*}{|v-v_*|} \cdot \sigma = \frac{v'-v_*}{|v'-v_*|} \cdot \omega.
\end{align*}

\centerline{
\includegraphics[width=8cm,clip]{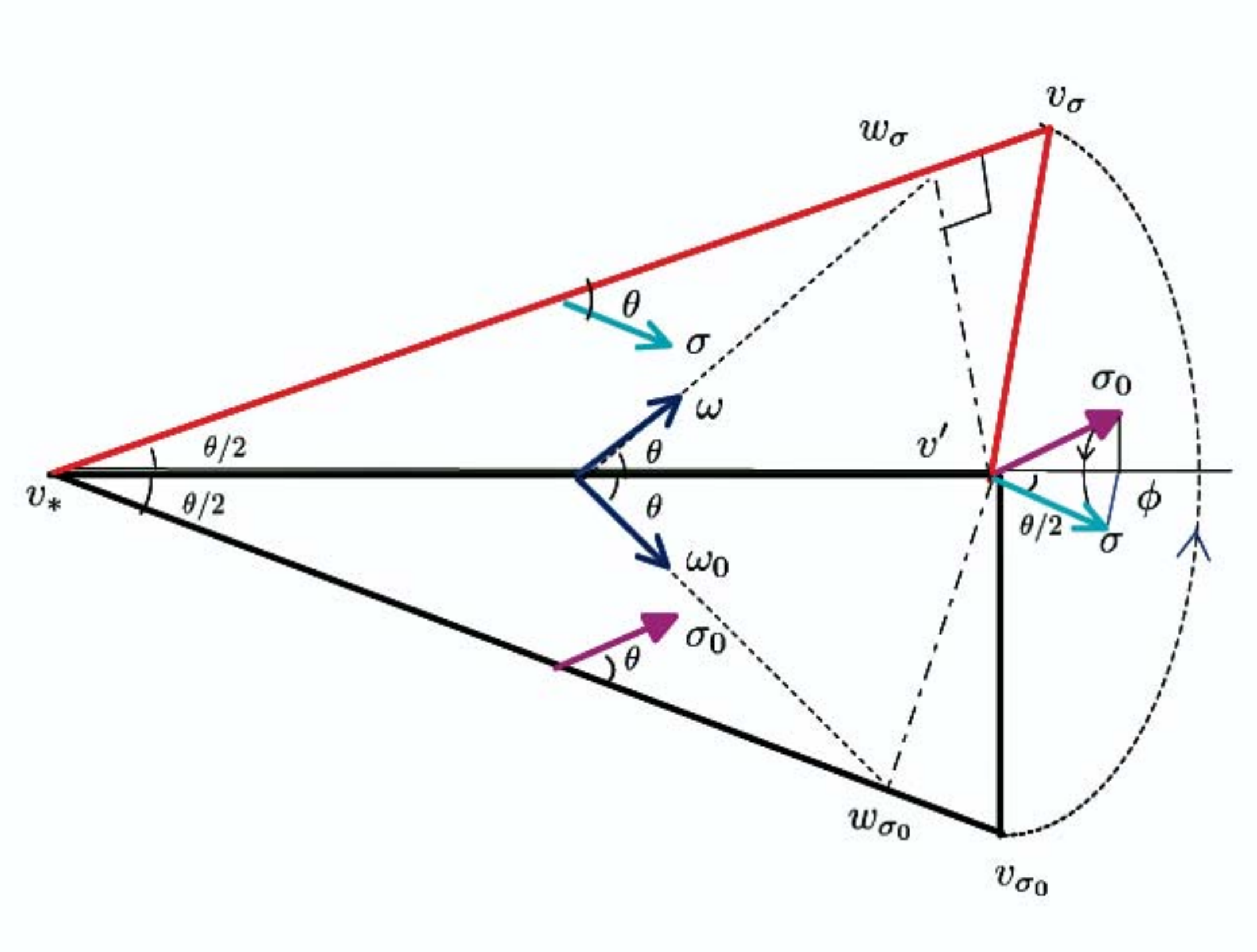}}
\centerline{Figure 1. $\sigma = (\theta/2, \phi), \sigma_0 = (\theta/2,0)$,
$d\sigma = \displaystyle \sin \frac{\theta}{2}d \left(\frac{\theta}{2}\right)d \phi$,
$d\omega = \sin \theta d \theta d \phi$}
\vspace{3mm}
{Using  the Cauchy-Schwarz} inequality, we have
\begin{align*}
|S|^2 &\le \iiint B F_*(g'-g)^2 d\sigma dv  dv_* \times \iiint B F_*(h'-h)^2d\sigma dv  dv_* \\
&\lesssim\|F\|^2_{L^1_{\max\{2, 2s+\gamma\}}}
\Big(  J_1^{\Phi_0}( \langle v \rangle^{\gamma/2} g) + \|g\|^2_{L^2_{\gamma/2}}\Big)
\Big(  J_1^{\Phi_0}( \langle v \rangle^{\gamma/2} h) + \|h\|^2_{L^2_{\gamma/2}}\Big)
\end{align*}
by means of Lemma \ref{upper-F-g-g}.

{For $M_0$, we write}
\begin{align*}
M_0 &= \int \Phi_{\bar R}(|v'-v_*|) b\Big(\frac{v-v_*}{|v-v_*|}\cdot \sigma\Big)
F_* g(v') (h( v') -h(w) ) dv dv_* d \sigma\\
& + \int \Big( \Phi_{\bar R}(|v-v_*|)- \Phi_{\bar R}(| v'-v_*|)\Big)
b(\frac{v-v_*}{|v-v_*|}\cdot \sigma)
F_* g( v') (h( v') -h(w) ) dv dv_* d \sigma\\
&= M_1 + R_3.
\end{align*}
Since
\[
 \Phi_{\bar R}(|v-v_*|)- \Phi_{\bar R}(| v'-v_*|)
= \Phi_{\bar R}(|v-v_*|)- \Phi_{\bar R}(|v-v_*|\cos (\theta/2)),\]
it is not difficult to see {that}
\[
|R_3| \lesssim \|F\|_{L^1_{\gamma}} \|g\|_{L^2_{\gamma/2}} \|h\|_{L^2_{\gamma/2}}\,.
\]

{For each fixed $(\sigma , v_*)$}, we perform the change of variables
$v \rightarrow v'$, as in \cite{ADVW2000}.  { Recall that}
\[
dv = \left|\frac{dv}{d v'}\right| d v'
= \frac{2^{3-1}}{\Big(\frac{ v'-v_*}{| v'-v_*|}\cdot \sigma \Big)^2} d v'\,,
\]
and we let  the inverse transformation {be} $v' \rightarrow v= v_\sigma(v', v_*)$.  Hence
\begin{align*}
M_1 &= \int_{\RR^3_{v_*}}  F(v_*)\Big( \int_{\mathbb{S}^2}\Big( \int_{\RR^3_v}
 \Phi_{\bar R}(| v'-v_*|)b\Big(\frac{ v-v_*}{| v-v_*|}\cdot \sigma \Big)\\
 &\quad \times g( v') \big(h(v') -h(v_* + \frac{1}{2}\Big(1+ \frac{ v-v_*}{| v-v_*|}\cdot \sigma \Big)(v-v_*)\big)dv\Big)
d\sigma \Big) dv_* \\
&= \int_{\RR^3_{v_*}}  F(v_*)\Big( \int_{\mathbb{S}^2\times \RR^3_{ v'}\cap
\{ \frac{ v'-v_*}{| v'-v_*|}\cdot \sigma \ge 1/\sqrt 2\}}
 \Phi_{\bar R}(| v'-v_*|)b(2 \Big(\frac{ v'-v_*}{| v'-v_*|}\cdot \sigma \Big)^2 -1)\\
 & \quad \times g(v') \big\{h(v') -h(
v_* + \Big(\frac{ v'-v_*}{| v'-v_*|}\cdot \sigma \Big)^2
(v_\sigma-v_*))\big\} 
\frac{2^{3-1}}{\Big(\frac{ v'-v_*}{| v'-v_*|}\cdot \sigma \Big)^2} d v'
d\sigma \Big) dv_*.
\end{align*}
Take polar {coordinates} $\sigma = (\vartheta, \phi) \in [0,\pi]\times [0,2 \pi]$ with {pole} $v' -v_*$.
Then  we have
\begin{align*}
& \quad 
\int_{\mathbb{S}^2 \cap \{ \frac{ v'-v_*}{| v'-v_*|}\cdot \sigma \ge 1/\sqrt 2\}}
\cdots -h(
v_* + \Big(\frac{ v'-v_*}{| v'-v_*|}\cdot \sigma \Big)^2
(v_\sigma-v_*))\big\}
\frac{2^{3-1}}{\Big(\frac{ v'-v_*}{| v'-v_*|}\cdot \sigma \Big)^2} d\sigma\\
&=\int_0^{2\pi} d \phi  \int_0^{\pi/4}\cdots
-h(
v_* + \cos^2 \vartheta
(v_{(\vartheta, \phi)} -v_*))\big\}
\frac{2^{3-1} \sin \vartheta d \vartheta}
{\cos^2 \vartheta}   \\
&= \int_{\mathbb{S}^2 \cap \{ \frac{ v'-v_*}{| v'-v_*|}\cdot \omega \ge 0\}}\cdots
-h(
\frac{v'+v_*}{2}+ \frac{|v'-v_*|}{2}\omega
)\big\}
\frac{d \omega}{\cos^3 (\theta/2)},
\end{align*}
because
$\vartheta = \theta/2$ and $\displaystyle \cos \theta =
\frac{ v' -v_*}{| v' -v_*|}\cdot \omega$ (see Figure 1).
Therefore, writing $v$ and $\sigma$ instead of $v'$ and $\omega$, we have
\begin{align*}
M_1 
= \int b \Phi_{\bar R} F_* g(v)(h(v) -h(v'))\frac{d \sigma }{\cos^3 (\theta/2)}dv dv_*
= - A + R_4,
\end{align*}
where
\[
R_4 = \int b \Phi_{\bar R} F_* g(v)(h(v) -h(v'))\Big(\frac{1}{\cos^3 (\theta/2)} -1 \Big)dv dv_*
d\sigma.
\]
It is easy to {check that}
\[
|R_4| \lesssim \|F\|_{L^1_{\gamma}} \|g\|_{L^2_{\gamma/2}} \|h\|_{L^2_{\gamma/2}}\,.
\]
Now  we concentrate {on} the term
\[
R_1 = \int b \Phi_{\bar R} F_* g(v) (h(w) -h(v) ) dv dv_* d \sigma,
\]
where
\[
w = v_* + \frac {1}{2}\Big(1+ \frac{(v-v_*)}{|v-v_*|}\cdot \sigma \Big)(v-v_*).
\]
Note that
\[
dw = (\cos ^2 (\theta/2))^3 dv, \enskip \enskip |w-v| = |v-v_*|\sin^2(\theta/2).
\]
Take the {dyadic} decomposition 
\[
B
= \sum_{\ell=1}^\infty   |v-v_*|^\gamma \varphi(2^{-\ell} |v-v_*|))  b( \cos \theta )  .
\]
We choose $\psi$ such that $\varphi \subset \subset \psi$.
Writing $\varphi_\ell(z) =
\varphi(2^{-\ell}z)$ and
$\psi_\ell(z) =  (2^{-\ell}|z|)^\gamma \psi(2^{-\ell}z)$, we obtain
by the change of variables $v \rightarrow v_* +z$, 
\begin{align*}
\tilde R_1(v_*)  &= \sum_{2^\ell \ge R} 2^{\gamma \ell}
\int_0^{\pi/2} b(\cos \theta)\sin \theta
\int_{\RR_z^3}
 \varphi_\ell(z)
\Big(h(v_*+ \cos^2 \frac{\theta}{2} z)- h(v_*+z) \Big) \psi_\ell(z) g(v_*+z) dz d\theta\\
&  \Denote  \sum_\ell 2^{\gamma \ell} J_\ell.
\end{align*}
{We} divide
\begin{align*}
J_\ell &=
\int_0^{\pi/2} b(\cos \theta)\sin \theta
\int_{\RR_z^3}
 \Big(\varphi_\ell(z)  - \varphi_\ell(\cos^2 \frac{\theta}{2}z) \Big)
h(v_*+ \cos^2 \frac{\theta}{2} z) \psi_\ell(z) g(v_*+z) dz d\theta\\
& \quad +
\int_0^{\pi/2} b(\cos \theta)\sin \theta
\int_{\RR_z^3}
 \Big(\varphi_\ell( \cos^2 \frac{\theta}{2}   z)
h(v_*+ \cos^2 \frac{\theta}{2} z)- \varphi_\ell(z)h(v_*+z) \Big)\\
&\qquad \qquad \qquad \qquad \qquad \qquad \qquad \times  \psi_\ell(z) g(v_*+z) dz d\theta\\
&\Denote  J_\ell^{(1)} + J_\ell^{(2)}.
\end{align*}
{Let} $H_\ell(z;v_*)= \varphi_\ell(z)h(v_*+z), G_\ell(z;v_*) =\psi_\ell(z) g(v_*+z)$ and
denote the Fourier transforms of $G_\ell, H_\ell$ with respect to $z$ by
$\hat G_\ell(\xi ;v_*), \hat H_\ell(\xi;v_*)$, respectively.
Then {Plancherel} formula gives
\begin{align*}
J_\ell^{(2)}& = \int_0^{\pi/2} b(\cos \theta)\sin \theta
\Big(\int_{\RR_\xi^3}
 \Big( \frac{1}{(\cos^2 \frac{\theta}{2})^3} \hat H_\ell(\frac{\xi}{
\cos^2 \frac{\theta}{2}};v_*)
- \hat H_\ell(\xi;v_*)\Big)\overline{\hat G_\ell(\xi;v_*)}d\xi\Big) d\theta\\
&=\int_0^{\pi/2} b(\cos \theta)\sin \theta\Big( \frac{1}{(\cos^2 \frac{\theta}{2})^3}-1\Big)
\Big(\hat H_\ell(\frac{\xi}{
\cos^2 \frac{\theta}{2}};v_*)\overline{\hat G_\ell(\xi;v_*)}d\xi\Big) d\theta\\
&\quad + \int_0^{\pi/2} b(\cos \theta)\sin \theta
\Big(\int_{\RR_\xi^3}
 \Big( \hat H_\ell(\frac{\xi}{
\cos^2 \frac{\theta}{2}};v_*)
- \hat H_\ell(\xi;v_*)\Big)\overline{\hat G_\ell(\xi;v_*)}d\xi\Big) d\theta\\
&= J_\ell^{(2,1)}+J_\ell^{(2,2)}\,.
\end{align*}
It is easy to see {that}
\begin{align}\label{first}
|J_\ell^{(2,1)}| \lesssim \|\varphi_\ell(\cdot -v_*) h\|_{L^2}\|\psi_\ell(\cdot -v_*) g\|_{L^2}\,.
\end{align}
{A similar} estimate {is also true} for $J_\ell^{(1)}$, by replacing $\varphi_\ell$ by $\tilde \varphi_\ell$
which {is} defined from a suitable $\tilde \varphi$ satisfying $\varphi \subset \subset \tilde \varphi$.
Write
\begin{align*}
J_\ell^{(2,2)}&=
\int_{\RR_\xi^3} \overline{\hat G_\ell(\xi;v_*)}\Big(
\int_0^{2^{-\ell/2}\la \xi\ra^{-{1/2}}}  b(\cos \theta)\sin \theta \tan^2 \frac{\theta}{2} 
\int_0^1 \xi \cdot
\nabla_{\xi}\hat H_\ell(\xi+ \tau \tan^2 \frac{\theta}{2}  \xi  ;v_*)
d \theta d\tau \Big)d\xi\\
&\quad+ \int_{\RR_\xi^3} \overline{\hat G_\ell(\xi;v_*)}
\Big(
\int_{2^{-\ell/2}\la \xi\ra^{-{1/2}}}^{\pi/2}  b(\cos \theta)\sin \theta
\Big( \hat H_\ell(\frac{\xi}{
\cos^2 \frac{\theta}{2}};v_*)
- \hat H_\ell(\xi;v_*)\Big) d\theta \Big) d\xi \\
& = A_\ell + B_\ell.
\end{align*}
By {Cauchy-Schwarz} inequality, we have
\begin{align*}
|A_\ell|^2&\lesssim \int_{\RR_\xi^3} \la \xi \ra^{1\mp s} |\hat G_\ell(\xi;v_*)|^2
\int_0^{2^{-\ell/2}\la \xi\ra^{-{1/2}}}  b(\cos \theta)\sin \theta \tan^2 \frac{\theta}{2}d \theta d\xi\\
&\quad \times
\int_{\RR_\xi^3} \la \xi \ra^{1 \pm s} |\nabla_\xi \hat H_\ell(\xi;v_*)|^2
\int_0^{2^{-\ell/2}\la \xi\ra^{-{1/2}}}  b(\cos \theta)\sin \theta \tan^2 \frac{\theta}{2}d \theta d\xi\,,
\end{align*}
where, in the second factor, we have used the change of variable
\[ \xi \rightarrow (1+ \tau \tan^2 \frac{\theta}{2} ) \xi \]
after exchanging $d\theta d \xi$ by $d\xi d\theta$.
Therefore, we get
\begin{align}\label{main-1}
|A_\ell|^2&\lesssim \left\{ \begin{array}{l}
\displaystyle \Big( \int_{\RR_\xi^3} 2^{2s \ell}|\hat G_\ell(\xi;v_*)|^2 d\xi \Big)
\Big(\int_{\RR_\xi^3} |\la \xi\ra^{s} \nabla_\xi (2^{-\ell}\hat H_\ell)(\xi;v_*)|^2 d\xi\Big), \\\\
\displaystyle \Big( \int_{\RR_\xi^3}|\la \xi\ra^s \hat G_\ell(\xi;v_*)|^2 d\xi \Big)
\Big(\int_{\RR_\xi^3}  2^{2s \ell}|\nabla_\xi (2^{-\ell}\hat H_\ell)(\xi;v_*)|^2 d\xi\Big).
\end{array} \right.
\end{align}
On the other hand,  we have
\begin{align}\label{main-2}
|B_\ell|^2&\lesssim  \left\{ \begin{array}{l}
\displaystyle \Big( \int_{\RR_\xi^3} 2^{2s \ell}|\hat G_\ell(\xi;v_*)|^2 d\xi \Big)
\displaystyle \Big(\int_{\RR_\xi^3} |\la \xi\ra^{s}\hat H_\ell(\xi;v_*)|^2 d\xi\Big),\\\\
\displaystyle \Big( \int_{\RR_\xi^3} |\la \xi\ra^{s}\hat G_\ell(\xi;v_*)|^2 d\xi \Big)
\displaystyle \Big(\int_{\RR_\xi^3} 2^{2s \ell}|\hat H_\ell(\xi;v_*)|^2 d\xi\Big),
\end{array} \right.
\end{align}
Summing up {the} above estimates of the first case
we obtain
\begin{align*}
\sum 2^{\gamma \ell} |J_\ell| 
&\lesssim \sum \la v_* \ra^{\gamma+s}\|\tilde \varphi_\ell \la v \ra^{s+ \gamma/2} g(v)\|_{L^2}
\|\tilde \varphi_\ell \la D \ra^{s} h(v)\|_{L^2}\\
&\lesssim \la v_* \ra^{\gamma+s}
\Big( \sum \|\tilde \varphi_\ell \la v \ra^{s +\gamma/2}g(v)\|^2_{L^2} \Big)^{1/2} 
\Big( \sum \|\tilde \varphi_\ell \la D \ra^{s} h(v)\|^2_{L^2} \Big)^{1/2}\\
& \lesssim  \la v_* \ra^{\gamma+s} \|g\|_{L^2_{s+\gamma/2}}\|h\|_{H_{\gamma/2}^{s}},
\end{align*}
because
\[
\sum (\tilde \varphi_\ell F,\tilde \varphi_\ell F)_{L^2} = ((\sum (\tilde \varphi_\ell)^2 F, F)_{L^2}
\lesssim \|F\|_{L^2}^2.
\]
Here it should be noted that the commutator
\[
\|[ \la D \ra^s, \varphi_\ell ] g\|_{L^2} \lesssim 2^{-\ell}\|g\|_{L^2}
\]
is harmless to {the above summation process}.
Since $\|h\|^2_{H^s_{\gamma/2}} \lesssim J_1^{\Phi_0}(\langle  v \rangle^{ \gamma/2} h)
+ \|h\|^2_{L^2_{\gamma/2}}$, 
we obtain
\[
|R_1| \lesssim \|F\|_{L^1_{\gamma+s} }\|g\|_{L^2_{s+\gamma/2}} 
\Big(J_1^{\Phi_0}(\langle  v \rangle^{ \gamma/2} h)
+ \|h\|^2_{L^2_{\gamma/2}}\Big)^{1/2}.
 \]
Another case for $R_1$ is now obvious. 

As for $R_2$, it follows from {Cauchy-Schwarz} inequality that
\[
|R_2|^2 \le  \Big(\int B F_* (g( v') - g(v))^2dv dv_* d \sigma\Big)
\Big(\int b \Phi_{\bar R} F_* (h(w) -h(v) )^2 dv dv_* d \sigma\Big).
\]
The first factor is estimated by using Lemma \ref{upper-F-g-g}.
Note that the second factor is estimated above from
\[
2\Big( \int b \Phi_{\bar R} F_* (h(w) -h(v') )^2 dv dv_* d \sigma
+ \int b \Phi_{\bar R} F_* (h(v') -h(v) )^2 dv dv_* d \sigma\Big).
\]
The first term can be handled in the same way as for $M_0$, regarding $(w,v')$  as $(v',v)$,
and it is estimated by 
$J_1^{\Phi_0}(\langle  v \rangle^{ \gamma/2} h)
+ \|h\|^2_{L^2_{\gamma/2}}$ up to  a constant factor. And this completes the proof of the lemma.
\end{proof}

%
%


\bigskip
\noindent {\bf Acknowledgements:}
The research of the first author was supported by the
Conselho Nacional de Desenvolvimento
Cientifico e Tecnologico. The research of the second author was supported by Grant-in-Aid for Scientific Research No. 17K05318. The research of the third author was supported by
NSERC Discovery Individual Grant R611626 . And the research of the fourth author was supported by General Research Fund of Hong Kong 11303614.

\bibliographystyle{amsxport}
\bibliography{kinetic}
\end{document}